\documentclass[oneside,english]{amsart}
\usepackage{lmodern}
\usepackage[T1]{fontenc}
\usepackage[latin9]{inputenc}
\usepackage{geometry}
\usepackage{color}
\usepackage{units}
\usepackage{mathrsfs}
\usepackage{amsbsy}
\usepackage{amstext}
\usepackage{amsthm}
\usepackage{amssymb}
\usepackage{stmaryrd}

\makeatletter
\numberwithin{equation}{section}
\numberwithin{figure}{section}
\theoremstyle{plain}
\newtheorem{thm}{\protect\theoremname}
  \theoremstyle{plain}
  \newtheorem{lem}[thm]{\protect\lemmaname}
  \theoremstyle{remark}
  \newtheorem{rem}[thm]{\protect\remarkname}
  \theoremstyle{plain}
  \newtheorem{cor}[thm]{\protect\corollaryname}
  \theoremstyle{plain}
  \newtheorem{prop}[thm]{\protect\propositionname}

\makeatother

\usepackage{babel}
  \providecommand{\corollaryname}{Corollary}
  \providecommand{\lemmaname}{Lemma}
  \providecommand{\propositionname}{Proposition}
  \providecommand{\remarkname}{Remark}
\providecommand{\theoremname}{Theorem}

\begin{document}
\global\long\def\AS{Assumption~M}
\global\long\def\CDTN{Condition~M}
\global\long\def\bs{\boldsymbol{\sigma}}
\global\long\def\bx{\mathbf{x}}
\global\long\def\by{\mathbf{y}}
\global\long\def\grad{\nabla_{{\rm sp}}}
\global\long\def\Hess{\nabla_{{\rm sp}}^{2}}
\global\long\def\ddq{\frac{d}{dR}}
\global\long\def\qs{q_{\star}}
\global\long\def\qss{q_{\star\star}}
\global\long\def\Es{E_{\star}}
\global\long\def\ds{d_{\star}}
\global\long\def\Cs{\mathscr{C}_{\star}}
\global\long\def\nh{\boldsymbol{\hat{\mathbf{n}}}}
\global\long\def\BN{\mathbb{B}^{N}(\sqrt{N})}
\global\long\def\SN{\mathbb{S}^{N-1}(\sqrt{N})}
\global\long\def\SNm{\mathbb{S}^{N-2}(\sqrt{N-1})}
\global\long\def\nd{\nu^{(\delta)}}
\global\long\def\nz{\nu^{(0)}}
\global\long\def\cls{c_{LS}}
\global\long\def\qls{q_{LS}}
\global\long\def\dls{\delta_{LS}}
\global\long\def\Els{E_{LS}}
\newcommand{\GS}{\mbox{\rm GS}}
\newcommand{\E}{\mathbb{E}}
\newcommand{\R}{\mathbb{R}}
\newcommand{\corO}{}
\newcommand{\corOU}{}
\newcommand{\corE}{}
\title[Mixed spherical models]{\corO{Geometry and temperature chaos in mixed spherical spin glasses at low temperature - the perturbative regime}}

\author{G\'{e}rard Ben Arous, Eliran Subag and Ofer Zeitouni}
\begin{abstract}
We study the Gibbs measure of mixed spherical $p$-spin 
glass models at low temperature,
in (part of) the 1-RSB regime, including, in particular, models close to pure in an appropriate sense. We show that the Gibbs measure concentrates on spherical bands around deep critical points of the (extended) Hamiltonian restricted to the sphere of radius  $\sqrt N\qs$, where $\qs^2$ is the rightmost point in the support of the overlap distribution.
We also show that the relevant critical points are pairwise orthogonal for two different low temperatures.
This allows us to explain why temperature chaos occurs for those models, in contrast to the pure spherical models.
\tableofcontents{}
\end{abstract}

\maketitle

\section{Introduction}
\label{sec-intro}

We study in this paper the Gibbs measure of the mixed spherical $p$-spin 
glass model in low temperature,
in (part of) the 1-RSB regime. The model, which is a variant
of that introduced in the seminal paper \cite{SK75},
is defined as follows. Let $\SN$ 
denote the (Euclidean) sphere of radius 
$\sqrt{N}$ in dimension $N$.
Let $J_{i_{1},...,i_{p}}^{(p)}$ denote i.i.d real
standard Gaussian random variables, and let $\{\gamma_p\}_{p\geq 2}$ be 
a sequence of non-negative deterministic constants. The \textit{Hamiltonian}
is defined as
\begin{equation}
H_{N}\left(\bs\right)=H_{N,\nu}\left(\bs\right):=\sum_{p=2}^{\infty}\frac{\gamma_{p}}{N^{\left(p-1\right)/2}}\sum_{i_{1},...,i_{p}=1}^{N}J_{i_{1},...,i_{p}}^{(p)}\sigma_{i_{1}}\cdots\sigma_{i_{p}},
\quad\boldsymbol{\sigma}=\left(\sigma_{1},...,\sigma_{N}\right)\in
\SN,\label{eq:Hamiltonian}
\end{equation}
with \textit{ground state} 
\begin{equation}
\label{eq-ground}
\GS_N=\min_{\boldsymbol{\sigma}\in \SN }
H_N(\boldsymbol{\sigma}),
\end{equation}
and the associated Gibbs measure 
is the random probability measure on 
$\SN$ given by
\begin{equation}
\label{eq-gibbsmeas}
\frac{dG_{\corE{N,\beta}}}{d\boldsymbol{\sigma}}
\left(\boldsymbol{\sigma}\right):=\frac{1}{Z_{\corO{N,\beta}}}e^{-\beta H_{N}\left(\boldsymbol{\sigma}\right)},
\end{equation}
where $Z_{\corO{N,\beta}}$ is a normalization constant and $d\boldsymbol{\sigma}$
denotes normalized Haar measure on $\SN$. 

Let  
\begin{equation}
\label{eq-nu}
\nu(x)=\sum_{p= 2}^\infty \gamma_{p}^{2}x^{p}.
\end{equation}
We refer to the model as \textit{pure} if 
$\nu(x)$ is a monomial, and \textit{mixed} otherwise.  
Throughout the paper, we assume that $\gamma_p$ decays exponentially, so that 
$\nu(\cdot )$ is defined on an open interval that includes $(0,1]$. 
\corO{Following \cite{ABA2}, we normalize $\nu$ by setting $\nu(1)=1$.}

For many models of spin glasses including the spherical models,
properties of the Gibbs measure in terms of their overlaps (i.e., the
distribution of the distance between two or more points sampled
independently from the Gibbs measure,) which serve here as order parameter,
are available through 
a version of the Parisi formula, see  \cite{Talag} and 
\cite{Parisi,GuerraBound,Talag2}, and through the
ultrametricity properties of $G_{\beta,N}$, see \cite{ultramet}. 
We
refer to \cite{TalagrandBookII,PanchenkoBook} for comprehensive
introductions to the mathematical theory of spin glasses.

Our goal in this paper is different: we aim at developing a 
\textit{geometric} description of the Gibbs measure, at low temperature. For
the pure model, this was achieved by one of us in \cite{geometryGibbs},
where it was shown that at low temperature $\beta\gg 1$,
$G_{N,\beta}$ concentrates 
in thin bands (or rings) centered at 
the locations of the deepest local minima of $H_N$\footnote{
	The value of the global minimum can be inferred from Parisi's formula. It
	was 
	evaluated 
	via a study of the limiting expected complexity
	in \cite{A-BA-C} (pure) and \cite{ABA2} (mixed),
	and complemented (for the pure $p$-spin spherical model) by a study of
	second moments in \cite{2nd}.}. 

Our  results apply to 
spherical mixed models that satisfy a
certain decoupling condition (\CDTN,  defined below)
related to critical points. 
As we shall see, \CDTN\ dictates 1-RSB at very low temperature.
A particularly important class of models that satisfy our conditions
are perturbations of pure $p$-spin spherical models, see
Section \ref{sec:CDTNandCt} below. 
We show that the geometric description of 
the support of the Gibbs measure in low temperature, 
developed in  \cite{geometryGibbs} for the pure $p$-spin model,
needs to be modified in the mixed case. 
The Gibbs measure in the mixed case is supported on
thin bands that are centered 
at critical points not of the Hamiltonian 
\eqref{eq:Hamiltonian}
but rather
of  its extension to
the sphere 
$\mathbb{S}^{N-1}(\qs\sqrt{N})$, for appropriate
$\qs=\qs(\beta)<1$, see Theorem \ref{thm:Geometry}.
(These centers are close to critical points of the Hamiltonian 
\eqref{eq:Hamiltonian} with
low, but not minimal, 
energy.)
As a byproduct,
we are able to show that states of asymptotic positive mass are
pure and nearly orthogonal, see Theorem \ref{thm:geometry1}, and
explain why those models
exhibit chaos in temperature while the pure models
do not, see Theorem \ref{thm:Chaos} below and Section \ref{sec:PfsGeometry}.

\subsection{Main results}
We turn to a detailed description of our results. 
Introduce the function
\begin{equation}
\label{eq-G}
G(\nu)  =\log\frac{\nu''(1)}{\nu'(1)}-\frac{(\nu''(1)+\nu'(1))(\nu''(1)\nu(1)+\nu'(1)^{2}-\nu'(1)\nu(1))}{\nu''(1)\nu'(1)^{2}}.
\end{equation}
Following \cite{ABA2},
we call the model   \textit{pure-like}, \textit{critical} or \textit{full}
according to whether $G(\nu)>0$,
$=0$ or $<0$, respectively. In the sequel, we deal exclusively
with pure-like models.

Next, consider the limiting expected complexity at level $u$ and 
radial derivative $x$, defined as
\begin{eqnarray}
\label{eq-gerardcom1}
\Theta_{\nu,1}(u,x)&=&\lim_{\delta\to 0} \lim_{N\to\infty}\frac1N
\log\left(\E \# \left\{\mbox{\rm critical points $\bs\in \SN$  with}
\nonumber \right.\right. \\
&& \quad \quad \quad \quad \quad \quad \quad \quad \left.\left.
\mbox{\rm $|\frac1N  H_N(\bs)-u|\leq \delta$ 
	and
	$| \frac{1}{\sqrt{N}}\frac{d}{dR}H_N(\bs)-x|\leq \delta$
}\right\}\right),
\end{eqnarray}
where $\frac{d}{dR} H_N(\bs):=\frac{1}{\sqrt{N}} \frac{d}{dq} H_N(q\bs)|_{q=1}$ is the radial 
derivative of $H_N(\bs)$ at $\bs\in \SN$.
Recapitulating one of the
main results in \cite{ABA2}, one has 
that the limit in \eqref{eq-gerardcom1} exists and is explicit, see
Theorem \ref{thm:1stmoment} below.
Let
\begin{equation}
\label{eq:E0}
 -E_{0}:  =-E_{0}(\nu)=\min\Big\{ E:\,\sup_{x\in \R}\Theta_{\nu,1}\left(E,x\right)=0\Big\}.
\end{equation}
The level $-E_{0}N$ is the threshold beyond which
the number of critical points decays exponentially in expectation, see
\cite{ABA2}.\footnote{The definition of $E_{0}$ given in this paper
	coincides with the definition in \cite{ABA2}
	for pure-like or critical models, but not for full models. See Sections
	4 and 5 of \cite{ABA2}.}  
In particular, by Markov's inequality,  for large $N$, 
\begin{equation}
\frac{1}{N}\min_{\bs\in \SN}H_{N}(\bs)\geq-E_{0}-o(1),
\quad
\mbox{\rm with high  probability}.\label{eq:GSlb}
\end{equation}
As we show below in  Lemma \ref{lem:x0},
there exists a unique maximizer 
\begin{equation}
-x_{0}:=-x_{0}(\nu)=\arg\max_{x\in\mathbb{R}}\Theta_{\nu,1}\left(-E_{0},x\right).\label{eq:x0}
\end{equation}

Define the \textit{overlap} between $\bs,\bs'\in \R^N$ as
\begin{equation}
\label{eq-defoverlap}
R(\bs,\bs'):=\langle\bs,\bs'\rangle/\|\bs\|\|\bs'\|,
\end{equation}
where $\langle \cdot,\cdot\rangle$ denotes the standard inner product in
$\R^N$.
We next introduce a function $\Psi_{\nu,1,1}(r,u,x)$ that will play a 
crucial role in second moment 
computations, and which we refer to as the pair complexity at level $u$, radial derivative $x$ and 
overlap $r$: 
\begin{eqnarray}
\Psi_{\nu,1,1}(r,u,x)&=&\lim_{\delta\to 0} \lim_{N\to\infty}\frac1N
\log\Big(\E \# \Big\{\mbox{\rm pairs of critical points $\bs,\bs'$  with
	$|R(\bs,\bs')-r|\leq \delta$, } \nonumber\\
&&
\Big|\frac1N H_N(\bs)-u\Big|\leq \delta,
\Big|\frac1N H_N(\bs')-u\Big|\leq \delta,
\Big|\frac{1}{\sqrt{N}} \frac{d}{dR}H_N(\bs)-x\Big|\leq \delta,
\Big| \frac{1}{\sqrt{N}}\frac{d}{dR}H_N(\bs')-x\Big|\leq \delta
\Big\}\Big).
\end{eqnarray}
An explicit expression for the function $\Psi_{\nu,1,1}$ appears in
\eqref{eq:Psishort}.
%
Finally set,
for $r\in(-1,1)$, 
\begin{equation}
\Psi_{\nu}^{0}(r):=\Psi_{\nu,1,1}(r,-E_{0},-x_{0}),\label{eq:0204-02}
\end{equation}
and for $r=\pm1$, let $\Psi_{\nu}^{0}(\pm1)$ be the corresponding
$r\to\pm1$ limit.\footnote{The $r\to1$ limit always exists and finite, see the proof of Lemma
	\ref{lem:conccentration} below. The $r\to-1$ limit is finite if and only
	if $\nu$ is even. Otherwise, it is $-\infty$.} 
The function $\Psi_{\nu}^{0}(\cdot)$, 
which is determined
by $\nu$,
is a continuous function
from $[-1,1]$ to $\mathbb{R}\cup\{-\infty\}$, and determines the pair 
complexity at level $-NE_0$ and radial derivative $x$ as function of the overlap. 

All our results will be under the following assumption.

\noindent\textbf{\CDTN.} Assume that $\nu$ is mixed and
pure-like, that $\frac{d^{2}}{dr^{2}}\Psi_{\nu}^{0}(0)<0$,
and that the maximum of $\Psi_{\nu}^{0}(r)$ on the interval $[-1,1]$
is obtained uniquely at $r=0$.

\CDTN\ implies that the pair complexity at the relevant levels $(-E_0,-x_0)$
is maximal at zero overlap. We will see below, see Theorem \ref{thm:geometry1},
that \CDTN\  implies in particular that the model 
belongs to the so-called 1-RSB class. We will also see that \CDTN\ is an 
open condition, and that small perturbations
of pure models satisfy \CDTN, see Proposition \ref{prop:CDTNp}.

We are ready to state our first result. Recall the ground state 
$\GS_N$, see \eqref{eq-ground}.
\begin{thm}
	\label{theo-ground}
	Assume \CDTN. Then,
	\begin{equation}
	\label{eq-groundeq}
	\lim_{N\to\infty} \frac{\GS_N}{N}=-E_0, \quad {\rm
		a.s.}.
	\end{equation}
\end{thm}
Thus, expected 
complexity determines the ground state.
Our proof of Theorem \ref{theo-ground} relies on a two moments analysis
presented in Sections \ref{sec:CritPts}, \ref{sec:Moments} and \ref{sec:Matching}, and avoids the use of  Parisi's formula.

We next turn to the description of the support of the 
Gibbs measure $G_{\beta,N}$ 
for large $\beta$.
Let $F_{1},...,F_{N-1}$ be a piecewise smooth frame field on $\SN$,
and extend it to $\bx\in \mathbb{R}^{N}\setminus\{0\}$ by setting 
$F_{i}(\bx)=F_{i}(\sqrt{N} \bx/\|\bx\|)$
(under the usual identification of
tangent spaces with affine subspaces of $\mathbb{R}^{N}$). Denote
\begin{align}
\grad H_{N}\left(\boldsymbol{\sigma}\right) & :=\left\{ F_{i}H_{N}\left(\boldsymbol{\sigma}\right)\right\} _{i\leq N-1},\nonumber \\
\Hess H_{N}\left(\boldsymbol{\sigma}\right) & :=\left\{ F_{i}F_{j}H_{N}\left(\boldsymbol{\sigma}\right)\right\} _{i,j\leq N-1},\label{eq:derivatives}
\end{align}
We shall call the points $\boldsymbol{\sigma}\in\mathbb{S}^{N-1}(q\sqrt{N})$
with $\grad H_{N}\left(\boldsymbol{\sigma}\right)=0$, \textit{$q$-critical
	points}. For any set $B\subset\mathbb{R}$ let 
\begin{equation}
\mathscr{C}_{N,q}(B):=\left\{ \boldsymbol{\sigma}\in\mathbb{S}^{N-1}(q\sqrt{N}):\,\nabla_{{\rm sp}}H_{N}\left(\boldsymbol{\sigma}\right)=0,\,H_{N}\left(\boldsymbol{\sigma}\right)\in B
\right\} \label{eq:0204-01}
\end{equation}
denote the set of $q$-critical points with values in $B$.
As we are about to see, the support of the Gibbs measure at inverse
temperature $\beta\gg1$ is asymptotically contained in 
thin bands around $q$-critical points at level $-NE$, for particular 
values of 
$\qs=\qs(\beta)$ and $\Es=\Es(\beta)$ defined in \corO{\eqref{eq:qs} and \eqref{eq:Es}, respectively}. ($\corE{-\Es}$ is the 
normalized ground state of the Hamiltonian  \eqref{eq:Hamiltonian} on 
$\mathbb{S}^{N-1}(\qs\sqrt{N})$.)
To define 
what we mean by bands, set
\begin{equation}
{\rm Band}\left(\bs_{0},\epsilon\right):=\left\{ \boldsymbol{\sigma}\in\SN:
\,|R(\bs,\bs_{0})-\|\bs_{0}\|/\sqrt{N}|\leq\epsilon\right\}. \label{eq:band}
\end{equation}
Given a
sequence $\epsilon_{N}>0$, set $B=(-\Es-\epsilon_{N},-\Es+\epsilon_{N})$
and $\Cs=\Cs(\beta):=\mathscr{C}_{N,\qs}(NB)$. As the next theorem
shows,
the union (over $\Cs$) of these bands asymptotically supports 
$G_{N,\beta}$.
\begin{thm}
	\label{thm:Geometry}(Support of the Gibbs measure) 
	Assume that
	$\nu$ satisfies \CDTN. Then there
	exist positive $\epsilon_{N}\to0$
	such that for large enough $\beta$ the following hold.
	\begin{enumerate}
		\item \label{sub-exp}Subexponential number of critical points:
		\begin{equation}
		\label{eq-subsexp}
		\lim_{N\to\infty}\frac{1}{N}\log\mathbb{E}\left\{ |\Cs|\right\} =0.
		\end{equation}
		\item \label{enu:Geometry1}Asymptotic support: 
		\begin{equation}
		\lim_{N\to\infty}\mathbb{E}\left\{ G_{N,\beta}\left(\cup_{\bs_{0}\in\Cs}{\rm Band}\left(\bs_{0},\epsilon_{N}\right)\right)\right\} =1.\label{eq:thm1_1}
		\end{equation}
	\end{enumerate}
\end{thm}
We also obtain a detailed description of the states associated with 
the Gibbs measure $G_{N,\beta}$.
\begin{thm}
	\label{thm:geometry1}
	Assume that
	$\nu$ satisfies \CDTN. Let $\epsilon_{N}\to0$ as in Theorem 
	\ref{thm:Geometry} and $\beta$ large enough.
	Let $\bs$ and $\bs'$ be independent samples from the Gibbs measure
	$G_{N,\beta}$. Then,
	for any $\delta>0$, the following holds.
	\begin{enumerate}
		\item
		\label{enu:Geometry2}States are pure\footnote{Note that the second limit 
			in \eqref{eq:2601} is only relevant in case $\nu$ is an even
			polynomial, since otherwise $\Cs\cap-\Cs=\emptyset$ for  $\epsilon_{N}$
			small. }: 
		\begin{equation}
		\begin{aligned} & \lim_{N\to\infty}\mathbb{P}\left\{ 
		\bs,\,\bs'\in{\rm Band}(\bs_{0},\epsilon_{N})\,
		\mbox{\rm for some $\bs_{0}\in\Cs$, and} \,
		\left|R\left(\boldsymbol{\sigma},\boldsymbol{\sigma}^{\prime}\right)-\qs^{2}\right|>\delta\right\} =0,\\
		& \lim_{N\to\infty}\mathbb{P}\left\{
		\bs\in{\rm Band}(\bs_{0},\epsilon_{N}),
		\bs'\in{\rm Band}(-\bs_{0},\epsilon_{N})\,
		\mbox{\rm for some $\bs_{0}\in\Cs$, and} \,
		\left|R\left(\boldsymbol{\sigma},\boldsymbol{\sigma}^{\prime}\right)+\qs^{2}\right|>\delta\right\} =0.
		\end{aligned}
		\label{eq:2601}
		\end{equation}
		\item \label{enu:Geometry3}Orthogonality of states: 
		\begin{equation}
		\begin{aligned} & \lim_{N\to\infty}\mathbb{P}\left\{ 
		\bs\in{\rm Band}(\bs_{0},\epsilon_{N}),\,\bs'
		\in{\rm Band}(\bs_{0}',\epsilon_{N})\,
		\mbox{\rm for some $\bs_{0}\neq \pm \bs_0^{\prime}\in\Cs$, and} \,
		\left|R\left(\boldsymbol{\sigma},
		\boldsymbol{\sigma}^{\prime}\right)\right|>\delta
		\right\} =0.
		\end{aligned}
		\label{eq:1402}
		\end{equation}
	\end{enumerate}
\end{thm}
Theorem \ref{thm:geometry1} implies that at low temperature, 
the model is in the 1-RSB phase\footnote{Theorem \ref{thm:geometry1} per se 
	actually does not preclude the possibility that the model is in
	the replica 
	symmetric phase, with the Gibbs measure not giving mass to any region of
	vanishing normalized volume. However, the last possibility is easily ruled out by a use of Parisi's formula.}. Its free energy can 
be computed from Parisi's formula as a minimization over a manageable space
of measures. 
We obtain an alternative description of the free energy as
a simple maximization over a subinterval of $[0,1]$, 
see Remark \ref{corinintro} below.

We next turn to the temperature chaos. Recall that one of the main
results in \cite{geometryGibbs} is that pure spherical models do not exhibit
temperature chaos: the supports of $G_{\beta,N}$  are close to each other 
for 
different (large) $\beta$s. On the other hand, 
generically one expects to find temperature chaos in 
spin glasses, see e.g. \cite{PaCheChaos}, \cite{PaChaos}.
We can confirm that indeed, temperature chaos exists in the
mixed models we consider.
\begin{thm}
	\label{thm:Chaos}
	(Chaos in low temperature) Assume that $\nu$ satisfies \CDTN\ \corE{and let $\beta\neq \beta'$ be large enough. Let $\epsilon_N,\, \epsilon'\to0$ and $\Cs,\,\Cs'$ be the corresponding widths and sets of critical points as in Theorem \ref{thm:Geometry}. If $\bs$ and $\bs'$ are independent samples from $G_{N,\beta}$ and $G_{N,\beta'}$, respectively, then for any $\delta>0$,
		\begin{equation*}
	\lim_{N\to\infty}\mathbb{P}\left\{ 
		\bs\in{\rm Band}(\bs_{0},\epsilon_{N}),\,\bs'
		\in{\rm Band}(\bs_{0}',\epsilon_{N}')\,
		\mbox{\rm for some $ \bs_0\in\Cs,\,\bs_0'\in\Cs'$, and} \,
		\left|R\left(\boldsymbol{\sigma},
		\boldsymbol{\sigma}^{\prime}\right)\right|>\delta
		\right\} =0.
		\end{equation*}
}
\end{thm}

\section{Outline of proofs}

At low temperature, the Gibbs measure concentrates on regions where
the Hamiltonian is very low. For some models, those regions are believed
to have fairly simple topology, sometimes referred to as `deep, separated
valleys' in the physics literature, at the bottom of which one finds
a local minimum of the landscape. A natural approach to analyze the
Gibbs measure is therefore to study the distribution of critical points, investigate
the local structure of the Hamiltonian $H_{N}(\bs)$ around them,
and use those to analyze the Gibbs weights of various regions around
the deep critical points. Below we outline how we use this approach in
our setting, and compare with \cite{geometryGibbs} where a similar
method was used for pure models.

\subsection{Critical points}

Before directly considering Gibbs weights of different regions of
the sphere, we need to investigate the distribution of $q$-critical
points of the Hamiltonian, \corO{whose analysis will}
be based on moment computations.
The first moment calculation was carried out in \cite{ABA2} for the
original sphere $\mathbb{S}^{N-1}(\sqrt{N})$, i.e., for $1$-critical
points. In Theorem \ref{thm:1stmoment} we generalize the latter to
general $q$. To establish the concentration of the number of $q$-critical
points at low enough energies, we carry out the corresponding second
moment computation. See Theorem \ref{thm:2ndmomUBBK-1-1} and Corollary
\ref{cor:matching}, which generalize to mixed models the second moment
computation of \cite{2nd} for the pure case at logarithmic scale. 

While our second moment calculation is valid for general mixed models,
its matching to the first moment squared is not guaranteed. In fact,
for deep levels, matching at exponential scale is equivalent to $\Psi_{\nu}^{0}(r)$
being maximized at $r=0$, an assumption we make in \CDTN.

Another consequence of the second moment calculation is that deep
$q$-critical points are approximately orthogonal, see Corollary \ref{cor:orth}.
Moreover, for different values $q_{1}$, $q_{2}$ close to $1$, the
corollary implies pairwise approximate orthogonality of the collection
of deep $q_{1}$-critical points and $q_{2}$-critical points. This
will be crucial to understanding chaos in temperature.

\subsection{\label{subsec:deep-sub-lvl}Deep sub-level sets, critical points
and bands}

Given a lower bound on the free energy, for an appropriate energy
level the Gibbs mass of the corresponding super-level is negligible.
Exploiting this, we will be able to restrict our attention to 
\corO{the}  temperature
dependent sub-level set
\begin{equation}
\{\bs\in\mathbb{S}^{N-1}(\sqrt{N}):\,H_{N}(\bs)\leq-N(E_{0}(\nu)-\tau(\beta))\},\label{eq:sublvl}
\end{equation}
with $\tau(\beta)\to0$ as $\beta\to\infty$, which asymptotically
carries all the mass; see Corollary \ref{cor:UB}.

The geometry of \corO{the set in
  (\ref{eq:sublvl}) and its relation} to $q$-critical
points will play an important role in our analysis. We shall see (Proposition
\ref{prop:sublvl}) that for large $\beta$, the sub-level set (\ref{eq:sublvl})
splits into \emph{exponentially many} connected components, each of
which contains exactly one local minimum of $H_{N}(\bs)$ on the sphere
$\mathbb{S}^{N-1}(\sqrt{N})$. By increasing $\beta$, we can make
those $1$-critical points be as close to orthogonal as we wish (Corollary
\ref{cor:orth}) and the diameter of the components as small as we
wish (Proposition \ref{prop:sublvl}).

\corO{Only a small  number of these many connected components 
 will significantly contribute to the partition function. To
 characterize which do and 
identify the relevant region inside those components responsible for
such contribution,} we will
`scan' them using bands  \corO{as in}  (\ref{eq:band}). More precisely, defining
for any $\bs_{0}$ in $\BN$, \corO{the  $N$ dimensional ball of radius $\sqrt{N}$,}
\begin{equation}
\mathcal{S}(\bs_{0}):=\left\{ \bs\in\mathbb{S}^{N-1}(\sqrt{N}):\,R(\bs,\bs_{0})=\|\bs_{0}\|/\sqrt{N}\right\} ,\label{eq:subsp}
\end{equation}
we will see that for some constant $c$, for each $1$-critical point
$\bs_{*}$ in the sub-level set \corO{\eqref{eq:sublvl}}, there is a differentiable path $\bs_{q}$,
$q\in[1-c\tau(\beta),1]$, such that $\bs_{1}=\bs_{*}$, each $\bs_{q}\corO{\in \mathbb{S}^{N-1}(q\sqrt{N})}$
is a $q$-critical point, and the union $\cup_{q}\mathcal{S}(\bs_{q})$
of \corO{sections} covers the corresponding connected component of (\ref{eq:sublvl})
(see Proposition \ref{prop:sublvl}, Lemma \ref{lem:43}). 

\corO{Recall the normalization constant $Z_{N,\beta}$, 
see \eqref{eq-gibbsmeas}, 
set $F_{N,\beta}=\frac1N \log Z_{N,\beta}$ and define
the \textit{free energy} $F_\beta$ as
\begin{equation}
  \label{eq-Fbeta}
  F_\beta=\lim_{N\to\infty} \frac1N \E F_{N,\beta}.
\end{equation}
(The existence of the limit in \eqref{eq-Fbeta} is well known, see e.g.
\cite{Talag}, and follows also from our analysis.)}
  For any point $\bs_{0}\corO{\in \BN,}$ we define \corO{the}  corresponding weight
\begin{equation}
Z_{N,\beta}(\bs_{0}):=(1-\|\bs_{0}\|/\sqrt{N})^{\frac{N}{2}}\int_{\mathcal{S}(\bs_{0})}e^{-\beta H_{N}(\bs)}d\bs,\label{eq:Zsigma}
\end{equation}
where the integration over $\corE{\mathcal{S}(\bs_0)}$ is with respect to  the uniform
\corO{probability measure on $\corE{\mathcal{S}(\bs_0)}$}, and where the factor before the
integral accounts for the volume of a thin band centered at $\bs$,
and \corO{logarithmically  scales like ${{\rm Vol_{N-2}}(\mathcal{S}(\bs_{0}))}/{{\rm Vol_{N-1}}(\mathbb{S}^{N-1}(\sqrt{N}))}$.}

\corE{Using the information above on the sub-level set in \eqref{eq:sublvl}, to prove Part \ref{enu:Geometry1} of Theorem \ref{thm:Geometry}, we will need to prove a lower bound and upper bounds on certain free energies. The lower bound is on the free energy related to the collection of sections around $\qs$-critical with energy $-\Es$. That is, we will need to show that 
\begin{equation}
\lim_{N\to\infty}\mathbb{P}\bigg\{ \frac{1}{N}\log\sum_{\bs_{0}\in\mathscr{C}_{N,\qs}(N[-\Es-\epsilon,-\Es+\epsilon])}Z_{N,\beta}(\bs_{0})\geq F_{\beta}-\delta\bigg\} =1,\label{eq:3011-02}
\end{equation}
for some $\epsilon$ and $\delta$. 
The upper bounds we shall need are of the form
\begin{equation}
\lim_{N\to\infty}\mathbb{P}\bigg\{ \frac{1}{N}\log\sum_{\bs_{0}\in\mathscr{C}_{N,q}(N[E-\epsilon,E+\epsilon])}Z_{N,\beta}(\bs_{0})\leq F_{\beta}-\delta\bigg\} =1,\label{eq:2709-01}
\end{equation}
and we will need to prove them for any  $q\in[1-c\tau(\beta),1]$ and $E\in[-E_{0}(\nu),-E_{0}(\nu)-c\tau(\beta)]$ such that $(E,q)\neq(-\Es,\qs)$ (and $\epsilon$ and $\delta$ that are allowed to depend on $(E,q)$).
}
\subsection{\label{subsec:Conditional-models}Conditional models on bands}

To obtain bounds of the form \eqref{eq:3011-02} and (\ref{eq:2709-01}), we shall compute
certain related expectations. By an application of the Kac-Rice formula,
those will be expressed through various probabilities or expectations
involving the restriction of $H_{N}(\bs)$ to \corO{sections} $\mathcal{S}(\bs_{0})$
for points $\bs_{0}\in\BN$, conditional on 
\begin{equation}
H_{N}(\bs_{0})=NE,\quad\grad H_{N}(\bs_{0})=0.\label{eq:2709-02}
\end{equation}

This restriction can be mapped from $\mathcal{S}(\bs_{0})$ to the
`standard' sphere of the same dimension, $\SNm$, and the random field
thus obtained should be thought of as a random spherical Hamiltonian.
In Section \ref{sec:models_on_bands}, we will extend the decomposition
obtained in \cite[Section 3]{geometryGibbs} for pure models to general
mixed models, and show that this spherical Hamiltonian is in fact
a mixed $p$-spin model, with mixture coefficients depending on $q=\|\bs_{0}\|/\sqrt{N}$.
The key to analyzing the model under the conditioning of (\ref{eq:2709-02})
is the observation that each of the different $p$-spin interactions
can be written in terms of the Euclidean derivatives of order $p$
of $H_{N}(\bx)$ at $\bs_{0}$. In particular, the conditioning only
affects the lowest two interactions, the $0$- and $1$-spins. Here,
by $0$-spin we mean a model $H_{N-1}(\bs)\equiv J_{0}$ which is
constant on $\SNm$ and is determined by a single Gaussian variable
$J_{0}\sim N(0,N)$, and by $1$-spin, a model of the from $H_{N-1}(\bs)=\sum\sigma_{i}J_{i}$
with $J_{i}\sim N(0,1)$. We will see that the conditioning amounts
to determining the constant value of the $0$-spin to be $NE$, and
removing the $1$-spin interaction term. The corresponding conditional
model is, therefore, a mixture of $p$-spins with $p\geq2$, `shifted'
by a factor of $NE$. In particular, we remark that for $\|\bs_{0}\|=\qs\sqrt{N}$,
this model is replica symmetric for large $\beta$.

\subsection{A comparison with the pure case}

In Section \ref{subsec:deep-sub-lvl} we explained how the connected
components of the sub-level set (\ref{eq:sublvl}) can be covered
by `moving' \corO{sections} $\mathcal{S}(\bs_{q})$. A simpler approach
would be to use concentric \corO{sections} centered at the $1$-critical
point $\bs_{1}$ and avoid altogether constructing the path $\bs_{q}$
and investigating it. In fact, this is exactly what was done for the
pure case in \cite{geometryGibbs}. 

However, the pure case is very special: since its Hamiltonian is a
homogeneous polynomial, $\bs_{q}=q\bs_{1}$ is a (degenerate) path
of $q$-critical points, and (\ref{eq:2709-02}) with $\bs_0\in\SN$ dictates the same
for any point $q\bs_{0}$ on the same `fiber', with $E$ scaled to
$q^{p}E$. In particular, in \cite{geometryGibbs} the derivation
of bounds of the form (\ref{eq:2709-01}) always (i.e., independently
of $q$) involved conditioning as in (\ref{eq:2709-02}) with $\bs_{0}=\bs_{1}$
being the corresponding $1$-critical point. 

The problem with applying the same approach in the mixed is that if
we work with $\bs_{q}=q\bs_{1}$, and thus do not impose the condition
that $\grad H_{N}(\bs_{q})=0$, the conditional models we have to
deal with involves a non-zero $1$-spin component, leading to a more
complicated analysis of the corresponding weights. In particular,
since we need to analyze the weights of exponentially many critical
points, we must understand their large deviation probabilities, which
for non-zero $1$-spin, have speed $N$ matching the complexity, i.e.,
the number of critical points at a given energy. 

By working with $q$-critical $\bs_{q}$, we manage to avoid the difficult
analysis of weights, and obtain a replica symmetric description for
the restriction of $H_{N}(\bs)$ to relevant bands. 

\subsection{Upper bounds on the free energies of the conditional models}

We now return to the bounds (\ref{eq:2709-01}). For any $q\in(0,1)$ and $E<0$, we will show that with
$B=B(E,\epsilon)=[E-\epsilon,E+\epsilon]$,
\begin{equation}
\frac{1}{N}\log\mathbb{E}\Big\{\sum_{\bs_{0}\in\mathscr{C}_{N,q}(NB)}Z_{N,\beta}(\bs_{0})\Big\}\leq\sup_{x\in\mathbb{R}}\Theta_{\nu,q}(E,x)+\Lambda_{Z,\beta}(E,q)+O(\epsilon),\label{eq:1410-02}
\end{equation}
where $\Lambda_{Z,\beta}(E,q)$ is $\frac{1}{N}\log$ the expectation
of a single weight $Z_{N,\beta}(\bs_{0})$, conditional on (\ref{eq:2709-02})
(see Corollary \ref{cor:smallbandZ} and Lemma \ref{lem:bound1}).
By Markov's inequality, asymptotically we have that (\ref{eq:1410-02})
holds without the expectation, with high probability. This bound,
however, will be useful only for large enough $q$ in $[1-c\tau(\beta),1]$
\textendash{} the range we used to cover the  sub-level set of \eqref{eq:sublvl}.

More precisely, we will define a critical value $q_{c}=q_{c}(\beta)$
(see (\ref{eq:qc})), such that for $q\geq q_{c}$ the conditional
model of Section \ref{subsec:Conditional-models} is replica symmetric
and typically the corresponding free energy matches the simple bound
we get from expectations, and for $q<q_{c}$ that free energy is typically
smaller at exponential scale.

For the latter range, we will use the fact that for large $\beta$
the conditional model is dominated, in an appropriate sense, by its
$2$-spin component (see Section \ref{subsec:Conditional-models}).
In Section \ref{sec:Upper-bounds} we will prove that for $q\in[1-c\tau(\beta),q_{c}]$
and large $\beta$, with high probability,
\begin{equation}
\frac{1}{N}\log\sum_{\bs_{0}\in\mathscr{C}_{N,q}(NB)}Z_{N,\beta}(\bs_{0})  \leq\sup_{x\in\mathbb{R}}\Theta_{\nu,q}(E,x)+\Lambda_{F,\beta}^{2-}(E,q)+O(\epsilon)
 +K_{\beta}(E,q)+T_{\beta}.
\label{eq:1410-03}
\end{equation}
The function $\Lambda_{F,\beta}^{2-}(E,q)$ is the asymptotic (normalized)
free energy corresponding to the $2$-spin component only, conditional
on (\ref{eq:2709-02}). The term $K_{\beta}(E,q)$ accounts for atypically
large weights $Z_{N,\beta}(\bs_{0})$,\footnote{Atypical weights for $Z_{N,\beta}(\bs_{0})$ with arbitrary fixed
$\bs_{0}$ under the conditioning, however, typical for the maximal
$Z_{N,\beta}(\bs_{0})$ over all $\bs_{0}\in\mathscr{C}_{N,q}(NB)$.} which may occur for some of the exponentially many points in $\mathscr{C}_{N,q}(NB)$
(for $E$ not too negative). $T_{\beta}$ bounds the error resulting
from using only the $2$-spin part in our computation, which for large
$\beta$ becomes negligible compared to the other terms. 

By an abuse of notation, let $-\Es(q)$ denote the limiting ground state of $H_{N}(\bs)$
restricted to $\mathbb{S}^{N-1}(\sqrt{N}q)$. Note that the complexity
$\Theta_{\nu,q}(E,x)$ does not scale with $\beta$. Therefore, from
the definition (\ref{eq:LambdaZ}) of $\Lambda_{Z,\beta}(E,q)$, one
has that for any $q\in[q_{c},1]$, for large $\beta$ the right-hand
side of (\ref{eq:1410-02}) is maximized over $(-\Es(q),\infty)$
with $E=-\Es(q)$ (where, of course, lower $E$ are not relevant
as there are typically no $q$-critical points with such energy).
For the range $q\in[1-c\tau(\beta),q_{c}]$ further analysis based
on the concentration of the free energy will be needed, but our conclusion
will be the same \textendash{} the right-hand side of (\ref{eq:1410-03})
is also maximized with $E=-\Es(q)$. 

For the ground state $E=-\Es(q)$, the complexity term $\sup_{x\in\mathbb{R}}\Theta_{\nu,q}(E,x)$
vanishes (see Remarks \ref{rem:scaling} and \ref{rem:GSq}). Moreover, if we take $\epsilon$
to be small, so that at exponential scale the number of points in
$\mathscr{C}_{N,q}(NB)$ is small, the large deviation term $K_{\beta}(E,q)$
becomes negligible. Combining the above, to prove (\ref{eq:2709-01})
roughly what we will have to show is that for $q\in[1-c\tau(\beta),1]\setminus\{\qs\}$,
\begin{equation}
\Lambda_{\beta}(\qs)=F_{\beta}>\Lambda_{\beta}(q):=\begin{cases}
\Lambda_{Z,\beta}(-\Es(q),q) & q\in[q_{c},1]\\
\Lambda_{F,\beta}^{2-}(-\Es(q),q) & q\in[1-c\tau(\beta),q_{c}],
\end{cases}\label{eq:1510-01}
\end{equation}
for large $\beta$. In fact, we will only need to prove the inequality,
but our proof will go through showing that $\Lambda_{\beta}(\qs)=F_{\beta}$.

\subsection{The lower bound on the free energy}

In order to lower bound the limiting free energy $F_{\beta}$, we
shall consider the collections of bands corresponding to $\qs$ and $-\Es=-\Es(\qs)$. From the moment matching
of the number of $\qs$-critical points in $\Cs$ (see Corollary \ref{cor:matching}),
we have that with probability that decays slower \corO{than}  exponentially
in $N$, $\Cs$ is non-empty.\footnote{This is what follows from matching at exponential scale.
\corO{Had  we} established matching at scale $O(1)$, as in the pure case
\cite{2nd}, the same probability would have gone to $1$.} To prove the lower bound we will show in Section \ref{sec:lower-bound}
that for any $\delta>0$, there are no points $\bs_{0}$ in $\Cs$
for which 
\[
Z_{N,\beta}(\bs_{0})<\Lambda_{Z,\beta}(-\Es,\qs)-\delta.
\]
Consequently, with probability that decays slower \corO{than} exponentially
in $N$, 
\begin{equation}
\label{eq-Nov12}
F_{\beta}>\Lambda_{Z,\beta}(-\Es,\qs)-\delta.
\end{equation}
From the concentration of the free energy around its mean, we \corO{then conclude that \eqref{eq-Nov12} occurs}
with probability that tends to $1$ as $N\to\infty$.

\section{\label{sec:CritPts}Critical points: main results and notation}
A crucial step \corO{in  the analysis of the Gibbs measure $G_{N,\beta}$
is the study of} the critical points of $H_{N}(\cdot)$ \corO{on $\mathbb{S}^{N-1}(q\sqrt{N})$.}
We carry out this analysis by applying the second moment
method to
\begin{equation}
\mbox{Crt}_{N,q}\left(B,D\right):=\left|\mathscr{C}_{N,q}\left(NB,\sqrt{N}D\right)\right|,\label{eq:Crt}
\end{equation}
where 
\[
\mathscr{C}_{N,q}(B,D):=\Big\{ \boldsymbol{\sigma}\in\corE{\mathbb{S}^{N-1}(q\sqrt{N})}:\,\nabla_{{\rm sp}}H_{N}\left(\boldsymbol{\sigma}\right)=0,\,H_{N}\left(\boldsymbol{\sigma}\right)\in B,\,\ddq H_{N}\left(\bs\right)\in D\Big\} 
\]
denotes the set of $q$-critical points with values in $B$ and `normal'
derivative 
\[
\ddq H_{N}\left(\boldsymbol{\sigma}\right):=\left.\frac{d}{dq}\right|_{q=\left\Vert \boldsymbol{\sigma}\right\Vert }H_{N}\left(q\bs/\|\bs\|\right)
\corO{=\frac{1}{\|\bs\|} \left.\frac{d}{dq} H_N(q\bs)\right|_{q=1}}
\]
in $D$. 

The logarithmic asymptotics of $\mathbb{E}\mbox{Crt}_{N,1}\left(B,\mathbb{R}\right)$
were calculated by Auffinger and Ben Arous \cite{ABA2} (see also
\cite{A-BA-C} for the pure case). We shall need the next theorem
which is a direct extension of the latter, accounting for general
$D$ and $q$.

Let $\mu^{*}$ denote the semicircle \corO{probability}  measure, the density of which
with respect to Lebesgue measure is 
\begin{equation}
\frac{d\mu^{*}}{dx}=\frac{1}{2\pi}\sqrt{4-x^{2}}\mathbf{1}_{\left|x\right|\leq2}.\label{eq:semicirc}
\end{equation}
Define the functions (see, e.g., \cite[Proposition II.1.2]{logpotential})
\begin{align}
\Omega(x) & \triangleq\int_{\mathbb{R}}\log\left|\lambda-x\right|d\mu^{*}\left(\lambda\right)\label{eq:Omega}\\
 & =\begin{cases}
\frac{x^{2}}{4}-\frac{1}{2} & \mbox{ if }0\leq\left|x\right|\leq2\\
\frac{x^{2}}{4}-\frac{1}{2}-\left[\frac{\left|x\right|}{4}\sqrt{x^{2}-4}-\log\left(\sqrt{\frac{x^{2}}{4}-1}+\frac{\left|x\right|}{2}\right)\right] & \mbox{ if }\left|x\right|>2,
\end{cases}\nonumber 
\end{align}
\begin{equation}
\Theta_{\nu,q}\left(u,x\right):=\frac{1}{2}+\frac{1}{2}\log\left(q^{2}\frac{\nu''(q^{2})}{\nu'(q^{2})}\right)-\frac{1}{2}(u,x)\Sigma_{q}^{-1}(u,x)^{T}+\Omega\left(\frac{x}{q\sqrt{\nu''(q^{2})}}\right),\label{eq:Theta}
\end{equation}
where 
\begin{equation}
\Sigma_{q}:=\left(\begin{array}{cc}
\nu(q^{2}) & q^{2}\nu'(q^{2})\\
q^{2}\nu'(q^{2}) & q^{4}\nu''(q^{2})+q^{2}\nu'(q^{2})
\end{array}\right)\label{eq:Sigma1}
\end{equation}
is \corO{the covariance matrix of the vector ($H_N(\bs)/\sqrt{N},\ddq H_N(\bs)$), and is}  invertible by Lemma \ref{lem:invertibility}, \corO{whose statement and proof are given at} the end
of this section. \corO{The next theorem, whose proof appears in Section \ref{sec:Moments}, is an evaluation of the exponential rate of growth of the expectation of
\eqref{eq:Crt}.}
\begin{thm}
\label{thm:1stmoment}(First moment) For any intervals $B$ and $D$,
with $\Theta_{\nu,q}\left(u,x\right)$ as defined in (\ref{eq:Theta}),
\begin{equation}
\lim_{N\to\infty}\frac{1}{N}\log\mathbb{E}{\rm Crt}_{N,q}\left(B,D\right)=\sup_{u\in B,\,x\in D}\Theta_{\nu,q}\left(u,x\right).\label{eq:1106-1-1}
\end{equation}
\end{thm}
We shall also need an asymptotic upper bound on the corresponding
second moment. For any subsets $I\subset\left[-1,1\right]$, $B_{i},\,D_{i}\subset\mathbb{R}$,
define the `contribution' of pairs of points with overlap in $I$
to $\prod_{i=1,2}|\mathscr{C}_{N,q_{i}}(NB_{i},\sqrt{N}D_{i})|$ \corO{by}
\[ \negthickspace\negthickspace\negthickspace[\mbox{Crt}_{N,q_{1},q_{2}}(B_{1},B_{2},D_{1},D_{2},I)]_2:=
\#\Big\{\left(\boldsymbol{\sigma},\boldsymbol{\sigma}'\right)\in\prod_{i=1,2}\mathscr{C}_{N,q_{i}}\big(NB_{i},\sqrt{N}D_{i}\big):\;R(\boldsymbol{\sigma},\boldsymbol{\sigma}')\in I\,\Big\}.
\]

Define the function
\begin{align}
\Psi_{\nu,q_{1},q_{2}}\left(r,u_{1},u_{2},x_{1},x_{2}\right) & :=1+\frac{1}{2}\log\left((1-r^{2})\frac{q_{1}^{2}q_{2}^{2}\nu''(q_{1}^{2})\nu''(q_{2}^{2})}{\nu'(q_{1}^{2})\nu'(q_{2}^{2})-(\nu'(q_{1}q_{2}r))^{2}}\right)\nonumber \\
 & -\frac{1}{2}\left(u_{1},u_{2},x_{1},x_{2}\right)\Sigma_{U,X}^{-1}\left(r,q_{1},q_{2}\right)\left(u_{1},u_{2},x_{1},x_{2}\right)^{T}\label{eq:Psi}\\
 & +\Omega\Bigg(\frac{x_{1}}{q_{1}\sqrt{\nu''(q_{1}^{2})}}\Bigg)+\Omega\Bigg(\frac{x_{2}}{q_{2}\sqrt{\nu''(q_{2}^{2})}}\Bigg),\nonumber 
\end{align}
where $\Sigma_{U,X}\left(r,q_{1},q_{2}\right)$ is 
\corO{the covariance matrix of the vector 
\begin{equation}
\label{eq-rocknroll}
(H_N(\bs)/\sqrt{N},H_N(\bs')/\sqrt{N},\frac{d}{dR} H_N(\bs),\frac{d}{dR}H_N(\bs'))
\end{equation} 
 with $(\bs,\bs')\in \mathbb{S}^{N-1}(q_1\sqrt{N})
\times \mathbb{S}^{N-1}(q_2 \sqrt{N})$ with overlap $r$, conditioned on
$\nabla_{{\rm sp}}H_{N}(\bs),\nabla_{{\rm sp}}H_{N}(\bs')$, see Lemma
\ref{lem:condHamiltonian}, and is}
\corO{defined explicitly in}  (\ref{eq:SigmaUXbar}); \corO{ Lemma \ref{lem:invertibility} implies that this covariance matrix is}
invertible. \corO{With a slight} abuse of
notation, we  write 
\begin{equation}
\label{eq:Psishort}\Psi_{\nu,q_{1},q_{2}}\left(r,u,x\right)=\Psi_{\nu,q_{1},q_{2}}\left(r,u,u,x,x\right).
\end{equation}
We note for later use that if we substitute $r=0$, by simple algebra,
\begin{equation}
\Psi_{\nu,q_{1},q_{2}}\left(0,u_{1},u_{2},x_{1},x_{2}\right)=\Theta_{\nu,q_{1}}\left(u_{1},x_{1}\right)+\Theta_{\nu,q_{2}}\left(u_{2},x_{2}\right).\label{eq:1703-08}
\end{equation}

The following \corO{theorem and lemma are}   extensions of Theorem 5 of \cite{2nd} which
concerned the pure case with $q_{i}=1$. \corO{The proofs are given in Section \ref{sec:Moments}.}
\begin{thm}
\label{thm:2ndmomUBBK-1-1}(Second moment) For any intervals $I\subset(-1,1)$,
$B_{i},D_{i}\subset\mathbb{R}$, with $\Psi_{\nu,q_{1},q_{2}}$ as
defined in (\ref{eq:Psi}),

\begin{equation}
\limsup_{N\to\infty}\frac{1}{N}\log\left(\mathbb{E}\corO{\left[ {\rm Crt}_{N,q_{1},q_{2}}(B_{1},B_{2},D_{1},D_{2},I)\right]_2} \right)\leq\sup_{r\in I,u_{i}\in B_{i},x_{i}\in D_{i}}\Psi_{\nu,q_{1},q_{2}}\left(r,u_{1},u_{2},x_{1},x_{2}\right).\label{eq:2ndUB_BK-1-1-1}
\end{equation}
\end{thm}

\begin{lem}
\label{lem:2ndLB}\corO{If}  
$D_{i}\subset(-\infty,\corO{-2\sqrt{\nu''(q_{i}^{2})}q_{i}}-\tau)$
for some $\tau>0$ \corO{then} (\ref{eq:2ndUB_BK-1-1-1}) holds \corO{with}  equality. 
\end{lem}

\corO{We next turn to the study of consequences of \CDTN.} In addition to $E_0$ which was introduced in \corO{\eqref{eq:E0}}, the energy level 
\begin{equation}
E_{\infty}:  =E_{\infty}(\nu)=\frac{\nu''(1)\nu(1)+\nu'(1)^{2}-\nu'(1)\nu(1)}{\nu'(1)\sqrt{\nu''(1)}}
\label{eq:Einfty}
\end{equation}
was defined in  \cite{A-BA-C,ABA2}, \corO{as a}
threshold energy related to the spectrum
of the Hessian matrix $\Hess H_{N}(\bs)$ at critical points. \corO{As we will see, it plays a role in}
several large deviations problems and concentration
of statistics, and for pure-like models it is related to \corO{the function 
$G(\nu)$ defined in \eqref{eq-G}}  (see \cite[Eq. (1.22), (4.1)]{ABA2}) \corO{through the relation}
\begin{equation}
G(\nu)=\sup_{x\in\mathbb{R}}\Theta_{\nu,1}\left(-E_{\infty},x\right).\label{eq:Gpl}
\end{equation}

\corO{For $q\in (0,1]$, set} $\nu_{q}(x)=\sum_{p}\gamma_{p}^{2}q^{2p}x^{p}$ and define $E_{0}(q):=E_{0}(\nu,q)=E_{0}(\nu_{q})$ and similarly
define $E_{\infty}(q)$. Assuming $\nu_{q}$ is pure-like, also define
$x_{0}(q):=x_{0}(\nu,q)=\frac{1}{q}x_{0}(\nu_{q})$. \corO{The next remark summarizes scaling relations associated to these quantitites.}
\begin{rem} 
\label{rem:scaling}
\corO{Fix the disorder coefficients  $J_{i_{1},...,i_{p}}^{(p)}$  in \eqref{eq:Hamiltonian}, and let
$H_{N}^{\nu_{q}}(\bs)$ be the Hamiltonian
corresponding to the mixture $\nu_{q}(x)$.} Then $H_{N}^{\nu_{q}}(\bs)=H_{N}(q\bs)$
and $\ddq H_{N}\left(q\bs\right)=\frac{1}{q}\ddq H_{N}^{\nu_{q}}\left(\bs\right)$.
Therefore, $\Theta_{\nu,q}\left(u,x\right)=\Theta_{\nu_{q},1}\left(u,qx\right)$ and similarly to (\ref{eq:E0}) and (\ref{eq:x0}), we have
\begin{align}
-E_{0}(q)&=\min\Big\{ E:\,\sup_{x\in \R}\Theta_{\nu,q}\left(E,x\right)=0\Big\},\label{eq:E0q}\\
-x_{0}(q)&=\arg\max_{x\in\mathbb{R}}\Theta_{\nu,q}\left(-E_{0}(q),x\right).\label{eq:x0q}
\end{align}
\end{rem}

The following lemma implies the matching of the second and first moment
squared of (\ref{eq:Crt}) at exponential scale as $N\to\infty$,
for small $E$ and $D$ around $-E_{0}(q)$ and $-x_{0}(q)$. \corO{The proof is contained in Section \ref{sec:Matching}.} 
\begin{lem}
\label{lem:conccentration}For any $\nu$ satisfying \CDTN, there
exists some $\delta>0$ such that 
\corO{if  $|q_{i}-1|<\delta$, $B_{i}\subset-E_{0}(q_{i})+(-\delta,\delta)$ and 
$D_{i}\subset-x_{0}(q_{i})+(-\delta,\delta)$,}
then \corO{for any $\epsilon>0$},
\begin{equation}
\sup_{\corO{|r|\in[\epsilon,1)},u_{i}\in B_{i},x_{i}\in D_{i}}\Psi_{\nu,q_{1},q_{2}}\left(r,u_{1},u_{2},x_{1},x_{2}\right)<\sum_{i=1,2}\sup_{u_{i}\in B_{i},\,x_{i}\in D_{i}}\Theta_{\nu,q_{i}}\left(u_{i},x_{i}\right),\label{eq:1703-06}
\end{equation}
\corO{whenever both}  summands on the right-hand side of (\ref{eq:1703-06})
are nonnegative.
\end{lem}

As consequence of Lemma \ref{lem:conccentration} and
Theorem \ref{thm:1stmoment}, we have the matching of the moments
at exponential scale.
\begin{cor}
\label{cor:matching}(Matching of moments) \corO{With $\nu$ and $\delta$ as in Lemma \ref{lem:conccentration}, let $|q-1|<\delta$,
$B\subset -E_{0}(q)+(-\delta,\delta)$
and $D\subset  -x_0(q)+(-\delta,\delta)$. Then}
\begin{equation}
\lim_{N\to\infty}\frac{1}{N}\log\mathbb{E}\left({\rm Crt}_{N,q}\left(B,D\right)^{2}\right)=2\lim_{N\to\infty}\frac{1}{N}\log\mathbb{E}{\rm Crt}_{N,q}\left(B,D\right),\label{eq:1008-01}
\end{equation}
\corO{as soon} \corE{as} the expectation on the right-hand side tends to
$\infty$ as $N\to\infty$. \corO{The equality in \eqref{eq:1008-01} continues to hold}  if we let $B=B_{N}=u+(-\delta_{N},\delta_{N})$
and $D=D_{N}=x+(-\delta_{N},\delta_{N})$ \corO{with $|u+E_0(q)|<\delta$ and $|x+x_0(q)|<\delta$, as soon as} 
$\delta_{N}\to0$ slow enough so that $\frac{1}{N}\log\mathbb{E}{\rm Crt}_{N,q}(B_{N},D_{N})\to\Theta_{q}(u,x)$.
\end{cor}

Another interesting consequence of Lemma \ref{lem:conccentration}
is the following \corE{corollary,} \corO{ whose proof appears in Section  \ref{sec:Matching}.}
\begin{cor}
\label{cor:orth}(Orthogonality of deep critical points) \corO{Assume \CDTN. Then there exists $\delta_0>0$ so that for any 
$0<\delta<\delta_0$,} \corE{$q_1,\,q_2\in (1-\delta,1+\delta)$ and $\epsilon>0$, there exist constants $\eta=\eta(\epsilon), c=c(\epsilon)>0$ so that, with
$B_i(\eta):=(-E_{0}(q_i)-\eta,-E_{0}(q_i)+\eta)$,
\begin{equation}
\mathbb{P}\Big\{\exists\bs_i\in\mathscr{C}_{N,q_i}(NB_i(\eta)),\,\bs_1\neq\pm\bs_2:\,|R(\bs_1,\bs_2)|\geq\epsilon\Big\}<e^{-cN},\label{eq:0707-01}
\end{equation}
 Moreover, for any
$\eta_{N}=o(1)$, setting $B_{N,i}:=(-\infty,-E_{0}(q_i)+\eta_{N})$,
there exists a sequence $\epsilon_{N}=o(1)$ such that
\begin{equation}
\lim_{N\to\infty}\mathbb{P}\Big\{\exists\bs_i\in\mathscr{C}_{N,q_i}(NB_{N,i}),\,\bs_1\neq\pm\bs_2:\,|R(\bs_1,\bs_2)|\geq\epsilon_{N}\Big\}=0.\label{eq:0707-02}
\end{equation}
}
\end{cor}
%
\corO{We finish the section with the statement and proof of the following
lemma concerning the invertibility of the matrices  $\Sigma_{U,X}\left(r,q_{1},q_{2}\right)$ and $\Sigma_{q}$ }.
\begin{lem}
\label{lem:invertibility}If $\nu$ is not a monomial, then for any
$r\in(-1,1)$, $\Sigma_{U,X}\left(r,q_{1},q_{2}\right)$ and $\Sigma_{q}$
are invertible (and therefore \corO{strictly} positive definite) for any $q_{1},\,q_{2},\,q\in(0,1]$. 
\end{lem}

\begin{proof} \corO{Recall that $\Sigma_{U,X}\left(r,q_{1},q_{2}\right)$
is the covariance matrix of the vector \eqref{eq-rocknroll} 
conditional
on the gradients at the two corresponding points.}
Suppose $H_{N}^{\nu}(\bs)=H_{N}^{\nu_{1}}(\bs)+H_{N}^{\nu_{2}}(\bs)$
is the Hamiltonian corresponding to $\nu=\nu_{1}+\nu_{2}$, where
$H_{N}^{\nu_{i}}(\bs)$ are independent. Using a similar notation
for $\Sigma_{U,X}^{\nu}\left(r,q_{1},q_{2}\right)$, we have that
if $\Sigma_{U,X}^{\nu_{1}}\left(r,q_{1},q_{2}\right)+\Sigma_{U,X}^{\nu_{2}}\left(r,q_{1},q_{2}\right)$
is invertible then so is $\Sigma_{U,X}^{\nu}\left(r,q_{1},q_{2}\right)$
(since the former corresponds to the distribution obtained by conditioning
on each the gradients corresponding to $H_{N}^{\nu_{i}}(\bs)$, $i=1,2$,
and the latter corresponds to conditioning on the sum of
those gradients being $0$). Thus, since $\Sigma_{U,X}^{\nu}\left(r,q_{1},q_{2}\right)$
is positive semi-definite, to prove invertibility for general non-pure
mixture, it is enough to prove that $\Sigma_{U,X}^{\nu_{1}}\left(r,q_{1},q_{2}\right)+\Sigma_{U,X}^{\nu_{2}}\left(r,q_{1},q_{2}\right)$
is invertible for any $\nu_{1}(x)=\gamma x^{p}$ and $\nu_{2}(x)=\gamma'x^{p'}$
with $p\neq p'$ and $\gamma,\,\gamma'>0$. 

From the formula (\ref{eq:SigmaUXbar}) for 
$\Sigma_{U,X}\left(r,q_{1},q_{2}\right)$, 
\corO{we have that if
\corE{$\nu_i(x)=x^{p_i}$} and
\begin{align*}
\mathfrak{a}(r,q,p): & =q^{2(p-1)}\left[1-\frac{pr^{2(p-1)}\left(1-r^{2}\right)}{1-\left(r^{p}-(p-1)r^{p-2}(1-r^{2})\right)^{2}}\right],\\
\mathfrak{b}(r,p): & =r^{p}\left(q_{1}q_{2}\right)^{p-1}\left[1-pr^{p-2}\left(1-r^{2}\right)\frac{r^{p}-(p-1)r^{p-2}(1-r^{2})}{1-\left(r^{p}-(p-1)r^{p-2}(1-r^{2})\right)^{2}}\right],
\end{align*}
then}
\[
  \Sigma_{U,X}^{\nu_i}\left(r,q_{1},q_{2}\right)=\left(\begin{array}{cccc}
q_{1}^{2}\mathfrak{a}(r,q_{1},p_i) & q_{1}q_{2}\mathfrak{b}(r,p_i) & p_iq_{1}\mathfrak{a}(r,q_{1},p_i) & p_iq_{1}\mathfrak{b}(r,p_i)\\
q_{1}q_{2}\mathfrak{b}(r,p_i) & q_{2}^{2}\mathfrak{a}(r,q_{2},p_i) & p_iq_{2}\mathfrak{b}(r,p_i) & p_iq_{2}\mathfrak{a}(r,q_{2},p_i)\\
p_iq_{1}\mathfrak{a}(r,q_{1},p_i) & pq_{2}\mathfrak{b}(r,p_i) & p_i^{2}\mathfrak{a}(r,q_{1},p_i) & p^{2}\mathfrak{b}(r,p_i)\\
p_iq_{1}\mathfrak{b}(r,p_i) & pq_{2}\mathfrak{a}(r,q_{2},p_i) & p^{2}\mathfrak{b}(r,p_i) & p^{2}\mathfrak{a}(r,q_{2},p_i)
\end{array}\right).
\]
\corO{Therefore,
  \corE{if} $(U_{1},U_{2})\sim N(0,\Sigma_{U}^{\nu_i}\left(r,q_{1},q_{2}\right))$, \corE{where $\Sigma_{U}^{\nu_i}(r,q_{1},q_{2})$ is the upper-left $2\times 2$ sub-matrix of $\Sigma_{U,X}^{\nu_i}(r,q_{1},q_{2})$,}
then 
\[
  (U_{1},U_{2},\frac{p_i}{q_{1}}U_{1},\frac{p_i}{q_{2}}U_{2})\sim N(0,\Sigma_{U,X}^{\nu_i}\left(r,q_{1},q_{2}\right)).
\]
Since $\Sigma_U^{\nu_i}(r,q_1,q_2)$ is invertible whenever $|r|\neq 1$, we 
have that 
$(x_{1},x_{2},y_{1},y_{2})\Sigma_{U,X}^{\nu_i}\left(r,q_{1},q_{2}\right)=0$
if and only if $x_{1}+\frac{p_i}{q_{1}}y_{1}=0$ and 
$x_{2}+\frac{p_i}{q_{2}}y_{2}=0$. 
Using the positive definiteness of 
  $\Sigma_{U,X}^{\nu_i}$ we deduce that
$(x_{1},x_{2},y_{1},y_{2})\Sigma_{U,X}^{\nu}\left(r,q_{1},q_{2}\right)=0$
  iff $(x_{1},x_{2},y_{1},y_{2})=0$.
This proves the invertibility of $\Sigma_{U,X}\left(r,q_{1},q_{2}\right)$
for general mixtures. }
%
%

Noting that $\Sigma_{q_{1}}$ is the $2\times2$ sub-matrix obtained
from $\Sigma_{U,X}\left(0,q_{1},q_{2}\right)$ by deleting the second
and fourth rows and columns, we \corO{conclude the invertibility of}  $\Sigma_{q}$. 

\corO{Strict} positive definiteness follows from invertibility, since both matrices
are covariance matrices.
\end{proof}

\section{\label{sec:Moments}Moments computations: proofs of Theorems \ref{thm:1stmoment},
\ref{thm:2ndmomUBBK-1-1} and Lemma \ref{lem:2ndLB}}

This section is devoted to the proofs of the results in its title.
\corO{The proofs rely on tedious computations of certain covariance matrices, which
are contained in Appendix \ref{sec:Covariances}.}

\subsection{\label{subsec:pf_1stmoment}Proof of Theorem \ref{thm:1stmoment}}

By an application of the Kac-Rice formula \cite[Theorem 12.1.1]{RFG},
using the stationarity of $(H_{N}\left(\boldsymbol{\sigma}\right),\,\ddq H_{N}\left(\boldsymbol{\sigma}\right))$
on $\mathbb{S}^{N-1}(\sqrt{N}q)$,
\begin{align*}
\mathbb{E}\mbox{Crt}_{N,q}\left(B,D\right) & =q^{N-1}\omega_{N}\varphi_{\grad H_{N}\left(\boldsymbol{\sigma}\right)}(0)\\
&\times \mathbb{E}\Big\{\left|\det(\Hess H_{N}\left(\boldsymbol{\sigma}\right))\right|
  \mathbf{1}\{H_{N}\left(\boldsymbol{\sigma}\right)\in NB,\,\ddq H_{N}\left(\boldsymbol{\sigma}\right)\in\sqrt{N}D\}\,\Big|\,\grad H_{N}\left(\boldsymbol{\sigma}\right)=0\Big\},
\end{align*}
where \corO{$\bs\in \mathbb{S}^{N-1}(\sqrt{N}q)$ is arbitrary},
$\varphi_{\grad H_{N}(\boldsymbol{\sigma})}(0)$ is the
density of $\grad H_{N}(\boldsymbol{\sigma})$ at $0$,
and $\omega_{N}={2\pi^{N/2}}/{\Gamma(N/2)}$
is the surface area of the $N-1$-dimensional unit sphere. \corO{By a covariance computation contained in   Lemma
\ref{lem:cov} of Appendix \ref{sec:Covariances}} (applied with $r=1$), the three variables 
\[
\left(H_{N}\left(\boldsymbol{\sigma}\right),\,\ddq H_{N}\left(\boldsymbol{\sigma}\right)\right),\,\Hess H_{N}\left(\boldsymbol{\sigma}\right)+\frac{1}{\sqrt{N}q}\ddq H_{N}\left(\boldsymbol{\sigma}\right)\mathbf{I},\,\grad H_{N}\left(\boldsymbol{\sigma}\right)
\]
are independent, $\grad H_{N}\left(\boldsymbol{\sigma}\right)\sim N(0,\nu'(q^{2})\mathbf{I})$,
\begin{equation}
\Big(\frac{1}{\sqrt{N}}H_{N}\left(\boldsymbol{\sigma}\right),\,\ddq H_{N}\left(\boldsymbol{\sigma}\right)\Big)\sim N\left(0,\Sigma_{q}\right).\label{eq:0503-02}
\end{equation} 
 with $\Sigma_{q}$ as defined in (\ref{eq:Sigma1}), and
\[
\mathbf{G}=\sqrt{\frac{N}{(N-1)\nu''(q^{2})}}\Big(\Hess H_{N}\left(\boldsymbol{\sigma}\right)+\frac{1}{\sqrt{N}q}\ddq H_{N}\left(\boldsymbol{\sigma}\right)\mathbf{I}\Big)
\]
is a (normalized) GOE matrix, \corO{that is, 
 a real, symmetric \corE{$N-1\times N-1$}  matrix such that all elements are centered Gaussian
variables which, up to symmetry, are independent with variance given
by
\[
\mathbb{E}\left\{ \mathbf{G}_{ij}^{2}\right\} =\begin{cases}
1/\corE{(N-1)}, & \,i\neq j\\
2/\corE{(N-1)}, & \,i=j.
\end{cases}
\]
}
Combining the above, after some algebra, we have that as $N\to\infty$,
\begin{equation}
\begin{aligned}\mathbb{E}\mbox{Crt}_{N,q}\left(B,D\right) & =e^{\frac{N}{2}+o(N)}\left(\frac{\nu''(q^{2})}{q^{2}\nu'(q^{2})}\right)^{\frac{N}{2}}\\
&\times \int_{\sqrt{N}B}du\int_{\sqrt{N}D}dx
 \exp\left\{ -\frac{1}{2}(u,x)\Sigma_{q}^{-1}(u,x)^{T}\right\} \mathbb{E}\left\{ \left|\det\left(\mathbf{G}-\frac{x}{q\sqrt{(N-1)\nu''(q^{2})}}\mathbf{I}\right)\right|\right\} .
\end{aligned}
\label{eq:1102-01}
\end{equation}

\corO{The determinant in \eqref{eq:1102-01} can be written as 
$\exp(\sum\log|\lambda_{i}|)$,} where $\lambda_{i}$ are the corresponding
eigenvalues. \corO{An upper bound on the right hand  side of  \eqref{eq:1102-01}, which gives the inequality $\leq$ in  (\ref{eq:1106-1-1}),
is obtained by combining}
Varadhan's
integral lemma \cite[Theorem 4.3.1, Exercise 4.3.11]{LDbook} and
the large deviation principle satisfied by the empirical measure
of eigenvalues of GOE matrices \cite[Theorem 2.1.1]{BAG97} (together
with a truncation argument based on the upper bound for top eigenvalue
\cite[Lemma 6.3]{BDG} of GOE matrices). We will discuss a similar
argument in the much more complicated case of bounding the expectation
of \corE{${[\rm Crt}_{N,q_{1},q_{2}}(B_{1},B_{2},D_{1},D_{2},I)]_2$} in the
proof of Theorem \ref{thm:2ndmomUBBK-1-1}. \corO{Therefore, we refrain from going
into the details here.}

To \corO{obtain the reverse inequality $\geq$ in} (\ref{eq:1106-1-1}),
\corO{it is enough} to show that for any $t\in\mathbb{R}$,
\begin{equation}
\lim_{\epsilon\to0}\lim_{N\to\infty}\frac{1}{N}\log\int_{t_{0}-\epsilon}^{t_{0}+\epsilon}\mathbb{E}\left\{ \left|\det\left(\mathbf{G}-t\mathbf{I}\right)\right|\right\} dt=\Omega(t),\label{1102-03-1}
\end{equation}
\corO{where $\Omega$ is as in \eqref{eq:Omega}.}
For the pure case $\nu_{p}(x)=x^{p}$, it was shown in \cite[(3.21)]{A-BA-C} that
\[
\mathbb{E}\mbox{Crt}_{N,1}\left(B,\mathbb{R}\right)=e^{\frac{N}{2}+o(N)}\left(p-1\right)^{\frac{N}{2}}\int_{\sqrt{N}B}e^{-\frac{u^{2}}{2}}\mathbb{E}\left\{ \left|\det\left(\mathbf{G}-u\sqrt{\frac{p}{(N-1)(p-1)}}\mathbf{I}\right)\right|\right\} du .
\]
On the other hand, it is proved in \cite[Theorem 2.8]{A-BA-C} that
\[
\lim_{N\to\infty}\frac{1}{N}\log\mathbb{E}\mbox{Crt}_{N,1}\left(B,\mathbb{R}\right)=\sup_{u\in B}\left\{ \frac{1}{2}+\frac{1}{2}\log\left(p-1\right)-\frac{u^{2}}{2}+\Omega\left(u\sqrt{\frac{p}{p-1}}\right)\right\} .
\]
By considering the  intervals $B=(t_{0}-\epsilon,t_{0}+\epsilon)$, the
above implies (\ref{1102-03-1}) and completes the proof.\qed

\subsection{\label{subsec:pf_2ndmoment}Proof of Theorem \ref{thm:2ndmomUBBK-1-1}}
\corO{Throughout the proof we fix the intervals
$B_{i},\,D_{i}\subset\mathbb{R}$ and $I\subset\left(-1,1\right)$
and the numbers $q_{1},\,q_{2}\in(0,1]$.}
The proof follows closely that of \cite[Theorem 5]{2nd} (see Section
5.4 there) and requires, in particular, variants of the auxiliary
Lemmas 11-16 of \cite{2nd}.  An application of the Kac-Rice
formula \cite[Theorem 12.1.1]{RFG} and isotropy yields the integral formula 
\begin{align}
 & \mathbb{E}\left\{ \left[\mbox{Crt}_{N,q_{1},q_{2}}(B_{1},B_{2},D_{1},D_{2},I)\right]_{2}\right\} =\nonumber \\
 & \omega_{N}\omega_{N-1}\left(N-1\right)^{N-1}\left(q_{1}^{2}q_{2}^{2}\nu''(q_{1}^{2})\nu''(q_{1}^{2})\right)^{\frac{N-1}{2}}\int_{I_{R}}dr\cdot\left(1-r^{2}\right)^{\frac{N-3}{2}}\varphi_{\grad H_{N}\left(q_{1}\hat{\mathbf{n}}\right),\grad H_{N}\left(q_{2}\boldsymbol{\sigma}\left(r\right)\right)}\left(0,0\right)\nonumber \\
 & \mathbb{E}\Bigg\{\Big|\det\Big(\sqrt{\frac{N}{\left(N-1\right)\nu''(q_{1}^{2})}}\Hess H_{N}(q_{1}\hat{\mathbf{n}})\Big)\Big|\cdot\Big|\det\Big(\sqrt{\frac{N}{\left(N-1\right)\nu''(q_{2}^{2})}}\Hess H_{N}(q_{2}\boldsymbol{\sigma}\left(r\right))\Big)\Big|\label{eq:KR-2nd}\\
 & \mathbf{1}\Big\{\big(H_{N}(q_{1}\hat{\mathbf{n}}),\,H_{N}(q_{2}\boldsymbol{\sigma}(r))\big)\in NB_{1}\times NB_{2},\,\big(\ddq H_{N}(q_{1}\hat{\mathbf{n}}),\,\ddq H_{N}(q_{2}\boldsymbol{\sigma}(r))\big)\in\sqrt{N}D_{1}\times\sqrt{N}D_{2}\Big\}\nonumber \\
 & \,\quad\quad\quad
\;\;\;\;\;\;\;\;\;\;\;\;\;\;\;\;\;\;\;\;\;\;\;\;\;\;\;\;\;\;\;\;\;\;\;\;\;\;\;\;\;\;\;\;\;\;\;\;\;\;\;\;\;\;\;\;\;\;\;\;\;\;\;\;\;\;\;\;\;\;\;\;\;\;\;\;\;\;\;\;\;\;\;\;\;
\Bigg|\,\grad H_{N}(q_{1}\hat{\mathbf{n}})=\grad H_{N}(q_{2}\boldsymbol{\sigma}(r))=0\Bigg\},\nonumber 
\end{align}
where $\hat{\mathbf{n}}=\left(0,...0,\sqrt{N}\right)$, 
\begin{equation}
\boldsymbol{\sigma}\left(r\right):=\sqrt{N}\left(0,...,0,\sqrt{1-r^{2}},r\right),\label{eq:sig_r-1}
\end{equation}
and $\varphi_{\grad H_{N}\left(q_{1}\hat{\mathbf{n}}\right),\grad H_{N}\left(q_{2}\boldsymbol{\sigma}\left(r\right)\right)}$
is the joint Gaussian density of $\grad H_{N}\left(q_{1}\hat{\mathbf{n}}\right)$
and $\grad H_{N}\left(q_{2}\boldsymbol{\sigma}\left(r\right)\right)$.
This has been worked out in \cite[Lemma 11]{2nd} for pure spherical
models and $q_{1}=q_{2}=1$, $B_{1}=B_{2}$ and $D_{1}=D_{2}=\mathbb{R}$,
but the proof in the mixed case is similar.

The following three lemmas, generalizing \cite[Lemmas 12 and 13]{2nd}
to the mixed case, are concerned with the joint law of the random
variables appearing in (\ref{eq:KR-2nd}). Their computationally heavy
proof is postponed to Appendix \ref{sec:Covariances}.
\begin{lem}
\label{lem:dens}(Density of gradients) For any $r\in\left(-1,1\right)$
and $q_{1},q_{2}\in(0,1]$, there exists a choice of the orthonormal
frame field  $F=\left(F_{i}\right)_{i=1}^{N-1}$, such that the density
of $\left(\grad H_{N}\left(q_{1}\hat{\mathbf{n}}\right),\grad H_{N}\left(q_{2}\boldsymbol{\sigma}\left(r\right)\right)\right)$
at $\left(0,0\right)\in\mathbb{R}^{N-1}\times\mathbb{R}^{N-1}$ is
given by 
\begin{align}
 & \varphi_{\grad H_{N}\left(q_{1}\hat{\mathbf{n}}\right),\grad H_{N}\left(q_{2}\boldsymbol{\sigma}\left(r\right)\right)}\left(0,0\right)\label{eq:grad_dens}\\
 &= \left(2\pi\right)^{-\left(N-1\right)}\left[\nu'(q_{1}^{2})\nu'(q_{2}^{2})-(\nu'(q_{1}q_{2}r))^{2}\right]^{-\frac{N-2}{2}}
 \left[\nu'(q_{1}^{2})\nu'(q_{2}^{2})-\left(r\nu'(q_{1}q_{2}r)-q_{1}q_{2}\nu''(q_{1}q_{2}r)(1-r^{2})\right)^{2}\right]^{-\frac{1}{2}}\,. \nonumber
\end{align}
\end{lem}

\begin{lem}
\label{lem:condHamiltonian}(Conditional law of Hamiltonians and normal
derivatives) \corO{With notation as in Lemma \ref{lem:dens}, }
 conditional on 
\begin{equation}
\left(\grad H_{N}\left(q_{1}\hat{\mathbf{n}}\right),\grad H_{N}\left(q_{2}\boldsymbol{\sigma}\left(r\right)\right)\right)=\left(0,0\right),\label{eq:grad=00003D0}
\end{equation}
the vector 
\begin{equation}
\left(\frac{1}{\sqrt{N}}H_{N}\left(q_{1}\hat{\mathbf{n}}\right),\frac{1}{\sqrt{N}}H_{N}\left(q_{2}\boldsymbol{\sigma}\left(r\right)\right),\frac{d}{dR}H_{N}\left(q_{1}\hat{\mathbf{n}}\right),\frac{d}{dR}H_{N}\left(q_{2}\boldsymbol{\sigma}\left(r\right)\right)\right)\label{eq:0702-01}
\end{equation}
is a centered Gaussian vector with covariance matrix $\Sigma_{U,X}\left(r,q_{1},q_{2}\right)$
(see (\ref{eq:SigmaUXbar})).
\end{lem}

\begin{lem}
\label{lem:Hess_struct_2}(Conditional law of Hessians) 
\corO{With notation as in Lemma \ref{lem:dens}, }
conditional on (\ref{eq:grad=00003D0})
the joint distribution of the matrices 
\begin{align*}
 & \left(\sqrt{\frac{N}{\left(N-1\right)\nu''(q_{1}^{2})}}\left(\Hess H_{N}\left(q_{1}\hat{\mathbf{n}}\right)+\frac{1}{\sqrt{N}q_{1}}\ddq H_{N}\left(q_{1}\hat{\mathbf{n}}\right)\mathbf{I}\right),\right.\\
 & \left.\sqrt{\frac{N}{\left(N-1\right)\nu''(q_{2}^{2})}}\left(\Hess H_{N}\left(q_{2}\boldsymbol{\sigma}\left(r\right)\right)+\frac{1}{\sqrt{N}q_{2}}\ddq H_{N}\left(q_{2}\boldsymbol{\sigma}\left(r\right)\right)\mathbf{I}\right)\right)
\end{align*}
is the same as that of
\[
\left(\mathbf{A}_{N-1}^{\left(1\right)}\left(r,q_{1},q_{2}\right),\,\mathbf{A}_{N-1}^{\left(2\right)}\left(r,q_{1},q_{2}\right)\right),
\]
 where 
\begin{align}
&\mathbf{A}_{N-1}^{\left(i\right)}\left(r,q_{1},q_{2}\right) =\hat{\mathbf{M}}_{N-1}^{\left(i\right)}\left(r,q_{1},q_{2}\right)+\sqrt{\frac{N}{\left(N-1\right)\nu''(q_{i}^{2})}}m_{i}\left(r,q_{1},q_{2}\right)e_{N-1,N-1},\label{eq:100}\\
&m_{i}\left(r,q_{1},q_{2}\right) =\frac{1}{N}\left(1-r^{2}\right)
\left(H_{N}\left(q_{1}\hat{\mathbf{n}}\right),H_{N}\left(q_{2}\boldsymbol{\sigma}\left(r\right)\right),\ddq H_{N}\left(q_{1}\hat{\mathbf{n}}\right),\ddq H_{N}\left(q_{2}\boldsymbol{\sigma}\left(r\right)\right)\right)\Sigma_{U,X}^{-1}\left(r,q_{1},q_{2}\right)\boldsymbol{\varsigma}_{i}\left(r,q_{1},q_{2}\right),\nonumber 
\end{align}
$e_{N-1,N-1}$ is \corE{the} $N-1\times N-1$ matrix whose $N-1,N-1$ entry
is equal to $1$ and all other entries \corE{are $0$}, $\Sigma_{U,X}\left(r,q_{1},q_{2}\right)$
and $\boldsymbol{\varsigma}_{i}\left(r,q_{1},q_{2}\right)$ are given
by (\ref{eq:SigmaUXbar}) and (\ref{eq:varsigma}), and $\hat{\mathbf{M}}_{N-1}^{\left(i\right)}\left(r,q_{1},q_{2}\right)$
are $N-1\times N-1$ Gaussian random matrices independent of (\ref{eq:0702-01})
with block structure
\begin{align}
\hat{\mathbf{M}}_{N-1}^{\left(i\right)}\left(r,q_{1},q_{2}\right) & =\left(\begin{array}{cc}
\hat{\mathbf{G}}_{N-2}^{\left(i\right)}\left(r\right) & Z^{\left(i\right)}\left(r\right)\\
\left(Z^{\left(i\right)}\left(r\right)\right)^{T} & Q^{\left(i\right)}\left(r\right)
\end{array}\right),\label{eq:Mhat}
\end{align}
satisfying the following: 

\begin{enumerate}
\item $\big(\hat{\mathbf{G}}_{N-2}^{\left(1\right)}\left(r\right),\hat{\mathbf{G}}_{N-2}^{\left(2\right)}\left(r\right)\big)$,
$\left(Z^{\left(1\right)}\left(r\right),Z^{\left(2\right)}\left(r\right)\right)$,
and $\left(Q^{\left(1\right)}\left(r\right),Q^{\left(2\right)}\left(r\right)\right)$
are independent.
\item $\hat{\mathbf{G}}^{\left(i\right)}\left(r\right)=\hat{\mathbf{G}}_{N-2}^{\left(i\right)}\left(r\right)$
are $N-2\times N-2$ random matrices such that $\sqrt{\frac{N-1}{N-2}}\hat{\mathbf{G}}^{\left(i\right)}\left(r\right)$
is a GOE matrix and, in distribution,
\[
\left(\begin{array}{c}
\vphantom{\left(1-\frac{\left|\nu''\left(q_{1}q_{2}r\right)\right|}{\sqrt{\nu''\left(q_{1}^{2}\right)\nu''\left(q_{2}^{2}\right)}}\right)^{1/2}}\hat{\mathbf{G}}^{\left(1\right)}\left(r\right)\\
\vphantom{\left(1-\frac{\left|\nu''\left(q_{1}q_{2}r\right)\right|}{\sqrt{\nu''\left(q_{1}^{2}\right)\nu''\left(q_{2}^{2}\right)}}\right)^{1/2}}\hat{\mathbf{G}}^{\left(2\right)}\left(r\right)
\end{array}\right)=\left(\begin{array}{c}
\left(1-\frac{\left|\nu''\left(q_{1}q_{2}r\right)\right|}{\sqrt{\nu''\left(q_{1}^{2}\right)\nu''\left(q_{2}^{2}\right)}}\right)^{1/2}\bar{\mathbf{G}}^{\left(1\right)}+{\rm sgn}\left(\nu''\left(q_{1}q_{2}r\right)\right)\left(\frac{\left|\nu''\left(q_{1}q_{2}r\right)\right|}{\sqrt{\nu''\left(q_{1}^{2}\right)\nu''\left(q_{2}^{2}\right)}}\right)^{1/2}\bar{\mathbf{G}}\\
\left(1-\frac{\left|\nu''\left(q_{1}q_{2}r\right)\right|}{\sqrt{\nu''\left(q_{1}^{2}\right)\nu''\left(q_{2}^{2}\right)}}\right)^{1/2}\bar{\mathbf{G}}^{\left(2\right)}+\left(\frac{\left|\nu''\left(q_{1}q_{2}r\right)\right|}{\sqrt{\nu''\left(q_{1}^{2}\right)\nu''\left(q_{2}^{2}\right)}}\right)^{1/2}\bar{\mathbf{G}}
\end{array}\right),
\]
where $\bar{\mathbf{G}}=\bar{\mathbf{G}}_{N-2}$, $\bar{\mathbf{G}}^{\left(1\right)}=\bar{\mathbf{G}}_{N-2}^{\left(1\right)}$,
and $\bar{\mathbf{G}}^{\left(2\right)}=\bar{\mathcal{\mathbf{G}}}_{N-2}^{\left(2\right)}$
are independent \corO{of each other} and have the same law as $\hat{\mathbf{G}}^{\left(i\right)}\left(r\right)$, \corO{that is, scaled GOE}. 
\item $Z^{\left(i\right)}\left(r\right)=\left(Z_{j}^{\left(i\right)}\left(r\right)\right)_{j=1}^{N-2}$
are Gaussian vectors such that $\left(Z_{j}^{\left(1\right)}\left(r\right),\,Z_{j}^{\left(2\right)}\left(r\right)\right)$
are independent for different $j$ and 
\[
\left(Z_{j}^{\left(1\right)}\left(r\right),\,Z_{j}^{\left(2\right)}\left(r\right)\right)\sim N\left(0,\,\frac{1}{\left(N-1\right)}\Sigma_{Z}\left(r,q_{1},q_{2}\right)\right),
\]
where $\Sigma_{Z}\left(r\right)$ is given in (\ref{eq:SigmaZ}).
\item With $\Sigma_{Q}\left(r\right)$ given by (\ref{eq:SigmaQ}),
\[
\left(Q^{\left(1\right)}\left(r\right),\,Q^{\left(2\right)}\left(r\right)\right)\sim N\left(0,\,\frac{1}{\left(N-1\right)}\Sigma_{Q}\left(r,q_{1},q_{2}\right)\right).
\]
\end{enumerate}
\end{lem}
By  Lemma \ref{lem:condHamiltonian}, conditional on (\ref{eq:grad=00003D0}),
the vector (\ref{eq:0702-01}) has the same distribution as 
\begin{equation}
\left(U_{1}\left(r\right),U_{2}\left(r\right),X_{1}\left(r\right),X_{2}\left(r\right)\right)\sim N\left(0,\Sigma_{U,X}\left(r,q_{1},q_{2}\right)\right).\label{eq:2102-01}
\end{equation}
From (\ref{eq:KR-2nd}), Lemmas \ref{lem:dens} and \ref{lem:Hess_struct_2}
and some calculus, 
\begin{equation}
\mathbb{E} \left[\mbox{Crt}_{N,q_{1},q_{2}}(B_{1},B_{2},D_{1},D_{2},I)\right]_{2}
=C_{N}\int_{I}dr\cdot\left(\corO{\mathcal{D}}\left(r\right)\right)^{N-3}\mathcal{F}\left(r\right)\mathbb{E}\left\{ \prod_{i=1,2}\left|\det\left(\mathcal{\mathbf{M}}_{N-1}^{\left(i\right)}\left(r\right)\right)\right|\mathbf{1}_{E_{i}}\right\} ,\label{eq:KR_conditional}
\end{equation}
where 
\begin{align}
E_{i} & =\Big\{ U_{i}\left(r\right)\in\sqrt{N}B_{i},\,X_{i}\left(r\right)\in\sqrt{N}D_{i}\Big\},\nonumber \\
\mathcal{\mathbf{M}}_{N-1}^{\left(i\right)}\left(r\right) & =\mathcal{\mathbf{A}}_{N-1}^{\left(i\right)}\left(r\right)-\sqrt{\frac{1}{\left(N-1\right)\nu''(q_{i}^{2})}}\frac{1}{q_{i}}X_{i}\left(r\right)\mathbf{I},\label{eq:1903-01}
\end{align}
 $\mathcal{\mathbf{A}}_{N-1}^{\left(i\right)}\left(r\right):=\mathcal{\mathbf{A}}_{N-1}^{\left(i\right)}\left(r,q_{1},q_{2}\right)$
are defined by (\ref{eq:100}) and assumed to be independent of (\ref{eq:2102-01}),
and
\begin{align}
C_{N} & =\omega_{N}\omega_{N-1}\left(\frac{1}{2\pi}\left(N-1\right)q_{1}q_{2}\sqrt{\frac{\nu''(q_{1}^{2})\nu''(q_{2}^{2})}{\nu'(q_{1}^{2})\nu'(q_{2}^{2})}}\right)^{N-1},\nonumber \\
\corO{\mathcal{D}}\left(r\right) & =\left(1-r^{2}\right)^{\frac{1}{2}}\left(1-\frac{(\nu'(q_{1}q_{2}r))^{2}}{\nu'(q_{1}^{2})\nu'(q_{2}^{2})}\right)^{-\frac{1}{2}},\label{eq:cgf}\\
\mathcal{F}\left(r\right) & =\left(1-\frac{(\nu'(q_{1}q_{2}r))^{2}}{\nu'(q_{1}^{2})\nu'(q_{2}^{2})}\right)^{-\frac{1}{2}}\left(1-\left(\frac{r\nu'(q_{1}q_{2}r)-q_{1}q_{2}\nu''(q_{1}q_{2}r)(1-r^{2})}{\sqrt{\nu'(q_{1}^{2})\nu'(q_{2}^{2})}}\right)^{2}\right)^{-\frac{1}{2}}.\nonumber 
\end{align}

Next, we relate the determinant of $\mathcal{\mathbf{M}}_{N-1}^{\left(i\right)}\left(r\right)$
to that of its $N-2\times N-2$ upper-left submatrix, which we denote
by $\mathbf{G}_{N-2}^{\left(i\right)}\left(r\right)$. With $\hat{\mathbf{G}}_{N-2}^{\left(i\right)}\left(r\right)$
as  in (\ref{eq:Mhat}) we have 
\begin{equation}
\mathbf{G}_{N-2}^{\left(i\right)}\left(r\right)=\hat{\mathbf{G}}_{N-2}^{\left(i\right)}\left(r\right)-\sqrt{\frac{1}{\left(N-1\right)\nu''(q_{i}^{2})}}\frac{1}{q_{i}}X_{i}\left(r\right)\mathbf{I}.\label{eq:M}
\end{equation}
Set 
\begin{equation}
W_{i}\left(r\right)=W_{i,N}\left(r\right):=\left(2\sum_{j=1}^{N-2}\left(\left(\mathbf{M}_{N-1}^{\left(i\right)}\left(r\right)\right)_{j,N-1}\right)^{2}+\left(\left(\mathbf{M}_{N-1}^{\left(i\right)}\left(r\right)\right)_{N-1,N-1}\right)^{2}\right)^{1/2}.\label{eq:W}
\end{equation}

For any $\kappa>\epsilon>0$ define
\[
h_{\epsilon}\left(x\right)=\max\left\{ \epsilon,x\right\} ,
\]
and
\begin{equation}
h_{\epsilon}^{\kappa}\left(x\right)=\begin{cases}
\epsilon & \,\,\mbox{if }x<\epsilon,\\
x & \,\,\mbox{if }x\in\left[\epsilon,\kappa\right],\\
\kappa & \,\,\mbox{if }x>\kappa,
\end{cases}\,\,\,\mbox{and}\,\,\,h_{\kappa}^{\infty}\left(x\right)=\begin{cases}
1 & \mbox{if }x\leq\kappa,\\
x/\kappa & \mbox{if }x>\kappa,
\end{cases}\label{eq:h_trunc}
\end{equation}
so that $h_{\epsilon}^{\kappa}\left(x\right)h_{\kappa}^{\infty}\left(x\right)=h_{\epsilon}\left(x\right)$.
Lastly, define 
\begin{align}
\log_{\epsilon}^{\kappa}\left(x\right) & =\log\left(h_{\epsilon}^{\kappa}\left(x\right)\right).\label{eq:log trunc}
\end{align}
For a general real symmetric matrix $\mathbf{C}$ let $\lambda_{j}\left(\mathbf{C}\right)$
denote the eigenvalues of $\mathbf{C}$.

By exactly the same proof as for Lemma 14 of \cite{2nd} (which does
not involve probabilistic arguments) we have that for any $\epsilon>0$,
$r\in\left(-1,1\right)$, almost surely, 
\begin{align}
\left|\det\left(\mathcal{\mathbf{M}}_{N-1}^{\left(i\right)}\left(r\right)\right)\right| & \leq\frac{W_{i}\left(r\right)\left(W_{i}\left(r\right)+\epsilon\right)}{\epsilon}\prod_{j=1}^{N-2}h_{\epsilon}\left(\left|\lambda_{j}\left(\mathbf{G}_{N-2}^{\left(i\right)}\left(r\right)\right)\right|\right).\label{eq:0908-1}
\end{align}

\corO{With  $\kappa>\epsilon$ and $2\leq m\in\mathbb{N}$ arbitrary, set
 $t=t\left(m\right):=m/\left(m-1\right)$, and }
\begin{align}
\mathcal{E}_{\epsilon,\kappa}^{\left(1\right)}\left(r\right) & =\mathbb{E}\left\{ \prod_{i=1,2}\prod_{j=1}^{N-2}\left(h_{\epsilon}^{\kappa}\left(\left|\lambda_{j}\left(\mathbf{G}_{N-2}^{\left(i\right)}\left(r\right)\right)\right|\right)\right)^{t}\cdot\mathbf{1}_{E_{i}}\right\} ,\nonumber \\
\mathcal{E}_{\epsilon,\kappa}^{\left(2\right)}\left(r\right) & =\mathbb{E}\left\{ \prod_{i=1,2}\prod_{j=1}^{N-2}\left(h_{\kappa}^{\infty}\left(\left|\lambda_{j}\left(\mathbf{G}_{N-2}^{\left(i\right)}\left(r\right)\right)\right|\right)\right)^{2m}\right\} ,\label{eq:84}\\
\mathcal{E}_{\epsilon,\kappa}^{\left(3\right)}\left(r\right) & =\mathbb{E}\left\{ \left(\frac{W_{1}\left(r\right)\left(W_{1}\left(r\right)+\epsilon\right)}{\epsilon}\right)^{4m}\right\} \mathbb{E}\left\{ \left(\frac{W_{2}\left(r\right)\left(W_{2}\left(r\right)+\epsilon\right)}{\epsilon}\right)^{4m}\right\} .\nonumber 
\end{align}
Then, from (\ref{eq:KR_conditional}), (\ref{eq:0908-1}) and H{\"{o}}lder's
inequality we have that
\[
\mathbb{E}\left\{ \prod_{i=1,2}\left|\det\left(\mathcal{\mathbf{M}}_{N-1}^{\left(i\right)}\left(r\right)\right)\right|\mathbf{1}_{E_{i}}\right\} \leq\left(\mathcal{E}_{\epsilon,\kappa}^{\left(1\right)}\left(r\right)\right)^{\nicefrac{1}{t}}\left(\mathcal{E}_{\epsilon,\kappa}^{\left(2\right)}\left(r\right)\right)^{\nicefrac{1}{2m}}\left(\mathcal{E}_{\epsilon,\kappa}^{\left(3\right)}\left(r\right)\right)^{\nicefrac{1}{4m}}
\]
and
\begin{align}
 & \negthickspace\negthickspace\negthickspace\negthickspace\limsup_{N\to\infty}\frac{1}{N}\log\left(\mathbb{E}\left\{ \left[\mbox{Crt}_{N,q_{1},q_{2}}(B_{1},B_{2},D_{1},D_{2},I)\right]_{2}\right\} \right)\leq\limsup_{N\to\infty}\frac{1}{N}\log\left(C_{N}\right)\nonumber \\
 & +\limsup_{N\to\infty}\frac{1}{mN}\log\left(\int_{I}\left(\mathcal{F}\left(r\right)\right)^{m}\left(\mathcal{E}_{\epsilon,\kappa}^{\left(2\right)}\left(r\right)\right)^{\nicefrac{1}{2}}\left(\mathcal{E}_{\epsilon,\kappa}^{\left(3\right)}\left(r\right)\right)^{\nicefrac{1}{4}}dr\right)\label{eq:45}\\
 & +\limsup_{N\to\infty}\frac{1}{tN}\log\left(\int_{I}\left(\corO{\mathcal{D}}\left(r\right)\right)^{t(N-3)}\mathcal{E}_{\epsilon,\kappa}^{\left(1\right)}\left(r\right)dr\right)=:\Delta_{I}+\Delta_{II}+\Delta_{III}.\nonumber 
\end{align}

We note that $\Delta_{I}=1+\log\left(q_{1}q_{2}\sqrt{\frac{\nu''(q_{1}^{2})\nu''(q_{2}^{2})}{\nu'(q_{1}^{2})\nu'(q_{2}^{2})}}\right)$.
To complete \corO{the proof}  of Theorem \ref{thm:2ndmomUBBK-1-1} we will show that
$\Delta_{II}\leq0$, if $\kappa$ is large enough, and that 
\begin{equation}
\lim_{\epsilon\to0}\lim_{\kappa\to\infty}\lim_{m\to\infty}\Delta_{III}=\sup_{r\in I,u_{i}\in B_{i},x_{i}\in D_{i}}\Psi_{\nu,q_{1},q_{2}}\left(r,u_{1},u_{2},x_{1},x_{2}\right)-\Delta_{I}.\label{eq:0403-01}
\end{equation}
By a similar proof to that of Lemma 16 of \cite{2nd} (essentially
all that is needed is to replace $\bar{U}_{i}\left(r\right)$ by $\sqrt{\frac{1}{\left(N-1\right)\nu''(q_{i}^{2})}}\frac{1}{q_{i}}X_{i}\left(r\right)$
everywhere in the proof), using the large deviation principle satisfied
by the empirical measure of eigenvalues \cite[Theorem 2.1.1]{BAG97}
and the upper bound for top eigenvalue \cite[Lemma 6.3]{BDG} of GOE
matrices,  we have
the following two inequalities:
\begin{enumerate}
\item For all
  $\corE{t,\epsilon>0}$ and $\kappa>\max\left\{ \epsilon,1\right\} $
there exists a constant $c=c\left(\epsilon,\kappa\right)>0$, such
that for 
\corE{any intervals $B_i\subset\mathbb{R}$ and}
large enough $N$, uniformly in $r\in\left(-1,1\right)$,
\begin{align}
 & \mathbb{E}\left\{ \prod_{i=1,2}\prod_{j=1}^{N-2}\left(h_{\epsilon}^{\kappa}\left(\left|\lambda_{j}\left(\mathbf{G}_{N-2}^{\left(i\right)}\left(r\right)\right)\right|\right)\right)^{t}\cdot\mathbf{1}_{E_{i}}\right\} \leq\exp\left\{ -cN^{2}\right\}\label{eq:40}\\
& +\exp\left\{ 2t\epsilon N\right\}
\mathbb{E}\left\{ \mathbf{1}_{E_{1}}\mathbf{1}_{E_{2}}\exp\left\{ \sum_{i=1,2}tN\int\log_{\epsilon}^{\kappa}\left(\left|\lambda-\sqrt{\frac{1}{\left(N-1\right)\nu''(q_{i}^{2})}}\frac{1}{q_{i}}X_{i}\left(r\right)\right|\right)d\mu^{*}\corO{(\lambda)}\right\} \right\} ,\nonumber 
\end{align}
where $\mu^{*}$ is the semicircle law, given by (\ref{eq:semicirc}).
\item \corE{For any $m\geq 1$, there exists some large  $\kappa>0$, so that uniformly in $r\in\left(-1,1\right)$ and $N$},
\begin{equation}
\mathbb{E}\left\{ \prod_{i=1,2}\prod_{j=1}^{N-2}\left(h_{\kappa}^{\infty}\left(\left|\lambda_{j}\left(\mathbf{G}_{N-2}^{\left(i\right)}\left(r\right)\right)\right|\right)\right)^{\corO{2m}}\right\} \leq \corE{2}.\label{eq:44}
\end{equation}
\end{enumerate}
The equality \corO{in} (\ref{eq:0403-01}) \corO{follows from} (\ref{eq:40})
and Varadhan's integral lemma \cite[Theorem 4.3.1, Exercise 4.3.11]{LDbook}.
(See the proof of \cite[Theorem 5]{2nd} for details.)

\corO{It remains to consider $\Delta_{II}$.} By (\ref{eq:44}), for fixed $m$ and large enough $\kappa$,
\begin{equation}
\Delta_{II}\leq\limsup_{N\to\infty}\frac{1}{mN}\log\left(\int_{\corO{I}}\left(\mathcal{F}\left(r\right)\right)^{m}\left(\mathcal{E}_{\epsilon,\kappa}^{\left(3\right)}\left(r\right)\right)^{\nicefrac{1}{4}}dr\right).\label{eq:0304-02}
\end{equation}
\corO{We control the right hand side of \eqref{eq:0304-02} differently according to whether $q_1=q_2$ or not.} 

\corO{Assume first that
  $q_1=q_2=q$. In that case, we will show that the integrand
in (\ref{eq:0304-02}) is bounded uniformly in $r\in(-1,1)$ by some \corO{(possibly $m$-dependent)} 
constant independent of $N$.} 
\corO{We have to be particularly careful with the limit as $|r|\to 1$ since
$\left(\mathcal{F}\left(r\right)\right)^{m}$ can
explode as $|r|\to1$}. 
Note that
\[
\lim_{r\to-1}(\nu'(q^{2}r))^{2}=(\nu'(q^{2}))^{2}
\]
if and only if $\nu$ is either an even or odd polynomial. Thus, only
in this case $\lim_{r\to-1}\mathcal{F}\left(r\right)=\infty$. Therefore,
to complete the proof it is sufficient to show that for $q_{1}=q_{2}=q$,
\begin{equation}
\limsup_{r\nearrow1}\left(\mathcal{F}\left(r\right)\right)^{m}\left(\mathcal{E}_{\epsilon,\kappa}^{\left(3\right)}\left(r\right)\right)^{\nicefrac{1}{4}}<\infty,\label{0304-03}
\end{equation}
and in the case that $|\nu(r)|=|\nu(-r)|$, the same holds for the
$r\searrow-1$ limit.

By the same proof as in Lemma 15 of \cite{2nd}, if $\chi_{N-1}=\sum_{i=1}^{N-1}X_{i}^{2}$,
$X_{i}\sim N(0,1)$ i.i.d., is a Chi-squared variable of $N-1$ degrees
of freedom, then
\begin{align}
\mathbb{E}\left\{ \left(W_{i}\left(r\right)\right)^{2m}\right\}  & \leq\left(2V(r)\right)^{m}\mathbb{E}\left\{ \chi_{N-1}^{m}\right\} \label{eq:0304-04}\\
 & =\left(2V(r)\right)^{m}(N-1)(N+1)\cdots(N-3+2m),\nonumber 
\end{align}
\corE{where $V(r)$ is the maximum of the  variance of $(\mathbf{M}_{N-1}^{\left(i\right)}\left(r\right))_{1,N-1}$
and of $(\mathbf{M}_{N-1}^{\left(i\right)}\left(r\right))_{N-1,N-1}$. By (\ref{eq:cov1}), the former variance is equal to $(N-1)^{-1}\Sigma_{Z,11}\left(r,q,q\right)$ (see (\ref{eq:SigmaZ})).
By Lemmas \ref{lem:condHamiltonian} and \ref{lem:Hess_struct_2}, the latter variance is equal to the conditional variance of the $N-1,\,N-1$ entry of  $\sqrt{\frac{N}{\left(N-1\right)\nu''(q_{2}^{2})}}\Hess H_{N}\left(q_{2}\boldsymbol{\sigma}\left(r\right)\right)$ conditioned on  $\grad H_{N}\left(q\hat{\mathbf{n}}\right)$ and $\grad H_{N}\left(q\boldsymbol{\sigma}\left(r\right)\right)$.
By the same calculation by which we arrive to (\ref{eq:cov1}), this
conditional variance is equal to}  
\begin{align*}
 & (N-1)^{-1}\left[q^{4}\nu''''\left(q^{2}r\right)\left(1-r^{2}\right)^{2}-6q^{2}\nu'''\left(q^{2}r\right)r\left(1-r^{2}\right)\right.\\
 & +\nu''\left(q^{2}r\right)\left(3r^{2}-4\left(1-r^{2}\right)\right)+rq^{-2}\nu'\left(q^{2}r\right)-a_{2}(r,q,q)\left(1-r^{2}\right)\\
 & \left.\times\left(-q^{3}\nu'''\left(q^{2}r\right)\left(1-r^{2}\right)+3rq\nu''\left(q^{2}r\right)+q^{-1}\nu'\left(q^{2}r\right)\right)^{2}\right].
\end{align*}
By some calculus, one can verify that each of those variances multiplied
by $(N-1)(1-r)^{-1}$ converges as $r\nearrow1$ to a constant. On
the other hand, as $r\nearrow1$, $\mathcal{F}\left(r\right)(1-r)$
converges to a positive constant. Similarly, when $|\nu(r)|=|\nu(-r)|$
the same convergences hold $r\searrow-1$ with $1-r$ replaced by
$1+r$. Combined with (\ref{eq:0304-04}) this implies (\ref{0304-03})
and the corresponding $r\searrow-1$ limit, when needed, and completes
the proof \corO{in case $q_1=q_2$}.

Assume next that $q_1\neq q_2$. In that case, \corO{the argument 
  is simpler, since}
$\mathcal{F}(r)$ is uniformly bounded in $r$, 
while $W_i(r)^2$ \corOU{has  the
law of a sum of $N-1$ squares 
of Gaussian variables, conditioned on 
the vanishing of the spherical gradient, compare with
\eqref{eq:0304-04}. Before the conditioning, these variables are independent
and 
each of them  has variance uniformly bounded
by a multiple of $1/(N-1)$, and the variance of the sum
after the conditioning is not larger than the variance before conditioning.}
\corO{Therefore, $\mathcal{E}_{\epsilon,\kappa}^{\left(3\right)}
\leq \corE{C(m)}$, and thus $\Delta_{II}\leq 0$ if $q_1\neq q_2$. This completes
the proof in the case $q_1\neq q_2$, and
thus the proof of Theorem \ref{subsec:pf_2ndmoment} is complete.}
\qed

\subsection{Proof of Lemma \ref{lem:2ndLB}}

The lemma will follow from (\ref{eq:KR_conditional}) if we can show
that for $x_{i}\in D_{i}$, $r_{0}\in(-1,1)$, and $\epsilon_{N}\to0$
slowly enough, say $\epsilon_{N}=N^{-1}$, uniformly in $r\in(r_{0}-\epsilon_{N},r_{0}+\epsilon_{N})$,
\begin{align}
\liminf_{N\to\infty}\frac{1}{N}\log\mathbb{E}\left\{ \left.\prod_{i=1,2}\left|\det\left(\mathcal{\mathbf{M}}_{N-1}^{\left(i\right)}\left(r\right)\right)\right|\,\right|\,E^{\prime}\right\}  & \geq\Omega\left(\frac{x_{1}}{q_{1}\sqrt{\nu''(q_{1}^{2})}}\right)+\Omega\left(\frac{x_{2}}{q_{2}\sqrt{\nu''(q_{2}^{2})}}\right),\label{eq:1903-02}
\end{align}
where
\[
E^{\prime}=\left\{ \forall i=1,2:\,U_{i}\left(r\right)\in\sqrt{N}(u_{i}-\epsilon_{N},u_{i}+\epsilon_{N}),\,X_{i}\left(r\right)\in\sqrt{N}(x_{i}-\epsilon_{N},x_{i}+\epsilon_{N})\right\} ,
\]
since by (\ref{eq:2102-01}),
\[
\liminf_{N\to\infty}\frac{1}{N}\log\mathbb{P}\left\{ E_{i}^{\prime}\right\} \geq-\frac{1}{2}\left(u_{1},u_{2},x_{1},x_{2}\right)\Sigma_{U,X}^{-1}\left(r,q_{1},q_{2}\right)\left(u_{1},u_{2},x_{1},x_{2}\right)^{T}.
\]
By (\ref{eq:1903-01}), (\ref{eq:100}) and (\ref{eq:Mhat}), the
matrix $\mathcal{\mathbf{M}}_{N-1}^{\left(i\right)}\left(r\right)$
is of the form 
\begin{equation}
\mathcal{\mathbf{M}}_{N-1}^{\left(i\right)}\left(r\right)=\left(\begin{array}{cc}
\hat{\mathbf{G}}_{N-2}^{\left(i\right)}\left(r\right) & 0\\
0 & 0
\end{array}\right)-\sqrt{\frac{1}{\left(N-1\right)\nu''(q_{i}^{2})}}\frac{1}{q_{i}}X_{i}\left(r\right)\mathbf{I}+\mathbf{T}^{(i)},\label{eq:1903-03}
\end{equation}
where $\hat{\mathbf{G}}_{N-2}^{\left(i\right)}\left(r\right)$ is
a GOE matrix and $\mathbf{T}^{(i)}$ is a matrix whose only nonzero
elements are in the last column and row. 

Let $\tilde{\mathcal{\mathbf{M}}}_{N-2}^{\left(i\right)}\left(r\right)$
be the upper-left $N-2\times N-2$ submatrix of $\mathcal{\mathbf{M}}_{N-1}^{\left(i\right)}\left(r\right)$
and let  $\delta>0$ be arbitrary. 
From our assumption on $D_i$ and \corO{the convergence of the
      top eigenvalue of Wigner matrices, see \cite[Theorem 2.1.22]{Matrices},} 
      the eigenvalues
$\lambda_{j}^{(i)}$ of $\tilde{\mathcal{\mathbf{M}}}_{N-2}^{\left(i\right)}\left(r\right)$
are smaller than some $-\tau'<0$, independent of $N$, \corO{with
probability approaching $1$ as $N\to\infty$. Therefore,}
  by \corO{Wigner's theorem, see \cite[Theorem 2.1.1]{Matrices}, 
again with probability approaching $1$ as $N\to\infty$,}
\[
\frac{1}{N}\log\left|
\det\tilde{\mathcal{\mathbf{M}}}_{N-2}^{\left(i\right)}\left(r\right)
\right|=F^{(i)}\geq\Omega\Bigg(\frac{x_{i}}{q_{i}\sqrt{\nu''(q_{1}^{2})}}\Bigg)-\delta,
\]
where $F^{(i)}=\frac{1}{N}\sum_{j=1}^{N-2}\log|\lambda_{j}^{(i)}|$.
\corO{Finally},
since the variance of $X_{i}\left(r\right)$ is bounded \corE{from}
below by some $c=c(r_{0})>0$ uniformly in 
$r\in(r_{0}-\epsilon_{N},r_{0}+\epsilon_{N})$,
\[ \frac{1}{N}\log\left(\det\mathcal{\mathbf{M}}_{N-1}^{\left(i\right)}\left(r\right)/\det\tilde{\mathcal{\mathbf{M}}}_{N-2}^{\left(i\right)}\left(r\right)\right)=\frac{1}{N}\log\left(Z-V^{T}(\tilde{\mathcal{\mathbf{M}}}_{N-2}^{\left(i\right)}\left(r\right))^{-1}V\right)\in(-\delta,\delta),\]
\corO{again with probability approaching $1$ as $N\to\infty$,}
where $V$ is the vector composed of the first $N-2$ elements of
the last column of $\mathcal{\mathbf{M}}_{N-1}^{\left(i\right)}\left(r\right)$
and $Z$ is the $N-1,\,N-1$ element of $\mathcal{\mathbf{M}}_{N-1}^{\left(i\right)}\left(r\right)$.
This completes the proof.\qed
\section{\label{sec:Matching}Matching of moments and orthogonality: proofs
of Theorem \ref{theo-ground}, Lemma \ref{lem:conccentration} and Corollaries \ref{cor:matching}
and \ref{cor:orth}}
\corO{In this section we use the second moment computations of Section \ref{sec:Moments} to show that the ground state is 
determined by a first moment computation. We also show that deep $q$-critical points are nearly orthogonal.
We begin with stating and proving several consequences of \CDTN. This is then followed by the proofs of the statements in the title of the section.}

 \subsection{\label{sec:CDTNconseq} Consequences of \CDTN} 
For pure-like models, $-x_{0}(q)$ was defined as the maximizer of
the complexity $\Theta_{\nu,q}(-E_{0}(q),x)$. \corO{As the next lemma shows,} this definition makes
sense.
\begin{lem}
\label{lem:x0}If $\nu_{q}$ is pure-like, then there exists a unique
$x_{0}(q)$ such that 
\[
-x_{0}(q)=\arg\max_{x\in\mathbb{R}}\Theta_{\nu,q}(-E_{0}(q),x),
\]
and it satisfies $-\frac{x_{0}(q)}{q\sqrt{\nu''(q)}}<-2$.
\end{lem}

Next we state several auxiliary lemmas 
regarding $E_{0}(q)$ and $x_{0}(q)$, which will \corE{be} needed later.
\begin{lem}
	\label{lem:ddE}For  $q$  such that $\nu_{q}$ is pure-like, $(E,x)\mapsto\Theta_{\nu,q}(E,x)$
	is strictly concave on $\{(E,x):\,x<-2q\sqrt{\nu''(q^{2})}\}$. Moreover,
	for any $E$ and $x<-2q\sqrt{\nu''(q^{2})}$ ,  \corO{if $\Theta_{\nu,q}(E,x)=0$ then $\frac{d}{dE}\Theta_{\nu,q}(E,x)\neq0$.}
\end{lem}
\corO{Note that the conclusion of Lemma \ref{lem:ddE} holds if  $E=-E_{0}(q)$ and
$x=-x_{0}(q)$. A useful corollary of Lemma \ref{lem:ddE} is the following.} 
\begin{lem}
	\label{lem:Exsmooth}For any $q$ such that $\nu_q$ is pure-like, $E_0(q)$ and $x_0(q)$ are smooth functions at $q$.
\end{lem}

The operator $\|\nu\|=\sum_{p=2}^{\infty}\gamma_{p}p^{4}$
defines a norm on the space of (possibly infinite) polynomials $\nu(x)=\sum_{p=2}^{\infty}\gamma_{p}x^{p}$
such that $\|\nu\|<\infty$. We will denote $\nd\to\nz$ whenever
$\|\nd-\nz\|\to0$, as $\delta\to0$.
\begin{rem}
	\label{rem:norm}If $\|\bar{\nu}-\nu\|<\delta$, then $\sup_{r\in[-1,1]}|\frac{d^{i}}{dr^{i}}\bar{\nu}(r)-\frac{d^{i}}{dr^{i}}\nu(r)|<\delta$
	for $i=0,1,...,4$.
\end{rem}

\corO{We will  need the following stability results with respect to the norm $\|\cdot\|$.}
\begin{lem}
\label{lem:Excts}If $\nd\to\nz$ for some pure-like or pure mixture
$\nz$, then $E_{0}(\nd)\to E_{0}(\nz)$ and $x_{0}(\nd)\to x_{0}(\nz)$,
as $\delta\to0$.
\end{lem}

In the next lemma  we implicitly assume that $q$ and $q_{i}$ are
positive numbers for which $\Psi_{\nu,q,q}$ and $\Psi_{\nu,q_{1},q_{2}}$
are well-defined.\footnote{That is, $q$ and $q_{i}$ are such that expressions like $\nu(q^{2})$
or its derivatives, which appear in the definitions of $\Psi_{\nu,q,q}$
and $\Psi_{\nu,q_{1},q_{2}}$, are finite. From our assumption that
$\varlimsup p^{-1}\log\gamma_{p}<0$, the range that $q$ and $q_{i}$
are allowed to be in contains an open interval containing $(0,1]$.}
\begin{lem}
\label{lem:PhiCt} \corO{Assume $\nu$ is non-pure. Then the following hold.}
\begin{enumerate}
\item \label{enu:PsiCt1}$\Psi_{\nu,q_{1},q_{2}}\left(r,u_{1},u_{2},x_{1},x_{2}\right)$
and its first and second derivatives in $r$ are continuous functions
of $q_{1},\,q_{2}$, $r\in(-1,1)$, $u_{i}\in\mathbb{R}$, $x_{i}\in\mathbb{R}$
and $\nu$ (w.r.t. the norm $\|\cdot\|)$. 
\item \label{enu:PsiCt2} \corO{If $\nu$ satisfies \CDTN\  then,} for any $\tau>0$,
for small enough $\delta=\delta(\tau)>0$ and any $\epsilon>0$, the following holds.
 If for some mixture $\bar{\nu}$, $\|\bar{\nu}-\nu\|$,
$|1-q_{i}|$, $|u_{i}+E_{0}(1)|$ and $|x_{i}+x_{0}(1)|$ are all
smaller than $\delta$, then
\[
\sup_{\epsilon\leq|r|\leq1-\tau}\Psi_{\bar{\nu},q_{1},q_{2}}\left(r,u_{1},u_{2},x_{1},x_{2}\right)<\Psi_{\bar{\nu},q_{1},q_{2}}\left(0,u_{1},u_{2},x_{1},x_{2}\right).
\]
\item \label{enu:PsiCt4}With $E=-E_{0}(\nu)$ and $x=-x_{0}(\nu)$,
\[ \limsup_{\left(\bar{\nu},r,q,u_{1},u_{2},x_{1},x_{2}\right)\to\left(\nu,1,1,E,E,x,x\right)}\Psi_{\bar{\nu},q,q}\left(r,u_{1},u_{2},x_{1},x_{2}\right)
  \leq\limsup_{r\to1}\Psi_{\nu,1,1}\left(r,E,E,x,x\right)=\Psi_{\nu}^{0}(1).
\]
The same also holds with the $r\to1$ limit replaced by $r\to-1$
and $\Psi_{\nu}^{0}(1)$ by $\Psi_{\nu}^{0}(-1)$.
\end{enumerate}
\end{lem}

The rest of the subsection is devoted to proofs.
\begin{proof}[Proof of Lemma \ref{lem:Excts}]\corO{We begin with an auxiliary computation.
 Let \corO{$\Sigma=\Sigma_1$, see}
 (\ref{eq:Sigma1}), and denote
its elements by $\Sigma_{ij}$. By Lemma \ref{lem:invertibility},
$\Sigma$ is invertible. Note that, \corO{by a direct computation or}
from  the standard formula
for conditional Gaussian distribution \cite[p. 10-11]{RFG},
\begin{equation}
(u,x)\Sigma^{-1}(u,x)^{T}=u^{2}\Sigma_{11}^{-1}+(x-\Sigma_{12}\Sigma_{11}^{-1}u)^{2}(\Sigma_{22}-\Sigma_{12}^{2}\Sigma_{11}^{-1})^{-1}.\label{eq:0503-06-1}
\end{equation}
By the definition (\ref{eq:Theta}) of $\Theta_{\nu,q}$, (\ref{eq:0503-06-1})
and substitution of the values of $\Sigma_{ij}$,
\begin{equation}
\Theta_{\nu,1}\left(-E,x\right)=\frac{1}{2}+\frac{1}{2}\log\left(\frac{\nu''(1)}{\nu'(1)}\right)-\frac{E^{2}}{2}-\frac{(x+\nu'(1)E)^{2}}{2(\nu''(1)+\nu'(1)-\nu'(1)^{2})}+\Omega\left(\frac{x}{\sqrt{\nu''(1)}}\right).\label{eq:0503-07-1}
\end{equation}}

\corO{Turning to the proof of the lemma, assume first} that the limiting polynomial $\nz$ is not a pure mixture.
Since $\Omega(x)$ \corO{from}  (\ref{eq:Omega}) is a Lipschitz \corO{function} and $\Sigma_{1}$
is positive definite, for small $\delta$, $E_{0}(\nd)$ and $x_{0}(\nd)$
belong to some compact set $[-T,T]$, and the same for $\nz$. On
$[-T,T]^{2}$, $\Theta_{\nd,1}(E,x)$ converges uniformly to $\Theta_{\nz,1}(E,x)$,
as $\delta\to0$. Since $-E_{0}(\nz)<-E_{\infty}(\nz)$ and $-x_{0}(\nz)$
are unique\footnote{\label{fn:EUniq}\corOU{Uniqueness of $E_0(\nz)$ follows from \cite[Proposition 1, Theorem 1.4]{ABA2};  Uniqueness of $x_0$ follows
from  Lemma \ref{lem:x0} below.}},
this proves the lemma in the current case.

Next, assume that \corO{$\nz(x)=\nu_p(x)=x^{p}$} is pure.
We may and will assume that any of the $\nd$
is not a pure mixture. 
Setting $\alpha_{\nu}^{2}=\nu''(1)+\nu'(1)-\nu'(1)^{2}$
and denoting by $C>0$ the Lipschitz constant of $\Omega(x)$, for
any $\nu$, {we obtain from \eqref{eq:0503-07-1} that}
\begin{equation}
\Theta_{\nu}\left(-E,x-\nu'(1)E\right)-\Theta_{\nu}\left(-E,-\nu'(1)E\right)\leq-\frac{x^{2}}{2\alpha_{\nu}^{2}}+\frac{C|x|}{\sqrt{\nu''\left(1\right)}},\label{eq:0904-02-1}
\end{equation}
and the left-hand side of (\ref{eq:0904-02-1}) is negative whenever
$|x|>2C\alpha_{\nu}/\sqrt{\nu''\left(1\right)}.$ Moreover,
\begin{equation}
\max_{x\in\mathbb{R}}\Theta_{\nu}\left(-E,x-\nu'(1)E\right)-\Theta_{\nu}\left(-E,-\nu'(1)E\right)\leq C^{2}\alpha_{\nu}^{2}/\nu''\left(1\right),\label{eq:0210-01}
\end{equation}
with the maximum above being obtained with some $|x|\leq2C\alpha_{\nu}^{2}/\sqrt{\nu''\left(1\right)}$.
Since $\alpha_{\nd}^{2}\to0$, this proves that $x_{0}(\nd)\to x_{0}(\nz)$. 

Since $\Theta_{\nd}\left(-E,-(\nd)'(1)E\right)$ converges to $\Theta_{p}\left(-E\right)$
uniformly on compacts, the convergence $E_{0}(\nd)\to E_{0}(\nz)$
can be deduced by first restricting to a compact range $E\in[-T,T]$,
as we did for the mixed case.
\end{proof}
\begin{proof}[Proof of Lemma \ref{lem:x0}]

By remark \ref{rem:scaling} it is enough to prove the lemma assuming
that $q=1$ and $\nu(1)=1$, which we will.
Since $\Omega$ \corO{from  (\ref{eq:Omega}) is a symmetric function satisfying}
 ${d}\Omega(x)/dx=0$ if and only if $x=0$, \corO{we deduce  from \eqref{eq:0503-07-1} that}
\[
\sup_{x<-\nu'(1)E_{0}}\Theta_{\nu,1}\left(-E_{0},x\right)>\sup_{x\geq-\nu'(1)E_{0}}\Theta_{\nu,1}\left(-E_{0},x\right).
\]
For any $-2\sqrt{\nu''(1)}<x<-\nu'(1)E_{0}$, since $\Omega(x)=x^{2}/4-1/2$
for $|x|\leq2$,
\begin{align*}
\frac{d}{dx}\Theta_{\nu,1}\left(-E_{0},x\right) & =x\left(\frac{1}{2\nu''(1)}-\frac{1}{\nu''(1)+\nu'(1)-\nu'(1)^{2}}\right)-\frac{\nu'(1)E_{0}}{\nu''(1)+\nu'(1)-\nu'(1)^{2}}\\
 & <-\frac{1}{\sqrt{\nu''(1)}}+\frac{2\sqrt{\nu''(1)}}{\nu''(1)+\nu'(1)-\nu'(1)^{2}}-\frac{\nu'(1)E_{0}}{\nu''(1)+\nu'(1)-\nu'(1)^{2}}\leq0,
\end{align*}
where the first inequality follows since the expression in the parentheses
is negative \corO{as can be checked using that $\nu(1)=1$}, and the second inequality follows by calculus since for
pure-like or critical models $E_{0}\geq E_{\infty}$ from (\ref{eq:Gpl}). 
Therefore, 
\[
\sup_{x\leq-2\sqrt{\nu''(1)}}\Theta_{\nu,1}\left(-E_{0},x\right)\geq\sup_{x>-2\sqrt{\nu''(1)}}\Theta_{\nu,1}\left(-E_{0},x\right),
\]
with strict inequality if the supremum of the left-hand side is obtained
at some $x<-2\sqrt{\nu''(1)}$.
\corO{We conclude that}
\begin{align}
\sup_{x\leq-2\sqrt{\nu''(1)}}\Theta_{\nu,1}\left(-E_{0},x\right) & =\sup_{x\leq-\sqrt{2}}\Theta_{\nu,1}\left(-E_{0},x\sqrt{2\nu''(1)}\right)\label{eq:1108-1-1-1}\\
 & =\sup_{x\leq-\sqrt{2}}\frac{1}{2}\log\left(\frac{\nu''(1)}{\nu'(1)}\right)-\frac{u^{2}}{2}+\frac{x^{2}}{2}-\frac{(\sqrt{2\nu''(1)}x+\nu'(1)E_{0})^{2}}{2(\nu''(1)+\nu'(1)-\nu'(1)^{2})}-I_{1}\left(|x|\right),\nonumber 
\end{align}
where 
\[
I_{1}\left(x\right)=\frac{1}{2}\left(x\sqrt{x^{2}-2}+\log2-2\log\left(x+\sqrt{x^{2}-2})\right)\right),\quad \corO{x\geq 2}
\]
is the rate function for the top eigenvalue of a GOE matrix (with
a different normalization than we use), see \cite[Eq. (2.9)]{ABA2}.
In the proof of \cite[Theorem 1.1]{ABA2} it is shown that replacing
$-E_{0}$ with some $u<-E_{\infty}$ in (\ref{eq:1108-1-1-1}), the
supremum in (\ref{eq:1108-1-1-1}) is obtained at a unique point $x_{*}<-\sqrt{2}$.
This completes the proof since, for pure-like models, $-E_{0}<-E_{\infty}$
from (\ref{eq:Gpl}).
\end{proof}

\begin{proof}[Proof of Lemma \ref{lem:ddE}]
On $D:=\{(E,x):\,x<-2q\sqrt{\nu''(q^{2})}\}$, the term involving
$\Omega$ in the definition of $\Theta_{\nu,q}\left(E,x\right)$ is
strictly concave. Since $\Sigma_{q}$ is positive-definite (see Lemma
\ref{lem:invertibility}), $(E,x)\mapsto\Theta_{\nu,q}(E,x)$  is
strictly concave \corO{on $D$}  as well. 

\corO{To see the second part of the lemma, assume} towards contradiction that
\corO{$\Theta_{\nu,q}(E,x)=0$ and $\frac{d}{dE}\Theta_{\nu,q}(E,x)=0$ for some $(E,x)\in D$.}
\corO{Then, by the concavity of $\Theta_{\nu,q}$ on $D$, it follows that}
the maximum of $\Theta_{\nu,q}\left(E,x\right)$ over
$D$ is equal to $0$. Observe that in the proof of Lemma \ref{lem:x0}
the only information used on $E_{0}(q)$ was that $E_{0}(q)\geq E_{\infty}(q)$.
Therefore, following the same proof we have that for any $E\in(-E_{0}(q),-E_{\infty}(q))$,
there exists a unique $x$ such that 
\begin{equation}
\Theta_{\nu,q}(E,x)=\max_{x\in\mathbb{R}}\Theta_{\nu,q}(E,x),\label{eq:2707-01}
\end{equation}
and for that $x$, $(E,x)\in D$. Also, since $\nu_{q}$ is pure-like,
by (\ref{eq:Gpl}) and Remark \ref{rem:scaling}, the right-hand side
of (\ref{eq:2707-01}) with $E=-E_{\infty}(q)$ is strictly positive.
By continuity, we also have some $E\in(-E_{0}(q),-E_{\infty}(q))$
for which the right-hand side of (\ref{eq:2707-01}) is strictly positive,
which implies that the maximum of $\Theta_{\nu,q}\left(E,x\right)$
over $D$ is strictly positive.  Since we arrived at a contradiction,
the proof is completed.
\end{proof}

\begin{proof}[Proof of Lemma \ref{lem:Exsmooth}.]
The proof is an application of the implicit function
theorem. Let $\bar{q}$ be a positive number such that $\nu_{\bar{q}}$
is pure-like. Define the function
\[
F(q,(E,x))=\left(\Theta_{\nu,q}(E,x),\,\frac{d}{dx}\Theta_{\nu,q}(E,x)\right).
\]
By Remark \ref{rem:scaling}, $-x_{0}(\bar{q})$ is the maximum point
of $x\mapsto\Theta_{\nu,\bar{q}}(-E_{0}(\bar{q}),x)$, and thus $F(\bar{q},g(\bar{q}))=0$.
By Lemma \ref{lem:x0}, $-x_{0}(\bar{q})\neq-2\bar{q}\sqrt{\nu''(\bar{q})}$.
From the definition (\ref{eq:Theta}) of $\Theta_{\nu,\bar{q}}(E,x)$
one can verify that $F$ is a smooth function of $q$, $E$ and $x$
on a neighborhood of $(\bar{q},-E_{0}(\bar{q}),-x_{0}(\bar{q}))$.
Therefore, by the implicit function theorem, if we show that the Jacobian
\[
\left(\begin{array}{cc}
\frac{d}{dE}\Theta_{\nu,q}(E,x) & \frac{d}{dx}\Theta_{\nu,q}(E,x)\\
\frac{d}{dx}\frac{d}{dE}\Theta_{\nu,q}(E,x) & \frac{d^{2}}{dx^{2}}\Theta_{\nu,q}(E,x)
\end{array}\right)
\]
is invertible at $(E,x)=(-E_{0}(\bar{q}),-x_{0}(\bar{q}))$, then
there exists a smooth function $g(q)=(E(q),x(q))$ on the neighborhood
of $\bar{q}$, such that $F(q,(E(q),x(q)))=0$. Hence, $(E(q),x(q))=(-E_{0}(q),-x_{0}(q))$
(since by Lemmas \ref{lem:ddE} and \ref{lem:x0} there is a
unique $-E<0$ such that $\sup_{x}\Theta_{\nu,q}(-E,x)=0$).

What \corO{remains} is to prove the invertibility of the matrix above. Since
$\frac{d}{dx}\Theta_{\nu,q}(-E_{0}(q),-x_{0}(q))=0$, the proof is
completed by \corO{invoking}  Lemma \ref{lem:ddE}. 
\end{proof}

\begin{proof}[Proof of Lemma \ref{lem:PhiCt}]
Using Remark \ref{rem:norm}, Point \ref{enu:PsiCt1} follows directly
from the formula (\ref{eq:Psi}) for $\Psi_{\nu,q_{1},q_{2}}$ and
the definitions (\ref{eq:Omega}) and (\ref{eq:SigmaUXbar}) of $\Omega(x)$
and $\Sigma_{U,X}\left(r,q_{1},q_{2}\right)$. 

\corO{Turning to the proof of Point  \ref{enu:PsiCt2}, for any $\rho>0$, define} 
\[
\eta(\rho)=\Psi_{\nu}^{0}(0)-\sup_{|r|\in[\rho,1-\tau]}\Psi_{\nu}^{0}(r)  \corE{>0},
\]
\corO{where the inequality is due to \CDTN.} From continuity, for small enough $\rho$
and $\delta=\delta(\tau,\rho)$, uniformly in $q_{i}$, $u_{i}$,
$x_{i}$ and $\bar{\nu}$ as in Point \ref{enu:PsiCt2},
\begin{align*}
\sup_{|r|\leq\rho}\frac{d^{2}}{dr^{2}}\Psi_{\bar{\nu},q_{1},q_{2}}\left(r,u_{1},u_{2},x_{1},x_{2}\right) & <0,\\
\sup_{|r|\leq1-\tau}\left|\Psi_{\bar{\nu},q_{1},q_{2}}\left(r,u_{1},u_{2},x_{1},x_{2}\right)-\Psi_{\nu}^{0}(r)\right| & <\frac{\eta(\rho)}{2}.
\end{align*}
Since the derivative in $r$ at $r=0$ of any of the entries of $\Sigma_{U,X}\left(r,q_{1},q_{2}\right)$
(see (\ref{eq:SigmaUXbar})) is $0$, 
\[
\corO{\left.\frac{d}{dr}\Psi_{\bar{\nu},q_{1},q_{2}}\left(r,u_{1},u_{2},x_{1},x_{2}\right)\right|_{r=0}}=0.
\]
The above imply Point \ref{enu:PsiCt2}.

\corO{We next prove  Point \ref{enu:PsiCt4} for the $r\to1$ limit}. Since
$\Omega(x)$ is continuous, it is enough to prove that 
\begin{equation}
\begin{aligned} & \liminf_{\left(\bar{\nu},r,q,u_{1},u_{2},x_{1},x_{2}\right)\to\left(\nu,1,1,-E_{0},-E_{0},-x_{0},-x_{0}\right)}\left(u_{1},u_{2},x_{1},x_{2}\right)\Sigma_{U,X,\bar{\nu}}^{-1}\left(r,q,q\right)\left(u_{1},u_{2},x_{1},x_{2}\right)^{T}\\
 & \geq\liminf_{r\to1}\left(r,-E_{0},-E_{0},-x_{0},-x_{0}\right)\Sigma_{U,X,\nu}^{-1}\left(r,1,1\right)\left(r,-E_{0},-E_{0},-x_{0},-x_{0}\right)^{T}
\end{aligned}
\label{eq:1609-01}
\end{equation}
(where the dependence of $\Sigma_{U,X}$ in $\bar{\nu}$ is expressed
by its addition to the subscript). 

\corE{In Appendix \ref{sec:Covariances} we define, see (\ref{eq:SigmaUXbar}), (\ref{eq:26}), (\ref{eq:SigmaX}),
	(\ref{eq:Sigmab}),}
\begin{equation}
\Sigma_{U,X}\left(r,q,q\right)=\left(\begin{array}{cc}
\Sigma_{U}\left(r,q,q\right) & \Sigma_{b}\left(r,q,q\right)\\
\Sigma_{b}^{T}\left(r,q,q\right) & \Sigma_{X}\left(r,q,q\right)
\end{array}\right).\label{eq:1003-01}
\end{equation}
By setting 
\[
\Sigma_{U,X}\left(1,q,q\right)=\left(\begin{array}{cc}
z_{1}(q)\mathbf{1}_{2\times2} & z_{2}(q)\mathbf{1}_{2\times2}\\
z_{2}(q)\mathbf{1}_{2\times2} & z_{3}(q)\mathbf{1}_{2\times2}
\end{array}\right),
\]
where $\mathbf{1}_{2\times2}$ is the $2\times2$ matrix whose all
entries are $1$ and 
\begin{align*}
z_{1}(q) & =\bar{\nu}(q^{2})-q^{2}\frac{\bar{\nu}'(q^{2})^{2}}{\bar{\nu}'(q^{2})+3q^{2}\bar{\nu}''(q^{2})},\\
z_{2}(q) & =q\bar{\nu}'\left(q^{2}\right)-q\bar{\nu}'\left(q^{2}\right)\frac{q^{2}\bar{\nu}''\left(q^{2}\right)+\bar{\nu}'\left(q^{2}\right)}{\bar{\nu}'(q^{2})+3q^{2}\bar{\nu}''(q^{2})},\\
z_{3}(q) & =q^{2}\bar{\nu}''\left(q^{2}\right)+\bar{\nu}'\left(q^{2}\right)-\frac{\left(q^{2}\bar{\nu}''\left(q^{2}\right)+\bar{\nu}'\left(q^{2}\right)\right)^{2}}{\bar{\nu}'(q^{2})+3q^{2}\bar{\nu}''(q^{2})},
\end{align*}
we continuously extend the elements of (\ref{eq:1003-01}) for $r=1$. 

For any $r\in(-1,1]$, each of the blocks of (\ref{eq:1003-01}) is
a $2\times2$ matrix of the from
\[
\left(\begin{array}{cc}
a & b\\
b & a
\end{array}\right).
\]
Therefore, the covariance matrix $\Sigma_{U,X}\left(r,q,q\right)$
has two orthogonal unit eigenvectors, say $v_{i}(r,q)$, $i=1,2$,
of the form $(u,u,x,x)$ and two orthogonal unit eigenvectors of the
form $(u,-u,x,-x)$, say $v_{i}(r,q)$, $i=3,4$. By \cite[P. 106-108]{Kato},
from the continuity of the elements of $\Sigma_{U,X}\left(r,q,q\right)$,
we can choose the eigenvectors $v_{i}(r,q)$ and the corresponding
eigenvalues $\lambda_{i}(r,q)$ so that they are continuous in $r$,$q$
and $\bar{\nu}$ at the point $(r,q,\bar{\nu})=(1,1,\nu)$ (though
we do not necessarily have continuity on a neighborhood of this point).

For a general vector $w\in\mathbb{R}^{4}$, 
\[
w^{T}\Sigma_{U,X}^{-1}\left(r,q,q\right)w\geq\sum_{i=1,2}\frac{1}{\lambda_{i}(r,q)}\langle w,v_{i}\left(r,q\right)\rangle^{2},
\]
and for $w_{0}=\left(-E_{0},-E_{0},-x_{0},-x_{0}\right)$ we have
an equality, since it is orthogonal to $v_{i}(r,q)$, $i=3,4$. This
implies (\ref{eq:1609-01}), and therefore Point \ref{enu:PsiCt4}
for the $r\to1$ limit.

Finally, we prove Point \ref{enu:PsiCt4} for the $r\to-1$ limit.
If $\nu(x)$ is an even or an odd function, a similar argument can
be applied using eigen-values and -vectors. Lastly, if $\nu(x)$ is
neither even nor odd, then the logarithmic term in the definition
of $\Psi_{\nu,1,1}\left(r,u_{1},u_{2},x_{1},x_{2}\right)$ goes to
$-\infty$ as $r\to-1$, and so do both the limits in Point \ref{enu:PsiCt4}.
\end{proof}

\subsection{\label{subsec:pf}Proof of Lemma \ref{lem:conccentration}}

Recall that at $r=0$, (\ref{eq:1703-08}) and that $E_{0}(q)$ and
$x_{0}(q)$ are continuous in $q$ by Lemma \ref{lem:Excts}. Thus,
to prove Lemma \ref{lem:conccentration} we need to show that for
small enough $\delta>0$, if $|1-q_{i}|<\delta$ and 
\[
B_{i}\times D_{i}\subset A(\delta)=\left\{ (u,x):\,|u+E_{0}(\nu)|,\,|x+x_{0}(\nu)|<\delta\right\} ,
\]
then, \corO{with $I(\epsilon)=\{x: |x|\in [\epsilon,1)\}$,} 
\begin{equation}
\sup_{r\in I(\epsilon),\,u_{i}\in B_{i},x_{i}\in D_{i}}\Psi_{\nu,q_{1},q_{2}}\left(r,u_{1},u_{2},x_{1},x_{2}\right)<\sup_{u_{i}\in B_{i},x_{i}\in D_{i}}\Psi_{\nu,q_{1},q_{2}}\left(0,u_{1},u_{2},x_{1},x_{2}\right).\label{eq:1909-02}
\end{equation}
We proceed by treating the cases $q_{1}=q_{2}$ and $q_{1}\neq q_{2}$
separately. 

\subsection*{The case $q_{1}=q_{2}=q$ }

From the continuity \corO{stated in}  Point \ref{enu:PsiCt1} of Lemma \ref{lem:PhiCt},
\[
\lim_{(q,u_{i},x_{i})\to(1,-E_{0}(\nu),-x_{0}(\nu))}\Psi_{\nu,q,q}\left(0,u_{1},u_{2},x_{1},x_{2}\right)=\Psi_{\nu}^{0}(0)>\max\{\Psi_{\nu}^{0}(1),\,\Psi_{\nu}^{0}(-1)\}.
\]
Thus, from Point \ref{enu:PsiCt4} of Lemma \ref{lem:PhiCt}, for
small enough $\delta$, it is enough to prove (\ref{eq:1909-02})
with $I(\epsilon)$ replaced by $I(\epsilon,\tau)=[-1+\tau,-\epsilon]\cup[\epsilon,1-\tau]$,
for some small $\tau$. Moreover, from \corO{the} continuity of $\Psi_{\nu,q,q}$
in all its variables,
\begin{equation}
\sup_{r\in I(\epsilon,\tau),\,u_{i}\in B_{i},x_{i}\in D_{i}}\Psi_{\nu,q,q}\left(r,u_{1},u_{2},x_{1},x_{2}\right)=\Psi_{\nu,q,q}\left(r',u_{1}',u_{2}',x_{1}',x_{2}'\right)\label{eq:1909-03}
\end{equation}
for some $r'\in I(\epsilon,\tau)$ and $u_{i}'$ and $x_{i}'$ in
the closure of $B_{i}$ and $D_{i}$, respectively. (\corO{Recall that the latter sets are bounded by assumption.})
Therefore, the \corO{proof of Lemma \ref{lem:conccentration} in the case $q_1=q_2$ follows from}
Point \ref{enu:PsiCt2} of
Lemma \ref{lem:PhiCt}.

\subsection*{The case $q_{1}\protect\neq q_{2}$ }

The main ingredient in the proof of the current case is the following
lemma, the proof of which is deferred to the end of the \corO{subsection}.
\begin{lem}
\label{lem:minTheta}For small enough $\tau,\,\delta>0$, for any
$q_{1}\neq q_{2}$ such that $|1-q_{i}|<\delta$ and any $(u_{i},x_{i})\in A(\delta)$,
\begin{equation}
\sup_{1-\tau\leq|r|<1}\Psi_{\nu,q_{1},q_{2}}\left(r,u_{1},u_{2},x_{1},x_{2}\right)\leq\min_{i=1,2}\Theta_{\nu,q_{i}}(u_{i},x_{i}).\label{eq:ineq1}
\end{equation}
\end{lem}

Note that for $q_{1}\neq q_{2}$, as $|r|\to1$ the logarithmic term
in the definition of $\Psi_{\nu,q_{1},q_{2}}$ (\ref{eq:Psi}) goes
to $-\infty$. Since the \corO{quadratic} term involving $\Sigma_{U,X}^{-1}\left(r,q_{1},q_{2}\right)$
in (\ref{eq:Psi}) is nonpositive, we conclude \corO{that, for fixed $q_1\neq q_2$, }
\begin{equation}
\label{cor:q1neqq2:eq}
\lim_{|r|\to1}\Psi_{\nu,q_{1},q_{2}}\left(r,u_{1},u_{2},x_{1},x_{2}\right)=-\infty,
\end{equation}
\corO{and the convergence is uniform in  $u_{1},u_{2},x_{1},x_{2}$ in compact sets (but not in $q_1,q_2$).}
%
%
Combining \corO{\eqref{cor:q1neqq2:eq}} with Point \ref{enu:PsiCt1} of Lemma \ref{lem:PhiCt},
we have that
\[
\sup_{r\in I(\epsilon),\,u_{i}\in B_{i},x_{i}\in D_{i}}\Psi_{\nu,q_{1},q_{2}}\left(r,u_{1},u_{2},x_{1},x_{2}\right)=\Psi_{\nu,q_{1},q_{2}}\left(r',u_{1}',u_{2}',x_{1}',x_{2}'\right)
\]
for some $r\in I(\epsilon)$ and $u_{i}$ and $x_{i}$ in the closure
of $B_{i}$ and $D_{i}$, respectively, which are assumed here to
be bounded sets. Thus, in light of Point \ref{enu:PsiCt2} of Lemma
\ref{lem:PhiCt}, to prove (\ref{eq:1909-02}) it is enough to show
that for small enough $\delta>0$, if $|1-q_{i}|<\delta$ and $(u_{i},x_{i})\in A(\delta)$,
then 
\begin{equation}
\sup_{1-\tau\leq|r|<1}\Psi_{\nu,q_{1},q_{2}}\left(r,u_{1},u_{2},x_{1},x_{2}\right)  <\Psi_{\nu,q_{1},q_{2}}\left(0,u_{1},u_{2},x_{1},x_{2}\right)
\corO{(
  =\Theta_{\nu,q_{1}}\left(u_{1},x_{1}\right)+\Theta_{\nu,q_{2}}\left(u_{2},x_{2}\right)),}
\label{eq:1909-02-1}
\end{equation}
where the equality follows from (\ref{eq:1703-08}) and $\tau$ is
a fixed number which can be assumed to small. In fact, in light of
our assumption that the summands in (\ref{eq:1703-06}) are nonnegative,
it will be enough to prove (\ref{eq:1909-02-1}) only for $q_{i}$
, $u_{i}$ and $x_{i}$ such that $\Theta_{\nu,q_{i}}(u_{i},x_{i})\geq0$.
In this case, if one of $\Theta_{\nu,q_{i}}(u_{i},x_{i})$ is strictly
positive, then (\ref{eq:1909-02-1})  follows from (\ref{eq:ineq1}).
Hence, to complete the proof of Lemma \ref{lem:conccentration} it
remains to prove Lemma \ref{lem:minTheta} and the following one.
\begin{lem}
\label{lem:Th12=00003D0}
For small
enough $\tau,\,\delta>0$, if $|1-q_{i}|<\delta$,
$(u_{i},x_{i})\in A(\delta)$
and
$\Theta_{\nu,q_{1}}(u_{1},x_{1})=\Theta_{\nu,q_{2}}(u_{2},x_{2})=0,$
then
\[
\sup_{1-\tau\leq|r|<1}\Psi_{\nu,q_{1},q_{2}}\left(r,u_{1},u_{2},x_{1},x_{2}\right)<0.
\]
\end{lem}
Lemma \ref{lem:Th12=00003D0} is a 
direct consequence of \corO{Lemma \ref{lem:minTheta} and} the following
two short lemmas applied with $g_{i}=\Theta_{\nu,q_{i}}$, $g=\Psi_{\nu,q_{1},q_{2}}$
and $B=(-1,1)\setminus(-1+\tau,1-\tau)$.
\begin{lem}
Let $B\subset \mathbb{R}$, let 
$g_{i}(t)$, $i=1,2$, be real
functions defined on open sets $T_{i}\subset\mathbb{R}^{k}$, and
let 
$g(s,t_{1},t_{2})$ be a real function defined on $B\times T_{1}\times T_{2}$.
Suppose that for any $t_{1}$, $t_{2}$,
\begin{equation}
\sup_{s\in B}g\left(s,t_{1},t_{2}\right)\leq\min_{i=1,2}g_{i}(t_{i}).\label{eq:0105-02}
\end{equation}
If for some $t_{1}^{*}$, $t_{2}^{*}$ and $s^{*}\in B$,
\begin{enumerate}
  \item  $g(s^{*},t_{1}^{*},t_{2}^{*})=\sup_{s\in B}g(s,t_{1}^{*},t_{2}^{*})$,
  \item $g_{1}(t_{1}^{*})=g_{2}(t_{2}^{*})$, 
  \item  $\corO{(\nabla g_{1})}(t_{1}^{*})\neq0$
(in particular, the gradient exists), 
\item the gradient of $g$ in the
  coordinate $t_{1}$ only, $\corO{(\nabla_{t_{1}}g)}(s^{*},t_{1}^{*},t_{2}^{*})$,
exists,
\end{enumerate}
then
\[
\sup_{s\in B}g\left(s,t_{1}^{*},t_{2}^{*}\right)<g_{1}(t_{1}^{*})=g_{2}(t_{2}^{*}).
\]
\end{lem}
\begin{proof}
Assume towards contradiction that $g\left(s^{*},t_{1}^{*},t_{2}^{*}\right)=\min_{i=1,2}g_{i}(t_{i}^{*})$.
Let $v\in\mathbb{R}^{k}$ be a vector for which $\langle\nabla g_{1}(t_{1}^{*}),v\rangle<0$.
From (\ref{eq:0105-02}) we also must have that 
$\langle\nabla_{t_{1}}g(s^{*},t_{1}^{*},t_{2}^{*})v\rangle<0$.
However, then, for small $\epsilon>0$, we have that
$g(s^{*},t_{1}^{*}-\epsilon v,t_{2}^{*})>g(s^{*},t_{1}^{*},t_{2}^{*})=\min\{g_{1}(t_{1}^{*}-\epsilon v),\,g_{2}(t_{2}^{*})\}$,
in contradiction to (\ref{eq:0105-02}).
\end{proof}
\begin{lem}
  Assume that $q_{1}\neq q_{2}$. \corO{Then,} the supremum 
\corO{in the left hand side of} (\ref{eq:ineq1}) is
obtained at some point $r^{*}\in(-1,1)\setminus(-1+\tau,1-\tau)$,
and the gradient in $(u_{1},x_{1})$ only of $\Psi_{\nu,q_{1},q_{2}}$
at $(r^{*},u_{1},u_{2},x_{1},x_{2})$ exists. Further,
for small enough $\delta$,
if $|1-q|<\delta$, $(u,x)\in A(\delta)$ and $\Theta_{\nu,q}\left(u,x\right)=0$,
then $\nabla\Theta_{\nu,q}\left(u,x\right)\neq0$.
\end{lem}

\begin{proof}
The existence of $r^{*}$ as in the lemma 
follows from Point \ref{enu:PsiCt1} of Lemma \ref{lem:PhiCt} and
\corO{\eqref{cor:q1neqq2:eq}}. The fact that $\nabla\Theta_{\nu,q}\left(u,x\right)\neq0$
follows from Lemmas \ref{lem:x0} and \ref{lem:ddE}.
\end{proof}

\corO{It thus remains to prove  Lemma \ref{lem:minTheta}. For the proof,
we need the following deterministic inequality.}
\begin{lem}
\label{lem:2to1-1}For any $r\in(\cos\pi/8,1)$ setting $r_{0}:=\cos(4\cos^{-1}(r))\in(0,1)$,
we have that, deterministically, 
 \[\left[\mbox{\rm Crt}_{N,q_{1},q_{2}}(B_{1},B_{2},D_{1},D_{2},I(r))\right]_{2}
 \leq\sum_{i=1,2}\left[\mbox{\rm Crt}_{N,q_{i},q_{i}}(B_{i},B_{i},D_{i},D_{i},I(r_{0}))\right]_{2}+2\mbox{\rm Crt}_{N,q_{1}}(B_{1},D_{1}),\]
where $I(r)=(-1,1)\setminus[-r,r]$. 
\end{lem}
\begin{proof}
In the current proof, for a point $\bs\in\mathbb{S}^{N-1}(1)$ denote
by $\mathcal{S}_{\pm}^{*}(\bs,\delta)$ the set of points in $\mathbb{S}^{N-1}\left(1\right)$
which are different from both $\bs$ and $-\bs$ and which have minimal
distance from $\bs$ or $-\bs$ less than $\delta$, under the usual
metric on the sphere. For any point $\bs\in\mathbb{R}^{N}\setminus\{0\}$
define the cone 
\[
\mathcal{B}_{\pm}^{*}(\bs,\delta):=\{c\bs':\,c\in\mathbb{R},\,\bs'\in\mathcal{S}_{\pm}^{*}(\bs/\|\bs\|,\delta)\}.
\]
The overlap $r$ defines the distance 
\[
\epsilon:=\cos^{-1}(r)\in[0,\pi]
\]
on the sphere $\mathbb{S}^{N-1}(1)$. Assuming that $r\in(\cos\pi/8,1)$,
we define $r_{0}$ as the overlap that corresponds to $4$ times that
distance, $r_{0}:=\cos(4\cos^{-1}(r))\in(0,1)$. 

Note that 
\begin{equation}
 \left[\mbox{Crt}_{N,q_{1},q_{2}}(B_{1},B_{2},D_{1},D_{2},I(r))\right]_{2}
 =\sum_{\bs\in\mathscr{C}_{N,q_{1}}(NB_{1},\sqrt{N}D_{1})}|\mathscr{C}_{N,q_{2}}(NB_{2},\sqrt{N}D_{2})\cap\mathcal{B}_{\pm}^{*}(\bs,\epsilon)|.\label{eq:1104-02-1}
\end{equation}

Denote by $A_{0}$ the set of points $\bs\in\mathscr{C}_{N,q_{1}}(NB_{1},\sqrt{N}D_{1})$
for which 
\begin{equation}
|\mathscr{C}_{N,q_{2}}(NB_{2},\sqrt{N}D_{2})\cap\mathcal{\mathcal{B}}_{\pm}^{*}(\bs,\epsilon)|>|\mathscr{C}_{N,q_{1}}(NB_{1},\sqrt{N}D_{1})\cap\mathcal{B}_{\pm}^{*}(\bs,4\epsilon)|+2.\label{eq:1104-03-1}
\end{equation}
Denoting $a_{c}=|\mathscr{C}_{N,q_{1}}(NB_{1},\sqrt{N}D_{1})\setminus A_{0}|$, \corO{we have that}
\begin{align*}
 & \sum_{\bs\in\mathscr{C}_{N,q_{1}}(NB_{1},\sqrt{N}D_{1})\setminus A_{0}}|\mathscr{C}_{N,q_{2}}(NB_{2},\sqrt{N}D_{2})\cap\mathcal{B}_{\pm}^{*}(\bs,\epsilon)|\\
 & \leq2a_{c}+\sum_{\bs\in\mathscr{C}_{N,q_{1}}(NB_{1},\sqrt{N}D_{1})\setminus A_{0}}|\mathscr{C}_{N,q_{1}}(NB_{1},\sqrt{N}D_{1})\cap\mathcal{B}_{\pm}^{*}(\bs,4\epsilon)|\\
 & \leq2a_{c}+\left[\mbox{Crt}_{N,q_{1},q_{1}}(B_{1},B_{1},D_{1},D_{1},I(r_{0}))\right]_{2}.
\end{align*}

To complete the proof we will show that there exists an injective
mapping 
\[
\mathfrak{X}:\,A_{0}\to\mathscr{C}_{N,q_{2}}(NB_{2},\sqrt{N}D_{2}),
\]
such that for any $\bs\in A_{0}$,
\begin{equation}
|\mathscr{C}_{N,q_{2}}(NB_{2},\sqrt{N}D_{2})\cap\mathcal{B}_{\pm}^{*}(\bs,\epsilon)|\leq|\mathscr{C}_{N,q_{2}}(NB_{2},\sqrt{N}D_{2})\cap\mathcal{B}_{\pm}^{*}(\mathfrak{X}(\bs),4\epsilon)|+2.\label{eq:1703-05-1}
\end{equation}
This will imply that the sum in (\ref{eq:1104-02-1}) over $A_{0}$
only is bounded from above by 
\[
\left[\mbox{Crt}_{N,q_{2},q_{2}}(B_{2},B_{2},D_{2},D_{2},I(r_{0}))\right]_{2}+2|A_{0}|.
\]

Our definition is inductive, starting with an arbitrary point $\bs_{0}\in A_{0}$.
From the definition of $A_{0}$, the number of points in $A_{0}\cap\mathcal{B}_{\pm}^{*}(\bs_{0},2\epsilon)$
is smaller than $|\mathscr{C}_{N,q_{1}}(NB_{1},\sqrt{N}D_{1})\cap\mathcal{B}_{\pm}^{*}(\bs,2\epsilon)|$
and smaller than the number of points in
\begin{equation}
\mathscr{C}_{N,q_{2}}(NB_{2},\sqrt{N}D_{2})\cap\mathcal{B}_{\pm}^{*}(\bs_{0},\epsilon),\label{eq:1703-04-1}
\end{equation}
so we can define $\mathfrak{X}$ injectively on $A_{0}\cap\mathcal{B}_{\pm}^{*}(\bs_{0},2\epsilon)$
such that any point is mapped to a point in (\ref{eq:1703-04-1}).
Doing so we have that for any $\bs\in A_{0}\cap\mathcal{B}_{\pm}^{*}(\bs_{0},2\epsilon)$,
by the triangle inequality, $\bs\in\mathcal{B}_{\pm}^{*}(\mathfrak{X}(\bs),3\epsilon)$.
Therefore (\ref{eq:1703-05-1}) holds, since the set on the left-hand
side of (\ref{eq:1703-05-1}) is contained in the one on the right-hand
side, and up to the $2$ points removed by the `puncturing' of $\mathcal{B}_{\pm}^{*}(\bs,\epsilon)$.

We continue defining $\mathfrak{X}$ as follows. At step $k$ let
$A_{k}$ be defined as $A_{0}$ minus the set of points for which
$\mathfrak{X}$ was already defined, choose $\bs_{k}\in A_{k}$ arbitrarily,
and define $\mathfrak{X}$ injectively from $A_{k}\cap\mathcal{B}_{\pm}^{*}(\bs_{k},2\epsilon)$
to 
\[
\mathscr{C}_{N,q_{2}}(NB_{2},\sqrt{N}D_{2})\cap\mathcal{B}_{\pm}^{*}(\bs_{k},\epsilon).
\]
As before, this implies that (\ref{eq:1703-05-1}) for the points
in $A_{k}\cap\mathcal{B}_{\pm}^{*}(\bs_{k},2\epsilon)$. We continue
defining $\mathfrak{X}$ in the same manner until no points are left,
i.e., until $A_{k}=\varnothing$. 

So far we made sure that (the restriction of) $\mathfrak{X}$ is injective
as a function on $A_{k}\cap\mathcal{B}_{\pm}^{*}(\bs_{k},2\epsilon)$
for all $k$. However, since $\bs_{k}\notin\cup_{j<k}\mathcal{B}_{\pm}^{*}(\bs_{j},2\epsilon)$,
the image of $A_{k}\cap\mathcal{B}_{\pm}^{*}(\bs_{k},2\epsilon)$,
which is contained in $\mathcal{B}_{\pm}^{*}(\bs_{k},\epsilon)$,
is disjoint from $\cup_{j<k}\mathcal{B}_{\pm}^{*}(\bs_{j},\epsilon)$.
Hence, $\mathfrak{X}$ is injective on $A_{0}$ and the proof is completed.
\end{proof}
\corO{We have completed all preparatory steps and can proceed to
  the proof of Lemma \ref{lem:minTheta}.}
  \begin{proof}[Proof of Lemma \ref{lem:minTheta}.]
    By Lemma \ref{lem:x0},
$-x_{0}/\sqrt{\nu''(1)}<-2$. Therefore for small enough $\delta$
if $q_{i}$ and $D_{i}$ are as in Lemma \ref{lem:conccentration},
then $D_{i}$ are as in Lemma \ref{lem:2ndLB}. Hence, by letting
$B_{i}$ and $D_{i}$ shrink to a point as $N\to\infty$, we obtain
from Lemma \ref{lem:2to1-1}, Theorem \ref{thm:1stmoment}, Point
\ref{enu:PsiCt1} of Lemma \ref{lem:PhiCt} and \corO{\eqref{cor:q1neqq2:eq}}
that
\begin{align}
 & \sup_{|r|\in[1-\tau,1)}\Psi_{\nu,q_{1},q_{2}}\left(r,u_{1},u_{2},x_{1},x_{2}\right)\leq\label{eq:2304-01}\\
&\quad\quad\max\Big\{\min_{i=1,2}\Theta_{\nu,q_{i}}\left(u_{i},x_{i}\right)
  \sup_{|r|\in[1-\tau_{0},1)}\Psi_{\nu,q_{1},q_{1}}\left(r,u_{1},u_{1},x_{1},x_{1}\right),\,\sup_{|r|\in[1-\tau_{0},1)}\Psi_{\nu,q_{2},q_{2}}\left(r,u_{2},u_{2},x_{2},x_{2}\right)\Big\},\nonumber 
\end{align}
for $q_{i}\in(1-\delta,1+\delta)$ and $(u_{i},x_{i})\in A(\delta)$,
where $1-\tau_{0}$ is related to $1-\tau$ as $r_{0}$ is related
to $r$ in Lemma \ref{lem:2to1-1}. By \CDTN, $\Psi_{\nu}^{0}(\pm1)<\Psi_{\nu}^{0}(0)=0$.
Hence, from Points \ref{enu:PsiCt1} and \ref{enu:PsiCt4} of Lemma
\ref{lem:PhiCt}, uniformly in $q_{i}\in(1-\delta,1+\delta)$ and
$(u_{i},x_{i})\in A(\delta)$, 
\begin{align*}
\sup_{|r|\in[1-\tau_{0},1)}\Psi_{\nu,q_{i},q_{i}}\left(r,u_{i},u_{i},x_{i},x_{i}\right) & \leq-c,
\end{align*}
for some constant $c>0$, assuming $\delta$ and $\tau$, and therefore
$\tau_{0}$, are small enough. Since $\Theta_{\nu,q}\left(u,x\right)$
is continuous in $q$, $u$ and $x$, and $\Theta_{\nu,1}\left(-E_{0},-x_{0}\right)=0$,
the maximum in (\ref{eq:2304-01}) is equal to the first term, for
small $\delta$. Namely, we proved (\ref{eq:ineq1}).\end{proof}

\subsection{Proof of Corollary \ref{cor:matching}}

Note that
\[
\mathbb{E}\left({\rm Crt}_{N,q}\left(B,D\right)^{2}\right)\leq2\mathbb{E}{\rm Crt}_{N,q}\left(B,D\right)+\mathbb{E}{\rm Crt}_{N,q,q}\left(B,B,D,D,(-1,1)\right),
\]
so that, since we assume that $\mathbb{E}{\rm Crt}_{N,q}\left(B,D\right)\to\infty$,
\begin{align*}
\limsup_{N\to\infty}\frac{1}{N}\log\mathbb{E}\left({\rm Crt}_{N,q}\left(B,D\right)^{2}\right) & \leq\lim_{N\to\infty}\frac{1}{N}\log\mathbb{E}{\rm Crt}_{N,q,q}\left(B,B,D,D,(-1,1)\right)\\
 & \leq2\sup_{u\in B,\,x\in D}\Theta_{\nu,q}\left(u,x\right),
\end{align*}
where the second inequality follows from Theorem \ref{thm:2ndmomUBBK-1-1},
Lemma \ref{lem:conccentration} and (\ref{eq:1703-08}). The fact
that
\[
\liminf_{N\to\infty}\frac{1}{N}\log\mathbb{E}\left({\rm Crt}_{N,q}\left(B,D\right)^{2}\right)\geq2\liminf_{N\to\infty}\frac{1}{N}\log\mathbb{E}{\rm Crt}\left(B,D\right)=2\sup_{u\in B,\,x\in D}\Theta_{\nu,q}\left(u,x\right)
\]
follows from Theorem \ref{thm:1stmoment}. The $N$-dependent case
follows from a standard \corO{diagonalization} argument.\qed

\subsection{\label{subsec:pfGS}Proof of Theorem \ref{theo-ground}}

From Theorem \ref{thm:1stmoment} and the definition of $-E_{0}$ (combined with Markov's inequality and the Borel-Cantelli Lemma), 
\begin{equation}
\label{eq:GSlbas}\liminf_{N\to\infty}\frac{1}{N}\GS_N\geq -E_0,\quad \mbox{almost surely}.
\end{equation}

From Corollary\ref{cor:matching}, for some $\delta_{N}=o(1)$, with $B=B_{N}=-E_{0}+(-\delta_{N},\delta_{N})$
and $D=D_{N}=-x_{0}+(-\delta_{N},\delta_{N})$, 
\begin{equation}
\lim_{N\to\infty}\frac{1}{N}\log\mathbb{E}\left(\mbox{Crt}_{N,1}\left(B,D\right)\right)^{2}=2\lim_{N\to\infty}\frac{1}{N}\log\mathbb{E}\mbox{Crt}_{N,1}\left(B,D\right).\label{eq:0304-05-1}
\end{equation}
By appealing to the \corO{Cauchy-Schwarz} inequality, this implies that
\[
\mathbb P(\frac{1}{N}\GS_N\leq -E_0+\frac{\delta_N}{N})
\geq \mathbb P(\mbox{Crt}_{N,1}(B,D) > 0)
\geq \corE{\frac{\big(\E(\mbox{Crt}_{N,1}\left(B,D\right))\big)^2}{\E\big(\mbox{Crt}_{N,1}(B,D)\big)^2}}
\]
does not decay exponentially in $N$. 

Using the Borell-TIS inequality \cite{Borell,TIS}
(see also \cite[Theorem 2.1.1]{RFG}), which implies that the $\GS_N/N$ has exponential in $N$ tails, this is in fact sufficient to conclude the matching upper bound to \eqref{eq:GSlbas}. For the full argument, see Appendix IV of \cite{2nd}, where this is \corO{carried out} in the pure setting.
\qed

\begin{rem} \label{rem:GSq}
  By Proposition \ref{prop:openCDTN} \corO{below}, 
  for $q$ close enough to $1$, 
	$\nu_q(x)=\sum_{p=2}^\infty\gamma_p^2q^{2p}x^p$ satisfies \CDTN. Thus, by Theorem \ref{theo-ground},
	\[
	\lim_{N\to\infty} \frac{1}{N}\min_{\boldsymbol{\sigma}\in \SN }
	H_N(q\boldsymbol{\sigma})=-E_0(q), \quad {\rm
		a.s.}.
	\] 
\end{rem}

\subsection{Proof of Corollary \ref{cor:orth}}
\corO{We begin with two preparatory lemmas.}
\begin{lem}
\label{lem:compact}For any compact set $K\subset(0,\infty)$, there
exists some large $T>0$, such that $\Theta_{\nu,q}(E,x)<0$ uniformly
in $q\in K$ and $(E,x)\in\{(E,x):\,\max(|E|,|x|)>T\}$.
\end{lem}
\begin{proof}
Since the matrix $\Sigma_{q}$ (appearing in the definition (\ref{eq:Theta})
of $\Theta_{\nu,q}(E,x)$)  is positive definite and its elements are
continuous in $q$, the eigenvalues of $\Sigma_{q}$ are bounded from
below by some positive constant uniformly in $q\in K$. The lemma
therefore follows from the definition of $\Theta_{\nu,q}(E,x)$ and
the fact that $\Omega$ is Lipschitz continuous. 
\end{proof}
\begin{lem}
  \label{lem:epsx0}\corO{Assume that $\nu$ is pure-like.}
 \corO{Then there exist  $\delta>0$ so that for
 any $q\in1+(-\delta,\delta)$, any $\epsilon'>0$, there exists}
 \corOU{$c=c(\epsilon')>0$} so that 
 for \corO{small enough $\epsilon>0$ and large enough $N$,}
setting $B(\epsilon):=-E_{0}(q)+(-\epsilon,\epsilon)$
and $D(\epsilon'):=-x_{0}(q)+(-\epsilon',\epsilon')$,
\begin{equation}
\mathbb{P}\left\{ \mbox{\rm Crt}_{N,q}\left(B(\epsilon),\mathbb{R}\setminus D(\epsilon')\right)>0\right\} \leq e^{-cN},\label{eq:1507-05}
\end{equation}
where $B(\epsilon):=-E_{0}(q)+(-\epsilon,\epsilon)$
and $D(\epsilon'):=-x_{0}(q)+(-\epsilon',\epsilon')$.
\end{lem}
\begin{proof}
From Theorem \ref{thm:1stmoment} and Lemma \ref{lem:compact}, to
prove (\ref{eq:1507-05}) it will be enough to show that for some
$c>0$ and large enough $T$,
\[
(E,x)\in B(\epsilon)\times[-T,T]\setminus D(\epsilon')\implies\Theta_{\nu,q}\left(-E,-x\right)<-2c.
\]
Assume $\delta$ is small enough so that $\nu_{q}$ is pure-like,
and from Lemma \ref{lem:x0}, $\Theta_{\nu,q}(-E_{0}(q),x)<0$ for
any $x\neq-x_{0}(q)$. Lemma \ref{lem:epsx0} therefore follows from
the continuity of $\Theta_{\nu,q}(u,x)$ in $u$ and $x$.
\end{proof}
\corO{We can now provide the proof of 
Corollary \ref{cor:orth}.}
\begin{proof}[Proof of Corollary \ref{cor:orth}.]
From Lemma \ref{lem:epsx0}, to conclude the proof it will be enough
to show that for any $\epsilon>0$, if $\eta,\,c>0$ are small enough
then\corE{
\begin{equation}
\mathbb{P}\Big\{\exists\bs_i\in\mathscr{C}_{N,q_i}(NB_i(\eta),\mathbb{R}\setminus\sqrt{N}D_i(\eta)),\,\bs_1\neq\pm\bs_2:\,|R(\bs_1,\bs_2)|\geq\epsilon\Big\}<e^{-cN},\label{eq:0707-01-1}
\end{equation}
where we define $B_i(\eta)=-E_{0}(q_i)+(-\eta,\eta)$ and $D_i(\eta)=-x_{0}(q_i)+(-\eta,\eta)$.}

\corE{By Theorem \ref{thm:2ndmomUBBK-1-1}, the corresponding number of
pairs $(\bs_1,\,\bs_2)$ of points, 
\[
K(\epsilon,\eta):=\left[\mbox{Crt}_{N,q_1,q_2}(I(\epsilon),B_1(\eta),B_2(\eta),D_1(\eta),D_2(\eta))\right]_{2},
\]
where $I(\epsilon)=(-1,1)\setminus(-\epsilon,\epsilon)$, satisfies
\begin{equation}
\limsup_{N\to\infty}\frac{1}{N}\log\mathbb{E}K(\epsilon,\eta)\leq\sup_{r\in I(\epsilon),u_{i}\in B_i(\eta),x_{i}\in D_i(\eta)}\Psi_{\nu,q_1,q_2}\left(r,u_{1},u_{2},x_{1},x_{2}\right).\label{eq:0707-03}
\end{equation}
From  Lemma \ref{lem:conccentration}, the fact that $\Theta_{\nu,q_i}\left(u_{i},x_{i}\right)$
is continuous in $u_{i}$ and $x_{i}$  and equal to $0$ when $u_{i}=-E_{0}(q_i)$
and $x_{i}=-x_{0}(q_i)$, we have that if $|1-q_i|$ and $\eta$ are small
enough then the right-hand of (\ref{eq:0707-03}) in negative. By
Markov's inequality, this proves (\ref{eq:0707-01-1}). The bound
of (\ref{eq:0707-02}) follows by a standard diagonalization argument.}\end{proof}

\section{\label{sec:CDTNandCt}\corO{Stability of \CDTN\ and 
further consequences}}
\corO{In this short section, we provide several further consequences of \CDTN. 
The proofs utilize some of the results in Section \ref{sec:Matching}.}
\begin{lem}
  \label{lem:ddq}\corO{Assume \CDTN. Then, there exists $\delta>0$ so that,} 
    for any
	$q\in(1-\delta,1]$, $\frac{d}{dq}E_{0}(q)=x_{0}(q)$.
\end{lem}
The next 
two propositions concern the stability of \CDTN\ under perturbations of $\nu$.
\begin{prop}
	\label{prop:CDTNp}For any $p\geq3$ there exists a $\delta>0$ such
	that if $\|\nu(x)-x^{p}\|<\delta$, then $\nu$ also satisfies \CDTN.
\end{prop}
\begin{prop}
	\label{prop:openCDTN}If $\nu$ is a non-pure mixture satisfying \CDTN,
	then for some $\delta>0$, any $\bar{\nu}$ for which $\|\bar{\nu}-\nu\|<\delta$
	also satisfies \CDTN.
\end{prop}
\corO{The rest of the section is devoted to the proofs.}
\begin{proof}[Proof of Proposition \ref{prop:openCDTN}.]
Since all the expressions in the definition of $\Psi_{\nu,q_{1},q_{2}}$,
including the elements of the matrix $\Sigma_{U,X}$, involve only
polynomials in $r$ which do not have a linear term, we have that
$\frac{d}{dr}\Psi_{\nu,q_{1},q_{2}}\left(0,u_{1},u_{2},x_{1},x_{2}\right)=0$
for any $\nu$, $q_{i}$, $u_{i}$ and $x_{i}$. Therefore Proposition
\ref{prop:openCDTN} is a direct consequence of \corO{Lemmas \ref{lem:Excts} and \ref{lem:PhiCt}.} 
\end{proof}
\begin{proof}[Proof of Lemma \ref{lem:ddq}.]
By Remark \ref{rem:GSq},
\[
\lim_{N\to\infty}\frac{1}{N}\min_{\bs\in\mathbb{S}^{N-1}(\sqrt{N})}H_{N}(q\bs)=-E_{0}(q),\quad\text{almost surely},
\]
for any $q\in(1-\delta,1]$, if $\delta>0$ is small enough. 

Combined with (\ref{eq:1507-05}), this implies that, with probability
tending to $1$, if $\bs_{q}$ is a global minimum point of $\mathbb{S}^{N-1}(\sqrt{N}q)\ni\bs\mapsto H_{N}(\bs)$,
then 
\[
\Big|\frac{1}{N}H_{N}(\bs_{q})+E_{0}(q)\Big|,\,\Big|\frac{1}{\sqrt{N}}\ddq H_{N}(\bs_{q})+x_{0}(q)\Big|<\epsilon_{N},
\]
for some sequence $\epsilon_{N}=o(1)$.

From the uniform bound on the Lipschitz constant of the Hessian from
Corollary \ref{cor:gradbd} (with $k=2$), with probability tending
to $1$,
\begin{align*}
\frac{1}{N}\left|H_{N}(\frac{1}{q}\bs_{q})-\left(H_{N}(\bs_{q})+(1-q)\frac{1}{\sqrt{N}}\ddq H_{N}(\bs_{q})\right)\right| & \leq\tilde{C}_{2}(1-q)^{2},\\
\frac{1}{N}\left|H_{N}(q\bs_{1})-\left(H_{N}(\bs_{1})-(1-q)\frac{1}{\sqrt{N}}\ddq H_{N}(\bs_{1})\right)\right| & \leq\tilde{C}_{2}(1-q)^{2},
\end{align*}
for some constant $\tilde{C}_{2}>0$. Therefore, 
\begin{align*}
-E_{0}(1) & \leq-E_{0}(q)-x_{0}(q)(1-q)+\tilde{C}_{2}(1-q)^{2},\\
-E_{0}(q) & \leq-E_{0}(1)+x_{0}(1)(1-q)+\tilde{C}_{2}(1-q)^{2},
\end{align*}
and
\[
x_{0}(q)-\tilde{C}_{2}(1-q)\leq\frac{E_{0}(1)-E_{0}(q)}{1-q}\leq x_{0}(1)+\tilde{C}_{2}(1-q).
\]
Since \corO{$q\mapsto x_{0}(q)$}
is continuous \corO{by} Lemma \ref{lem:Excts}, the proof
is completed.\end{proof}

\begin{proof}[Proof of Proposition \ref{prop:CDTNp}.]
Throughout the proof we shall use the notation
\[
\Psi_{\nu}\left(r,u,x\right):=\Psi_{\nu,1,1}\left(r,u,u,x,x\right).
\]
We will also always assume that $q$, $q_{1}$ and $q_{2}$ are equal
to $1$ and omit them from notation, writing, for example, ${\rm Crt}_{N}(B,\mathbb{R})$
or $\Sigma_{U}(r)$ for 
${\rm Crt}_{N,1}(B,\mathbb{R})$ and
$\Sigma_{U}(r,1,1)$.

For the pure case $\nu_{p}(x)=x^{p}$, similarly to (\ref{eq:1106-1-1}),
it was proved in \cite[Theorem 2.8]{A-BA-C} that for any intervals
$B\subset(-\infty,0$), 
\begin{align*}
\lim_{N\to\infty}\frac{1}{N}\log\left(\mathbb{E}{\rm Crt}_{N}\left(B,\mathbb{R}\right)\right) & =\sup_{u\in B}\Theta_{p}(u),
\end{align*}
where
\begin{align*}
\Theta_{p}(u) & :=\frac{1}{2}+\frac{1}{2}\log\left(p-1\right)-\frac{u^{2}}{2}+\Omega\left(\sqrt{\frac{p}{p-1}}u\right).
\end{align*}
Define $E_{\infty}(\nu_{p})=E_{\infty}(p)=2\sqrt{(p-1)/p}$ and $E_{0}(\nu_{p})=E_{0}(p)$
as the unique number $E\in(E_{\infty}(p),\infty)$ such that $\Theta_{p}(E)=0$,
and set $-x_{0}(p)=-\nu_{p}^{\prime}(1)E_{0}(p)$. The above definitions
can be appropriately extended to unnormalized pure models $\nu(x)=\gamma^{2}x^{p}$. 

Similarly to (\ref{eq:2ndUB_BK-1-1-1}), it was proved in \cite[Theorem 5]{2nd},
that for the pure case $\nu_{p}(x)=x^{p}$, for any intervals $B\subset(-\infty,0$),
$I\subset(-1,1)$, 
\begin{align*}
\limsup_{N\to\infty}\frac{1}{N}\log\left(\mathbb{E} \left[{\rm Crt}_{N}\left(B,B,\mathbb{R},\mathbb{R}I\right)\right]_{2} \right) & \leq\sup_{r\in I,u_{i}\in B}\Psi_{p}\left(r,u_{1},u_{2}\right),
\end{align*}
where, with $\Sigma_{U}\left(r,1,1\right)$ as defined for the mixed
case (\ref{eq:26}),
\begin{align*}
&\Psi_{p}\left(r,u_{1},u_{2}\right) \\
&:=1+\frac{1}{2}\log\left((p-1)^{2}\frac{1-r^{2}}{1-r^{2p-2}}\right)
  -\frac{1}{2}\left(u_{1},u_{2}\right)\Sigma_{U}^{-1}\left(r,1,1\right)\left(\begin{array}{c}
u_{1}\\
u_{2}
\end{array}\right)+\Omega\left(\sqrt{\frac{p}{p-1}}u_{1}\right)+\Omega\left(\sqrt{\frac{p}{p-1}}u_{2}\right).
\end{align*}
By Lemma 7 of \cite{2nd}, for any $\epsilon>0$,
\begin{equation}
\Psi_{p}\left(0,-E_{0}(p),-E_{0}(p)\right)>\sup_{|r|\in(\epsilon,1)}\Psi_{p}\left(r,-E_{0}(p),-E_{0}(p)\right).\label{eq:201}
\end{equation}

From Lemma \ref{lem:condHamiltonian}, for mixed $H_{N}(\bs)$, the
conditional mean and covariance of $(\frac{d}{dR}H_{N}(\hat{\mathbf{n}}),\frac{d}{dR}H_{N}(\boldsymbol{\sigma}(r)))$
given 
\begin{equation}
H_{N}(\hat{\mathbf{n}})=H_{N}(\boldsymbol{\sigma}(r))=u,\,\,\grad H_{N}(\hat{\mathbf{n}})=\grad H_{N}(\boldsymbol{\sigma}(r))=0,\label{eq:0104-01}
\end{equation}
are equal respectively to
\begin{align}
m(r,u) & =(\Sigma_{b,11}\left(r\right),\Sigma_{b,21}\left(r\right))\Sigma_{U}^{-1}\left(r\right)(u,u)^{T},\nonumber \\
\Sigma_{\bar{X}}\left(r\right) & =\Sigma_{X}\left(r\right)-\Sigma_{b}^{T}\left(r\right)\Sigma_{U}^{-1}\left(r\right)\Sigma_{b}\left(r\right),\label{eq:SigmaXbar}
\end{align}
where invertibility follows from Lemma \ref{lem:invertibility}. 

Denoting $\bar{x}=x-m(r,u)$, for any non-pure $\nu$, 
\begin{align*}
 & \Psi_{\nu}\left(r,u,x\right)=\\
 & 1+\frac{1}{2}\log\left((1-r^{2})\frac{\nu''(1)^{2}}{\nu'(1)^{2}-(\nu'(r))^{2}}\right)+2\Omega\left(\frac{x}{\sqrt{\nu''(1)}}\right)
  -\frac{1}{2}\left(u,u\right)\Sigma_{U}^{-1}\left(r\right)\left(u,u\right)^{T}-\frac{1}{2}\left(\bar{x},\bar{x}\right)\Sigma_{\bar{X}}^{-1}\left(r\right)\left(\bar{x},\bar{x}\right)^{T}.
\end{align*}
We note that $\Sigma_{U,11}\left(r\right)=\Sigma_{U,22}\left(r\right)$
and $\Sigma_{U,12}\left(r\right)=\Sigma_{U,21}\left(r\right)$, and
therefore $\Sigma_{U,11}(r)\pm\Sigma_{U,12}(r)$ are the eigenvalues
of $\Sigma_{U}\left(r\right)$ that correspond to the eigenvectors
$(1,\pm1)$. The same holds for $\Sigma_{X}\left(r\right)$ and $\Sigma_{\bar{X}}\left(r\right)$.
Thus, 
\[
m(r,u)=\frac{\Sigma_{b,11}\left(r\right)+\Sigma_{b,21}\left(r\right)}{\Sigma_{U,11}(r)+\Sigma_{U,12}(r)}u,
\]
and we have that 
\begin{equation}
  \Psi_{\nu}\left(r,u,x\right)=
 1+\frac{1}{2}\log\left((1-r^{2})\frac{\nu''(1)^{2}}{\nu'(1)^{2}-(\nu'(r))^{2}}\right)+2\Omega\left(\frac{x}{\sqrt{\nu''(1)}}\right)
  -\frac{u^{2}}{\Sigma_{U,11}(r)+\Sigma_{U,12}(r)}-\frac{\bar{x}^{2}}{\Sigma_{\bar{X},11}(r)+\Sigma_{\bar{X},12}(r)}.\nonumber
\label{eq:Psi2}
\end{equation}

Set, for any mixture $\nu$,
\[\tilde{\Psi}_{\nu}\left(r,u\right)  =1+\frac{1}{2}\log\left((1-r^{2})\frac{\nu''(1)^{2}}{\nu'(1)^{2}-(\nu'(r))^{2}}\right)
-\frac{u^{2}}{\Sigma_{U,11}(r)+\Sigma_{U,12}(r)}.\]
We note that
\begin{equation}
\tilde{\Psi}_{\nu_{p}}\left(r,u\right)=\Psi_{p}\left(r,u,u\right)-2\Omega\left(\sqrt{\frac{p}{p-1}}u\right).\label{eq:purePsitilde}
\end{equation}
\begin{lem}
\label{lem:Psitilde}\corO{Assume $\nu$ is non-pure. Then, 
for any $p\geq3$, the following holds.}
\begin{enumerate}
\item \label{enu:Psitilde1}$\tilde{\Psi}_{\nu}\left(r,u\right)$ and its
first and second derivatives in $r$ are continuous functions of $r\in(-1,1)$,
$u\in\mathbb{R}$, and $\nu$ in a small neighborhood of $\nu_{p}$
(w.r.t. the norm $\|\cdot\|)$. 
\item \label{enu:Psitilde2}We have that 
\[
\limsup_{\left(\nu,r,u\right)\to\left(\nu_{p},1,-E_{0}(p)\right)}\tilde{\Psi}_{\nu}\left(r,u\right)\leq\limsup_{r\to1}\tilde{\Psi}_{\nu_{p}}\left(r,-E_{0}(p)\right),
\]
and the same also holds with the $r\to1$ limits replaced by $r\to-1$.
\end{enumerate}
\end{lem}

\begin{proof}
The lemma follows from (\ref{eq:purePsitilde}) and the definition
of the function $\tilde{\Psi}_{\nu}\left(r,u\right)$, since as $\nu\to\nu_{p}$
the corresponding derivatives in $r$, up to order $4$, converge
by Remark \ref{rem:norm}, and since $\Omega$ is smooth in the
neighborhood of 
\[
\lim_{(\nu,u)\to(\nu_{p},-E_{0}(p))}\frac{\nu'(1)u}{\sqrt{\nu''(1)}}=-\sqrt{\frac{p}{p-1}}E_{0}(p)<-2,
\]
\corO{due to} $E_{0}(p)>E_{\infty}(p)=2\sqrt{(p-1)/p}$.
\end{proof}
\corO{Continuing with the proof of 
  Proposition \ref{prop:CDTNp}, we use the notation} 
\[
\mathcal{Q}(r)=\mathcal{Q}_{\nu}(r):=-\frac{(-x_{0}(\nu)-m(r,-E_{0}(\nu)))^{2}}{\Sigma_{\bar{X},11}(r)+\Sigma_{\bar{X},12}(r)},
\]
for the quadratic term as in (\ref{eq:Psi2}) with $(u,x)=(-E_{0}(\nu),-x_{0}(\nu))$.
Recall that using (\ref{eq:0210-01}) we had that $|-x_{0}(\nu)+\nu'(1)E_{0}(\nu)|\leq2C\alpha_{\nu}^{2}/\sqrt{\nu''\left(1\right)}$,
for an appropriate constant $C$. For $r=0$, 
\[
m(0,-E_{0}(\nu))=-\nu'(1)E_{0}(\nu)\text{\,\, and \,\,}\Sigma_{\bar{X},11}(r)+\Sigma_{\bar{X},12}(r)=\alpha_{\nu}^{2},
\]
and therefore 
\begin{equation}
|\mathcal{Q}(0)|\leq4C^{2}\alpha_{\nu}^{2}/\nu''\left(1\right).\label{eq:0210-03}
\end{equation}

By straightforward algebra, (assuming $p\geq3$)
\begin{equation}
\frac{d^{2}}{dr^{2}}\tilde{\Psi}_{\nu_{p}}\left(0,-E_{0}(p)\right)=-1,\quad\forall\nu,\,E:\,\,\frac{d}{dr}\tilde{\Psi}_{\nu}\left(0,E\right)=0,\label{eq:0210-02}
\end{equation}
and, by (\ref{eq:201}) and (\ref{eq:purePsitilde}),
\begin{equation}
\tilde{\Psi}_{\nu_{p}}\left(0,-E_{0}(p)\right)>\sup_{|r|\in(\epsilon,1)}\tilde{\Psi}_{\nu_{p}}\left(r,-E_{0}(p)\right).\label{0210-03}
\end{equation}
Therefore, from Lemmas \ref{lem:Excts} and \ref{lem:Psitilde}, if
$\|\nu-\nu_{p}\|$ is small enough, then for some $\epsilon>0$ and
any $r$ with $|r|\leq\epsilon$,
\begin{equation}
\tilde{\Psi}_{\nu}\left(r,-E_{0}(\nu)\right)\leq\tilde{\Psi}_{\nu}\left(0,-E_{0}(\nu)\right)-r^{2}/4.\label{eq:0210-04}
\end{equation}
Since $\mathcal{Q}(r)\leq0$ and (\ref{eq:0210-03}), from the above,
if $\|\nu-\nu_{p}\|$ is small enough, then for any $r\in(-1,1)$
with $|r|>4C\alpha_{2}/\sqrt{\nu''(1)}$,
\[
\Psi_{\nu}\left(r,-E_{0}(\nu),-x_{0}(\nu)\right)<\Psi_{\nu}\left(0,-E_{0}(\nu),-x_{0}(\nu)\right).
\]

Since for any $\nu$, $E$ and $x$, $\frac{d}{dr}\Psi_{\nu}\left(0,E,x\right)=0$,
in light of (\ref{eq:0210-04}), the proof of Proposition \ref{prop:CDTNp}
will be completed if we prove that 
\begin{equation}
|r|\leq4C\alpha_{2}/\sqrt{\nu''(1)}\,\Longrightarrow\,|\mathcal{Q}(r)-\mathcal{Q}(0)|\leq r^{2}/8,\label{eq:0310-01}
\end{equation}
assuming $\|\nu-\nu_{p}\|$ is small enough. This follows from Taylor
expanding each of the terms in the definition of $\mathcal{Q}(r)$.
More precisely,\footnote{We remind the reader that we are omitting $q=1$ from our notation,
so that for example, $a_{2}\left(r\right)$ stands for $a_{2}\left(r,1,1\right)$,
etc. } for any $r\in(-\eta,\eta)$ we have that $a_{2}\left(r\right)$ and
$a_{4}\left(r\right)$ are bounded in absolute value by some constant
$c'>0$ and 
\begin{align*}
 & \Sigma_{U,11}(r)-\nu\left(1\right),\,\,\Sigma_{U,12}(r),\,\,\Sigma_{b,11}\left(r\right)-\nu'(1),\\
 & \Sigma_{b,12}\left(r\right),\,\,\Sigma_{X,11}\left(r\right)-\nu''(1)-\nu'(1)\,\,\text{and}\,\,\Sigma_{X,12}\left(r\right)
\end{align*}
are bounded in absolute value by $cr^{2}$, where $c$ can be taken
to be as small as we wish provided that $\eta$ and $\nu''(1)$ are
small enough. For arbitrary $C'>0$, assuming that $\|\nu-\nu_{p}\|$
is sufficiently small, we therefore have that for any $r\in[-C'\alpha_{\nu},C'\alpha_{\nu}]$,
\begin{align*}
|\Sigma_{\bar{X},11}\left(r\right)+\Sigma_{\bar{X},12}\left(r\right)-\alpha_{\nu}^{2}| & \leq c_{2}^{\prime}r^{2},\\
|m(r,-E_{0}(\nu))-\nu'(1)E_{0}(\nu)| & \leq c_{2}^{\prime}r^{2},
\end{align*}
from which (\ref{eq:0310-01}) follows, since $\alpha_{2}\to0$ as
$\nu\to\nu_{p}$. This completes the proof of 
Proposition \ref{prop:CDTNp}.
\end{proof}

\section{\label{sec:models_on_bands}Conditional models on sections}
\corO{In this section, we show that conditionally on
  the value of $H_N$ at a point and its first order derivatives there,
  one obtains an effective mixed-model on the 
  $N-2$ dimensional sections of the
  sphere determined by a fixed overlap with that point, that is, on appropriate ``bands''. The main result
  are Lemmas \ref{cor:conditional} and \ref{lem:apxest}.}

For indices $i_{1},...,i_{p}$, denote by $\bar{J}_{i_{1},...,i_{p}}^{(p)}$
the sum of all $J_{i_{1}^{\prime},...,i_{p}^{\prime}}^{(p)}$ such
that $\{i_{1}^{\prime},...,i_{p}^{\prime}\}=\{i_{1},...,i_{p}\}$
as multisets. For $\bx\in\mathbb{R}^{N}$ with $\|\bx\|\leq\sqrt{N}$,
we may write
\begin{equation}
H_{N}\left(\bx\right)=\sum_{p=2}^{\infty}\frac{\gamma_{p}}{N^{\left(p-1\right)/2}}\sum_{i_{1}\leq...\leq i_{p}\leq N}\bar{J}_{i_{1},...,i_{p}}^{(p)}x_{i_{1}}\cdots x_{i_{p}}.\label{eq:Havg}
\end{equation}
We are interested in the structure of the Hamiltonian $H_{N}\left(\bs\right)$
on the \corO{section} 
\begin{equation}
\mathcal{S}(q\nh)=\{\bs\in\mathbb{S}^{N-1}(\sqrt{N}):\,R(\bs,\hat{\mathbf{n}})=q\}\label{eq:Sq}
\end{equation}
formed by points with fixed overlap relative to the point $\hat{\mathbf{n}}=(0,...,0,\sqrt{N})$,
\corO{see \eqref{eq:subsp}.}
We therefore introduce, for $\bs\in\SN$,
\begin{equation}
\tilde{\boldsymbol{\sigma}}=\sqrt{\frac{N-1}{N}}\frac{\left(\sigma_{1},...,\sigma_{N-1}\right)}{\sqrt{1-q^{2}\left(\boldsymbol{\sigma}\right)}}\in\mathbb{S}^{N-2}(\sqrt{N-1})\,\,\,{\rm and}\,\,\,q\left(\boldsymbol{\sigma}\right)=\frac{\sigma_{N}}{\sqrt{N}}=R(\bs,\hat{\mathbf{n}}).\label{eq:2610-1}
\end{equation}
Thinking of $\tilde{\boldsymbol{\sigma}}$ as coordinates in an $N-1$-dimensional
sphere, we group the terms in (\ref{eq:Havg}) corresponding to $k$-spin
interactions,
\begin{equation}
\begin{aligned}\bar{H}_{N}^{\hat{\mathbf{n}},k}\left(\boldsymbol{\sigma}\right)= & \sum_{p=k}^{\infty}\frac{\gamma_{p}}{N^{\left(p-1\right)/2}}\sum_{1\leq i_{1}\leq...\leq i_{k}\leq N-1}\bar{J}_{i_{1},...,i_{k},N,...,N}^{(p)}\sigma_{i_{1}}\cdots\sigma_{i_{k}}\sigma_{N}^{p-k}\\
= & \sum_{p=k}^{\infty}\gamma_{p}\binom{p}{k}^{1/2}\left(1-q^{2}\left(\boldsymbol{\sigma}\right)\right)^{k/2}(q\left(\boldsymbol{\sigma}\right))^{p-k}\sqrt{\frac{N}{N-1}}H_{N-1,k}^{p}\left(\tilde{\boldsymbol{\sigma}}\right),
\end{aligned}
\label{eq:1907-01}
\end{equation}
where
\[
H_{N-1,k}^{p}\left(\tilde{\boldsymbol{\sigma}}\right):=(N-1)^{-\frac{k-1}{2}}\binom{p}{k}^{-1/2}\sum_{1\leq i_{1}\leq...\leq i_{k}\leq N-1}\bar{J}_{i_{1},...,i_{k},N,...,N}^{(p)}\tilde{\sigma}_{i_{1}}\cdots\tilde{\sigma}_{i_{k}}.
\]
Since for different $(p,k)$ the models $H_{N-1,k}^{p}\left(\tilde{\boldsymbol{\sigma}}\right)$
are measurable w.r.t to disjoint sets of the coefficients $(\bar{J}_{i_{1},...,i_{p}}^{(p)})$,
we have that
\begin{equation}
H_{N}\left(\boldsymbol{\sigma}\right)=\sum_{k=0}^{\infty}\bar{H}_{N}^{\hat{\mathbf{n}},k}\left(\boldsymbol{\sigma}\right),\text{ and \ensuremath{\bar{H}_{N}^{\hat{\mathbf{n}},k}\left(\boldsymbol{\sigma}\right)} are independent.}
\label{eq:1511-E1}
\end{equation}
Also note that, since $\text{Var}(\bar{J}_{i_{1},...,i_{k},N,...,N}^{(p)})=\binom{p}{k}\text{Var}(\bar{J}_{i_{1},...,i_{k}}^{(k)})$,
for each $k$, $H_{N-1,k}^{p}\left(\tilde{\boldsymbol{\sigma}}\right)$
is a pure $k$-spin models on $\mathbb{S}^{N-2}(\sqrt{N-1})$ (where
for $k=0$, the `$0$-spin' model $H_{N-1,0}^{p}\left(\tilde{\boldsymbol{\sigma}}\right)\equiv(N-1)^{\frac{1}{2}}\bar{J}_{N,...,N}^{(p)}
\corO{=(N-1)^{\frac{1}{2}}{J}_{N,...,N}^{(p)}}$
is a random variable which is \corE{constant as function of $\tilde{\bs}$}). Hence, setting
\begin{equation}
\alpha_{k}(q):=\left(1-q^{2}\right)^{k/2}\left(\sum_{p=k}^{\infty}\gamma_{p}^{2}\binom{p}{k}q^{2(p-k)}\right)^{\frac{1}{2}}\label{eq:alpha}
\end{equation}
and letting $H_{N-1,k}\left(\tilde{\boldsymbol{\sigma}}\right)$ denote
a pure $k$-spin model, we have that
\begin{equation}
\bar{H}_{N}^{\hat{\mathbf{n}},k}\left(\boldsymbol{\sigma}\right)\stackrel{d}{=}\sqrt{\frac{N}{N-1}}\alpha_{k}(q(\boldsymbol{\sigma}))H_{N-1,k}\left(\tilde{\boldsymbol{\sigma}}\right).\label{eq:HNhatk}
\end{equation}

From (\ref{eq:Havg}), for any $1\leq i_{1}\leq...\leq i_{k}\leq N-1$,
the Euclidean derivatives at $q\hat{\mathbf{n}}$ are given by 
\begin{equation}
\frac{d}{dx_{i_{1}}}\cdots\frac{d}{dx_{i_{k}}}H_{N}\left(q\hat{\mathbf{n}}\right)=\sum_{p=k}^{\infty}\gamma_{p}q^{p-k}\frac{\prod_{j=1}^{N-1}|\{l\leq k:\,i_{l}=j\}|!}{N^{\left(k-1\right)/2}}\bar{J}_{i_{1},...,i_{k},N,...,N}^{(p)}.\label{eq:Eucdrv}
\end{equation}
Thus, on $\mathcal{S}(q\nh)$, we can view the representation (\ref{eq:1907-01})
as a Taylor series of $H_{N}\left(\boldsymbol{\sigma}\right)$ around
$q\hat{\mathbf{n}}$. Namely, for any $\bs\in\mathcal{S}(q\nh)$,
\begin{equation}
\begin{aligned}\bar{H}_{N}^{\hat{\mathbf{n}},0}\left(\boldsymbol{\sigma}\right) & =H_{N}\left(q\hat{\mathbf{n}}\right),\\
\bar{H}_{N}^{\hat{\mathbf{n}},k}\left(\boldsymbol{\sigma}\right) & =\frac{1}{k!}\sum_{1\leq i_{1},...,i_{k}\leq N-1}\frac{d}{dx_{i_{1}}}\cdots\frac{d}{dx_{i_{k}}}H_{N}\left(q\hat{\mathbf{n}}\right)\sigma_{i_{1}}\cdots\sigma_{i_{k}},
\quad \corO{k\geq 1}.
\end{aligned}
\label{eq:1907-02}
\end{equation}

We would like to relate $\grad H_{N}\left(\boldsymbol{\sigma}\right)$
and $\Hess H_{N}\left(\boldsymbol{\sigma}\right)$, defined using
the frame field $F_{i}$ \corO{as in} (\ref{eq:derivatives}), to the Euclidean
derivatives (\ref{eq:Eucdrv}). Therefore, in this section we will
assume that $F_{i}$ is chosen so that (see \cite[Footnote 7]{geometryGibbs})
\begin{equation}
\begin{aligned}\grad H_{N}\left(q\hat{\mathbf{n}}\right) & =\left(\left.\frac{d}{dx_{i}}\right|_{\bx=0}H_{N}\left((x_{1},...,x_{N-1},q\sqrt{N-\|\bx\|^{2}}\right)\right)_{i\leq N-1},\\
\Hess H_{N}\left(q\hat{\mathbf{n}}\right) & =\left(\left.\frac{d}{dx_{i}}\frac{d}{dx_{j}}\right|_{\bx=0}H_{N}\left((x_{1},...,x_{N-1},q\sqrt{N-\|\bx\|^{2}}\right)\right)_{i,j\leq N-1},
\end{aligned}
\label{eq:1507-08}
\end{equation}
where $\bx=\bx_{N-1}=(x_{1},...,x_{N-1})$. (Note that $F_i$ from Lemma \ref{lem:cov} satisfy \eqref{eq:1507-08}, see \eqref{eq:1511---}.) 
Under this assumption, see the proof of \cite[Lemma 2]{geometryGibbs},
\begin{align}
\grad H_{N}\left(q\hat{\mathbf{n}}\right) & =\left(\frac{d}{dx_{i}}H_{N}\left(q\hat{\mathbf{n}}\right)\right)_{i\leq N-1}=\left(\sum_{p=2}^{\infty}\gamma_{p}q^{p-1}\bar{J}_{i,N,...,N}^{(p)}\right)_{i\leq N-1},\nonumber \\
\Hess H_{N}\left(q\hat{\mathbf{n}}\right) & =\left(\frac{d}{dx_{i}}\frac{d}{dx_{j}}H_{N}\left(q\hat{\mathbf{n}}\right)\right)_{i,j\leq N-1}-\frac{1}{\sqrt{N}q}\ddq H_{N}(q\hat{\mathbf{n}})\mathbf{I},\label{eq:1507-09}
\end{align}
where the $N-1\times N-1$ matrix $\mathbf{G}\left(q\hat{\mathbf{n}}\right)=\mathbf{G}_{N-1}\left(q\hat{\mathbf{n}}\right)$
defined by 
\begin{equation}
\mathbf{G}\left(q\hat{\mathbf{n}}\right):=\left(\frac{d}{dx_{i}}\frac{d}{dx_{j}}H_{N}\left(q\hat{\mathbf{n}}\right)\right)_{i,j\leq N-1}=\sum_{p=2}^{\infty}\frac{\gamma_{p}q^{p-2}(1+\delta_{ij})}{N^{1/2}}\bar{J}_{i,j,N,...,N}^{(p)}\label{eq:1507-10}
\end{equation}
has the same law as $\sqrt{\frac{N-1}{N}\nu''(q^{2})}\mathbf{M}$,
where $\mathbf{M}=\mathbf{M}_{N-1}$ is a GOE matrix.

For random variables that are continuous functions of the Gaussian
disorder $(J_{i_{1},...,i_{p}})_{p\geq2}$, we will denote by $\mathbb{P}_{u,v,\mathbf{A}}^{q}\left\{ \,\cdot\,\right\} $
the conditional probability given
\begin{equation}
H_{N}(q\hat{\mathbf{n}})=u,\,\ddq H_{N}(q\hat{\mathbf{n}})=v,\,\grad H_{N}(q\hat{\mathbf{n}})=0{\rm \,\,and\,\,\mathbf{G}_{N-1}\left(q\hat{\mathbf{n}}\right)=\mathbf{A},}\label{eq:24-1}
\end{equation}
interpreted in the usual way by restricting $(J_{i_{1},...,i_{p}})_{p\geq2}$
to the appropriate affine subspace with the restriction of the density
of $(J_{i_{1},...,i_{p}})_{p\geq2}$, normalized. Similarly, $\mathbb{P}_{u,v}^{q}\left\{ \,\cdot\,\right\} $
(respectively, $\mathbb{P}_{u,\mathbf{A}}^{q}\left\{ \,\cdot\,\right\} $)
will denote the conditional probability given only the first three
equalities (respectively, all equalities but the second). The corresponding
expectations will be denoted with $\mathbb{P}$ replaced by $\mathbb{E}$.

\corO{Let $\theta_{q}:\mathbb{S}^{N-2}(\sqrt{N-1})\to\mathcal{S}(q\nh)$
denote the (left) inverse} 
of $\boldsymbol{\sigma}\mapsto\tilde{\boldsymbol{\sigma}}$,
see (\ref{eq:2610-1}), given by 
\[
\theta_{q}\left(\left(\sigma_{1},...,\sigma_{N-1}\right)\right)=\sqrt{\frac{N}{N-1}\left(1-q^{2}\right)}\left(\sigma_{1},...,\sigma_{N-1},0\right)+q\hat{\mathbf{n}}.
\]
For any function $h:\mathbb{S}^{N-1}(\sqrt{N})\to\mathbb{R}$, define
$h|_{q}:\mathbb{S}^{N-2}(\sqrt{N-1})\to\mathbb{R}$ by 
\begin{equation}
h|_{q}\left(\boldsymbol{\sigma}\right)=h\circ\theta_{q}\left(\boldsymbol{\sigma}\right).\label{eq:22}
\end{equation}
\begin{lem}
\label{cor:conditional}Under $\mathbb{P}_{u,v}^{q}\left\{ \,\cdot\,\right\} $
we have that
\begin{equation}
  \label{cor:conditional:eq1}
H_{N}|_{q}\left(\boldsymbol{\sigma}\right)\overset{d}{=}u+\sum_{k=2}^{\infty}\alpha_{k}(q)\sqrt{\frac{N}{N-1}}H_{N-1,k}\left(\bs\right),
\end{equation}
where $H_{N-1,k}\left(\sigma\right)$ are independent pure $k$-spin
models of dimension $N-1$. Under $\mathbb{P}_{u,v,\mathbf{A}}^{q}\left\{ \,\cdot\,\right\} $,
\begin{equation}
  \label{cor:conditional:eq2}
H_{N}|_{q}\left(\boldsymbol{\sigma}\right)  \overset{d}{=}u+\frac{1}{2}\frac{N}{N-1}(1-q^{2})\boldsymbol{\sigma}^{T}\mathbf{A}\boldsymbol{\sigma}+\sum_{k=3}^{\infty}\alpha_{k}(q)\sqrt{\frac{N}{N-1}}H_{N-1,k}\left(\bs\right).
\end{equation}
Also, if $H_{N}(\bs)$ is replaced by $\sum_{k=0}^{l}\bar{H}_{N}^{\hat{\mathbf{n}},k}\left(\boldsymbol{\sigma}\right)$
above, then the conditional law is given by the same formula but with
summation up to $l$ instead of $\infty$.
\end{lem}
\corO{An important aspect of Lemma \ref{cor:conditional} is that the
expressions in \eqref{cor:conditional:eq1} and
\eqref{cor:conditional:eq2} do not contain linear terms (i.e., with $k=1$).}

\begin{proof}
  By \eqref{eq:1511-E1}, we can write $H_{N}|_{q}=\sum_{k=0}^{\infty}\bar{H}_{N}^{\hat{\mathbf{n}},k}|_{q}\left(\boldsymbol{\sigma}\right)$. By \eqref{eq:1907-02}, if $H_N(q\nh)= u$, then $\bar{H}_{N}^{\hat{\mathbf{n}},0}|_{q}\left(\boldsymbol{\sigma}\right)\corO{=}u$.  By \eqref{eq:1907-02} and \eqref{eq:1507-09}, if $\grad H_N(q\nh)=0$, then 
  $\bar{H}_{N}^{\hat{\mathbf{n}},1}|_{q}\left(\boldsymbol{\sigma}\right)\corO{=} 0$. 
By \eqref{eq:1907-02} and \eqref{eq:1507-10}, if $\mathbf{G}_{N-1}\left(q\hat{\mathbf{n}}\right)=\mathbf{A}$, then
$\bar{H}_{N}^{\hat{\mathbf{n}},2}|_{q}\left(\boldsymbol{\sigma}\right)=
\frac{1}{2}\frac{N}{N-1}(1-q^{2})\boldsymbol{\sigma}^{T}\mathbf{A}\boldsymbol{\sigma}$.
Further, from \eqref{eq:HNhatk},
\begin{equation*}
\bar{H}_{N}^{\hat{\mathbf{n}},k}|_{q}\left(\boldsymbol{\sigma}\right)\stackrel{d}{=}\sqrt{\frac{N}{N-1}}\alpha_{k}(q)H_{N-1,k}\left(\boldsymbol{\sigma}\right),
\quad \corO{k\geq 1}.
\end{equation*}

Note that $\ddq H_{N}(q\hat{\mathbf{n}})$
\corO{and} $ H_{N}(q\hat{\mathbf{n}})$ 
are  measurable w.r.t the disorder coefficients $\bar{J}_{N,...,N}^{(p)}$. 
Similarly, by \eqref{eq:1507-09},  $\grad H_N(q\nh)$ 
\corO{is} measurable w.r.t the coefficients of the form 
$\bar{J}_{i_1,N,...,N}^{(p)}$. By \eqref{eq:1507-10}, 
$\mathbf{G}_{N-1}\left(q\hat{\mathbf{n}}\right)$  
\corO{is} measurable w.r.t the coefficients of the form $\bar{J}_{i_1,i_2,N,...,N}^{(p)}$. Lastly, for any $k\geq 3$, by \eqref{eq:1907-01}, $\{\bar{H}_{N}^{\hat{\mathbf{n}},k}|_{q}\left(\boldsymbol{\sigma}\right)\}_{\bs}$ is measurable w.r.t the coefficients of the form $\bar{J}_{i_1,\ldots,i_k,N,...,N}^{(p)}$.

The lemma 
follows by \corO{combining these facts, using that}
the disorder coefficients $\bar{J}_{i_1,\ldots,i_k,N,...,N}^{(p)}$ are independent for different values of $(k,p)$.
\end{proof}
The random fields $\bar{H}_{N}^{\nh,k}\left(\boldsymbol{\sigma}\right)$
can be developed around a general point $\bs_{0}\in\SN$ instead
of $\nh$. One way to do so is by using (\ref{eq:1907-02}) and rotating,
in an appropriate sense, our coordinate system to be aligned with
the direction corresponding to $\bs_{0}$. More intrinsically, we
can use the connection to Taylor expansions. That is, for any $\bs\in\SN$
such that $R(\bs,\bs_{0})=q$, we define $\bar{H}_{N}^{\bs_{0},k}\left(\boldsymbol{\sigma}\right)$
as the $k$-degree term in the Taylor series (in $\mathbb{R}^{N}$)
of $H_{N}(\bx)$ around $H_{N}(q\bs_{0})$ evaluated at $\bs$. 
Denote by $\nabla_{E}^{k}H_{N}(\bx)=(\frac{d}{dx_{i_{1}}}\cdots\frac{d}{dx_{i_{k}}}H_{N}(\bx))_{i_{1},...,i_{k}}$
the tensor of (Euclidean) derivatives of order $k$ of the Hamiltonian
$H_{N}(\bx)$, and for any tensor $\mathbf{T}=(t_{i_{1},...,i_{k}})_{i_{1},...,i_{k}\leq N}$
define
\begin{equation}
\|\mathbf{T}\|_{\infty}:=\frac{1}{\sqrt{N}}\sup_{\|\mathbf{y}^{(i)}\|=1}\left|\sum_{i_{1},...,i_{p}=1}^{N}t_{i_{1},...,i_{k}}y_{i_{1}}^{(1)}\cdots y_{i_{p}}^{(p)}\right|,\label{eq:inftynorm-1}
\end{equation}
where $\mathbf{y}^{(i)}=(y_{1}^{(i)},...,y_{N}^{(i)})$, and for a
linear subspace $V\subset\mathbb{R}^{N}$ define $\|\mathbf{T}\|_{\infty}^{V}$
similarly to (\ref{eq:inftynorm-1}) with the additional restriction
that $\mathbf{y}^{(i)}\in V$. 
\corO{For use later, we record the following
direct corollary of the definitions.}
\begin{cor}
\label{cor:tensorbd}Let $\bs_{0}\in\mathbb{S}^{N-1}(\sqrt{N})$ and
$q\in(0,1)$, and denote by $V_{0}$ the tangent space to $\mathbb{S}^{N-1}(\sqrt{N})$
at $\bs_{0}$. If $\|\nabla_{E}^{k}H_{N}(q\bs_{0})\|_{\infty}^{V_{0}}<CN^{-(k-1)/2}$,
then for all $\bs\in\mathbb{S}^{N-1}(\sqrt{N})$ with $R(\bs,\bs_{0})=q$,
\[
\big|\bar{H}_{N}^{\bs_{0},k}\left(\boldsymbol{\sigma}\right)\big|\leq NC(1-q^{2})^{k/2}/k!.
\]
\end{cor}

We \corO{conclude this section}
with a bound on the error of a finite degree approximation
of the expansion.
\begin{lem}
\label{lem:apxest}For any mixture $\nu$, there exist $C$, $c>0$
such that 
\begin{equation}
  \mathbb{P}\left\{ \exists k\geq 0,\bs_{0},
  \bs\in\mathbb{S}^{N-1}(\sqrt{N}):\,1-q(\bs)^{2}<c,
 \Big|H_{N}\left(\boldsymbol{\sigma}\right)-\sum_{i=0}^{k}\bar{H}_{N}^{\bs_{0},i}\left(\boldsymbol{\sigma}\right)\Big|\geq NC\Big(\frac{1-q(\bs)^{2}}{c}\Big)^{(k+1)/2}\right\} \leq e^{-cN},
\label{eq:0407-01-1}
\end{equation}
where we abbreviate $q(\bs):=R(\bs,\bs_{0})$.
\end{lem}

\begin{proof}
Since $H_{N}\left(\boldsymbol{\sigma}\right)=\sum_{i=0}^{\infty}\bar{H}_{N}^{\bs_{0},i}\left(\boldsymbol{\sigma}\right)$,
we need to derive an appropriate bound for $|\sum_{i=k+1}^{\infty}\bar{H}_{N}^{\bs_{0},i}\left(\boldsymbol{\sigma}\right)|$,
for any $k\geq0$. From Corollary \ref{cor:tensorbd} and \corO{the concentration inequality in \eqref{eq:1507-03}
of Corollary \ref{cor:gradbd},}
\corO{with $K$ the universal constant of Lemma \ref{lem:Lipschitz},}
we have the following. On an event whose complement has exponentially
small in $N$ probability, uniformly over all $\bs_{0}$, $\bs\in\mathbb{S}^{N-1}(\sqrt{N})$
and $i\geq1$, \corO{writing $\rho=\rho(\bs)=1-q(\bs)^2$,}
\[
  \big|\bar{H}_{N}^{\bs_{0},i}\left(\boldsymbol{\sigma}\right)\big|\leq 2NK\sum_{p\geq i}\gamma_{p}p^{1/2}\binom{p}{i}\corO{\rho^{i/2},}
\]
and
\begin{equation}
\big|\sum_{i=k+1}^{\infty}\bar{H}_{N}^{\bs_{0},i}\left(\boldsymbol{\sigma}\right)\big|\leq 2NK\sum_{i=k+1}^{\infty}\sum_{p= i}^\infty\gamma_{p}p^{1/2}\binom{p}{i}
\corO{c^{i/2} \left(\frac{\rho}{c}\right)^{i/2}\leq
2NK \left(\frac{\rho}{c}\right)^{(k+1)/2}
\sum_{i=k+1}^{\infty}\sum_{p= i}^\infty\gamma_{p}p^{1/2}\binom{p}{i}
c^{i/2},}\label{eq:1907-03}
\end{equation}
\corO{where in the last inequality we used that $\rho/c\leq 1$.}
We next claim that for small enough $c>0$, 
\begin{equation}
  \label{eq-1611-a}
  \sum_{i=1}^{\infty}\sum_{p= i}^\infty\gamma_{p}p^{1/2}\binom{p}{i}
c^{i/2}<\infty.
\end{equation}
\corO{
Indeed, interchanging the order of summation, the left hand side of 
\eqref{eq-1611-a} equals
\[\sum_{p= 1}^\infty \gamma_p p^{1/2}\sum_{ i=1}^p \binom{p}{i} c^{i/2}\leq
\sum_{p=1}^\infty \gamma_p (1+\sqrt{c})^p p^{1/2}<\infty,\]
where in the last inequality we used that 
$\sum_{p\geq 1} \gamma_p^2 \beta^p<\infty$ for some $\beta>1$ by assumption.}
\corO{Combining \eqref{eq:1907-03} with \eqref{eq-1611-a} 
completes the proof of the lemma.}
%
%
%
\end{proof}

\section{\label{sec:log-weight}The logarithmic weight functions $\Lambda_{Z,\beta}$
and $\Lambda_{F,\beta}^{2-}$}

In Section \ref{sec:models_on_bands}, we studied the structure of
the Hamiltonian on the \corO{section} $\mathcal{S}(\bs_{0})$ of co-dimension $1$ 
around an arbitrary point $\bs_{0}\in\BN$, see \eqref{eq:subsp}.
Those results can be used to compute the weight of thin
bands conditional on the center being a $q$-critical point $\bs_{q}$
such that $H_{N}(\bs_{q})\approx NE$. \corO{To discuss these, we introduce notation.}

For any measurable set $B\subset\mathbb{S}^{N-1}(\sqrt{N})$, set
\begin{equation}
\label{eq-16-11-b}
Z_{N,\beta}(B)=\int_{B}e^{-\beta H_{N}(\bs)}d\bs.
\end{equation}
 Using the uniform
Lipschitz bounds of Corollary \ref{cor:gradbd} for the gradient,
\corE{one has that with high probability, 
\[
\forall q\in (0,1),\,\delta>0:\quad
\lim_{\epsilon\to 0}\lim_{N\to\infty}\mathbb P\bigg\{ \sup_{\bs_0\in\mathbb S^{N-1}(\sqrt N q)} \bigg| \frac{1}{N}\log\frac{Z_{N,\beta}({\rm Band}(\bs_{0},\epsilon))}{Z_{N,\beta}(\bs_{0})} \bigg|>\delta \bigg\} = 0,
\]
where
$Z_{N,\beta}(\bs_{0})$ is defined in (\ref{eq:Zsigma}).}
Therefore,
\corO{we focus on  analyzing} weights of the form $Z_{N,\beta}(\bs_{0})$.
\corO{We begin with  the}
following  consequence of Lemma \ref{cor:conditional}. 
\begin{cor}
\label{cor:smallbandZ}For any $q\in(0,1)$,
\begin{equation}
\lim_{N\to\infty}\frac{1}{N}\log\mathbb{E}_{NE,v}^{q}\left\{ Z_{N,\beta}(q\hat{\mathbf{n}})\right\} =\Lambda_{Z,\beta}(E,q),\label{eq:Eband}
\end{equation}
where, with $\alpha_{k}(q)$ as defined in (\ref{eq:alpha}),
\begin{equation}
\begin{aligned}\Lambda_{Z,\beta}(E,q) & :=-\beta E+\frac{1}{2}\log(1-q^{2})+\frac{1}{2}\beta^{2}\sum_{k=2}^{\infty}\alpha_{k}^{2}(q)\\
 & =-\beta E+\frac{1}{2}\log(1-q^{2})+\frac{1}{2}\beta^{2}\left(\nu(1)-\alpha_{0}^{2}(q)-\alpha_{1}^{2}(q)\right).
\end{aligned}
\label{eq:LambdaZ}
\end{equation}
\end{cor}

\begin{proof}
The logarithmic term in (\ref{eq:LambdaZ}) is equal to the limit
of $\frac{1}{N}\log$ of the ratio of volumes in (\ref{eq:Zsigma}).
By Lemma \ref{cor:conditional}, 
\[
\frac{1}{N}\log\mathbb{E}_{NE,v}^{q}\left\{ H_{N}(\bs)\right\}  =E \quad \mbox{\rm and}\quad
\frac{1}{N}\log{\rm Var}_{NE,v}^{q}\left\{ H_{N}(\bs)\right\}  =\sum_{k=2}^{\infty}\alpha_{k}^{2}(q),
\]
for any $\bs\in\mathcal{S}(q\hat{\mathbf{n}})$, where ${\rm Var}_{NE,v}^{q}$
denotes the variance under $\mathbb{P}_{NE,v}^{q}$. \corO{Using the expression for the moment generating function
of Gaussian variables, 
$\mathbb{E}e^{\beta X}=e^{\beta\mu+\beta^{2}\sigma^{2}/2}$
for $X\sim N(\mu,\sigma^{2})$, completes the proof.}
\end{proof}
We shall see that for large $\beta$, the $q$-critical points that
contribute most to the partition function are the deepest ones, \corO{that is,}
points with $H_{N}(\bs_{q})=-N(E_{0}(q)+o(1))$. We are therefore
particularly interested in the function $q\mapsto\Lambda_{Z,\beta}(-E_{0}(q),q)$,
which by an abuse of notation we will denote by $\Lambda_{Z,\beta}(q)$. 

For pure-like $\nu$ and $q<1$ close enough to $1$, using Lemmas
\ref{lem:ddq} and \ref{lem:Excts}, 
\begin{align*}
\frac{d}{dq}\Lambda_{Z,\beta}(q) & =\beta x_{0}(q)-\frac{q}{1-q^{2}}-\beta^{2}\left(1-q^{2}\right)\sum_{p=2}^{\infty}\gamma_{p}^{2}p(p-1)q^{2p-3},\\
\frac{d^{2}}{dq^{2}}\Lambda_{Z,\beta}(q) & =\beta\frac{d}{dq}x_{0}(q)-\frac{1+q^{2}}{(1-q^{2})^{2}}+2\beta^{2}\sum_{p=2}^{\infty}\gamma_{p}^{2}p(p-1)q^{2p-2}
 -\beta^{2}\left(1-q^{2}\right)\sum_{p=2}^{\infty}\gamma_{p}^{2}p(p-1)(2p-3)q^{2p-4}.
\end{align*}
Hence, for any $0<T=T(\beta)=o(\sqrt{\beta})$, uniformly on $(0,T]$,
as $\beta\to\infty$, 

\begin{align}
\frac{1}{\beta}\frac{d}{dq}\Lambda_{Z,\beta}(1-\frac{t}{\beta}) & \longrightarrow x_{0}(1)-\frac{1}{2t}-2\nu''(1)t,\label{eq:2907-01}\\
\frac{1}{\beta^{2}}\frac{d^{2}}{dq^{2}}\Lambda_{Z,\beta}(1-\frac{t}{\beta}) & \longrightarrow-\frac{1}{2t^{2}}+2\nu''(1).\label{eq:2907-02}
\end{align}

We conclude that for pure-like $\nu$ and large $\beta$, $\Lambda_{Z,\beta}(q)$
has exactly two critical points in $[1-T/\beta,1)$, 
\begin{equation}
\qs:=\qs(\beta)=1-\frac{t_{-}}{\beta}+o\Big(\frac{1}{\beta}\Big)\text{\,\,and\,\,}\qss:=\qss(\beta)=1-\frac{t_{+}}{\beta}+o\Big(\frac{1}{\beta}\Big),\label{eq:qs}
\end{equation}
where
\[
t_{\pm}=\frac{x_{0}(1)\pm\sqrt{x_{0}^{2}(1)-4\nu''(1)}}{4\nu''(1)}
\]
are the roots of (\ref{eq:2907-01}). Note that by Lemma \ref{lem:x0},
the discriminant above is positive. Also, from (\ref{eq:2907-02})
and the fact that (\ref{eq:2907-01}) is $0$  at the critical
points, the corresponding second derivatives normalized by $\beta^{2}$
converge to
\[
-\frac{x_{0}(1)}{t_{\pm}}+4\nu''(1)=\mp\frac{4\nu''(1)\sqrt{x_{0}^{2}(1)-4\nu''(1)}}{x_{0}(1)\pm\sqrt{x_{0}^{2}(1)-4\nu''(1)}}.
\]
Therefore, for large $\beta$, $\qs(\beta)$ is a local maximum and
$\qss$ is a local minimum. 

The quantity $\Lambda_{Z,\beta}(E,q)$ always gives an upper bound
for the 
\corO{(conditional) free energy defined as}
\begin{equation}
\label{eq-16-11-c}
\corO{
{\bf F}_{E,v}^q:= \lim_{N\to\infty}\frac{1}{N}\mathbb{E}_{NE,v}^{q}\log \left\{ Z_{N,\beta}(q\hat{\mathbf{n}})\right\},
}
\end{equation} 
by appealing to Markov's
inequality \corO{(compare with \eqref{eq:Eband})}. However, to derive a matching lower bound we will need
the $2$-spin model corresponding to the expansion of Section \ref{sec:models_on_bands},
namely, $\bar{H}_{N}^{\bs_{0},2}|_{q}\left(\boldsymbol{\sigma}\right)$
, to have an effective high temperature. \corO{From known facts concerning the spherical $2$-spin, see e.g. \cite[Theorem 1.1, Proposition 2.2]{Talag}, 
the (first) transition in $q$
from high to low temperature occurs at 
\begin{equation}
  q_{c}:=q_{c}(\beta)=\max\left\{ q\in(0,1):\,\alpha_{2}(q)=\frac{1}{\beta\sqrt{2}}\right\},\quad \mbox{\rm where}\;\;\alpha_{2}(q)=(1-q^{2})\left(\sum_{p=2}^{\infty}\gamma_{p}^{2}\binom{p}{2}q^{2p-4}\right)^{1/2}.
  \label{eq:qc}
\end{equation}
We}
note for later use that, since $\alpha_{2}(q)=(1-q)\sqrt{2\nu''(1)}+O((1-q)^{2})$,
as $q\to1$, 
\begin{equation}
q_{c}(\beta)=1-\frac{t_c}{\beta}+O\left(\frac{1}{\beta^{2}}\right),\,\,\,{\rm as\,\,}\beta\to\infty,{\rm \quad where\ \ } t_c=\frac{1}{2\sqrt{\nu''(1)}}.\label{eq:qca}
\end{equation}

We can write
\[
t_{\pm}=\frac{1}{2\sqrt{\nu''(1)}}\left(z\pm\sqrt{z^{2}-1}\right),\quad z=\frac{x_{0}(1)}{2\sqrt{\nu''(1)}}>1.
\]
Since $z-\sqrt{z^{2}-1}$ decreases in $z>1$ and is equal to $1$
for $z=1$, by (\ref{eq:qs}) and (\ref{eq:qca}), for large $\beta$,
\begin{equation}
t_-<t_c<t_+{\rm \quad and \quad}\qss<q_{c}<\qs.\label{eq:1010-05}
\end{equation}

Also note for later use that, from (\ref{eq:LambdaZ}),
we have a constant gap in the limit:
\begin{equation}
 \lim_{\beta\to\infty}\corO{(}\Lambda_{Z,\beta}(\qs)-\Lambda_{Z,\beta}(q_c)\corO{)} =(t_{c}-t_{-})x_{0}(1)+\frac{1}{2}\log(t_{-}/t_{c})+\nu''(1)(t_{-}^{2}-t_{c}^{2})>0,
\label{eq:gap}
\end{equation}
where the strict inequality follows by writing the limit as the integral,
\corE{over $(q_c,\qs)$,}
of (\ref{eq:2907-01}) times $\beta$,  \corO{and noting that}
(\ref{eq:2907-01})
is positive \corO{in this range}.

As \corO{mentioned above}, for $q<q_{c}$, $\Lambda_{Z,\beta}(E,q)
\corO{\neq
{\bf F}_{E,v}^q,}$ \corO{see \eqref{eq-16-11-c}}.
%
For the asymptotics we shall consider, what will be relevant is the free
energy corresponding to the 2-and-below spins on $\mathcal{S}(\bs_{0})$,
 defined similarly to (\ref{eq:Zsigma}) by
\begin{equation}
Z_{N,\beta}^{2-}(\bs_{0}):=(1-\|\bs_{0}\|/\sqrt{N})^{\frac{N}{2}}\int_{\mathcal{S}(\bs_{0})}e^{-\beta\sum_{i=0}^{2}\bar{H}_{N}^{\bs_{0},i}\left(\boldsymbol{\sigma}\right)}d\bs.\label{eq:3007-02}
\end{equation}
\corO{The (logarithmic) error between  $Z_{N,\beta}^{2-}(\bs_{0})$ and  $Z_{N,\beta}(\bs_{0})$ will} be controlled using Lemma \ref{lem:apxest}, see
\eqref{1010-04} below. \corO{Our immediate task
is to provide, in the next lemma, an expression for the free energy of the former.}
\begin{lem}
\label{lem:LambdaF}For any $q\in(0,1)$ such that $\beta\alpha_{2}(q)\geq1/\sqrt{2}$,
(in particular, for $q\in(q_{c}-\delta,q_{c}]$ with small fixed 
$\delta$, independent of $\beta$),
\[
\lim_{N\to\infty}\frac{1}{N}\mathbb{E}_{NE,v}^{q}\left( \log Z_{N,\beta}^{2-}(q\hat{\mathbf{n}})\right) =\Lambda_{F,\beta}^{2-}(E,q),
\]
where
\begin{equation}
\begin{aligned}\Lambda_{F,\beta}^{2-}(E,q) & :=-\beta E+\frac{1}{2}\log(1-q^{2})+\sqrt{2}\beta\alpha_{2}(q)-\frac{1}{2}\log(\beta\alpha_{2}(q))-\frac{3}{4}-\frac{1}{4}\log2\\
 & =-\beta E+\sqrt{2}\beta\alpha_{2}(q)-\frac{1}{4}\log\Big(\beta^{2}\nu''(q^{2})\Big)-\frac{3}{4}.
\end{aligned}
\label{eq:LambdaZ-1}
\end{equation}
\end{lem}

\begin{proof}
The \corO{term} $\frac{1}{2}\log(1-q^{2})$ comes from the volumes ratio as in
Corollary \ref{cor:smallbandZ}. With $\bs_{0}=q\hat{\mathbf{n}}$
the integral \corO{in} 
(\ref{eq:3007-02}) is \corO{determined by}  $\sum_{i=0}^{2}\bar{H}_{N}^{\hat{\mathbf{n}},i}|_{q}\left(\boldsymbol{\sigma}\right)$.
By the argument used to prove Lemma \ref{cor:conditional}, under
$\mathbb{P}_{NE,v}^{q}$, 
\begin{equation}
\sum_{i=0}^{2}\bar{H}_{N}^{\hat{\mathbf{n}},i}|_{q}\left(\boldsymbol{\sigma}\right)\overset{d}{=}NE+\alpha_{2}(q)\sqrt{\frac{N}{N-1}}H_{N-1,2}\left(\bs\right),\label{eq:2909-01}
\end{equation}
where $H_{N-1,2}\left(\bs\right)$ is the pure $2$-spin model on
$\mathbb{S}^{N-2}(\sqrt{N-1})$. Our assumption on $q$ is equivalent
to requiring the effective inverse-temperature $\beta\alpha_{2}(q)$
being in the high temperature phase of the $2$-spin model and the
proof is completed by the expression for its free energy from \cite[Theorem 1.1, Proposition 2.2]{Talag}. 
\end{proof}
We note that for any $0<T=T(\beta)=o(\sqrt{\beta})$, uniformly on $(0,T]$,
as $\beta\to\infty$, 
\begin{equation}
\frac{1}{\beta}\frac{d}{dq}\Lambda_{F,\beta}^{2-}(1-\frac{t}{\beta})\longrightarrow x_{0}(1)-2\sqrt{\nu''(1)}>0,\label{eq:1010-03}
\end{equation}
where we denote $\Lambda_{F,\beta}^{2-}(q):=\Lambda_{F,\beta}^{2-}(-E_{0}(q),q)$.
\corO{Also,}
\[
\Lambda_{F,\beta}^{2-}(q_{c}) =\beta E_{0}(q_{c})+\frac{1}{2}\log(1-q_{c}^{2})+\frac{1}{2}\beta^{2}\alpha_{2}^{2}(q_{c})
,
\]
and therefore
\begin{equation}
\left|\Lambda_{F,\beta}^{2-}(q_{c})-\Lambda_{Z,\beta}(q_{c})\right|=O\Big(\frac{1}{\beta}\Big).\label{1010-04}
\end{equation}

\section{\label{sec:lvlsets}The structure of deep level sets}

In this section we study the structure of the sub-level set 
\begin{equation}
A_{t}:=\{\bs\in\mathbb{S}^{N-1}(\sqrt{N}):\,H_{N}(\bs)\leq-(E_{0}-t)N\}\label{eq:sublvl-1}
\end{equation}
 for small $t$, and relate it to deep critical points. The main result
we prove is the following.
\begin{prop}
\label{prop:sublvl}Assume that $\nu$ satisfies \CDTN. For large enough $\cls=\cls(\nu)$ and small enough $\dls=\dls(\nu)$ we have
the following. For any $t<\dls/\cls$ and any $\eta>0$, with probability
tending to $1$ as $N\to\infty$: 
\begin{enumerate}
\item \label{enu:sublvl1}Each connected component $A$ of the sub-level
set $A_{t}$ contains exactly one $1$-critical point $\bs_{1}$,
which is in particular a local minimum of $H_{N}(\bs)$ on $\mathbb{S}^{N-1}(\sqrt{N})$.
\item \label{enu:sublvl2}For each $1$-critical point $\bs_{1}\in A_{t}$,
there exists a differentiable path $\mathcal{G}:[1-\dls,1]\to\mathbb{S}^{N-1}(1)$,
such that $\sqrt{N}\mathcal{G}(1)=\bs_{1}$ and for any $q\in[1-\dls,1]$,
$\bs_{q}:=\sqrt{N}q\mathcal{G}(q)$ is a $q$-critical point. Moreover,
the speed $\|\frac{d}{dq}\mathcal{G}(q)\|$ of $\mathcal{G}(q)$ under the standard Riemannian metric
on $\mathbb{S}^{N-1}(1)$
is bounded by $\cls>0$.
\item \label{enu:sublvl3}For each $1$-critical point $\bs_1\in A_t$, along the path $\bs_{q}$ defined above, 
\begin{equation}
\label{eq:2110-01}-E_{0}(q)-\eta<\frac{1}{N}H_{N}(\bs_{q})<-E_{0}(1)+mx_{0}(1)(1-q),
\end{equation}
where  $m=m(\dls)$ is a constant determined by $\dls$ that satisfies $\lim_{\dls\to0}m(\corO{\dls})=1$.
\item \label{enu:sublvl4}For each $1$-critical point $\bs_{1}\in A_{t}$,
the spherical cap 
\begin{equation}
\label{def:cap}
{\rm Cap}(t):=\{\bs\in\mathbb{S}^{N-1}(\sqrt{N}):\,R(\bs,\bs_{1-\cls t})\geq1-\cls t\}
\end{equation}
 contains the connected component of $\bs_{1}$ in $A_{t}$.
\end{enumerate}
\end{prop}
\corO{The proof occupies the rest of this section. We first prove part  \ref{enu:sublvl2}, then parts  \ref{enu:sublvl3} and  \ref{enu:sublvl4}, and finally  part \ref{enu:sublvl1}.}

\subsection{Proof of Proposition \ref{prop:sublvl}, Part \ref{enu:sublvl2}}

The construction of the path $\mathcal{G}$ will be based on an application
of the implicit function theorem. Define $G:\,(0,1+\tau)\times U\to\mathbb{R}^{N-1}$,
where $U\subset\mathbb{R}^{N-1}$ is a small neighborhood of the origin,
by
\[
G(q,\mathbf{x})=\Big(F_{i}H_{N}(qT_{\hat{\mathbf{n}}}(\mathbf{x}))\Big)_{i\leq N-1},\quad T_{\hat{\mathbf{n}}}(\mathbf{x})=\left(x_{1},...,x_{N-1},\sqrt{N-\|\mathbf{x}\|_{2}}\right),
\]
where $\mathbf{x}=(x_{1},...,x_{N-1})$ and we recall that $F_{i}$,
$i\leq N-1$, is a piecewise smooth frame field which we defined before
(\ref{eq:derivatives}), and which we will assume to satisfy (\ref{eq:1507-08})
\corO{(in particular, the frame from Lemma \ref{lem:cov}).}
We choose this definition so that if $G(q,\mathbf{x})=0$, then $qT_{\hat{\mathbf{n}}}(\mathbf{x})$
is a $q$-critical point. Denote
\[
J_{\bx}G(q,\mathbf{x})=\left(\frac{d}{dx_{j}}F_{i}H_{N}(qT_{\hat{\mathbf{n}}}(\mathbf{x}))\right)_{i,j\leq N-1},\quad\frac{d}{dq}G(q,\mathbf{x})=\left(\frac{d}{dq}F_{i}H_{N}(qT_{\hat{\mathbf{n}}}(\mathbf{x}))\right)_{i\leq N-1},
\]
and note that at $\bx=0$,
\begin{align*}
J_{\bx}G(q,0) & =\left(qF_{i}F_{j}H_{N}(q\hat{\mathbf{n}})\right)_{i,j\leq N-1}=q\Hess H_{N}(q\hat{\mathbf{n}}),\\
\frac{d}{dq}G(q,0) & =\frac{1}{q}\sum\gamma_{p}(p-1)q^{p-1}\grad H_{N,p}\left(\hat{\mathbf{n}}\right),
\end{align*}
where for the second equality we used the fact that
\begin{equation}
\grad H_{N}\left(q\bs\right)=\sum\gamma_{p}q^{p-1}\grad H_{N,p}\left(\bs\right).\label{eq:2909-02}
\end{equation}

By the implicit function theorem, if $q_{0}\hat{\mathbf{n}}$ is $q_{0}$-critical,
that is, $G(q_{0},0)=0$, and $J_{\bx}G(q_{0},0)$ is invertible,
then on a small neighborhood of $q_{0}$ there exists a unique $g_{q_{0}\hat{\mathbf{n}}}(q)=(g_{q_{0}\hat{\mathbf{n}}}^{(i)}(q))_{i\leq N-1}\in\mathbb{R}^{N-1}$
such that $G(q,g_{q_{0}\hat{\mathbf{n}}}(q))=0$ and 
\[
\frac{d}{dq}g{}_{q_{0}\hat{\mathbf{n}}}(q_{0}):=\big(\frac{d}{dq}g_{q_{0}\hat{\mathbf{n}}}^{(i)}(q)\big)_{i\leq N-1}=-\left[J_{\bx}G(q_{0},0)\right]^{-1}\frac{d}{dq}G(q_{0},0).
\]
In this case, defining the path
\[
\mathcal{G}_{q_{0}\hat{\mathbf{n}}}(q)=T_{\hat{\mathbf{n}}}(g_{q_{0}\hat{\mathbf{n}}}(q))/\|T_{\hat{\mathbf{n}}}(g_{q_{0}\hat{\mathbf{n}}}(q))\|\in\mathbb{S}^{N-1}(1)
\]
we have that the speed of $\mathcal{G}_{q_{0}\hat{\mathbf{n}}}(q)$
at $q_{0}$, relative to the standard Riemannian metric, is given
by $\|\frac{d}{dq}g{}_{q_{0}\hat{\mathbf{n}}}(q_{0})\|/(q_{0}\sqrt{N})$.

Of course, the same argument can be applied to a general point $q\bs\in\mathbb{S}^{N-1}(\sqrt{N}q)$
instead of $q\hat{\mathbf{n}}$, and therefore, if $q\bs$ is a $q$-critical
point such that 
\begin{align}
\frac{1}{q}\left\Vert \sum\gamma_{p}(p-1)q^{p-1}\grad H_{N,p}\left(\bs\right)\right\Vert _{2} & <\sqrt{N}c',\label{eq:1007-01}\\
\frac{1}{q}\left\Vert [\Hess H_{N}(q\bs)]^{-1}\right\Vert _{op} & <c,\label{eq:1007-02}
\end{align}
\corO{where by (\ref{eq:1007-02}) we mean in particular
that the inverse exists,}
then there exists a path $\mathcal{G}_{\bs}(q')$ of $q'$-critical
points, defined on a neighborhood of $q$, whose speed at $q$ is
bounded by $cc'$. To complete the proof, we need to prove
that for some $c$, $c'$ and $\dls$, with probability tending to
$1$, for every $1$-critical point $\bs_{1}$ with $H_{N}(\bs_{1})\leq-(E_{0}-t)N$:
if $q\in[1-\dls,1]$ and $\bs\in\mathbb{S}^{N-1}(\sqrt{N})$ has geodesic
distance to $\bs_{1}$ smaller than $\sqrt{N}\dls cc'$, then (\ref{eq:1007-01})
and (\ref{eq:1007-02}) hold.\footnote{We note that when we prove the bounds (\ref{eq:1007-01}) and (\ref{eq:1007-02})
we may allow the frame field $F_{i}$ to depend on $\bs$.}

Proving (\ref{eq:1007-01}) is easier. If we define $\bar{\gamma}_{p}=\gamma_{p}(p-1)$
and let $\bar{H}_{N}(\bs)$ be the corresponding mixed model, then
by (\ref{eq:2909-02}) the left-hand side of (\ref{eq:1007-01}) is
exactly the norm of $\frac{1}{q}\grad\bar{H}_{N}\left(q\bs\right)$.
Therefore for large enough $c'$, \corO{the concentration inequality  \eqref{eq:1507-03}} of Corollary \ref{cor:gradbd} implies
(\ref{eq:1007-01}) with high probability uniformly over all $\bs\in\mathbb{S}^{N-1}(\sqrt{N})$
and $q\in[1-\dls,1]$. 

To prove (\ref{eq:1007-02}), we first relate the spherical \corO{Hessian} $\Hess H_{N}(\bx)$
to \corE{the} Euclidean Hessian matrix $\nabla_{E}^{2}H_{N}(\bx)=\{\frac{d}{dx_{i}}\frac{d}{dx_{j}}H_{N}(\bx)\}_{i,j\leq N}$,
where $\bx\in\mathbb{R}^{N}$, $\|\bx\|\leq\sqrt{N}$. Assuming $\|\bx\|=\sqrt{N}q$,
let $T_{\bx}=T_{\bx}\mathbb{S}^{N-1}(\sqrt{N}q)$ be the tangent space
to the sphere at $\bx$, viewed as a linear subspace of $\mathbb{R}^{N}$
using the usual identification. Let $\mathbf{A}(\bx)\subset\mathbb{R}^{N-1\times N}$
be some matrix whose rows form an orthonormal basis of $T_{\bx}$.
For an appropriate choice of the frame $F_{i}$ on a neighborhood
of $\bx$, see (\ref{eq:1507-09}),
\[
\Hess H_{N}(\bx)=\mathbf{A}(\bx)\nabla_{E}^{2}H_{N}(\bx)(\mathbf{A}(\bx))^{T}-\frac{1}{\|\bx\|}\ddq H_{N}(\bx)\mathbf{I}.
\]
Therefore,
$\left\Vert [\Hess H_{N}(\bx)]^{-1}\right\Vert _{op}=f_{N}(\bx)$ 
where
\begin{equation}
f_{N}(\bx):=\left(\min_{\mathbf{v}\in T_{\bx}:\,\|\mathbf{v}\|=1}\left|\left\langle \nabla_{E}^{2}H_{N}(\bx),\,\mathbf{v}\mathbf{v}^{T}\right\rangle -\frac{1}{\|\bx\|}\ddq H_{N}(\bx)\right|\right)^{-1},\label{eq:fN}
\end{equation}
\corO{and $f_N(\bx)=\infty$ whenever  $\Hess H_{N}(\bx)$ is not invertible.}

In light of the \corO{the concentration inequality \eqref{eq:1507-03}}  of Corollary \ref{cor:gradbd}, Part
\ref{enu:sublvl2} of Proposition \ref{prop:sublvl} will follow if
we prove the following two lemmas; \corE{in} \corO{ both lemmas,  $f_N$ is as in \eqref{eq:fN}.} 
\begin{lem}
\label{lem:1crtHessbd}For small enough $t>0$  and   large enough $c>0$, 
\begin{equation}
\lim_{N\to\infty}\mathbb{P}\left\{ \exists\bs_{1}\in\mathscr{C}_{N,1}((-\infty,-(E_0-t)N)):\,f_{N}(\bs_{1})\geq c/2\right\} =0.\label{eq:1507-04}
\end{equation}
\end{lem}

\begin{lem}
\label{lem:qcrtHessbd}For any $C$, $c>0$, for small enough $\delta>0$,
if $\bx$, $\bx'\in\mathbb{R}^{N}$ are points such that:
\begin{enumerate}
\item  $\|\bx\|/\sqrt{N}\in(1/2,1]$,
\item $\|\nabla_{E}H_{N}(\bx)\|/\sqrt N,\,\|\nabla_{E}^{2}H_{N}(\bx)\|_{op}\leq C$,
\item$\|\nabla_{E}H_{N}(\bx)-\nabla_{E}H_{N}(\bx')\|/\sqrt{N}$,\; $\|\bx-\bx'\|/\sqrt{N}$
and $\|\nabla_{E}^{2}H_{N}(\bx)-\nabla_{E}^{2}H_{N}(\bx')\|_{op}$
are all smaller than $\delta$,
\end{enumerate}
then, 
\[
f_{N}(\bx)<c/2\,\Longrightarrow\,f_{N}(\bx')<c.
\]
\end{lem}

\begin{proof}[Proof of Lemma \ref{lem:1crtHessbd}]
By Lemma \ref{lem:epsx0}, to prove (\ref{eq:1507-04}) for small  $t$ it will be sufficient to show that for small  $\epsilon$,
\[
\lim_{N\to\infty}\mathbb{P}\left\{ \exists\bs_{1}\in\mathscr{C}_{N,1}(NB(\epsilon),\sqrt ND(\epsilon)):\,f_{N}(\bs_{1})\geq c/2\right\} =0,
\]
where $B(\epsilon)=-E_0 +(-\epsilon,\epsilon)$ and $D(\epsilon)=-x_0+(-\epsilon,\epsilon)$.

\corO{We show below that}
\begin{equation}
\corO{\bar c:=}\limsup_{N\to\infty}\sup_{u\in B(\epsilon),\,x\in D(\epsilon)}\frac{1}{N}\log\left(\mathbb{P}_{Nu,\sqrt Nx}^{1}\left\{ f_{N}(\left(\hat{\mathbf{n}}\right))\geq c/2\right\} \right)<0.\label{eq:1607-01}
\end{equation}
Since $\Theta_{\nu,1}\left(u,x\right)$ is continuous and $\Theta_{\nu,1}\left(-E_{0}(1),-x_{0}(1)\right)=0$, 
\corO{an application of}
Lemma \ref{lem:15} \corO{with $\varphi(u,x)=-\bar c/2$ and $g_N=f_N$ completes the proof, reducing $\epsilon$ if needed.}

\corO{We thus turn to the proof of \eqref{eq:1607-01}.}
From (\ref{eq:1507-10}) and the sentence following it, the conditional
probability in (\ref{eq:1607-01}) is equal to 
\begin{equation}
\mathbb{P}\left\{ \exists i\leq N-1: | \lambda_{i}(\hat{\mathbf{n}}) -x | \leq 2/c\right\} ,\label{eq:1607-02}
\end{equation}
where $\lambda_{i}(\hat{\mathbf{n}})$ are the eigenvalues of $\mathbf{G}\left(\hat{\mathbf{n}}\right)$,
which have the same law as $\sqrt{\frac{N-1}{N}\nu''(1)}\mathbf{M}$
with $\mathbf{M}$ being a GOE matrix. By Lemma \ref{lem:x0}, $-x_{0}/\sqrt{\nu''(1)}<-2$, and for small $\epsilon$ and large $c$, \eqref{eq:1607-02} is exponentially small in $N$, by the bound on the top eigenvalue of \cite[Lemma 6.3]{BDG}. This implies (\ref{eq:1607-01})  and completes the proof of the lemma.
\end{proof}

\begin{proof}[Proof of Lemma \ref{lem:qcrtHessbd}]
\corO{Throughout the proof, we write $\epsilon_i(\delta)$ for various positive quantities satisfying
 $\lim_{\delta\to 0} \epsilon_i(\delta)=0$.}\

Since $\ddq H_{N}(\bx)=\langle\nabla_{E}H_{N}(\bx),\,\bx/\|\bx\|\rangle$,
from our assumptions,
\[
\left\Vert \frac{1}{\|\bx\|}\ddq H_{N}(\bx)-\frac{1}{\|\bx'\|}\ddq H_{N}(\bx')\right\Vert \leq\epsilon_{1}(\delta),
\]
for some \corO{$\epsilon_{1}(\delta)$.}

Since, \corO{by assumption}, $\|\bx\|/\sqrt{N}>1/2$ and $\|\bx-\bx'\|/\sqrt{N}<\delta$,
there exists $\epsilon_{2}(\delta)$ such that for any $\mathbf{v}'\in T_{\bx'}$
with $\|\mathbf{v}'\|=1$ there exists $\mathbf{v}\in T_{\bx}$ with
$\|\mathbf{v}\|=1$ such that $\|\mathbf{v}-\mathbf{v}'\|\leq\epsilon_{2}(\delta)$,
and vice versa. Write
\begin{equation}
 \left|\left\langle \nabla_{E}^{2}H_{N}(\bx),\,\mathbf{v}\mathbf{v}^{T}\right\rangle -\left\langle \nabla_{E}^{2}H_{N}(\bx'),\,\mathbf{v}'\mathbf{v}'^{T}\right\rangle \right|
  \leq\left|\left\langle \nabla_{E}^{2}H_{N}(\bx)-\nabla_{E}^{2}H_{N}(\bx'),\,\mathbf{v}'\mathbf{v}'^{T}\right\rangle \right|+\left|\left\langle \nabla_{E}^{2}H_{N}(\bx),\,\mathbf{v}\mathbf{v}^{T}-\mathbf{v}'\mathbf{v}'^{T}\right\rangle \right|,\label{eq:1707-03}
\end{equation}
where $\langle\mathbf{A},\mathbf{B}\rangle=\text{Tr}(\mathbf{A}^{T}\mathbf{B})=\sum_{i}\langle\mathbf{A}e_{i},\mathbf{B}e_{i}\rangle$,
with $e_{i}$ being an orthonormal basis, denotes the Hilbert-Schmidt
inner product. The first summand in (\ref{eq:1707-03}) is bounded by $\delta$,
by \corO{assumption (3)}. The second summand is bounded by
\begin{equation}
\|\nabla_{E}^{2}H_{N}(\bx)\|_{op}\|\mathbf{v}\mathbf{v}^{T}-\mathbf{v}'\mathbf{v}'^{T}\|_{op}\text{\ensuremath{\cdot}Rank}\left(\mathbf{v}\mathbf{v}^{T}-\mathbf{v}'\mathbf{v}'^{T}\right)\leq2C\|\mathbf{v}\mathbf{v}^{T}-\mathbf{v}'\mathbf{v}'^{T}\|_{op}.\label{eq:1707-04}
\end{equation}
Since $\|\mathbf{v}-\mathbf{v}'\|\leq\epsilon_{2}(\delta)$, (\ref{eq:1707-04})
is bounded from above by some \corO{$\epsilon_{3}(\delta)$.}

Combining the above we have that |$(f_{N}(\bx'))^{-1}-(f_{N}(\bx))^{-1}|\leq\epsilon_{4}(\delta)$,
for some \corO{$\epsilon_{4}(\delta)$.} In particular, for small $\delta$,
$f_{N}(\bx')<c$.
\end{proof}

\corO{Having completed the proof of Lemmas \ref{lem:1crtHessbd} and \ref{lem:qcrtHessbd}, the proof of  Proposition \ref{prop:sublvl}, Part \ref{enu:sublvl2}
is complete.}\qed
\subsection{Proof of Proposition \ref{prop:sublvl}, Parts \ref{enu:sublvl3}
and \ref{enu:sublvl4}}

For small $\epsilon$, if $\bs$, $\bs'\in\mathbb{S}^{N-1}(\sqrt{N})$
are points with geodesic distance less than $\sqrt{N}\epsilon$, then
\[
R(\bs,\bs')\geq\sqrt{1-\epsilon^{2}}=1-\epsilon^{2}/2+O(\epsilon^{4}).
\]
\corO{Let $\cls$ be the constant from the (already proved)} Part \ref{enu:sublvl2} of Proposition \ref{prop:sublvl}.
Then, assuming $t$ is small, \corO{with high probability},
any $1$-critical point $\bs_{1}$ \corO{satisfies}
\[\corO{d\Big(\bs_1,\frac{\sqrt{N}\bs_{1-\cls t}}{\|\bs_{1-\cls t}\|}\Big)\leq
\sqrt{N}\cls^{2}t}\]
 and therefore 

\[
R(\bs_{1},\bs_{1-\cls t})\geq1-(\cls^{2}t)^{2}\geq1-\cls t,
\]
that is, $\bs_{1}\in\text{Cap}(t)$.

\corO{Let $\partial\text{Cap}(t)=\{\bs\in\mathbb{S}^{N-1}(\sqrt{N}):\,R(\bs,\bs_{1-\cls t})=1-\cls t\}$
denote the} boundary of $\text{Cap}(t)$.
Let \corO{$\bs_{1}\in A_{t}\cap \text{Cap}(t)$ be some $1$-critical point and assume that $\partial\text{Cap}(t)\cap A_{t}=\varnothing$, that is}
\begin{equation}
\min_{\bs\in\partial\text{Cap}(t)}H_{N}(\bs)>-(E_{0}-t)N.\label{eq:1707-05}
\end{equation}
Let $\bs_{0}$ be some point in the connected component $K$ of $A_{t}$
containing $\bs_{1}$ and let $\mathcal{Q}:[0,1]\to\mathbb{S}^{N-1}(\sqrt{N})$,
be a continuous path, 
\corO{wholly contained in the connected component $K$,}
that connects $\bs_{0}=\mathcal{Q}(0)$ to $\bs_{1}=\mathcal{Q}(1)$.
Note that $\bs_0$  must be contained in $\text{Cap}(t)$, since otherwise we would have had some $s'$ such that $\mathcal{Q}(s')\in \partial \text{Cap}(t)\cap K$. In other words, if $\bs_{1}\in\text{Cap}(t)$ and (\ref{eq:1707-05}), then
 $K\subset\text{Cap}(t)$.
Therefore, to prove Part \ref{enu:sublvl4} of Proposition \ref{prop:sublvl} we need to show that with high probability,
for any $1$-critical
point $\bs_{1}\in A_{t}$, the corresponding $(1-\cls t)$-critical
point $\bs_{1-\cls t}$ satisfies (\ref{eq:1707-05}).

In the rest of the proof, $\epsilon>0$ will be a constant which can
be taken to be as small as we wish, provided $t$ and $\dls$ are
small enough. Also, whenever we fix a $1$-critical point $\bs_{1}\in A_{t}$,
$\bs_{q}$ will denote the corresponding $q$-critical point defined
by Part \ref{enu:sublvl2} of Proposition \ref{prop:sublvl} (we always
restrict to the event that those points exist for all $\bs_{1}\in A_{t}$).
We shall prove that, with probability tending to $1$, for any $1$-critical
point $\bs_{1}\in A_{t}$ and $q\in[1-\dls,1)$:
\begin{itemize}
  \item[(a)] \label{enu:pt1}With $T_{\bx}$ as defined in the proof of Part \ref{enu:sublvl2},
\begin{equation}
\min_{\mathbf{v}\in T_{\bs_{q}}:\,\|\mathbf{v}\|=1}\left|\left\langle \nabla_{E}^{2}H_{N}(\bs_{q}),\,\mathbf{v}\mathbf{v}^{T}\right\rangle \right|<2\sqrt{\nu''(1)}+\epsilon.\label{eq:2017-02}
\end{equation}
\item[(b)] \label{enu:pt2}For some sequence $\eta_{N}=o(1)$, $H_{N}(\bs_{q})>-NE_{0}(q)-N\eta_{N}$.
\item[(c)] \label{enu:pt3}For some constant $c_{2}>0$,
\[
\Big|H_{N}(\bs_{q})-\left(H_{N}(\bs_{1})+N(1-q)x_{0}\right)\Big|<N\epsilon(1-q)+Nc_{2}(1-q)^{2}.
\]
\end{itemize}
By an argument similar to the proof of Lemma \ref{lem:1crtHessbd}
\textendash{} replacing the condition $f_{N}(\bs_{1})\geq c/2$ by
(\ref{eq:2017-02}) with $q=1$ and using the bound 
\begin{equation}
\limsup_{N\to\infty}\sup_{u\in B'(t),\,x\in D'(\epsilon)}\frac{1}{N}\log\mathbb{P}_{u,x}^{1}\left\{ \min_{\mathbf{v}\in T_{\hat{\mathbf{n}}}:\,\|\mathbf{v}\|=1}\left|\left\langle \nabla_{E}^{2}H_{N}(\hat{\mathbf{n}}),\,\mathbf{v}\mathbf{v}^{T}\right\rangle \right|<2\sqrt{\nu''(1)}+\epsilon/2\right\} <0,\label{eq:1607-01-1}
\end{equation}
which follows from (\ref{eq:1507-10}), we conclude that
with probability tending
to $1$ as $N\to\infty$, 
\corO{for each of the points $\bs_{1}\in A_{t}$,}
(\ref{eq:2017-02}) holds with $\epsilon/2$
instead of $\epsilon$. Combined with Corollary \ref{cor:gradbd}
and the fact that, by Part \ref{enu:sublvl2} of Proposition \ref{prop:sublvl},
\begin{equation}
\|\bs_{1}-\bs_{q}\|_{2}^{2}\leq\Big(\|\bs_{1}-q\bs_{1}\|_{2}+\|q\bs_{1}-\bs_{q}\|_{2}\Big)^{2}\leq N(1-q)^{2}(1+\cls)^{2},\label{eq:2007-04}
\end{equation}
this implies Point  (a) 
above. 

Point (b)
follows from Theorem \ref{thm:1stmoment} and Remark \ref{rem:scaling}, by Markov's inequality.

By Lemma \ref{lem:epsx0}, for all $\bs_{1}\in A_{t}$ with small $t$,
with probability tending to $1$ as $N\to\infty$, 
\begin{equation}
|\ddq H_{N}\left(\bs_{1}\right)+\sqrt{N}x_{0}(1)|\leq\sqrt{N}\epsilon.\label{eq:2017-03}
\end{equation}

Let $\mathcal{E}=\mathcal{E}_{N}$ be the intersection of the event
that Part \ref{enu:sublvl2} of Proposition \ref{prop:sublvl} holds
(with some $\cls$) and 
\[
\left\{ \sup_{\bx\in\mathcal{B}_{N}}\|\nabla_{E}^{2}H_{N}(\bx)\|_{op}<2\tilde{C}_{2}\right\} ,
\]
where $\tilde{C}_{2}$ is defined in Corollary \ref{cor:gradbd} and
$\|\cdot\|_{op}$ is the operator norm. Note that by Corollary \ref{cor:gradbd},
$\lim_{N\to\infty}\mathbb{P}\{\mathcal{E}_{N}\}=1$. 

On the event $\mathcal{E}_{N}$, for any 1-critical point $\bs_{1}\in A_{t}$
and any $q\in[1-\dls,1]$, by a Taylor approximation, 
\begin{align*}
 & \Big|H_{N}(\bs_{q})-\Big[H_{N}(\bs_{1})-(1-q)\sqrt{N}\ddq H_{N}(\bs_{1})\Big]\Big|\\
 & \leq\tilde{C}_{2}\|\bs_{1}-\bs_{q}\|_{2}^{2}\leq N(1-q)^{2}\tilde{C}_{2}(1+\cls)^{2},
\end{align*}
where we used (\ref{eq:2007-04}) and the fact that $\langle\nabla_{E}H_{N}(\bs_{1}),\,\mathbf{v}\rangle=0$,
for any direction $\mathbf{v}$ tangent to $\mathbb{S}^{N-1}(\sqrt{N})$
at $\bs_{1}$. Combined with (\ref{eq:2017-03}), this proves Point
(c) 
above.

Points 
\corO{(a) and (b)}
above imply that \eqref{eq:2110-01} holds for any finite number of values of
$q$ simultaneously, with probability tending to $1$ as $N\to\infty$. 
To complete the proof of Part \ref{enu:sublvl3} of 
Proposition \ref{prop:sublvl}, the same needs to \corO{be proved for a}
whole range $q\in[1-\dls,1]$ simultaneously. The latter follows from
the bound on the speed of $\bs_q$ from  Part \ref{enu:sublvl2} of Proposition \ref{prop:sublvl}, the fact that by (\ref{eq:1507-03}), $H_{N}: \BN\to \R$ is a Lipschitz function, and the fact that $-E_0(q)$ is continuous in a neighborhood of $1$ by Lemma \ref{lem:Exsmooth}. 

Finally, we turn to the proof of (\ref{eq:1707-05}). Recall the definition
of $\bar{H}_{N}^{\bs_{0},k}\left(\boldsymbol{\sigma}\right)$ from
Section \ref{sec:models_on_bands} and that $H_{N}\left(\boldsymbol{\sigma}\right)=\sum_{k=0}^{\infty}\bar{H}_{N}^{\bs_{0},k}\left(\boldsymbol{\sigma}\right)$.
Set $\bar{\bs}_{q}=\bs_{q}/q$ and let $\bs\in\mathbb{S}^{N-1}(\sqrt{N})$
be some point such that $R(\bs,\bar{\bs}_{q})=q$. Assuming Points \corE{(b) and (c)}
 above,
\begin{align*}
\bar{H}_{N}^{\bar{\bs}_{q},0}\left(\boldsymbol{\sigma}\right)/N & =H_{N}(\bs_{q})/N>-E_{0}+(1-q)x_{0}\\
 & -\epsilon(1-q)-c_{2}(1-q)^{2}-\eta_{N}.
\end{align*}
Since $\bs_{q}$ is $q$-critical, $\bar{H}_{N}^{\bar{\bs}_{q},1}\left(\boldsymbol{\sigma}\right)=0$.
\corO{Using Point (a) above,}
by Corollary \ref{cor:tensorbd} we have
\[
\big|\bar{H}_{N}^{\bar{\bs}_{q},2}\left(\boldsymbol{\sigma}\right)\big|\leq N(1-q^{2})(\sqrt{\nu''(1)}+\epsilon/2).
\]
From Lemma \ref{lem:apxest}, with probability tending to
$1$, for some constants $c$, $\bar{C}>0$, $q$ close to $1$,
\corO{and all points $\bar{\bs}_q$,}
\[
\Big|\sum_{k=3}^{\infty}\bar{H}_{N}^{\bar{\bs}_{q},k}\left(\boldsymbol{\sigma}\right)\Big|\leq N\bar{C}\Big(\frac{1-q^{2}}{c}\Big)^{3/2}.
\]
Combining the above, for sufficiently small $\delta$, for any $q\in[1-\dls,1)$
and large enough $N$,
\[
H_{N}\left(\boldsymbol{\sigma}\right)=\sum_{k=0}^{\infty}\bar{H}_{N}^{\bar{\bs}_{q},k}\left(\boldsymbol{\sigma}\right)>-NE_{0}+N(1-q)(x_{0}-2\sqrt{\nu''(1)}-\epsilon'),
\]
where $\epsilon'>0$ is a constant which can be taken to be as small
as we wish assuming $\dls$ and $t$ are small enough. Finally, since
by Lemma \ref{lem:x0}, $x_{0}>2\sqrt{\nu''(1)}$, for some $\tau>0$,
\[
H_{N}\left(\boldsymbol{\sigma}\right)>-E_{0}N+(1-q)\tau N.
\]
For $q=1-\cls t$, assuming that $\cls>1/\tau$, we have that $H_{N}\left(\boldsymbol{\sigma}\right)>-E_{0}N+tN$.
This yields (\ref{eq:1707-05}) and completes the proof.\qed
\subsection{Proof of Proposition \ref{prop:sublvl}, Part \ref{enu:sublvl1}}

The fact that a.s. each component $A$ contains a local minimum point
$\bs_{1}$ of $H_{N}(\bs)$ on $\mathbb{S}^{N-1}(\sqrt{N})$ is a
direct consequence of the fact that $H_{N}:\,\mathbb{S}^{N-1}(\sqrt{N})\to\mathbb{R}$
is a Morse function a.s., as can be verified using \cite[Theorem 11.3.1]{RFG}.
To prove that $\bs_{1}$ is the only $1$-critical point in $A$,
assuming Parts \ref{enu:sublvl2} and \ref{enu:sublvl4} and that
$\dls/\cls$ is small enough, it is sufficient to show that for small fixed
$c>0$, if $t$ is small enough then there are no two critical points
$\bs$, $\bs'$ with $H_{N}(\bs)\leq-(E_{0}-t)N$ such that $R\left(\bs,\bs'\right)>1-c$,
with probability tending to $1$ as $N\to\infty$. This is a consequence
of Corollary \ref{cor:orth} (which, in fact, states that the critical
points are either antipodal or close to orthogonal for small $t$).
\qed
\section{\label{sec:lower-bound}\corE{The }\corO{ lower bound on the free energy}}

This section is devoted to proving the following lower bound on the
free energy. \corO{Recall the variables $Z_{N,\beta}$, $\Lambda_{Z,\beta}$, $Z_{N,\beta}(\bs_0)$, see
\eqref{eq-gibbsmeas}, \eqref{eq:LambdaZ} and \eqref{eq:Zsigma}, and the variables \corE{$\qs$ and $E_0(\qs)$}, see 
\eqref{eq:qs} and \eqref{eq:E0q}.}
\begin{prop}
\label{prop:Flb}Assuming \CDTN, for large enough $\beta$ and any
$\epsilon>0$, there exist constants $c,\,C>0$ (depending on $\epsilon$)
such that 
\begin{equation}
\mathbb{P}\Big\{\frac{1}{N}\log Z_{N,\beta}<\Lambda_{Z,\beta}(-E_{0}(\qs),\,\qs)-\epsilon\Big\}\leq Ce^{-Nc}.\label{eq:0808-03}
\end{equation}
\end{prop}
\corO{The rest of the section is devoted to the proof of Proposition
  \ref{prop:Flb}.}
Consider the following three statements:
\begin{enumerate}
\item \label{enu:Flb1}For some $\epsilon_{N}=o(1)$,
\begin{equation}
\lim_{N\to\infty}\frac{1}{N}\log\mathbb{P}\Big\{\frac{1}{N}\log Z_{N,\beta}>\Lambda_{Z,\beta}(-E_{0}(\qs),\,\qs)-\epsilon_{N}\Big\}=0.\label{eq:1008-02}
\end{equation}
\item \label{enu:Flb2}Let $\delta_{N}=o(1)$ be arbitrary and set $B_{N}=-E_{0}(\qs)+(-\delta_{N},\delta_{N})$
and $D_{N}=-x_{0}(\qs)+(-\delta_{N},\delta_{N})$. Then, for some
$\epsilon_{N}=o(1)$, as $N\to\infty$,
\begin{equation}
\label{eq-mas1}
\mathbb{E}\mbox{LCrt}_{N,\beta}=o(1)\cdot\mathbb{E}\mbox{Crt}_{N,\qs}(B_{N},D_{N}),
\end{equation}
where, with $Z_{N,\beta}(\bs_{0})$ defined by \eqref{eq:Zsigma},
\[
\mbox{LCrt}_{N,\beta}=\#\left\{ \bs_{0}\in\mathscr{C}_{N,\qs}(NB_{N},\sqrt{N}D_{N}):\,\frac{1}{N}\log Z_{N,\beta}(\bs_{0})<\Lambda_{Z,\beta}(-E_{0}(\qs),\qs)-\epsilon_{N}\right\} 
\]
denotes the number of `light' points.
\item \label{enu:Flb3}\corO{Fix a sequence $\delta_N=o(1)$ in the definition of $B_N,D_N$. Then, if} $\epsilon_{N}=\epsilon_{N}(\delta_{N})=o(1)$
decays slowly enough, then as $N\to\infty$, uniformly in $(u,x)\in NB_{N}\times\sqrt{N}D_{N}$,
\begin{equation}
\mathbb{P}_{u,x}^{\qs}\left\{ \frac{1}{N}\log Z_{N,\beta}(\qs\hat{\mathbf{n}})<\Lambda_{Z,\beta}(-E_{0}(\qs),\qs)-\epsilon_{N}\right\} =o(1).\label{eq:0808-01-1}
\end{equation}
\end{enumerate}
We will prove (\ref{enu:Flb3}) and the implications (\ref{enu:Flb3})$\implies$(\ref{enu:Flb2})$\implies$(\ref{enu:Flb1})$\implies$Proposition
\ref{prop:Flb}. 

\subsection*{(\ref{enu:Flb1})$\protect\implies$Proposition \ref{prop:Flb}}

By Corollary \ref{cor:concF},
\begin{equation}
\limsup_{N\to\infty}\frac{1}{N}\log\mathbb{P}\left\{ \Big|\frac{1}{N}\log Z_{N,\beta}-\frac{1}{N}\mathbb{E}\log Z_{N,\beta}\Big|>Nt\right\} \leq-t^{2}/2\beta^{2}\nu(1).\label{eq:2809-03}
\end{equation}
Therefore, (\ref{eq:1008-02}) implies that 
\begin{equation}
\liminf_{N\to\infty}\mathbb{E}\frac{1}{N}\log Z_{N,\beta}\geq\Lambda_{Z,\beta}(-E_{0}(\qs),\,\qs).\label{eq:2809-04}
\end{equation}
\corO{Another application of} Corollary \ref{cor:concF},
together with (\ref{eq:2809-04}), yield
(\ref{eq:0808-03}).\qed

\subsection*{(\ref{enu:Flb2})$\protect\implies$(\ref{enu:Flb1})}
\corOU{
Let
\begin{equation}
\mbox{HCrt}_{N,\beta}=\#\left\{ \bs_{0}\in\mathscr{C}_{N,\qs}(NB_{N},\sqrt{N}D_{N}):\,\frac{1}{N}\log Z_{N,\beta}(\bs_{0})\geq
\Lambda_{Z,\beta}(-E_{0}(\qs),\qs)-\epsilon_{N}\right\}
\label{eq:2811-01} 
\end{equation}
denote the number of `heavy' critical points.
We show below that one can choose $\delta_N,\epsilon_N\to_{N\to\infty} 0$ so that
\begin{equation}
\label{eq-mashachnasa}
\lim_{N\to\infty} \frac{1}{N}\log\mathbb{P}\{\mbox{HCrt}_{N,\beta}\geq1\}= 0.
\end{equation}
By \eqref{eq:1507-03} of Corollary \ref{cor:gradbd}, there exist constants $c,C>0$ and an event ${\mathcal A}_N$ with 
$\mathbb{P}({\mathcal A}_N^c)\leq e^{-cN}$ such that  \corE{$\sup_{\bs\in\SN}\|\nabla H_{N}(\bs)\|_\infty \leq C$} on ${\mathcal A}_N$, and therefore,
on that event, if $\mbox{HCrt}_{N,\beta}\geq1$ \corE{then for some $t_N\to_{N\to\infty} 0$ with $\bs_0$ being  an arbitrary point as in \eqref{eq:2811-01},
\[\frac1N \log Z_{N,\beta} \geq \frac{1}{N}\log Z_{N,\beta}(\bs_{0})-Ct_N\geq
\Lambda_{Z,\beta}(-E_{0}(\qs),\qs)-\epsilon_{N}-Ct_N.\]
Using \eqref{eq-mashachnasa}, this gives \eqref{enu:Flb1} with $\epsilon_N$ in the latter replaced by $\epsilon_N+Ct_N$.}}

\corOU{
It remains to prove \eqref{eq-mashachnasa}. Choose $\delta_N$ so that $\mathbb{E}\mbox{Crt}_{N,\qs}(B_{N},D_{N})\to_{N\to\infty}\infty$, which is possible by 
Theorem \ref{thm:1stmoment}. By \eqref{eq-mas1}, this implies that $\mathbb{E}\mbox{HCrt}_{N,\beta}\to_{N\to\infty} \infty$.
By the Cauchy-Schwartz  inequality,
\[
\liminf_{N\to\infty}\frac{1}{N}\log\mathbb{P}\{\mbox{HCrt}_{N,\beta}\geq1\}\geq\liminf_{N\to\infty}\frac{1}{N}\log\frac{\left(\mathbb{E}\mbox{HCrt}_{N,\beta}\right)^{2}}{\mathbb{E}\left[(\mbox{HCrt}_{N,\beta})^{2}\right]}.
\]
The conclusion \eqref{eq-mashachnasa} now follows from  Corollary \ref{cor:matching}, since obviously
$\mathbb{E}\left[(\mbox{HCrt}_{N,\beta})^{2}\right]\leq \mathbb{E}[(\mbox{Crt}_{N,\qs}(B_{N},D_{N}))^2]$. \qed}

\subsection*{(\ref{enu:Flb3})$\protect\implies$(\ref{enu:Flb2})}

In the proof of Theorem \ref{thm:1stmoment}, we expressed $\mathbb{E}\mbox{Crt}_{N,\qs}(B_{N},D_{N})$
by (\ref{eq:1102-01}) (with $q=\qs$, $B=B_{N}$, $D=D_{N}$). The
same argument gives a similar expression for $\mathbb{E}\mbox{LCrt}_{N,\beta}$,
but with the expectation in (\ref{eq:1102-01}) replaced by
\begin{equation}
\mathbb{E}_{\corE{\sqrt Nu},x}^{\qs}\Bigg[|\det(\mathbf{G}-t(x)\mathbf{I})|\mathbf{1}\left\{ \frac{1}{N}\log Z_{N,\beta}(\qs\hat{\mathbf{n}})<\Lambda_{Z,\beta}(-E_{0}(\qs),\qs)-\epsilon_{N}\right\} \Bigg],\label{eq:0808-01}
\end{equation}
where $t(x)=x/\qs\sqrt{(N-1)\nu''(\qs^{2})}$ and $\mathbf{G}$ is
a GOE matrix given by
\[
\mathbf{G}=\sqrt{\frac{N}{(N-1)\nu''(\qs^{2})}}\Big(\Hess H_{N}\left(\qs\hat{\mathbf{n}}\right)+\frac{1}{\sqrt{N}\qs}\ddq H_{N}\left(\qs\hat{\mathbf{n}}\right)\mathbf{I}\Big).
\]

Thus, (\ref{enu:Flb2}) follows if we show that the expectation of
(\ref{eq:0808-01}) is $o(1)$ times the same expectation without
the indicator, uniformly over the domain of integration in (\ref{eq:1102-01}),
i.e., over $(u,x)\in\sqrt{N}B_{N}\times\sqrt{N}D_{N}$. By Lemma \ref{lem:x0},
$t(x)>2+\delta$ for some $\delta>0$, uniformly in $x\in\sqrt{N}D$
(for large $\beta$, so that $x_{0}(\qs)$ is close to $x_{0}(1)$).
Therefore, by \cite[Corollaries 22, 23]{2nd}, as $N\to\infty$, 
\[
\mathbb{E}|\det(\mathbf{G}-t(x)\mathbf{I})|^{2}\leq C\left(\mathbb{E}|\det(\mathbf{G}-t(x)\mathbf{I})|\right)^{2},
\]
for some $C>0$, uniformly in $x\in\sqrt{N}D$. Hence, by the
Cauchy-Schwarz inequality, (\ref{enu:Flb3})$\implies$(\ref{enu:Flb2}).
\qed

\subsection*{Proof of (\ref{enu:Flb3})}
\corO{We need the following two lemmas. Recall the notation  $Z_{N,\beta}^{2-}(\bs)$, see  (\ref{eq:3007-02}).}
\begin{lem}
\label{lem:2sp_highT}For large enough $\beta$ and $\epsilon_{N}=o(1)$
decaying \corO{slowly} enough, uniformly in $(u,x)\in NB_{N}\times\sqrt{N}D_{N}$,

\begin{equation}
\mathbb{P}_{u,x}^{\qs}\left\{ \frac{1}{N}\log Z_{N,\beta}^{2-}(\qs\hat{\mathbf{n}})<\beta E_{0}(\qs)+\frac{1}{2}\log(1-\qs^{2})+\frac{1}{2}\beta^{2}\alpha_{2}^{2}(\qs)-\frac{\epsilon_{N}}{2}\right\} =o(1).\label{eq:1208-01}
\end{equation}
\end{lem}
\begin{lem}
For large enough $\beta$ and $\epsilon_{N}=o(1)$ decaying \corO{slowly} enough,
uniformly in $(u,x)\in NB_{N}\times\sqrt{N}D_{N}$, 
\begin{equation}
\mathbb{P}_{u,x}^{\qs}\left\{ \frac{1}{N}\log Z_{N,\beta}(\qs\hat{\mathbf{n}})-\frac{1}{N}\log Z_{N,\beta}^{2-}(\qs\hat{\mathbf{n}})<\frac{1}{2}\beta^{2}\sum_{k=3}^{\infty}\alpha_{2}^{2}(\qs)-\frac{\epsilon_{N}}{2}\right\} =o(1).\label{eq:1308-01}
\end{equation}
\label{lem-44old}
\end{lem}
\begin{proof}[\corO{Proof of Lemma \ref{lem:2sp_highT}.}]
With $H_{N-1,2}\left(\bs\right)$ denoting the pure $2$-spin mode \corO{as in \eqref{eq:3007-02}, set}
\begin{equation}
\bar{Z}_{N,\beta}:=\int_{\mathbb{S}^{N-2}(\sqrt{N-1})}\exp\{-\beta\alpha_{2}(\qs)\sqrt{\frac{N}{N-1}}H_{N-1,2}\left(\bs\right)\}d\bs,
\quad \corO{\bar{F}_{N,\beta}:=\frac{1}{N}\log\bar{Z}_{N,\beta}},\label{eq:1208-03}
\end{equation}
where the integration is with respect to  the uniform Hausdorff measure on
$\mathbb{S}^{N-2}(\sqrt{N-1})$. \corO{We} will show that
\begin{equation}
\mathbb{P}\left\{ \bar{F}_{N,\beta}<\frac{1}{2}\beta^{2}\alpha_{2}^{2}(\qs)+\beta\delta_{N}-\frac{\epsilon_{N}}{2}\right\} =o(1).\label{eq:1208-02}
\end{equation}
By Lemma \ref{cor:conditional},
the probability in (\ref{eq:1208-01}) is independent of $x$ and
depends on $u$ through the uniform `shift' in the conditional law,
and (\ref{eq:1208-02}) implies (\ref{eq:1208-01}) \corO{since the limit of} $\frac{1}{N}\log$
of the ratio of volumes from the definition (\ref{eq:3007-02})
equals $\frac{1}{2}\log(1-\qs^{2})$.

Let $\tilde{F}_{N,\beta}$ be defined similarly to $\bar{F}_{N,\beta}$,
only without the $\sqrt{\frac{N}{N-1}}$ term in (\ref{eq:1208-03}),
and note that 
it is enough 
to prove (\ref{eq:1208-02})
with $\tilde{F}_{N,\beta}$ instead of $\bar{F}_{N,\beta}$ (by increasing
$\epsilon_{N}$ if needed). Baik and Lee \cite[Theorem 1.2]{BaikLee}
proved that the free energy of the pure $2$-spin converges in distribution
as $N\to\infty$. In particular, their result show that $N(\tilde{F}_{N,\beta}-\frac{1}{2}\beta^{2}\alpha_{2}^{2}(\qs))$
converges to a Gaussian variable. This  implies (\ref{eq:1208-02})
provided that $N(\epsilon_{N}-2\beta\delta_{N})\to\infty$. 
\end{proof}

\begin{proof}[\corO{Proof of Lemma \ref{lem-44old}.}]
The proof builds on the argument used in \cite[Section 6.4]{geometryGibbs}.
First, note that by Lemma \ref{cor:conditional} the probability \corO{in}
(\ref{eq:1308-01}) does not depend on $u$ and $x$. In fact, the
difference of free energies in (\ref{eq:1308-01}) is equal in distribution
to $\frac{1}{N}\log \corO{X_N} $ where
\[
\corO{X_N}=\frac{\bar{Z}_{N,\beta}^{\prime}}{\bar{Z}_{N,\beta}}:=\frac{\int_{\mathbb{S}^{N-2}(\sqrt{N-1})}\exp\{-\beta\sum_{k=2}^{\infty}\alpha_{k}(\qs)\sqrt{\frac{N}{N-1}}H_{N-1,k}\left(\bs\right)\}d\bs}{\int_{\mathbb{S}^{N-2}(\sqrt{N-1})}\exp\{-\beta\alpha_{2}(\qs)\sqrt{\frac{N}{N-1}}H_{N-1,2}\left(\bs\right)\}d\bs},
\]
where $H_{N-1,k}\left(\bs\right)$ are independent pure models and
the integration is w.r.t. the probability Hausdorff measure on $\mathbb{S}^{N-2}(\sqrt{N-1})$.
To prove (\ref{eq:1308-01}) we will show that $\corO{X_N}$ concentrates around
its mean, given by 
\[
\frac{1}{N}\log\mathbb{E}\corO{X_N}=\frac{1}{2N}\beta^{2}\text{Var}\Big(\sum_{k=3}^{\infty}\bar{H}_{N}^{\hat{\mathbf{n}},k}(\qs\hat{\mathbf{n}})\Big)=\frac{1}{2}\beta^{2}\sum_{k=3}^{\infty}\alpha_{k}^{2}(\qs).
\]

We will define a sequence of events $\mathcal{E}_{N}$ measurable
w.r.t. $(H_{N-1,2}\left(\bs\right))_{\bs}$ such that $\lim_{N\to\infty}\mathbb{P}\{\mathcal{E}_{N}\}\to1$
and 
\begin{equation}
\lim_{N\to\infty}\frac{\mathbb{E}\{\corO{X_N^2}\,|\,\mathcal{E}_{N}\}}{(\mathbb{E}\corO{X_N^2})}\leq1.\label{eq:1308-03}
\end{equation}
Since $\mathbb{E}\{\corO{X_N}\,|\,(H_{N-1,2}\left(\bs\right))_{\bs}\}=\mathbb{E}\corO{X_N}$,
also $\mathbb{E}\{\corO{X_N}\,|\,\mathcal{E}_{N}\}=\mathbb{E}\corO{X_N}$. Thus, from
Chebyshev's inequality (\ref{eq:1308-03}) will imply (\ref{eq:1308-01}),
even with $\epsilon_{N}$ of any order larger than $1/N$.

Denote
\[
T_{N}\left(I\right):=\left\{ (\boldsymbol{\sigma}_{1},\boldsymbol{\sigma}_{2})\in(\mathbb{S}^{N-1}(\sqrt{N}))^{2}\,:\,R(\boldsymbol{\sigma}_{1},\boldsymbol{\sigma}_{2})\in I\right\} .
\]
Using the co-area formula with the mapping 
\[
(\boldsymbol{\sigma}_{1},\boldsymbol{\sigma}_{2})\mapsto R(\boldsymbol{\sigma}_{1},\boldsymbol{\sigma}_{2}),\quad\forall\bs_{i}\in\mathbb{S}^{N-1}(\sqrt{N}),
\]
we have that, for measurable $I\subset\left[-1,1\right]$, 
\begin{equation}
\begin{aligned}\mathbb{E}W(\bar{H}_{N},I) & :=\mathbb{E}\int_{T_{N}\left(I\right)}\exp\{-\beta\bar{H}_{N}(\bs_{1})-\beta\bar{H}_{N}(\bs_{2})\}d\bs_{1}d\bs_{2}\\
 & \,=\int_{I}\frac{\omega_{N-1}}{\omega_{N}}\left(1-\varrho^{2}\right)^{\frac{N-3}{2}}\exp\{\beta^{2}\vartheta_{N}(\varrho)\}d\varrho,
\end{aligned}
\label{eq:1110-02}
\end{equation}
where $\bar{H}_{N}(\bs)$ is a general mixed model and, with $\bs$
and $\bs_{\varrho}$ being two points with $R(\bs,\bs_{\varrho})=\varrho$,
\[
\vartheta_{N}(\varrho):=\text{Var}(\bar{H}_{N}(\bs))+\text{Cov}(\bar{H}_{N}(\bs),\bar{H}_{N}(\bs_{\varrho})).
\]
Thus, \corO{whenever} $\vartheta_{N}(\varrho)/N\to\vartheta(\varrho)$ uniformly
in $\varrho\in[-1,1]$,
\[
\lim_{N\to\infty}\frac{1}{N}\log\mathbb{E}W(\bar{H}_{N},I)=\sup_{\varrho\in I}\zeta(\bar{H}_{N},\varrho):=\sup_{\varrho\in I}\frac{1}{2}\log(1-\varrho^{2})+\beta^{2}\vartheta(\varrho).
\]

Now consider $\bar{H}_{N-1}^{(1)}(\bs)=\alpha_{2}(\qs)\sqrt{\frac{N}{N-1}}H_{N-1,2}\left(\bs\right)$.
The derivative 
\[
\frac{d}{d\varrho}\zeta(\bar{H}_{N-1}^{(1)},\varrho)=-\frac{\varrho}{1-\varrho^{2}}+2\beta^{2}\alpha_{2}^{2}(\qs)\varrho
\]
is negative for any $\varrho\in(0,1]$ and positive for any $\varrho\in[-1,0)$,
if $2\beta^{2}\alpha_{2}^{2}(\qs)<1$. From (\ref{eq:qs}),
\[
\lim_{\beta\to\infty}2\beta^{2}\alpha_{2}^{2}(\qs)-1=4t_{-}^{2}\nu''(1)-1<0,
\]
where the inequality follows since $4t_{-}^{2}\nu''(1)-1$ has the
same sign as the normalized limiting second derivative of (\ref{eq:2907-02})
at the local maximum $\qs$.
We conclude that for large enough $\beta$ and any $\tau>0$,
\[
\frac{1}{N}\log\mathbb{E}W(\bar{H}_{N-1}^{(1)},(-\tau,\tau))=\beta^{2}\alpha_{2}^{2}(\qs)>\frac{1}{N}\log\mathbb{E}W(\bar{H}_{N-1}^{(1)},[-1,1]\setminus(-\tau,\tau)).
\]
In particular, for some $\tau_{N}=o(1)$, 
\[
\lim_{N\to\infty}\frac{\mathbb{E}W\Big(\bar{H}_{N-1}^{(1)},\,[-1,1]\setminus(-\tau_{N},\tau_{N})\Big)}{\exp\{N\beta^{2}\alpha_{2}^{2}(\qs)\}}=0.
\]

Setting $I_{N}=[-\tau_{N},\tau_{N}]\setminus(-N^{-a},N^{-a})$ with
some $a\in(1/3,1/2)$, using the fact $\sqrt{N}\omega_{N}/\omega_{N-1}\to\sqrt{2\pi}$
and the change of variables $\sqrt{N}\varrho\mapsto\varrho'$, we
obtain 
\begin{equation}
\lim_{N\to\infty}\begin{aligned}\frac{\mathbb{E}W(\bar{H}_{N-1}^{(1)},I_{N})}{\exp\{N\beta^{2}\alpha_{2}^{2}(\qs)\}} & =\lim_{N\to\infty}\int_{\sqrt{N}I_{N}}\frac{1}{\sqrt{2\pi}}e^{-\varrho^{2}/2+\beta^{2}\alpha_{2}^{2}(\qs)\varrho^{2}+o(\varrho^{2})}d\varrho.\end{aligned}
\label{eq:1408-01}
\end{equation}
\corO{For} $\qs>q_{c}$ \corO{we have}  $\beta^{2}\alpha_{2}^{2}(\qs)<1/2$ \corO{and therefore} the
limit of (\ref{eq:1408-01}) is equal to $0$. By Markov's inequality,
we conclude that with probability tending to $1$ as $N\to\infty$,
\begin{equation}
W(\bar{H}_{N-1}^{(1)},[-1,1]\setminus(-N^{-a},N^{-a}))<\eta_{N}\exp\{N\beta^{2}\alpha_{2}^{2}(\qs)\},\label{eq:1408-02}
\end{equation}
 for some $\eta_{N}=o(1)$.

We are now ready to define the events $\mathcal{E}_{N}$. \corO{As in the proof of Lemma  \ref{lem:2sp_highT},    \cite[Theorem 1.2]{BaikLee}
implies that} \corE{with probability tending to $1$ as $N\to\infty$,
\[
\bar{Z}_{N,\beta}^{2}=W(\bar{H}_{N-1}^{(1)},[-1,1])>\eta_{N}^{1/2}\exp\left\{ N\beta^{2}\alpha_{2}^{2}(\qs)\right\} .
\]}
Define $\mathcal{E}_{N}$ as the intersection
of this event and (\ref{eq:1408-02}), so that on $\mathcal{E}_{N}$
we also have that
\[
W(\bar{H}_{N-1}^{(1)},(-N^{-a},N^{-a}))>(\eta_{N}^{1/2}-\eta_{N})\exp\{N\beta^{2}\alpha_{2}^{2}(\qs)\},
\]

Next consider $\bar{H}_{N-1}^{(2)}(\bs)=\sum_{k=2}^{\infty}\alpha_{k}(\qs)\sqrt{\frac{N}{N-1}}H_{N-1,k}\left(\bs\right)$.
In this case, as $\beta\to\infty$,
\[
\beta^{2}\vartheta_{N-1}(\varrho) =\beta^{2}\sum_{k=2}^{\infty}\alpha_{k}^{2}(\qs)N\left(1+\varrho^{k}\right)
 =\beta^{2}\sum_{k=2}^{\infty}\alpha_{k}^{2}(\qs)N+\beta^{2}\alpha_{2}^{2}(\qs)N\varrho^{2}+\varrho^{3}O\left(\frac{1}{\beta}\right).
\]
Using this and a similar argument to the one used above for the Hamiltonian
corresponding to $k=2$ only, we obtain that for the current Hamiltonian
\[
\mathbb{E}W(\bar{H}_{N-1}^{(2)},[-1,1]\setminus(-N^{-a},N^{-a}))<\eta_{N}\exp\{N\beta^{2}\sum_{k=2}^{\infty}\alpha_{k}^{2}(\qs)\},
\]
where we may need to increase $\eta_{N}=o(1)$. Therefore, for large $N$,
\begin{equation*}
\mathbb{E}\left[\frac{W(\bar{H}_{N-1}^{(2)},[-1,1]\setminus(-N^{-a},N^{-a}))}{\bar{Z}_{N,\beta}^{2}(\mathbb{E}\corE{X_N})^{2}}\,\Big|\,\mathcal{E}_{N}\right]<\eta_{N}^{1/2}(1+o(1)).
\end{equation*}

We conclude that 
\[
\mathbb{E}\{\corO{X_N^2}\,|\,\mathcal{E}_{N}\}=\mathbb{E}\left[\frac{W(\bar{H}_{N-1}^{(2)},[-1,1])}{W(\bar{H}_{N-1}^{(1)},[-1,1])}\,\Big|\,\mathcal{E}_{N}\right]=\mathbb{E}\left[\frac{W(\bar{H}_{N-1}^{(2)},(-N^{-a},N^{-a})}{W(\bar{H}_{N-1}^{(1)},(-N^{-a},N^{-a}))}\,\Big|\,\mathcal{E}_{N}\right]
(1+o(1)).
\]
By conditioning on $\bar{H}_{N-1}^{(1)}(\bs)$, similar to the above
\begin{equation*}
 \lim_{N\to\infty}\mathbb{E}\frac{W(\bar{H}_{N-1}^{(2)},(-N^{-a},N^{-a}))}{W(\bar{H}_{N-1}^{(1)},(-N^{-a},N^{-a}))}/(\mathbb{E}\corO{X_N})^{2}
=\lim_{N\to\infty}\int_{-N^{1/2-a}}^{N^{1/2-a}}\frac{1}{\sqrt{2\pi}}e^{-\varrho^{2}/2+\beta^{2}\sum_{k=3}^{\infty}\alpha_{k}(\qs)\varrho^{k}N^{-k/2+1}+o(\varrho^{2})}d\varrho=1.
\end{equation*}
This proves (\ref{eq:1308-03})
and completes the proof.
\end{proof}
\begin{proof}[\corO{Proof of Point (\ref{enu:Flb3}).}] \corO{Using
the formula (\ref{eq:LambdaZ}) for $\Lambda_{Z,\beta}(E,q)$), Point \eqref{enu:Flb3} follows from
Lemmas \ref{lem:2sp_highT} and \ref{lem-44old}.}
\end{proof}
\section{\label{sec:Upper-bounds}Upper bounds \corE{on the free energy}}
\corO{We prove in this section an  upper bound which is complementary to 
 the lower bound of Proposition \ref{prop:Flb}. Both bounds will play a crucial role in the proof of the first part of Theorem \ref{thm:Geometry}.}
 For any measurable $D\subset\mathbb{S}^{N-1}(\sqrt{N})$,
we define $Z_{N,\beta}(D)=\int_{D}e^{-\beta H_{N}(\bs)}d\bs$. 
\begin{prop}
\label{prop:UB}Assume \CDTN. For large enough $\beta$ and any
$\epsilon>0$, for small enough $\delta>0$ 
\begin{equation}
\frac{1}{N}\log Z_{N,\beta}\Big(\mathbb{S}^{N-1}(\sqrt{N})\setminus\cup_{\bs_{0}\in\mathscr{C}_{N,\qs}(NB)}{\rm Band}(\bs_{0},\epsilon)\Big)<\Lambda_{Z,\beta}(-E_{0}(\qs),\,\qs)-\delta\label{eq:0410-01}
\end{equation}
with probability tending to $1$ as $N\to\infty$, where $B=(-E_{0}(\qs)-\epsilon,-E_{0}(\qs)+\epsilon)$.
\end{prop}

The \corO{rest of the section is devoted to the proof of Proposition \ref{prop:UB}.} In Section \ref{subsec:UB1},
using the results on the structure of sub-level sets from Section
\ref{sec:lvlsets}, we show how Proposition \ref{prop:UB} can be
deduced from bounds on weights of \corO{sections} $Z_{N,\beta}(\bs_{q})$,
\corO{defined in \eqref{eq:Zsigma},} centered at $q$-critical points of a given
depth. In Section \ref{subsec:11.2} we prove two general upper bounds
for the latter, which are relevant for different ranges of $q$. Finally,
in Section \ref{subsec:11.3} we use those to conclude Proposition
\ref{prop:UB}.

\subsection{\label{subsec:UB1}\corO{A reduction to bounds
on weights $Z_{N,\beta}(\protect\bs_{q})$ of $q$-critical points}}
We begin with the observation that, since the integration is w.r.t.
the probability Haar measure on the sphere,
\[
\frac{1}{N}\log\int_{\mathbb{S}^{N-1}(\sqrt{N})\setminus A_{t}}e^{-\beta H_{N}(\bs)}d\bs\leq\beta(E_{0}-t),
\]
where $A_{t}$ is the sub-level of $-(E_{0}-t)N$ (see (\ref{eq:sublvl-1})).
The following is a direct consequence.
\begin{cor}
\label{cor:UB}It is enough to prove Proposition \corO{\ref{prop:UB}}
with $\mathbb{S}^{N-1}(\sqrt{N})$ in (\ref{eq:0410-01}) replaced
by the (random) subset $A_{\tau}$ \corO{from \eqref{eq:sublvl-1},} with 
\begin{equation}
\tau:=\tau(\beta,\delta)=E_{0}(\nu)-\Lambda_{Z,\beta}(-E_{0}(\qs),\,\qs)/\beta+\delta.\label{eq:tau}
\end{equation}
\end{cor}

From the asymptotics of $\qs$ \corO{in} (\ref{eq:qs}), the definition of $\Lambda_{Z,\beta}$, \corO{see}
(\ref{eq:LambdaZ}), and the continuity of $E_{0}(q)$ \corO{proved in}  Lemma \ref{lem:ddq},
$\tau\to0$ as $\beta\to\infty$ and $\delta\to0$. Hence, with high
probability, the sub-level set $A_{\tau}$ is covered by the caps \corO{in \eqref{def:cap} from Proposition \ref{prop:sublvl}.}
From this it is not difficult to move to a cover by
\corO{sections} of the form $\mathcal{S}(\bs)$, \corO{see} (\ref{eq:subsp}).
\begin{lem}
\label{lem:43}Assume that $\nu$ satisfies \CDTN. Then, for small enough
$t$, with the notation of Proposition \ref{prop:sublvl}, for each
connected component $A$ of $A_{t}$,
\begin{equation}
A\subset\cup_{q\in[1-\cls t,1]}\mathcal{S}(\bs_{q}),\label{eq:0510-01}
\end{equation}
 with probability tending to $1$ as $N\to\infty$.
\end{lem}

\begin{proof}
Let $\bs\in A$ and define the continuous function $d_{\bs}(q)=R(\bs,\bs_{q})-q$,
for $q\in[1-\cls t,1]$. If $\bs\neq\bs_{1}$ and $\bs\notin\mathcal{S}(\bs_{1-\cls t})$,
then $d_{\bs}(q)<0$ and $d_{\bs}(1-\cls t)>0$. By the mean value theorem,
there exists some $q\in[1-\cls t,1]$ such that $d_{\bs}(q)=0$, which
exactly means that $\bs\in\mathcal{S}(\bs_{q})$.
\end{proof}
Finally, we translate the bound we need to \corO{bounds treating}  each pair
$(q,E)$ separately, in an appropriate sense. Denote
\begin{equation}
W_{\tau,\eta}:=\big\{(q,E): q\in[1-\cls\tau,1],\;
  E\in[-E_{0}(q)-\eta,-E_{0}(1)+2x_{0}(1)\cls\tau]\big\}
\label{eq:0510-02}
\end{equation}
and 
\[
B_{\boxempty}(q,E,\epsilon):=(q-\epsilon,q+\epsilon)\times(E-\epsilon,E+\epsilon).
\]
\begin{lem}
\label{lem:44}Assume that $\nu$ satisfies \CDTN.
 \corO{Assume that} for large enough $\beta$
and any small $\epsilon>0$, there exist $\eta$, $\delta$ and $\upsilon$
(depending on $\epsilon$, $\nu$ and $\beta$), \corO{so that with
$\tau$ given by (\ref{eq:tau})}, for any 
\begin{equation}
(q,E)\in W_{\tau,\eta}\setminus B_{\boxempty}(\qs,-E_{0}(\qs),\epsilon),\label{eq:qEdist}
\end{equation}
 we have that, \corO{with probability tending to $1$ as $N\to\infty$,}
\begin{equation}
\frac{1}{N}\log\sum_{\bs\in\mathscr{C}_{N,q}([E-\upsilon,E+\upsilon])}Z_{N,\beta}(\bs)<\Lambda_{Z,\beta}(-E_{0}(\qs),\,\qs)-\delta.\label{eq:bd2}
\end{equation}
\corO{Then,  Proposition
\ref{prop:UB} holds true.}
\end{lem}

\begin{proof}
Throughout the proof we implicitly restrict to the event that all
the statements of Proposition \ref{prop:sublvl}, Lemma \ref{lem:43}
and Corollary \ref{cor:gradbd} hold. (\corO{The probability of this event}  converges
to $1$, as $N\to\infty$.) All the statements below should be interpreted
as `occurring with probability tending to $1$', and we will refrain
from repeatedly writing so. Note that since we assume that $\beta$
is large, $\tau$ given by (\ref{eq:tau}) \corO{can be made arbitrarily small.}

By an abuse of notation, denote by $\mathscr{C}_{N,q}(NE)$ the set of $q$-critical points
$\bs_{0}$ such that $H_{N}(\bs_{0})=NE$. From Proposition \ref{prop:sublvl},
for any $q\in[1-\cls\tau,1]$ and connected component $A$ of $A_{\tau}$,
there is a corresponding $q$-critical point $\bs_{q}$. Denote the
subset of those points for which $H_{N}(\bs_{q})=NE$ by $\mathscr{C}_{N,q}^{\tau}(NE)$.
For any $W\subset\mathbb{R}^{2}$ denote, by an abuse of notation,
\[
\begin{aligned}\mathcal{S}(W) & :=\cup_{(q,E)\in W}\cup_{\bs\in\mathscr{C}_{N,q}(NE)}\mathcal{S}(\bs),\\
{\rm Band}(W,\epsilon) & :=\cup_{(q,E)\in W}\cup_{\bs\in\mathscr{C}_{N,q}(NE)}{\rm Band}(\bs,\epsilon),
\end{aligned}
\]
and define $\mathcal{S}_{\tau}(W)$ and ${\rm Band}_{\tau}(W,\epsilon)$
similarly, with $\mathscr{C}_{N,q}(NE)$ replaced by $\mathscr{C}_{N,q}^{\tau}(NE)$.

Let $A$ be some connected component of $A_{\tau}$, and let $\bs_{q}$
be the corresponding path of $q$-critical points. From Part \ref{enu:sublvl3}
of Proposition \ref{prop:sublvl}, since $\tau$ is small, for arbitrary
$\eta$ and large enough $N$, 
\[
\frac{1}{N}H_{N}(\bs_{q})\in[-E_{0}(q)-\eta,-E_{0}(1)+2x_{0}(1)\cls\tau],
\]
for all $q\in[1-\cls t,1]$. Combining this with Lemma \ref{lem:43},
we obtain that 
\[
A_{\tau}\subset\mathcal{S}_{\tau}(W_{\tau,\eta}).
\]

Next, we construct a cover using bands corresponding to a finite number
of values of $q$. For two points $\bs,\,\bs'\in\BN$, if $(\|\bs\|-\|\bs'\|)/\sqrt{N}$
and the distance between $\bs/\|\bs\|$ and $\bs'/\|\bs'\|$ (w.r.t.
the standard metric on the sphere) are both in $(-\upsilon/2,\upsilon/2)$,
then
\[
\mathcal{S}(\bs)\subset{\rm Band}\Big(\frac{\|\bs'\|}{\|\bs\|}\bs,\frac{\upsilon}{2}\Big)\subset{\rm Band}(\bs',\upsilon).
\]

\corO{With $A$ and $\bs_q$ as above,}  fix some $q\in[1-\cls t,1]$. From the Lipschitz bound of (\ref{eq:1507-03})
and Point \ref{enu:sublvl2} of Proposition \ref{prop:sublvl}, we
have the following. For any given $\upsilon>0$, for small enough
$\upsilon/2>\upsilon'>0$ (independent of $q$ and $E$), if $|q-q'|<\upsilon'$
and $|\frac{1}{N}H_{N}(\bs_{q'})-E|<\upsilon'$, then, \corO{with notation as in Proposition \ref{prop:sublvl},} 
\[
|\mathcal{G}(q')-\mathcal{G}(q)|<\frac{\upsilon}{2}\,\text{ and }\,|\frac{1}{N}H_{N}(\bs_{q})-E|<\upsilon,
\]
and thus
\[
\mathcal{S}(\bs_{q'})\subset{\rm Band}(\bs_{q},\upsilon).
\]

Denoting
\[
B_{\shortmid}(q,E,\upsilon)=\{q\}\times(E-\upsilon,E+\upsilon),
\]
 we therefore have that
\begin{equation}
\mathcal{S}_{\tau}\left(B_{\boxempty}(q,E,\upsilon')\right)\subset{\rm Band}_{\tau}\left(B_{\shortmid}(q,E,\upsilon),\upsilon\right).\label{eq:SBrel}
\end{equation}

Applying the above with $(q,E)=(\qs-E_{0}(\qs))$ and $\upsilon=\epsilon$,
we have that for small enough $\epsilon'>0$
\[
\cup_{\bs_{0}\in\mathscr{C}_{N,\qs}(NB)}{\rm Band}(\bs_{0},\epsilon)  ={\rm Band}(B_{\shortmid}(\qs,-E_{0}(\qs),\epsilon),\epsilon)
 \supset\mathcal{S}_{\tau}\left(B_{\boxempty}(\qs,-E_{0}(\qs),\epsilon')\right).
\]
From Corollary \ref{cor:UB} and Lemma \ref{lem:43}, we conclude
that in order to prove Proposition \ref{cor:UB}, it is enough to
show that for any arbitrarily small $\epsilon$, there exists some
$\delta$ such that 
\begin{equation}
\frac{1}{N}\log Z_{N,\beta}\Big(\mathcal{S}_{\tau}\big(W_{\tau,\eta}\setminus B_{\boxempty}(\qs,-E_{0}(\qs),\epsilon)\big)\Big)<\Lambda_{Z,\beta}(-E_{0}(\qs),\,\qs)-\delta\label{eq:bd1}
\end{equation}
with probability tending to $1$ as $N\to\infty$.

Now, fix $\epsilon>0$ and let $\upsilon$ and $\eta$ (see (\ref{eq:0510-02}))
be some small numbers. Let $\upsilon'$ be the value corresponding
to $\upsilon$ by the relation above. Choose some cover for the region $W_{\tau,\eta}\setminus B_{\boxempty}(\qs,-E_{0}(\qs),\epsilon)$
by a finite number of boxes $B_{\boxempty}(q,E,\upsilon')$, with
each of the centers $(q,E)$ belonging to $W_{\tau,\eta}\setminus B_{\boxempty}(\qs,-E_{0}(\qs),\epsilon)$.

From (\ref{eq:SBrel}), 
\[
\frac{1}{N}\log Z_{N,\beta}\Big(\mathcal{S}_{\tau}\big(B_{\boxempty}(q,E,\upsilon')\big)\Big)\leq\frac{1}{N}\log Z_{N,\beta}\Big({\rm Band}_{\tau}\left(B_{\shortmid}(q,E,\upsilon),\upsilon\right)\Big).
\]
Hence, since we are dealing with a finite number of boxes, for (\ref{eq:bd1})
to hold, it is enough to establish that for each of the boxes $B_{\boxempty}(q,E,\upsilon')$,
\[
\frac{1}{N}\log Z_{N,\beta}\Big({\rm Band}_{\tau}\left(B_{\shortmid}(q,E,\upsilon),\upsilon\right)\Big)<\Lambda_{Z,\beta}(-E_{0}(\qs),\,\qs)-\delta
\]
with probability going to $1$. From the Lipschitz bound of (\ref{eq:1507-03}),
it is thus enough to show with such probability that
\[
\frac{1}{N}\log Z_{N,\beta}\Big(\mathcal{S}_{\tau}\left(B_{\shortmid}(q,E,\upsilon)\right)\Big)<\Lambda_{Z,\beta}(-E_{0}(\qs),\,\qs)-\delta,
\]
where we may need to decrease $\delta$. This completes the proof
of Lemma \ref{lem:44}.
\end{proof}

\subsection{\label{subsec:11.2}General bounds on weights $Z_{N,\beta}(\protect\bs_{q})$
at a given depth}

This section is devoted to the proof of the following three lemmas,
bounding from above the contribution to the free energy coming from
$q$-critical points. We recall that $Z_{N,\beta}(\bs_{0})$ and $Z_{N,\beta}^{2-}(\bs_{0})$
below are as defined in (\ref{eq:Zsigma}) and (\ref{eq:3007-02}). 
\begin{lem}
\label{lem:bound1}For any $\delta>0$ there exists a constant $c=c(\nu)>0$,
such that for any $q\in(\delta,1)$, $E\in(-2E_{0}(q),0)$ and $\epsilon>0$,
setting $B=(E-\epsilon,E+\epsilon)$,
\begin{equation}
\limsup_{N\to\infty}\frac{1}{N}\log\mathbb{E}\sum_{\boldsymbol{\sigma}_{0}\in\mathscr{C}_{N,q}(NB)}Z_{N,\beta}(\bs_{0})\leq\sup_{x\in\mathbb{R}}\Theta_{\nu,q}(E,x)+\Lambda_{Z,\beta}(E,q)+(\beta+c)\epsilon,\label{eq:71-1-1}
\end{equation}
where $\Theta_{\nu,q}(E,x)$ and $\Lambda_{Z,\beta}(E,q)$ are given
by (\ref{eq:Theta}) and (\ref{eq:LambdaZ}).
\end{lem}

\begin{lem}
\label{lem:bound2}For any $\delta>0$ there exists a constant $c=c(\nu)>0$,
such that for any $q\in(\delta,1)$ with $\beta\alpha_{2}(q)\geq1/\sqrt{2}$,
$E\in(-2E_{0}(q),0)$ and $\epsilon>0$, setting $B=(E-\epsilon,E+\epsilon)$,
\begin{equation}
\lim_{N\to\infty}\mathbb{P}\Bigg\{ \frac{1}{N}\log\sum_{\boldsymbol{\sigma}_{0}\in\mathscr{C}_{N,q}(NB)}Z_{N,\beta}^{2-}(\bs_{0})\geq\theta+\beta\alpha_{2}(q)\sqrt{2\theta+c\epsilon}+\Lambda_{F,\beta}^{2-}(E,q)+\beta\epsilon\Bigg\} =0,\label{eq:71-1-1-1}
\end{equation}
where $\theta=\sup_{x\in\mathbb{R}}\Theta_{\nu,q}(E,x)$ and $\Lambda_{F,\beta}^{2-}(E,q)$
is defined by (\ref{eq:LambdaZ-1}).
\end{lem}

\begin{lem}
\label{lem:bound3}There exist constants $C,\,c>0$, such that for
any $q\in(0,1)$,
\[
\mathbb{P}\bigg\{ \exists \bs_{0}\in\mathbb{S}^{N-1}(\sqrt{N}q):\,\frac{1}{N}|\log Z_{N,\beta}(\bs_{0})-\log Z_{N,\beta}^{2-}(\bs_{0})|>C(1-q^{2})^{3/2}\bigg\} \leq e^{-cN}.
\]
\end{lem}

\begin{proof}[Proof of Lemma \ref{lem:bound1}.]
From Corollary \ref{cor:smallbandZ}, for any $q>0$ and $v\in\mathbb{R}$,
\[
\lim_{N\to\infty}\frac{1}{N}\log\mathbb{E}_{NE,v}^{q}\left\{ Z_{N,\beta}(q\hat{\mathbf{n}})\right\} =\Lambda_{Z,\beta}(E,q).
\]
By Lemma \ref{cor:conditional}, replacing $E$ by $E+\delta$
in the conditional expectation above amounts to shifting the conditional
law of the random field $H_{N}|_{q}(\bs)$, uniformly in $\bs$, by
$N\delta$. Moreover, by the same corollary the conditional law $H_{N}|_{q}(\bs)$
is independent of $v$. Thus, 
\begin{equation}
\lim_{N\to\infty}\sup_{E'\in B,\,v\in\mathbb{R}}\frac{1}{N}\log\mathbb{E}_{NE',v}^{q}\left\{ Z_{N,\beta}(q\hat{\mathbf{n}})\right\} =\Lambda_{Z,\beta}(E,q)+\beta\epsilon.\label{eq:0310-04}
\end{equation}
Hence, by \corO{the Kac-Rice formula contained in} Lemma \ref{lem:17},
\[
\limsup_{N\to\infty}\frac{1}{N}\log\mathbb{E}\sum_{\boldsymbol{\sigma}_{0}\in\mathscr{C}_{N,q}(NB)}Z_{N,\beta}(\bs_{0})\leq\sup_{E'\in B,\,x\in\mathbb{R}}\Theta_{\nu,q}(E',x)+\Lambda_{Z,\beta}(E,q)+\beta\epsilon,
\]
from which (\ref{eq:71-1-1}) follows by the fact that $\sup_{x\in\mathbb{R}}\Theta_{\nu,q}(E,x)$
is Lipschitz in $E\in(-2E_{0}(q),0)$, uniformly over $q\in(\delta,1)$
and therefore for some $c>0$,
\begin{equation}
\sup_{E'\in B}\sup_{x\in\mathbb{R}}\Theta_{\nu,q}(E',x)<\sup_{x\in\mathbb{R}}\Theta_{\nu,q}(E,x)+c\epsilon.\label{eq:0310-02}
\end{equation}
\end{proof}

\begin{proof}[Proof of Lemma \ref{lem:bound2}.]
The conditional variance of $\sum_{i=0}^{2}\bar{H}_{N}^{\hat{\mathbf{n}},i}|_{q}\left(\boldsymbol{\sigma}\right)$
under $\mathbb{P}_{NE,v}^{q}$ is equal to $N\alpha_{2}^{2}(q)$ (see
(\ref{eq:2909-01})). Thus, setting
\[
\Delta_{N}(\boldsymbol{\sigma}_{0})=\Big|\frac{1}{N}\log Z_{N,\beta}^{2-}(\boldsymbol{\sigma}_{0})-\frac{1}{N}\mathbb{E}_{NE,v}^{q}\log Z_{N,\beta}^{2-}(q\hat{\mathbf{n}})\Big|,
\]
by Corollary \ref{cor:concF}, for any $v\in\mathbb{R}$ and $t>0$,
\[
\mathbb{P}_{-NE,v}^{q}\left\{ \Delta_{N}(q\hat{\mathbf{n}})>t\right\} \leq3\exp\left\{ -(N-1)^2t^{2}/2\beta^{2}\alpha_{2}^{2}(q)\right\} .
\]
Lemma \ref{lem:15}, therefore, implies that
\begin{equation}
\limsup_{N\to\infty}\frac{1}{N}\log\left(\mathbb{E}\left|\left\{ \boldsymbol{\sigma}_{0}\in\mathscr{C}_{N,q}(NB):\,\Delta_{N}(\boldsymbol{\sigma}_{0})>t\right\} \right|\right)\leq\sup_{E'\in B,\,x\in\mathbb{R}}\Theta_{\nu,q}(E',x)-t^{2}/2\beta^{2}\alpha_{2}^{2}(q).\label{eq:0310-03}
\end{equation}

By (\ref{eq:0310-02}), for
\[
t\geq\beta\alpha_{2}(q)\sqrt{2\left(\theta+c\epsilon\right)},
\]
where $\theta=\sup_{x\in\mathbb{R}}\Theta_{\nu,q}(E,x)$, the left-hand
side of (\ref{eq:0310-03}) is negative.

Thus, with probability tending to $1$ as $N\to\infty$, for all the
points $\boldsymbol{\sigma}_{0}\in\mathscr{C}_{N,q}(NB)$ we have
that $\Delta_{N}(\boldsymbol{\sigma}_{0})<t$. From Theorem \ref{thm:1stmoment},
the number of points in $\mathscr{C}_{N,q}(NB)$ is bounded by $\theta+c\epsilon$,
with probability tending to $1$. The proof of the lemma therefore
follows from the fact that complement of the event in (\ref{eq:71-1-1-1})
is contained in the intersection of those two events, and since, by
Lemma \ref{lem:LambdaF}, similarly to (\ref{eq:0310-04}),
\begin{equation}
\lim_{N\to\infty}\sup_{E'\in B}\frac{1}{N}\mathbb{E}_{NE',v}^{q}\left\{ \log Z_{N,\beta}(q\hat{\mathbf{n}})\right\} =\Lambda_{F,\beta}^{2-}(E,q)+\beta\epsilon.\label{eq:0310-04-1}
\end{equation}
\end{proof}

\begin{proof}[Proof of Lemma \ref{lem:bound3}.]
Recall that by  definition ((\ref{eq:Zsigma}) and (\ref{eq:3007-02})),
\[
\frac{Z_{N,\beta}(\bs_{0})}{Z_{N,\beta}^{2-}(\bs_{0})}=\frac{\int_{\mathcal{S}(\bs_{0})}\exp(-\beta H_{N}(\bs))d\bs}{\int_{\mathcal{S}(\bs_{0})}\exp(-\beta\sum_{i=0}^{2}\bar{H}_{N}^{\bs_{0},i}\left(\boldsymbol{\sigma}\right))d\bs}.
\]
From Lemma \ref{lem:apxest} applied with $k=2$, for some constants
$C,\,c>0$,
\[
\mathbb{P}\Big\{\exists \bs_{0}\in\mathbb{S}^{N-1}(\sqrt{N}q):\,\frac{1}{N}\sup_{\bs\in\mathcal{S}(\bs_{0})}\Big|H_{N}\left(\boldsymbol{\sigma}\right)-\sum_{i=0}^{k}\bar{H}_{N}^{\bs_{0},i}\left(\boldsymbol{\sigma}\right)\Big|>C(1-q^{2})^{3/2}\Big\}\leq e^{-cN}.
\]
Lemma \ref{lem:bound3} directly follows from those two facts.
\end{proof}

\subsection{\label{subsec:11.3}Proof of Proposition \ref{prop:UB}}

Let $\beta>0$ be some large number, and $\epsilon>0$ be some small
number. We will show that there exist $\eta$, $\delta$ and $\upsilon$
satisfying the bound (\ref{eq:bd2}) as in Lemma \ref{lem:44}.

Suppose that $(q,E)\in W_{\tau,\eta}\setminus B_{\boxempty}(\qs,-E_{0}(\qs),\epsilon)$
(see (\ref{eq:0510-02}) and (\ref{eq:tau})) and further assume that
$q\geq\qss$ (see (\ref{eq:qc})). From Lemma \ref{lem:bound1}, with
probability tending to $1$ as $N\to\infty$,
\[
\frac{1}{N}\log\sum_{\bs\in\mathscr{C}_{N,q}([E-\upsilon,E+\upsilon])}Z_{N,\beta}(\bs)<\sup_{x\in\mathbb{R}}\Theta_{\nu,q}(E,x)+\Lambda_{Z,\beta}(E,q)+(\beta+c)\upsilon,
\]
for some constant $c$. Since we can choose $\upsilon$ as small as
we wish, (\ref{eq:bd2}) follows if we show that for $(q,E)$ as above,
\begin{equation}
\sup_{x\in\mathbb{R}}\Theta_{\nu,q}(E,x)+\Lambda_{Z,\beta}(E,q)<\Lambda_{Z,\beta}(E_{0}(\qs),\,\qs)-\delta/2.\label{eq:0710-01}
\end{equation}
From continuity \corO{of} the left-hand side in $(q,E)$ (uniformly on compacts),
if we prove \eqref{eq:0710-01} for any $(q,E)\in W_{\tau,0}\setminus B_{\boxempty}(\qs,-E_{0}(\qs),\epsilon)$
with
$q\geq\qss$, i.e., with $\eta=0$, then the same will follow for some small $\eta>0$. 

 Note that the dependence of $\Lambda_{Z,\beta}(E,q)$ \corO{in} $E$ is through the term $\beta E$ in its definition 
 (\ref{eq:LambdaZ}),  $\Theta_{\nu,q}(E,x)$ does not depend on $\beta$, and $\sup_{x\in\mathbb{R}}\Theta_{\nu,q}(E,x)$ is a Lipschitz function of $E$ in a compact set uniformly over $q$ in a compact subset of $(0,1]$. 
Thus, for large enough $\beta$, any constant $C>0$  and fixed $q\in[1-\qss,1]$,
\begin{align*}
\forall E\in(-E_0(q),-E_0(q)+C] :\ \   &\sup_{x\in\mathbb{R}}\Theta_{\nu,q}(E,x)+\Lambda_{Z,\beta}(E,q)  \\ &= \sup_{x\in\mathbb{R}}\Theta_{\nu,q}(-E_0(q),x)+\Lambda_{Z,\beta}(-E_0(q),q)
=\Lambda_{Z,\beta}(-E_0(q),q).
\end{align*} 
 
As we saw in Section \ref{sec:log-weight}, the
maximum of $q\mapsto\Lambda_{Z,\beta}(-E_{0}(q),q)$ over $[\qss,1]$
is obtained at $\qs$. From continuity, this completes the proof in
the case where $q\geq\qss$.

Now assume that $(q,E)\in W_{\tau,\eta}\setminus B_{\boxempty}(\qs,-E_{0}(\qs),\epsilon)$
and $q\in[1-\cls\tau,\,\qss]$. From Lemma \ref{lem:bound2}, with
probability tending to $1$ as $N\to\infty$,
\[
\frac{1}{N}\log\sum_{\bs\in\mathscr{C}_{N,q}([E-\upsilon,E+\upsilon])}Z_{N,\beta}(\bs)<\theta_{q,E}+\beta\alpha_{2}(q)\sqrt{2\theta_{q,E}+c\upsilon}+\Lambda_{F,\beta}^{2-}(E,q)+\beta\upsilon,
\]
for some constant $c$, where $\theta_{q,E}=\sup_{x\in\mathbb{R}}\Theta_{\nu,q}(E,x)$.
As before, by assuming that $\upsilon$ is small enough, we absorb the $\beta\upsilon$ and $c\upsilon$ terms into
$\delta/2$, so that we need to show that 
\begin{equation}
\theta_{q,E}+\beta\alpha_{2}(q)\sqrt{2\theta_{q,E}}+\Lambda_{F,\beta}^{2-}(E,q)\leq\Lambda_{Z,\beta}(E_{0}(\qs),\,\qs)-\delta/2.\label{eq:1010-02}
\end{equation}

To prove (\ref{eq:1010-02}) we will develop the $\beta\to\infty$
asymptotics of the terms above. Below $c$ and $C$ will be constants
that are assumed to be sufficiently small or large, respectively,
whenever needed. We also allow them to change from line to line. Assume
henceforth, that $\delta$ and $\eta$, which are allowed to depend
on $\beta$, are both smaller than $c\log\beta/\beta$.

First, we note that with $t_{-}$ as in (\ref{eq:qs}), from (\ref{eq:LambdaZ})
and (\ref{eq:tau}), as $\beta\to\infty$,
\begin{equation}
\begin{aligned}\beta\alpha_{2}(\qs) & =\sqrt{2\nu''(1)}t_{-}+O(\frac{1}{\beta}),\\
\Lambda_{Z,\beta}(-E_{0}(\qs),\,\qs) & =\beta E_{0}(\qs)+\frac{1}{2}\log(\frac{2t_{-}}{\beta})+t_{-}^{2}\nu''(1)+O(\frac{1}{\beta}),\\
\tau(\beta,\delta) & =E_{0}(1)-E_{0}(\qs)+\frac{\log\beta}{2\beta}+\delta+O(\frac{1}{\beta})\leq\frac{\log\beta}{\beta},
\end{aligned}
\label{eq:1010-01}
\end{equation}
where the inequality follows since, by Lemma \ref{lem:ddq}, $E_{0}(q)$
is differentiable at $q=1$ and $\qs=1-O(1/\beta)$.

For any compact $K$, for large enough $T>0$, $\sup_{x\in\mathbb{R}}\Theta_{\nu,q}(E,x)=\sup_{|x|\leq T}\Theta_{\nu,q}(E,x)$
uniformly over $E\in K$ and $q$ close enough to $1$ (see Lemma \ref{lem:compact}). Using this
and the fact that $\theta_{1,-E_{0}(1)}=0$, one can verify that for
some $C>0$,
\[
\theta_{q,E}\leq|1-q|C+|E+E_{0}(1)|C,
\]
for $(q,E)$ in a small neighborhood of $(1,-E_{0}(1))$. Thus, for
$(q,E)\in W_{\tau,\eta}$ and $\beta$ large enough, we have from
(\ref{eq:1010-01}) and (\ref{eq:LambdaZ-1}) that $\beta\theta_{q,E}$,
$\beta\alpha_{2}(q)$ and 
\[
|\Lambda_{F,\beta}^{2-}(E,q)-\Lambda_{F,\beta}^{2-}(E_{0}(q),q)|
\]
 are all smaller than $C\log\beta$.

From the above, to prove (\ref{eq:1010-02}) it will be enough to
show that for $q\in[1-\cls\tau,\,\qss]$,
\[
C\frac{\log\beta}{\beta}+C\frac{(\log\beta)^{3/2}}{\beta^{1/2}}+\Lambda_{F,\beta}^{2-}(-E_{0}(q),q)\leq\Lambda_{Z,\beta}(-E_{0}(\qs),\,\qs)-\delta/2,
\]
or, since we assume $\beta$ is large,
\begin{equation}
\label{eq:2210-01}\Lambda_{F,\beta}^{2-}(-E_{0}(q),q)\leq\Lambda_{Z,\beta}(-E_{0}(\qs),\,\qs)-\delta/4.
\end{equation}

From (\ref{eq:1010-03}), reparameterizing  $$\tilde\Lambda_{F,\beta}^{2-}(t)=\Lambda_{F,\beta}^{2-}(-E_{0}(q),q)|_{q=1-\frac{t}{\beta}},$$
\corO{we deduce that,} uniformly in $t\in(0,\cls \log\beta]\supset(0,\cls \tau \beta]$, as $\beta\to\infty$,
\[
\tilde\Lambda_{F,\beta}^{2-}(t)=\kappa'-t\kappa (1+o_\beta(1)),
\]
where $\kappa=x_0(1)-2\sqrt{\nu''(1)}>0$ and $\kappa'=\lim_{s\to 0}\tilde\Lambda_{F,\beta}^{2-}(s)$.

Therefore, \eqref{eq:2210-01} follows from \eqref{eq:gap} and \eqref{1010-04}.
This concludes the proof of (\ref{eq:bd2}), for small enough $\eta$,
$\delta$ and $\upsilon$, and thus also the proof of Proposition
\ref{prop:UB}.\qed

\section{\label{sec:PfsGeometry}Proofs of the main results: Theorems  \ref{thm:Geometry}, \ref{thm:geometry1} and \ref{thm:Chaos}}

\corO{Recall the definition of \corE{$\qs$}, see \eqref{eq:qs}.} 
The energy $\Es$ which was used in the statements of the proofs is defined as the limiting normalized ground state (see Remark \ref{rem:GSq})
\begin{equation}
\label{eq:Es}\Es:=\Es(\beta)=E_0(\qs).
\end{equation}
Throughout the proofs we will use the notation $B(\epsilon)=-E_{0}(\qs)+(-\epsilon,\epsilon)$
and $D(\epsilon)=-x_{0}(\qs)+(-\epsilon,\epsilon)$. 

\subsection{Proof of Theorem \ref{thm:Geometry}}

From continuity, 
\[
\lim_{\epsilon\to0}\sup_{E\in B(\epsilon),\, x\in D(\epsilon)}\Theta_{\nu,\qs}(E,x)=\Theta_{\nu,\qs}(-E_0(\qs),-x_0(\qs))=0.
\]
Thus, by Theorem \ref{thm:1stmoment} and Lemma \ref{lem:epsx0},  \eqref{eq-subsexp}  holds for any choice of $\epsilon_N=o(1)$.

\corO{From the}  lower bound of Proposition \ref{prop:Flb} and the upper bound
of Proposition \ref{prop:UB}, we have that for any $\epsilon>0$
and small enough $\delta>0$, with probability tending to $1$ as
$N\to\infty$,
\begin{equation}
\frac{1}{N}\log Z_{N,\beta}>\Lambda_{Z,\beta}(-E_{0}(\qs),\,\qs)-\delta/2\label{eq:1010-06}
\end{equation}
and 
\begin{equation}
  \label{eq:1010-06a}
\frac{1}{N}\log Z_{N,\beta}\Big(\mathbb{S}^{N-1}(\sqrt{N})\setminus\cup_{\bs_{0}\in\mathscr{C}_{N,q}(NB(\epsilon))}{\rm Band}(\bs_{0},\epsilon)\Big)<\Lambda_{Z,\beta}(-E_{0}(\qs),\,\qs)-\delta.
\end{equation}
This proves the statement of Part \ref{enu:Geometry1} with $\epsilon>0$
instead of $\epsilon_{N}=o(1)$. By a standard diagonalization argument,
we obtain the same with $\epsilon_{N}=o(1)$, assuming the rate of
decay is slow enough. \qed
\corO{\begin{rem}
  \label{corinintro}
 Using \eqref{eq:1010-06}, \eqref{eq:1010-06a}, \eqref{eq:qs}  and
  \eqref{eq:qc}, we obtain that 
  \[ F_\beta=\Lambda_{Z,\beta} (-E_0(\qs), \qs) = \sup_{q\in[q_c, 1)} \Lambda_{Z,\beta} (-E_0(q),q).\]
\end{rem}
}

\subsection{Proof of Theorem \ref{thm:geometry1}, Part \ref{enu:Geometry2}}

Let $\epsilon_{N}$ be the sequence defined in Theorem \ref{thm:Geometry}. First note that by Lemma \ref{lem:epsx0}, instead of $\Cs$, it will be enough to prove the theorem with points only from 
\begin{equation}
\Cs^+:=\mathscr{C}_{N,\qs}(NB(\epsilon_{N}),\,\sqrt{N}D(\epsilon_{N})),\label{eq:1010-08}
\end{equation}
where $\epsilon_N$ may be needed to be increased.
By a union bound, the first limit of Part \ref{enu:Geometry2} will follow if we show that
\begin{equation}
\lim_{N\to\infty}\mathbb E\sum_{\bs_{0}\in\Cs^+} G_{N,\beta}\times G_{N,\beta }\left\{ \bs,\,\bs'\in{\rm Band}(\bs_{0},\epsilon_{N}),\,\left|R\left(\boldsymbol{\sigma},\boldsymbol{\sigma}^{\prime}\right)-\qs^{2}\right|>\delta\right\} =0,\label{eq:1210-01}
\end{equation}
where the above notation means that $\bs $ and $\bs'$ are sampled independently from the Gibbs measure $G_{N,\beta}$.

For any $\delta>0$ and $\bs_{0}\in\BN$, define 
\begin{align*}
\tilde{Z}_{N,\beta,\delta}^{\otimes2}(\bs_{0}) & =\int_{({\rm Band}(\bs_{0},\epsilon_{N}))^{2}}
  \mathbf{1}_{\{|R(\bs,\bs')-\qs^{2}|>\delta\}}e^{-\beta(H_{N}(\bs)+H_{N}(\bs'))}d\bs d\bs',
\end{align*}
where the integration is w.r.t. the product measure of the probability
Hausdorff measure on the sphere with itself. Note that the probability under the product Gibbs measure
in (\ref{eq:1210-01}) is equal to $\tilde{Z}_{N,\beta,\delta}^{\otimes2}(\bs_{0})/(Z_{N,\beta})^{2}$.

\corE{Denote $\Lambda_\beta:=\Lambda_{Z,\beta}(-E_{0}(\qs),\,\qs)$.} Assume that for some $C$ independent of $N$, $H_{N}(\bs)$ is $\sqrt{N}C$-Lipschitz
continuous on $\BN$. Then for $\eta>0$, since $\epsilon_{N}\to0$,
for large $N$, 
\begin{equation}
{\rm if}\ \ \frac{1}{N}\log Z_{N,\beta,\delta}^{\otimes2}(\bs_{0})<\corE{2\Lambda_\beta}-\eta, \ \ {\rm then}\ \ \frac{1}{N}\log\tilde{Z}_{N,\beta,\delta}^{\otimes2}(\bs_{0})<\corE{2\Lambda_\beta}-\eta/2,\label{eq:1110-01}
\end{equation}
where we define
\begin{align*}
Z_{N,\beta,\delta}^{\otimes2}(\bs_{0}) & =(1-\qs^{2})^{N}\int_{(\mathcal{S}(\bs_{0}))^{2}} \mathbf{1}_{\{|R(\bs,\bs')\mp\qs^{2}|>\delta\}}e^{-\beta(H_{N}(\bs)+H_{N}(\bs'))}d\bs d\bs',
\end{align*}
with the integration being w.r.t. the product of the probability Hausdorff
measure on $\mathcal{S}(\bs_{0})$ with itself, and where the term
$(1-\qs^{2})^{N}$ accounts for the volume of band (at exponential
level, for large $N$).

In light of the lower bound on the free energy (\ref{eq:1010-06}), the implication (\ref{eq:1110-01}) and the Lipschitz bound
of (\ref{eq:1507-03}), to prove \eqref{eq:1210-01} 
it will be enough to show that for any $\delta>0$,
\begin{equation}
\limsup_{N\to\infty}\frac{1}{N}\log\mathbb{E}\Big\{ \sum_{\bs_{0}\in\Cs^+}Z_{N,\beta,\delta}^{\otimes2}(\bs_{0})\Big\} <\corE{2\Lambda_\beta}.\label{eq:1010-07}
\end{equation}

Recall that $\Theta_{\nu,\qs}(-E_{0}(\qs),-x_{0}(\qs))=0$ and note
that $\Theta_{\nu,\qs}\left(E,x\right)$ is continuous. Thus, similarly
to Lemma \ref{lem:17}, to prove (\ref{eq:1010-07}) it is sufficient
to show that 
\begin{equation}
\limsup_{N\to\infty}\sup_{\substack{E\in B(\epsilon_{N})\\
x\in D(\epsilon_{N})
}
}\frac{1}{N}\log\mathbb{E}_{NE,\sqrt{N}x}^{\qs}\left\{ Z_{N,\beta,\delta}^{\otimes2}(\qs\hat{\mathbf{n}})\right\} <\corE{2\Lambda_\beta}.\label{eq:1110-03}
\end{equation}
From Lemma \ref{cor:conditional} and (\ref{eq:1110-02}), the
left-hand side of (\ref{eq:1110-03}) is equal to 
\begin{equation}
\log(1-\qs^{2})+2\beta E_{0}(\qs)+\sup_{|\varrho|\in\big(\frac{\delta}{1-\qs^{2}},1\big)}\Big\{ \frac{1}{2}\log\left(1-\varrho^{2}\right)+\beta^{2}\sum_{k=2}^{\infty}\alpha_{k}^{2}(\qs)(1+\varrho^{k})\Big\} ,\label{eq:1110-04}
\end{equation}
where we used the fact that for points in $\mathcal{S}(\qs\hat{\mathbf{n}})$,
$|R(\bs,\bs')|>\delta$ if and only if
\[
|R(\bs-\qs\hat{\mathbf{n}},\bs'-\qs\hat{\mathbf{n}})|>\frac{\delta}{1-\qs^{2}}.
\]

The right-hand side of (\ref{eq:1110-03}) is equal to the expression
in (\ref{eq:1110-04}) with $\varrho=0$ (and no supremum). The inequality
(\ref{eq:1110-03}) therefore follows from a similar analysis to that
following (\ref{eq:1110-02}). This proves the first of the two limits
in Part \ref{enu:Geometry2} of Theorem \ref{thm:geometry1}.

If $\nu$ is neither an odd nor even polynomial, then almost surely
$H_{N}(\bs)$ has no antipodal critical points on $\mathbb{S}^{N-1}(\qs\sqrt{N})$. \corE{(This can be verified by applying the Kac-Rice formula \cite[Theorem 12.1.1]{RFG} to compute the expected number of pairs $(\bs_1,\bs_2)\in(\mathbb{S}^{N-1}(\qs\sqrt{N}))^2$ of critical points with overlap $|R(\bs_1,\bs_2)+1|<\epsilon$, and  taking $\epsilon\to 0 $.)}
If $\nu$ is odd, then for any $\bs_{0}\in\Cs^+$,
$-\bs_{0}$ is also a critical point, but since $H_{N}(-\bs_{0})=-H_{N}(\bs_{0})$,
$-\bs_{0}\notin\Cs^+$. Lastly, if $\nu$ is even,
then for any $\bs_{0}\in\Cs^+$ , $H_{N}(-\bs_{0})=H_{N}(\bs_{0})$
and $-\bs_{0}\in\Cs^+$ .

Hence, we only need to prove the second limit of Part \ref{enu:Geometry2}
in the case where $\nu$ is even. This case, however, follows directly
from the symmetry $H_{N}(\bs)=H_{N}(-\bs)$.
\qed 

\subsection{Proof of Theorem \ref{thm:geometry1}, Part \ref{enu:Geometry3}}
\corE{The key element in the current proof is combining Part  \ref{enu:Geometry2} of Theorem \ref{thm:geometry1} with the following corollary, which is a direct conclusion of Lemma 11 in \cite{geometryGibbs} and the argument used in the proof of Theorem 3 in \cite{geometryGibbs}, both of which rely on basic linear algebra and do not involve probabilistic arguments.
\begin{cor}\label{cor:overlapcrt}
	For some function $\rho(\epsilon)>0$ satisfying $\lim_{\epsilon\to0}\rho(\epsilon)=0$ we have the following for every $N$. For $i=1,2$, let $\epsilon>0$,  $q_i\in(0,1)$ and $\bs_0^i\in \mathbb S^{N-1}(\sqrt N q_i)$. If $M_i$ is a measure supported on ${\rm Band}(\bs_{0}^i,\epsilon)$ such that
	\[
	M_i\times M_i \{  |R(\bs,\bs')-q_i^2|>\epsilon \}<\epsilon,
	\]
	then
	\[
	M_1 \times M_2 \{ |R(\bs,\bs')-q_1q_2R(\bs_0^1,\bs_0^2)| > \rho(\epsilon)\}<\rho(\epsilon).
	\]
\end{cor}
}
Denote
by $\Cs^{\delta}$ the set of points $\bs_{0}\in\Cs$ for which 
\[
\frac{1}{N}\log Z_{N,\beta}({\rm Band}(\bs_{0},\epsilon_{N}))>\Lambda_{Z,\beta}(-E_{0}(\qs),\,\qs)-\delta.
\]
Since by \eqref{eq-subsexp} the number of points in $\Cs$ is sub-exponential,  from the lower bound \corE{on the free energy in}  Proposition \ref{prop:Flb}, if 
 $\delta_{N}=o(1)$ decays slow enough, then 
\[
\lim_{N\to\infty}\mathbb{E}G_{N,\beta}(\cup_{\bs_{0}\in\Cs^{\delta_{N}}}{\rm Band}(\bs_{0},\epsilon_{N}))=1.
\]

From Corollary \ref{cor:orth}, for large $\beta$ (and therefore
 $\qs$ \corE{close to $1$}), the bands corresponding to different points in $\Cs$
are disjoint, with probability tending to $1$ as $N\to\infty$. 
Therefore, to prove (\ref{eq:1402}) it will be enough to show that 
\begin{equation}
\lim_{N\to\infty} \mathbb P\Big\{ \forall \bs_0\neq\pm\bs_0'\in \Cs^{\delta_N}:\
G_{N,\beta}^{\bs_{0}}\times G_{N,\beta}^{\bs_{0}'}\left\{ \left|R\left(\boldsymbol{\sigma},\boldsymbol{\sigma}^{\prime}\right)\right|>\delta\right\} <\rho_{N}
\Big\}=1,\label{eq:2210-05}
\end{equation}
for some $\rho_N=o(1)$, where we denote by $G_{N,\beta}^{\bs_{0}}$
the conditional Gibbs measure given ${\rm Band}(\bs_{0},\epsilon_{N})$.

Note that from (\ref{eq:1010-07}), \eqref{eq:1110-01} and Corollary \ref{cor:orth}, if $\epsilon_{N}'=o(1)$ decays sufficiently slow, then with probability tending to $1$: 
uniformly in $\bs_{0}\in\Cs^{\delta_{N}}$,
\begin{equation}
G_{N,\beta}^{\bs_{0}}\times G_{N,\beta}^{\bs_{0}}\left\{ \left|R\left(\boldsymbol{\sigma},\boldsymbol{\sigma}^{\prime}\right)-\qs^{2}\right|>\epsilon_{N}'\right\} <\epsilon_{N}',\label{eq:3011-01}
\end{equation}
and  uniformly in $\bs_{0},\bs_{0}'\in\Cs^{\delta_{N}}$
with $\bs_{0}\neq\pm\bs_{0}'$, 
$
|R(\bs_{0},\bs_{0}')|<\epsilon_{N}'
$.
\corE{Combined with Corollary \ref{cor:overlapcrt}, this implies \eqref{eq:2210-05} and completes the proof.}\qed

\corE{
\subsection{Proof of Theorem \ref{thm:Chaos}}
Let $\beta\neq\beta'$ be two different inverse-temperatures and define  $\qs=\qs(\beta)\neq\qs'=\qs(\beta')$ by \eqref{eq:qs}.
Note that almost surely there are no pairs of critical points $(\bs_1,\bs_2)\in\mathbb{S}^{N-1}(\qs\sqrt{N})\times \mathbb{S}^{N-1}(\qs'\sqrt{N})$ such that $\bs_1=\pm \bs_2$. This can be verified e.g. by applying the Kac-Rice formula \cite[Theorem 12.1.1]{RFG} to compute the expected number of pairs $(\bs_1,\bs_2)\in\mathbb{S}^{N-1}(\qs\sqrt{N})\times \mathbb{S}^{N-1}(\qs'\sqrt{N})$ of critical points with overlap $|R(\bs_1,\bs_2)\pm 1|<\epsilon$, and  taking $\epsilon\to 0 $. (We emphasize, however, that for pure models, for any $\qs$-critical point $\bs_1$, the point $\bs_2=\bs_1\cdot\qs'/\qs$ is  a $\qs'$-critical point deterministically, in which case the Kac-Rice formula cannot be applied due to the degeneracy of the
covariance matrix.)}

\corE{From the  argument that precedes \eqref{eq:2210-05} (which is based on Corollary \ref{cor:orth}, valid also for $q_1\neq q_2$), to complete the proof of Theorem \ref{thm:Chaos} it will be enough to show that for some $\rho_N=o(1)$,
\begin{equation}
\lim_{N\to\infty} \mathbb P\Big\{ \forall \bs_0\in \Cs^{\delta_N},\,\bs_0'\in \Cs^{\prime,\delta_N},\,\bs_0\neq\pm\bs_0':\
G_{N,\beta}^{\bs_{0}}\times G_{N,\beta'}^{\bs_{0}'}\left\{ \left|R\left(\boldsymbol{\sigma},\boldsymbol{\sigma}^{\prime}\right)\right|>\delta\right\} <\rho_{N}
\Big\}=1,\label{eq:2210-051}
\end{equation}
where $\bs_0\in \Cs^{\delta_N}$ and $G_{N,\beta}^{\bs_{0}}$ are as in the proof  of  Part \ref{enu:Geometry3} of Theorem \ref{thm:geometry1}, and $\bs_0\in \Cs^{\prime,\delta_N}$ and $G_{N,\beta'}^{\bs_{0}'}$ are defined similarly.
Similarly to \eqref{eq:3011-01} we have that uniformly in $\bs_{0}'\in\Cs^{\prime,\delta_{N}}$,
\begin{equation}
G_{N,\beta'}^{\bs_{0}'}\times G_{N,\beta'}^{\bs_{0}'}\left\{ \left|R\left(\boldsymbol{\sigma},\boldsymbol{\sigma}^{\prime}\right)-\qs'^{2}\right|>\epsilon_{N}'\right\} <\epsilon_{N}',\nonumber
\end{equation}
and  uniformly in $ \bs_0\in \Cs^{\delta_N}$, $\bs_0'\in \Cs^{\prime,\delta_N}$   with $\bs_0\neq\pm\bs_0'$, $|R(\bs_{0},\bs_{0}')|<\epsilon_{N}'$
with probability tending to $1$, for $\epsilon_N'$ decaying slowly enough.
Combining the above with Corollary \ref{cor:overlapcrt}, the proof is completed.\qed} 
\appendix

\section{\label{sec:Covariances}Covariances }

In this appendix we prove Lemmas \ref{lem:dens}, \ref{lem:condHamiltonian}
and \ref{lem:Hess_struct_2}. We begin with a study the joint covariance
of
\[
H_{N}\left(q\bs\right),\,\grad H_{N}\left(q\bs\right),\,\Hess H_{N}\left(q\bs\right),\,\ddq H_{N}\left(q\bs\right)
\]
at two points of the form $q\bs=q_{1}\hat{\mathbf{n}}$ and $q\bs=q_{2}\boldsymbol{\sigma}\left(r\right)=q_{2}\sqrt{N}\left(0,...,0,\sqrt{1-r^{2}},r\right)$. 

With the usual notation 
\[
\delta_{ij}=\begin{cases}
1 & \mbox{ if }i=j,\\
0 & \mbox{ otherwise},
\end{cases}
\]
in the lemma below we denote $\delta_{i=j}=\delta_{ij}$, $\delta_{i=j=k}=\delta_{ij}\delta_{jk}$,
$\delta_{i=j\neq k}=\delta_{ij}\left(1-\delta_{jk}\right)$, etc.
\begin{lem}
\label{lem:cov}For any $r\in\left[-1,1\right]$ and $q_{1},q_{2}\in(0,1]$
there exists a frame field $F=\left(F_{i}\right)$ (orthonormal when
restricted to any sphere centered at the origin) satisfying
\begin{equation}
\begin{aligned}
F_i H_{N}\left(q_1\hat{\mathbf{n}}\right) & =\left.\frac{d}{dx_{i}}\right|_{\bx=0}H_{N}\left((x_{1},...,x_{N-1},q_1\sqrt{N-\|\bx\|^{2}}\right),\\
F_iF_j H_{N}\left(q_1\hat{\mathbf{n}}\right) & =\left.\frac{d}{dx_{i}}\frac{d}{dx_{j}}\right|_{\bx=0}H_{N}\left((x_{1},...,x_{N-1},q_1\sqrt{N-\|\bx\|^{2}}\right),
\end{aligned}
\label{eq:1511---}
\end{equation}
such that 
\begin{align*}
\frac{1}{N}\mbox{\ensuremath{\mathbb{E}}}\left\{ H_{N}\left(q_{1}\hat{\mathbf{n}}\right)H_{N}\left(q_{2}\boldsymbol{\sigma}\left(r\right)\right)\right\}  & =\nu\left(q_{1}q_{2}r\right),\\
\frac{1}{\sqrt{N}}\mbox{\ensuremath{\mathbb{E}}}\left\{ H_{N}\left(q_{1}\hat{\mathbf{n}}\right)\ddq H_{N}\left(q_{2}\boldsymbol{\sigma}\left(r\right)\right)\right\}  & =\frac{1}{\sqrt{N}}\mbox{\ensuremath{\mathbb{E}}}\left\{ F_{N}H_{N}\left(q_{2}\hat{\mathbf{n}}\right)H_{N}\left(q_{1}\boldsymbol{\sigma}\left(r\right)\right)\right\}=q_{1}r\nu'\left(q_{1}q_{2}r\right), \\
\frac{1}{\sqrt{N}}\mbox{\ensuremath{\mathbb{E}}}\left\{ H_{N}\left(q_{1}\hat{\mathbf{n}}\right)F_{l}H_{N}\left(q_{2}\boldsymbol{\sigma}\left(r\right)\right)\right\}  & =-\frac{1}{\sqrt{N}}\mbox{\ensuremath{\mathbb{E}}}\left\{ F_{l}H_{N}\left(q_{2}\hat{\mathbf{n}}\right)H_{N}\left(q_{1}\boldsymbol{\sigma}\left(r\right)\right)\right\} =-q_{1}\nu'\left(q_{1}q_{2}r\right)\left(1-r^{2}\right)^{1/2}\delta_{l=N-1},\\
\mbox{\ensuremath{\mathbb{E}}}\left\{ H_{N}\left(q_{1}\hat{\mathbf{n}}\right)F_{k}F_{l}H_{N}\left(q_{2}\boldsymbol{\sigma}\left(r\right)\right)\right\}  & =\mbox{\ensuremath{\mathbb{E}}}\left\{ F_{k}F_{l}H_{N}\left(q_{2}\hat{\mathbf{n}}\right)H_{N}\left(q_{1}\boldsymbol{\sigma}\left(r\right)\right)\right\} \\
 & =q_{1}^{2}\nu''\left(q_{1}q_{2}r\right)\left(1-r^{2}\right)\delta_{l=k=N-1}-\frac{q_{1}}{q_{2}}r\nu'\left(q_{1}q_{2}r\right)\delta_{k=l},
\end{align*}
\begin{align*}
\mbox{\ensuremath{\mathbb{E}}}\left\{ \ddq H_{N}\left(q_{1}\hat{\mathbf{n}}\right)\ddq H_{N}\left(q_{2}\boldsymbol{\sigma}\left(r\right)\right)\right\}  & =q_{1}q_{2}r^{2}\nu''\left(q_{1}q_{2}r\right)+r\nu'\left(q_{1}q_{2}r\right),\\
\mbox{\ensuremath{\mathbb{E}}}\left\{ \ddq H_{N}\left(q_{1}\hat{\mathbf{n}}\right)F_{l}H_{N}\left(q_{2}\boldsymbol{\sigma}\left(r\right)\right)\right\}  & =-\mbox{\ensuremath{\mathbb{E}}}\left\{ F_{l}H_{N}\left(q_{2}\hat{\mathbf{n}}\right)\ddq H_{N}\left(q_{1}\boldsymbol{\sigma}\left(r\right)\right)\right\} \\
 & =-\left(q_{1}q_{2}r\nu''\left(q_{1}q_{2}r\right)+\nu'\left(q_{1}q_{2}r\right)\right)\left(1-r^{2}\right)^{1/2}\delta_{l=N-1},\\
\sqrt{N}\mbox{\ensuremath{\mathbb{E}}}\left\{ \ddq H_{N}\left(q_{1}\hat{\mathbf{n}}\right)F_{k}F_{l}H_{N}\left(q_{2}\boldsymbol{\sigma}\left(r\right)\right)\right\}  & =\sqrt{N}\mbox{\ensuremath{\mathbb{E}}}\left\{ F_{k}F_{l}H_{N}\left(q_{2}\hat{\mathbf{n}}\right)\ddq H_{N}\left(q_{1}\boldsymbol{\sigma}\left(r\right)\right)\right\} \\
 & =\left(q_{1}^{2}q_{2}r\nu'''\left(q_{1}q_{2}r\right)+2q_{1}\nu''\left(q_{1}q_{2}r\right)\right)\left(1-r^{2}\right)\delta_{l=k=N-1}\\
 & -\left(q_{1}r^{2}\nu''\left(q_{1}q_{2}r\right)+\frac{r}{q_{2}}\nu'\left(q_{1}q_{2}r\right)\right)\delta_{k=l},
\end{align*}
\begin{align*}
\mbox{\ensuremath{\mathbb{E}}}\left\{ F_{j}H_{N}\left(q_{1}\hat{\mathbf{n}}\right)F_{l}H_{N}\left(q_{2}\boldsymbol{\sigma}\left(r\right)\right)\right\}  & =\left(r\nu'\left(q_{1}q_{2}r\right)-q_{1}q_{2}\nu''\left(q_{1}q_{2}r\right)\left(1-r^{2}\right)\right)\delta_{l=j=N-1}\\
 & +\nu'\left(q_{1}q_{2}r\right)\delta_{l=j\neq N-1},\\
\sqrt{N}\mbox{\ensuremath{\mathbb{E}}}\left\{ F_{j}H_{N}\left(q_{1}\hat{\mathbf{n}}\right)F_{k}F_{l}H_{N}\left(q_{2}\boldsymbol{\sigma}\left(r\right)\right)\right\}  & =-\sqrt{N}\mbox{\ensuremath{\mathbb{E}}}\left\{ F_{k}F_{l}H_{N}\left(q_{2}\hat{\mathbf{n}}\right)F_{j}H_{N}\left(q_{1}\boldsymbol{\sigma}\left(r\right)\right)\right\} \\
 & =q_{1}^{2}q_{2}\nu'''\left(q_{1}q_{2}r\right)\left(1-r^{2}\right)^{3/2}\delta_{j=k=l=N-1}\\
 & -q_{1}\nu''\left(q_{1}q_{2}r\right)\left(1-r^{2}\right)^{1/2}\left[\delta_{j=k\neq N-1}\delta_{l=N-1}+\delta_{j=l\neq N-1}\delta_{k=N-1}+2r\delta_{j=k=l=N-1}\right]\\
 & -\left(rq_{1}\nu''\left(q_{1}q_{2}r\right)+q_{2}^{-1}\nu'\left(q_{1}q_{2}r\right)\right)\left(1-r^{2}\right)^{1/2}\delta_{k=l}\delta_{j=N-1},
\end{align*}
\begin{align*}
N\mbox{\ensuremath{\mathbb{E}}}\left\{ F_{i}F_{j}H_{N}\left(q_{1}\hat{\mathbf{n}}\right)F_{k}F_{l}H_{N}\left(q_{2}\boldsymbol{\sigma}\left(r\right)\right)\right\}  & =q_{1}^{2}q_{2}^{2}\nu''''\left(q_{1}q_{2}r\right)\left(1-r^{2}\right)^{2}\delta_{i=j=k=l=N-1}\\
 & -q_{1}q_{2}\nu'''\left(q_{1}q_{2}r\right)\left(1-r^{2}\right)\left[6r\delta_{i=j=k=l=N-1}+r\delta_{i=j\neq N-1}\delta_{k=l=N-1}\right.\\
 & +r\delta_{i=j=N-1}\delta_{k=l\neq N-1}+\delta_{i=k\neq N-1}\delta_{j=l=N-1}+\delta_{k=i=N-1}\delta_{j=l\neq N-1}\\
 & \left.+\delta_{j=k\neq N-1}\delta_{i=l=N-1}+\delta_{k=j=N-1}\delta_{i=l\neq N-1}\right]\\
 & +\nu''\left(q_{1}q_{2}r\right)\left[r^{2}\delta_{i=j}\delta_{k=l}-2\left(1-r^{2}\right)\left(\delta_{i=j=N-1}\delta_{k=l}+\delta_{i=j}\delta_{l=k=N-1}\right)\right.\\
 & +\left(\delta_{j=l\neq N-1}+r\delta_{j=l=N-1}\right)\left(\delta_{i=k\neq N-1}+r\delta_{i=k=N-1}\right)\\
 & \left.+\left(\delta_{i=l\neq N-1}+r\delta_{i=l=N-1}\right)\left(\delta_{j=k\neq N-1}+r\delta_{j=k=N-1}\right)\right]\\
 & +rq_{1}^{-1}q_{2}^{-1}\nu'\left(q_{1}q_{2}r\right)\delta_{i=j}\delta_{k=l}.
\end{align*}
\end{lem}

\begin{proof}
The lemma follows by straightforward algebra from the relations
\begin{align*}
H_{N}\left(q\bs\right) & =\sum\gamma_{p}q^{p}H_{N,p}\left(\bs\right),\,\,\grad H_{N}\left(q\bs\right)=\sum\gamma_{p}q^{p-1}\grad H_{N,p}\left(\bs\right),\\
\Hess H_{N}\left(q\bs\right) & =\sum\gamma_{p}q^{p-2}\Hess H_{N,p}\left(\bs\right),\,\,\ddq H_{N}\left(q\bs\right)=\frac{1}{\sqrt{N}}\sum\gamma_{p}pq^{p-1}H_{N,p}\left(\bs\right),
\end{align*}
and the covariance computations of \cite[Lemma 30]{2nd} which dealt
with the pure case.
\end{proof}
For any $r\in\left(-1,1\right)$ and $q_{1},q_{2}\in(0,1]$ define
\[
\begin{array}{ll}
a_{1}\left(r,q_{1},q_{2}\right)=\frac{\nu'(q_{2}^{2})}{\nu'(q_{1}^{2})\nu'(q_{2}^{2})-\left(\nu'(q_{1}q_{2}r)\right)^{2}}, & a_{2}\left(r,q_{1},q_{2}\right)=\frac{\nu'(q_{2}^{2})}{\nu'(q_{1}^{2})\nu'(q_{2}^{2})-\left(r\nu'(q_{1}q_{2}r)-q_{1}q_{2}\nu''(q_{1}q_{2}r)(1-r^{2})\right)^{2}},\\
a_{3}\left(r,q_{1},q_{2}\right)=\frac{-\nu'(q_{1}q_{2}r)}{\nu'(q_{1}^{2})\nu'(q_{2}^{2})-\left(\nu'(q_{1}q_{2}r)\right)^{2}}, & a_{4}\left(r,q_{1},q_{2}\right)=\frac{-\left(r\nu'(q_{1}q_{2}r)-q_{1}q_{2}\nu''(q_{1}q_{2}r)(1-r^{2})\right)}{\nu'(q_{1}^{2})\nu'(q_{2}^{2})-\left(r\nu'(q_{1}q_{2}r)-q_{1}q_{2}\nu''(q_{1}q_{2}r)(1-r^{2})\right)^{2}},\\
\upsilon_{1}\left(r,q_{1},q_{2}\right)=q_{1}q_{2}r\nu''\left(q_{1}q_{2}r\right)+\nu'\left(q_{1}q_{2}r\right), & \upsilon_{2}\left(r,q_{1},q_{2}\right)=-q_{1}q_{2}^{2}\nu'''\left(q_{1}q_{2}r\right)\left(1-r^{2}\right)+2rq_{2}\nu''\left(q_{1}q_{2}r\right),\\
\upsilon_{3}\left(r,q_{1},q_{2}\right)=q_{1}^{2}q_{2}r\nu'''\left(q_{1}q_{2}r\right)+2q_{1}\nu''\left(q_{1}q_{2}r\right),\\
b_{1}\left(r\right)=-\nu'\left(1\right) +a_{2}\left(r\right)\left(1-r^{2}\right)\nu'\left(r\right)\upsilon_{1}\left(r\right),
& b_{2}\left(r\right)=-r\nu'\left(r\right)-a_{4}\left(r\right)\left(1-r^{2}\right)\nu'\left(r\right)\upsilon_{1}\left(r\right),\\
b_{3}\left(r\right)=a_{2}\left(r\right)\left(1-r^{2}\right)\nu'\left(r\right)\upsilon_{2}\left(r\right), 
& b_{4}\left(r\right)=\nu''\left(r\right)\left(1-r^{2}\right)
-a_{4}\left(r\right)\left(1-r^{2}\right)\nu'\left(r\right)\upsilon_{2}\left(r\right)
\end{array}
\]

For any of $T=U,\,X,\,b,\,Z,\,Q$, define the matrix $\Sigma_{T}\left(r,q_{1},q_{2}\right)=\left(\Sigma_{T,ij}\left(r,q_{1},q_{2}\right)\right)_{i,j=1}^{2,2}$
by the following
\begin{align}
\Sigma_{U,11}\left(r,q_{1},q_{2}\right) & =\Sigma_{U,22}\left(r,q_{2},q_{1}\right)=\nu(q_{1}^{2})-q_{1}^{2}a_{2}\left(r,q_{2},q_{1}\right)\left(\nu'(q_{1}q_{2}r)\right)^{2}\left(1-r^{2}\right),\label{eq:26}\\
\Sigma_{U,12}\left(r,q_{1},q_{2}\right) & =\Sigma_{U,21}\left(r,q_{1},q_{2}\right)=\nu\left(q_{1}q_{2}r\right)+q_{1}q_{2}a_{4}\left(r,q_{1},q_{2}\right)\left(\nu'(q_{1}q_{2}r)\right)^{2}\left(1-r^{2}\right),\nonumber 
\end{align}
\begin{align}
  \label{eq:SigmaX}
 & \Sigma_{X,11}\left(r,q_{1},q_{2}\right)=\Sigma_{X,22}\left(r,q_{2},q_{1}\right)=q_{1}^{2}\nu''\left(q_{1}^{2}\right)+\nu'\left(q_{1}^{2}\right)-\left(\upsilon_{1}\left(r,q_{1},q_{2}\right)\right)^{2}\left(1-r^{2}\right)a_{2}\left(r,q_{2},q_{1}\right), \\
 & \Sigma_{X,12}\left(r,q_{2},q_{1}\right)=\Sigma_{X,21}\left(r,q_{1},q_{2}\right)
 =q_{1}q_{2}r^{2}\nu''\left(q_{1}q_{2}r\right)+r\nu'\left(q_{1}q_{2}r\right)+\left(\upsilon_{1}\left(r,q_{1},q_{2}\right)\right)^{2}\left(1-r^{2}\right)a_{4}\left(r,q_{1},q_{2}\right), \nonumber
\end{align}
\begin{align}
\Sigma_{b,11}\left(r,q_{1},q_{2}\right) & =\Sigma_{b,22}\left(r,q_{2},q_{1}\right)=q_{1}\nu'\left(q_{1}^{2}\right)-q_{1}\left(1-r^{2}\right)\nu'\left(q_{1}q_{2}r\right)\upsilon_{1}\left(r,q_{1},q_{2}\right)a_{2}\left(r,q_{2},q_{1}\right),
 \label{eq:Sigmab}\\
\Sigma_{b,12}\left(r,q_{1},q_{2}\right) & =\Sigma_{b,21}\left(r,q_{2},q_{1}\right)
  =q_{1}r\nu'\left(q_{1}q_{2}r\right)+q_{1}\left(1-r^{2}\right)\nu'\left(q_{1}q_{2}r\right)\upsilon_{1}\left(r,q_{1},q_{2}\right)a_{4}\left(r,q_{1},q_{2}\right),\nonumber 
\end{align}
\begin{align}
\Sigma_{Z,11}\left(r,q_{1},q_{2}\right) & =\Sigma_{Z,22}\left(r,q_{2},q_{1}\right)=1-\frac{1}{\nu''(q_{1}^{2})}q_{2}^{2}\left(1-r^{2}\right)\left(\nu''\left(q_{1}q_{2}r\right)\right)^{2}a_{1}(r,q_{2},q_{1}),\label{eq:SigmaZ}\\
\Sigma_{Z,12}\left(r,q_{1},q_{2}\right) & =\Sigma_{Z,21}\left(r,q_{1},q_{2}\right)=-\frac{1}{\sqrt{\nu''(q_{1}^{2})\nu''(q_{2}^{2})}}\left[q_{1}q_{2}\left(1-r^{2}\right)\nu'''\left(q_{1}q_{2}r\right)+r\nu''\left(q_{1}q_{2}r\right)\right.
\nonumber \\
 & \left.\hspace{2.5in}+q_{1}q_{2}\left(1-r^{2}\right)\left(\nu''\left(q_{1}q_{2}r\right)\right)^{2}a_{3}\left(r,q_{1},q_{2}\right)\right],\nonumber 
\end{align}
and 
\begin{align}
\Sigma_{Q,11}\left(r,q_{1},q_{2}\right) & =2-\frac{\left(1-r^{2}\right)}{\nu''(q_{1}^{2})}a_{2}\left(r,q_{2},q_{1}\right)\left(\upsilon_{2}\left(r,q_{1},q_{2}\right)\right)^{2}
-\left(1-r^{2}\right)^{2}\left(\boldsymbol{\varsigma}_{1}\left(r,q_{1},q_{2}\right)\right)^{T}\Sigma_{U,X}^{-1}\left(r,q_{1},q_{2}\right)\boldsymbol{\varsigma}_{1}\left(r,q_{1},q_{2}\right),\label{eq:SigmaQ} \\
\Sigma_{Q,22}\left(r,q_{1},q_{2}\right) & =2-\frac{\left(1-r^{2}\right)}{\nu''(q_{2}^{2})}a_{2}\left(r,q_{1},q_{2}\right)\left(\upsilon_{2}\left(r,q_{2},q_{1}\right)\right)^{2}
-\left(1-r^{2}\right)^{2}\left(\boldsymbol{\varsigma}_{2}\left(r,q_{1},q_{2}\right)\right)^{T}\Sigma_{U,X}^{-1}\left(r,q_{1},q_{2}\right)\boldsymbol{\varsigma}_{2}\left(r,q_{1},q_{2}\right),\nonumber\\
\Sigma_{Q,12}\left(r,q_{1},q_{2}\right) & =\Sigma_{Q,21}\left(r,q_{1},q_{2}\right)=\frac{1}{\sqrt{\nu''(q_{1}^{2})\nu''(q_{2}^{2})}}
\Big[q_{1}q_{2}\nu''''\left(q_{1}q_{2}r\right)\left(1-r^{2}\right)^{2}-2q_{1}q_{2}\left(1-r^{2}\right)r\nu'''\left(q_{1}q_{2}r\right)\nonumber \\
&\hspace{1.9in} +2r^{2}\left(1-r^{2}\right)\nu''\left(q_{1}q_{2}r\right)+\upsilon_{2}\left(r,q_{1},q_{2}\right)\upsilon_{2}\left(r,q_{2},q_{1}\right)a_{4}\left(r,q_{1},q_{2}\right)\Big]\nonumber \\
& \hspace{1.3in}
-\left(1-r^{2}\right)^{2}\left(\boldsymbol{\varsigma}_{1}\left(r,q_{1},q_{2}\right)\right)^{T}\Sigma_{U,X}^{-1}\left(r,q_{1},q_{2}\right)\boldsymbol{\varsigma}_{2}\left(r,q_{1},q_{2}\right),\nonumber 
\end{align}
where
\begin{equation}
\boldsymbol{\varsigma}_{1}\left(r,q_{1},q_{2}\right)=\frac{1}{\sqrt{\nu''(q_{1}^{2})}}\left(\begin{array}{c}
q_{1}\nu'\left(q_{1}q_{2}r\right)\upsilon_{2}\left(r,q_{1},q_{2}\right)a_{2}\left(r,q_{2},q_{1}\right)\\
q_{2}^{2}\nu''\left(q_{1}q_{2}r\right)-q_{2}\nu'\left(q_{1}q_{2}r\right)\upsilon_{2}\left(r,q_{1},q_{2}\right)a_{4}\left(r,q_{1},q_{2}\right)\\
\upsilon_{1}\left(r,q_{1},q_{2}\right)\upsilon_{2}\left(r,q_{1},q_{2}\right)a_{2}\left(r,q_{2},q_{1}\right)\\
\upsilon_{3}\left(r,q_{2},q_{1}\right)-\upsilon_{1}\left(r,q_{1},q_{2}\right)\upsilon_{2}\left(r,q_{1},q_{2}\right)a_{4}\left(r,q_{1},q_{2}\right)
\end{array}\right),\label{eq:varsigma}
\end{equation}
\[
\boldsymbol{\varsigma}_{2}\left(r,q_{1},q_{2}\right)=\frac{1}{\sqrt{\nu''(q_{2}^{2})}}\left(\begin{array}{c}
q_{1}^{2}\nu''\left(q_{1}q_{2}r\right)-q_{1}\nu'\left(q_{1}q_{2}r\right)\upsilon_{2}\left(r,q_{2},q_{1}\right)a_{4}\left(r,q_{1},q_{2}\right)\\
q_{2}\nu'\left(q_{1}q_{2}r\right)\upsilon_{2}\left(r,q_{2},q_{1}\right)a_{2}\left(r,q_{1},q_{2}\right)\\
\upsilon_{3}\left(r,q_{1},q_{2}\right)-\upsilon_{1}\left(r,q_{1},q_{2}\right)\upsilon_{2}\left(r,q_{2},q_{1}\right)a_{4}\left(r,q_{1},q_{2}\right)\\
\upsilon_{1}\left(r,q_{1},q_{2}\right)\upsilon_{2}\left(r,q_{2},q_{1}\right)\left(1-r^{2}\right)a_{2}\left(r,q_{1},q_{2}\right)
\end{array}\right),
\]
\begin{equation}
\Sigma_{U,X}\left(r,q_{1},q_{2}\right)=\left(\begin{array}{cc}
\Sigma_{U}\left(r,q_{1},q_{2}\right) & \Sigma_{b}\left(r,q_{1},q_{2}\right)\\
\Sigma_{b}^{T}\left(r,q_{1},q_{2}\right) & \Sigma_{X}\left(r,q_{1},q_{2}\right)
\end{array}\right).\label{eq:SigmaUXbar}
\end{equation}
\begin{rem}
\label{rem:a}For $r\in(-1,1)$ and $q_{1},q_{2}\in(0,1]$, we note
that 
\begin{align*}
\nu'(q_{1}^{2})\nu'(q_{2}^{2}) & =\sum_{p\geq m}\frac{\gamma_{p}^{2}\gamma_{m}^{2}pm}{1+\delta_{m=p}}\left(q_{1}^{2(p-1)}q_{2}^{2(m-1)}+q_{1}^{2(m-1)}q_{2}^{2(p-1)}\right),\\
\nu'(q_{1}q_{2}r)^{2} & =\sum_{p\geq m}\frac{\gamma_{p}^{2}\gamma_{m}^{2}pm}{1+\delta_{m=p}}2(q_{1}q_{2}r)^{p-1+m-1},
\end{align*}
and thus by comparing summand by summand, $\nu'(q_{1}^{2})\nu'(q_{2}^{2})>\nu'(q_{1}q_{2}r)^{2}$;
and that
\begin{equation}
  r\nu'(q_{1}q_{2}r)-q_{1}q_{2}\nu''(q_{1}q_{2}r)(1-r^{2})\label{eq:0902-01}\\
  =\sum_{p,m}\gamma_{p}^{2}\gamma_{m}^{2}pm\left(pr^{p}-(p-1)r^{p-2}\right)\left(mr^{m}-(m-1)r^{m-2}\right), 
\end{equation}
which, since $|pr^{p}-(p-1)r^{p-2}|<1$ for any $p\geq2$, similarly
implies that (\ref{eq:0902-01}) is strictly smaller, in absolute
value, than $\nu'(q_{1}^{2})\nu'(q_{2}^{2})$. Therefore, the denominators
in the definitions of $a_{i}\left(r,q_{1},q_{2}\right)$ are positive
for any $r\in\left(-1,1\right)$.
\end{rem}
%

\begin{proof}[Proof of Lemmas \ref{lem:dens}, \ref{lem:condHamiltonian}
and \ref{lem:Hess_struct_2}]
The proof closely follows that of \cite[Lemmas 12 and 13]{2nd}.
Fix $r\in\left(-1,1\right)$ and $q_{1},q_{2}\in(0,1]$. Assume all
vectors in the proof are column vectors and denote the concatenation
of any two vectors $v_{1}$, $v_{2}$ by $\left(v_{1};v_{2}\right)$.
The covariance matrix of the vector $\left(\grad H_{N}\left(q_{1}\hat{\mathbf{n}}\right);\grad H_{N}\left(q_{2}\boldsymbol{\sigma}\left(r\right)\right)\right)$
can be extracted from Lemma \ref{lem:cov}. By computation, one has
that (\ref{eq:grad_dens}) holds and that the inverse of the latter
covariance matrix is the block matrix
\[
\mathbf{A}_{0}\left(r\right)=\left(\begin{array}{cccc}
a_{1}\left(r,q_{1},q_{2}\right)\mathbf{I}_{N-1} & 0 & a_{3}\left(r,q_{1},q_{2}\right)\mathbf{I}_{N-1} & 0\\
0 & a_{2}\left(r,q_{1},q_{2}\right) & 0 & a_{4}\left(r,q_{1},q_{2}\right)\\
a_{3}\left(r,q_{1},q_{2}\right)\mathbf{I}_{N-1} & 0 & a_{1}\left(r,q_{2},q_{1}\right)\mathbf{I}_{N-1} & 0\\
0 & a_{4}\left(r,q_{1},q_{2}\right) & 0 & a_{2}\left(r,q_{2},q_{1}\right)
\end{array}\right),
\]
where $\mathbf{I}_{N-1}$ is the $N-1\times N-1$ identity matrix.

For any random vector $V$ let $\mathbb{E}V$ denote the corresponding
vector of expectations. From Lemma \ref{lem:cov}, denoting by $e_{i}$
the $1\times\left(2N-2\right)$ vector with the $i$-th entry equal
to $1$ and all others equal to $0$, we obtain
\begin{align*}
\frac{1}{\sqrt{N}}\mbox{\ensuremath{\mathbb{E}}}\left\{ H_{N}\left(q_{1}\hat{\mathbf{n}}\right)\cdot\left(\grad H_{N}\left(q_{1}\hat{\mathbf{n}}\right);\grad H_{N}\left(q_{2}\boldsymbol{\sigma}\left(r\right)\right)\right)\right\}  & =-q_{1}\nu'\left(q_{1}q_{2}r\right)\left(1-r^{2}\right)^{1/2}e_{2N-2},\\
\frac{1}{\sqrt{N}}\mbox{\ensuremath{\mathbb{E}}}\left\{ H_{N}\left(q_{2}\boldsymbol{\sigma}\left(r\right)\right)\cdot\left(\grad H_{N}\left(q_{1}\hat{\mathbf{n}}\right);\grad H_{N}\left(q_{2}\boldsymbol{\sigma}\left(r\right)\right)\right)\right\}  & =q_{2}\nu'\left(q_{1}q_{2}r\right)\left(1-r^{2}\right)^{1/2}e_{N-1},
\end{align*}
\begin{align*}
 & \mbox{\ensuremath{\mathbb{E}}}\left\{ \ddq H_{N}\left(q_{1}\hat{\mathbf{n}}\right)\cdot\left(\grad H_{N}\left(q_{1}\hat{\mathbf{n}}\right);\grad H_{N}\left(q_{2}\boldsymbol{\sigma}\left(r\right)\right)\right)\right\} 
 =-\left(q_{1}q_{2}r\nu''\left(q_{1}q_{2}r\right)+\nu'\left(q_{1}q_{2}r\right)\right)\left(1-r^{2}\right)^{1/2}e_{2N-2},\\
 & \mbox{\ensuremath{\mathbb{E}}}\left\{ \ddq H_{N}\left(q_{2}\boldsymbol{\sigma}\left(r\right)\right)\cdot\left(\grad H_{N}\left(q_{1}\hat{\mathbf{n}}\right);\grad H_{N}\left(q_{2}\boldsymbol{\sigma}\left(r\right)\right)\right)\right\} 
 =\left(q_{1}q_{2}r\nu''\left(q_{1}q_{2}r\right)+\nu'\left(q_{1}q_{2}r\right)\right)\left(1-r^{2}\right)^{1/2}e_{N-1},
\end{align*}
\begin{align*}
 & \sqrt{N}\mbox{\ensuremath{\mathbb{E}}}\left\{ E_{i}E_{j}H_{N}\left(q_{1}\hat{\mathbf{n}}\right)\cdot\left(\grad H_{N}\left(q_{1}\hat{\mathbf{n}}\right);\grad H_{N}\left(q_{2}\boldsymbol{\sigma}\left(r\right)\right)\right)\right\} \\
 & \quad=\begin{cases}
0 & ,\left|\left\{ i,j,N-1\right\} \right|=3\\
\left(rq_{2}\nu''\left(q_{1}q_{2}r\right)+q_{1}^{-1}\nu'\left(q_{1}q_{2}r\right)\right)\left(1-r^{2}\right)^{1/2}e_{2N-2} & ,i=j\neq N-1\\
q_{2}\nu''\left(q_{1}q_{2}r\right)\left(1-r^{2}\right)^{1/2}e_{N-1+i} & ,i\neq j=N-1\\
q_{2}\nu''\left(q_{1}q_{2}r\right)\left(1-r^{2}\right)^{1/2}e_{N-1+j} & ,j\neq i=N-1\\
\left(-q_{1}q_{2}^{2}\nu'''\left(q_{1}q_{2}r\right)\left(1-r^{2}\right)+3rq_{2}\nu''\left(q_{1}q_{2}r\right)+q_{1}^{-1}\nu'\left(q_{1}q_{2}r\right)\right)\left(1-r^{2}\right)^{1/2}e_{2N-2} & ,i=j=N-1.
\end{cases}\\
 & \sqrt{N}\mbox{\ensuremath{\mathbb{E}}}\left\{ E_{i}E_{j}H_{N}\left(q_{2}\boldsymbol{\sigma}\left(r\right)\right)\cdot\left(\grad H_{N}\left(q_{1}\hat{\mathbf{n}}\right);\grad H_{N}\left(q_{2}\boldsymbol{\sigma}\left(r\right)\right)\right)\right\} \\
 & \quad=\begin{cases}
0 & ,\left|\left\{ i,j,N-1\right\} \right|=3\\
-\left(rq_{1}\nu''\left(q_{1}q_{2}r\right)+q_{2}^{-1}\nu'\left(q_{1}q_{2}r\right)\right)\left(1-r^{2}\right)^{1/2}e_{N-1} & ,i=j\neq N-1\\
-q_{1}\nu''\left(q_{1}q_{2}r\right)\left(1-r^{2}\right)^{1/2}e_{i} & ,i\neq j=N-1\\
-q_{1}\nu''\left(q_{1}q_{2}r\right)\left(1-r^{2}\right)^{1/2}e_{j} & ,j\neq i=N-1\\
-\left(-q_{1}^{2}q_{2}\nu'''\left(q_{1}q_{2}r\right)\left(1-r^{2}\right)+3rq_{1}\nu''\left(q_{1}q_{2}r\right)+q_{2}^{-1}\nu'\left(q_{1}q_{2}r\right)\right)\left(1-r^{2}\right)^{1/2}e_{N-1} & ,i=j=N-1.
\end{cases}
\end{align*}

Denoting by $\mbox{Cov}_{\nabla}\left\{ X,Y\right\} $ the covariance
of two random variables $X$, $Y$ conditional on $\grad H_{N}\left(q_{1}\hat{\mathbf{n}}\right),\,\grad H_{N}\left(q_{2}\boldsymbol{\sigma}\left(r\right)\right)$
(and the covariance with no conditioning by $\mbox{Cov}\left\{ X,Y\right\} $),
we have (cf. \cite[p. 10-11]{RFG})
\begin{align*}
\mbox{Cov}_{\nabla}\left\{ X,Y\right\}  & =\mbox{Cov}\left\{ X,Y\right\} \\
&\;\;-\left(\mbox{\ensuremath{\mathbb{E}}}\left\{ X\cdot\left(\grad H_{N}\left(q_{1}\hat{\mathbf{n}}\right);\grad H_{N}\left(q_{2}\boldsymbol{\sigma}\left(r\right)\right)\right)\right\} \right)^{T}
  \mathbf{A}_{0}\left(r\right)\mbox{\ensuremath{\mathbb{E}}}\left\{ Y\cdot\left(\grad H_{N}\left(q_{1}\hat{\mathbf{n}}\right);\grad H_{N}\left(q_{2}\boldsymbol{\sigma}\left(r\right)\right)\right)\right\} .
\end{align*}

Using this formula and the above by straightforward algebra one finds
that conditional on $\grad H_{N}\left(q_{1}\hat{\mathbf{n}}\right)=\grad H_{N}\left(q_{2}\boldsymbol{\sigma}\left(r\right)\right)=0$,
the random vector (\ref{eq:0702-01}) is a centered Gaussian vector
with covariance matrix $\Sigma_{U,X}\left(r,q_{1},q_{2}\right)$,
completing the proof of Lemma \ref{lem:condHamiltonian}. Moreover,
defining the random matrices
\begin{align*}
\mathbf{K}^{(1)} & =\frac{1}{\sqrt{\nu''(q_{1}^{2})}}\left(\Hess H_{N}\left(q_{1}\hat{\mathbf{n}}\right)+\frac{1}{\sqrt{N}q_{1}}\ddq H_{N}\left(q_{1}\hat{\mathbf{n}}\right)\mathbf{I}\right),\\
\mathbf{K}^{(2)} & =\frac{1}{\sqrt{\nu''(q_{2}^{2})}}\left(\Hess H_{N}\left(q_{2}\boldsymbol{\sigma}\left(r\right)\right)+\frac{1}{\sqrt{N}q_{2}}\ddq H_{N}\left(q_{2}\boldsymbol{\sigma}\left(r\right)\right)\mathbf{I}\right),
\end{align*}
we have that
\begin{equation}
\left(\begin{array}{c}
\mbox{Cov}_{\nabla}\left\{ \mathbf{K}_{ij}^{(1)},H_{N}\left(q_{1}\hat{\mathbf{n}}\right)\right\} \\
\mbox{Cov}_{\nabla}\left\{ \mathbf{K}_{ij}^{(1)},H_{N}\left(q_{2}\boldsymbol{\sigma}\left(r\right)\right)\right\} \\
\mbox{Cov}_{\nabla}\left\{ \mathbf{K}_{ij}^{(1)},\sqrt{N}\ddq H_{N}\left(q_{1}\hat{\mathbf{n}}\right)\right\} \\
\mbox{Cov}_{\nabla}\left\{ \mathbf{K}_{ij}^{(1)},\sqrt{N}\ddq H_{N}\left(q_{2}\boldsymbol{\sigma}\left(r\right)\right)\right\} 
\end{array}\right)\text{\quad and\quad}\left(\begin{array}{c}
\mbox{Cov}_{\nabla}\left\{ \mathbf{K}_{ij}^{(2)},H_{N}\left(q_{1}\hat{\mathbf{n}}\right)\right\} \\
\mbox{Cov}_{\nabla}\left\{ \mathbf{K}_{ij}^{(2)},H_{N}\left(q_{2}\boldsymbol{\sigma}\left(r\right)\right)\right\} \\
\mbox{Cov}_{\nabla}\left\{ \mathbf{K}_{ij}^{(2)},\sqrt{N}\ddq H_{N}\left(q_{1}\hat{\mathbf{n}}\right)\right\} \\
\mbox{Cov}_{\nabla}\left\{ \mathbf{K}_{ij}^{(2)},\sqrt{N}\ddq H_{N}\left(q_{2}\boldsymbol{\sigma}\left(r\right)\right)\right\} 
\end{array}\right)\label{eq:0702-02}
\end{equation}
are equal to $\delta_{i=j=N-1}\left(1-r^{2}\right)\boldsymbol{\varsigma}_{1}\left(r,q_{1},q_{2}\right)$
and $\delta_{i=j=N-1}\left(1-r^{2}\right)\boldsymbol{\varsigma}_{2}\left(r,q_{1},q_{2}\right)$,
respectively, and that
\begin{align}
 & N\mbox{Cov}_{\nabla}\left\{ \mathbf{K}_{ij}^{(1)},\mathbf{K}_{kl}^{(1)}\right\} \label{eq:cov1}\\
 & =\begin{cases}
2-\delta_{i=N-1}\frac{1}{\nu''(q_{1}^{2})}\left(1-r^{2}\right)a_{2}\left(r,q_{2},q_{1}\right)\left(\upsilon_{2}\left(r,q_{1},q_{2}\right)\right)^{2} & ,\,i=j=k=l,\\
1 & ,\,N-1\notin\left\{ i,j\right\} =\left\{ k,l\right\} ,i\neq j,\\
\Sigma_{Z,11}\left(r,q_{1},q_{2}\right) & ,\,N-1\in\left\{ i,j\right\} =\left\{ k,l\right\} ,i\neq j,\\
0 & ,\,\mbox{otherwise},
\end{cases}\nonumber 
\end{align}
\begin{align*}
 & N\mbox{Cov}_{\nabla}\left\{ \mathbf{K}_{ij}^{(2)},\mathbf{K}_{kl}^{(2)}\right\} \\
 & =\begin{cases}
2-\delta_{i=N-1}\frac{1}{\nu''(q_{2}^{2})}\left(1-r^{2}\right)a_{2}\left(r,q_{1},q_{2}\right)\left(\upsilon_{2}\left(r,q_{2},q_{1}\right)\right)^{2} & ,\,i=j=k=l,\\
1 & ,\,N-1\notin\left\{ i,j\right\} =\left\{ k,l\right\} ,i\neq j,\\
\Sigma_{Z,22}\left(r,q_{1},q_{2}\right) & ,\,N-1\in\left\{ i,j\right\} =\left\{ k,l\right\} ,i\neq j,\\
0 & ,\,\mbox{otherwise},
\end{cases}
\end{align*}
\begin{eqnarray*}
  N\mbox{Cov}_{\nabla}\left\{ \mathbf{K}_{ii}^{(1)},\mathbf{K}_{jj}^{(2)}\right\}
 & =&\begin{cases}
0 & ,\,i\neq j\\
2\frac{\nu''\left(q_{1}q_{2}r\right)}{\sqrt{\nu''(q_{1}^{2})\nu''(q_{2}^{2})}} & ,\,i=j\neq N-1\\
\frac{1}{\sqrt{\nu''(q_{1}^{2})\nu''(q_{2}^{2})}}\left[q_{1}q_{2}\nu''''\left(q_{1}q_{2}r\right)\left(1-r^{2}\right)^{2}\right.\\
-2q_{1}q_{2}\nu'''\left(q_{1}q_{2}r\right)\left(1-r^{2}\right)r+2r^{2}\nu''\left(q_{1}q_{2}r\right)\\
\left.+\upsilon_{2}\left(r,q_{1},q_{2}\right)\upsilon_{2}\left(r,q_{2},q_{1}\right)a_{4}\left(r,q_{1},q_{2}\right)\left(1-r^{2}\right)\right] & ,\,i=j=N-1,
\end{cases}\\
  N\mbox{Cov}_{\nabla}\left\{ \mathbf{K}_{ij}^{(1)},\mathbf{K}_{ij}^{(2)}\right\} 
  & =&\begin{cases}
\frac{\nu''\left(q_{1}q_{2}r\right)}{\sqrt{\nu''(q_{1}^{2})\nu''(q_{2}^{2})}} & ,\,\left|\left\{ i,j,N-1\right\} \right|=3\\
\Sigma_{Z,12}\left(r,q_{1},q_{2}\right) & ,\,\left|\left\{ i,j,N-1\right\} \right|=2,\,i\neq j,
\end{cases}\\
N\mbox{Cov}_{\nabla}\left\{ \mathbf{K}_{ij}^{(1)},\mathbf{K}_{kl}^{(2)}\right\} &=&
0,\mbox{ \,\,\ if }\exists m:\,\sum_{t\in\left\{ i,j,k,l\right\} }\mathbf{1}\left\{ t=m\right\} =1.
\end{eqnarray*}
Let $\mbox{Cov}_{H_{N},\ddq,\nabla}\left\{ X,Y\right\} $ and $\mathbb{E}_{H_{N},\ddq,\nabla}\left\{ X\right\} $
denote the covariance of two random variables $X$, $Y$ and expectation
of $X$, respectively, conditional on 
\begin{equation}
\grad H_{N}\left(\bs\right)=0,\,H_{N}\left(\bs\right),\,\ddq H_{N}\left(\boldsymbol{\sigma}\right),\quad\bs\in\{q_{1}\hat{\mathbf{n}},\,q_{2}\bs(r)\}.\label{eq:cond}
\end{equation}

Since the covariances (\ref{eq:0702-02}) are non-zero only if $i=j=N-1$,
unless $i=j=k=l=N-1$, for any $\kappa,\,\kappa'\in\{1,2\}$,
\[
\mathbb{E}_{H_{N},\ddq,\nabla}\left\{ \mathbf{K}_{ij}^{(\kappa)}\right\} =0,\quad\mbox{Cov}_{H_{N},\ddq,\nabla}\left\{ \mathbf{K}_{ij}^{(\kappa)},\mathbf{K}_{kl}^{(\kappa')}\right\} =\mbox{Cov}_{\nabla}\left\{ \mathbf{K}_{ij}^{(\kappa)},\mathbf{K}_{kl}^{(\kappa')}\right\} .
\]
Lastly, for $\kappa,\,\kappa'\in\{1,2\}$,
\begin{align*}
 & N\mbox{Cov}_{H_{N},\ddq,\nabla}\left\{ \mathbf{K}_{N-1,N-1}^{(\kappa)},\mathbf{K}_{N-1,N-1}^{(\kappa')}\right\} =\Sigma_{Q,\kappa\kappa'}\left(r,q_{1},q_{2}\right)=\\
 & \quad N\mbox{Cov}_{\nabla}\left\{ \mathbf{K}_{N-1,N-1}^{(\kappa)},\mathbf{K}_{N-1,N-1}^{(\kappa')}\right\} -\left(1-r^{2}\right)^{2}\left(\boldsymbol{\varsigma}_{\kappa}\left(r,q_{1},q_{2}\right)\right)^{T}\Sigma_{U,X}^{-1}\left(r,q_{1},q_{2}\right)\boldsymbol{\varsigma}_{\kappa'}\left(r,q_{1},q_{2}\right),\\
 & \mathbb{E}_{H_{N},\ddq,\nabla}
 \left\{ \mathbf{K}_{N-1,N-1}^{(\kappa)}\right\}\\
 &\quad=\left(1-r^{2}\right)
 \frac{1}{N}\left(H_{N}\left(q_{1}\hat{\mathbf{n}}\right),H_{N}\left(q_{2}\boldsymbol{\sigma}\left(r\right)\right),\sqrt{N}\ddq H_{N}\left(q_{1}\hat{\mathbf{n}}\right),\sqrt{N}\ddq H_{N}\left(q_{2}\boldsymbol{\sigma}\left(r\right)\right)\right)\Sigma_{U,X}^{-1}\left(r,q_{1},q_{2}\right)\boldsymbol{\varsigma}_{\kappa}\left(r,q_{1},q_{2}\right).
\end{align*}
\end{proof}
\section{Auxiliary results}

\subsection*{The Kac-Rice formula}

Our analysis uses the two auxiliary lemmas below, based on a variant
of the Kac-Rice formula \cite[Theorem 12.1.1]{RFG}. In the notation
of \cite[Theorem 12.1.1]{RFG} we are interested in situations where,
with some random function $g_{N}\left(\boldsymbol{\sigma}\right)$,
\begin{align*}
 & M=\mathbb{S}^{N-1}(\sqrt{N}q),\,\,\,f\left(\boldsymbol{\sigma}\right)=\grad H_{N}\left(\boldsymbol{\sigma}\right),\,\,\,u=0\in\mathbb{R}^{N-1},\\
 & h\left(\boldsymbol{\sigma}\right)=\left(H_{N}\left(\boldsymbol{\sigma}\right),\,\ddq H_{N}\left(\bs\right),\,g_{N}\left(\boldsymbol{\sigma}\right)\right).
\end{align*}
The application of \cite[Theorem 12.1.1]{RFG} requires $h\left(\boldsymbol{\sigma}\right)$
to satisfy certain non-degeneracy conditions - namely, conditions
(a)-(g) in \cite[Theorem 12.1.1]{RFG}. We will say that $g_{N}\left(\boldsymbol{\sigma}\right)$
is tame if the conditions are satisfied and if $\{h\left(\boldsymbol{\sigma}\right)\}_{\boldsymbol{\sigma}}$
is a stationary random field. From (\ref{eq:Hamiltonian}), the conditions
are easy to check in any case we will apply the formula and this will
be left to the reader.
\begin{lem}
\label{lem:15}Let $B$,$D$ and $L$ be some intervals and let $\varphi:\mathbb{R}^{2}\to\mathbb{R}$
be a continuous function. Assume that $H_{N}(\bs)$ is a mixed model
such that
\begin{equation}
\limsup_{N\to\infty}\sup_{\substack{u\in B\\
x\in D
}
}\left\{ \frac{1}{N}\log\left(\mathbb{P}_{Nu,\sqrt{N}x}^{q}\left\{ g_{N}\left(\hat{\mathbf{n}}\right)\in NL\right\} \right)-\varphi\left(u,x\right)\right\} \leq0.\label{eq:1224-1}
\end{equation}
Then 
\begin{equation}
\limsup_{N\to\infty}\frac{1}{N}\log\left(\mathbb{E}\left|\left\{ \boldsymbol{\sigma}_{0}\in\mathscr{C}_{N,q}(NB,\sqrt{N}D):\,g_{N}\left(\boldsymbol{\sigma}_{0}\right)\in NL\right\} \right|\right)\leq\sup_{\substack{u\in B\\
x\in D
}
}\left\{ \Theta_{\nu,q}\left(u,x\right)+\varphi\left(u,x\right)\right\} .\label{eq:71}
\end{equation}
\end{lem}

\begin{lem}
\label{lem:17}Let $B$ and $D$ be some intervals and $\varphi:\mathbb{R}^{2}\to\mathbb{R}$
be a continuous function. Suppose that $H_{N}(\bs)$ is a mixed model
such that, with $Z_{N,\beta}(\bs_{0})$ defined by (\ref{eq:Zsigma}),
\begin{equation}
\limsup_{N\to\infty}\sup_{\substack{u\in B\\
x\in D
}
}\left\{ \frac{1}{N}\log\left(\mathbb{E}_{Nu,\sqrt{N}x}^{q}\left\{ Z_{N,\beta}(q\hat{\mathbf{n}})\right\} \right)-\varphi\left(u,x\right)\right\} \leq0.\label{eq:81}
\end{equation}
Then 
\begin{equation}
\limsup_{N\to\infty}\frac{1}{N}\log\left(\mathbb{E}\left\{ \sum_{\boldsymbol{\sigma}_{0}\in\mathscr{C}_{N,q}(NB,\sqrt{N}D)}Z_{N,\beta}(\bs_{0})\right\} \right)\leq\sup_{\substack{u\in B\\
x\in D
}
}\left\{ \Theta_{\nu,q}\left(u,x\right)+\varphi\left(u,x\right)\right\} .\label{eq:71-1}
\end{equation}
\end{lem}

Lemmas \ref{lem:15} and \ref{lem:17} follow from a similar argument
to the proof of \cite[Lemmas 14, 16]{geometryGibbs} which dealt with
the pure case, with one exception. After applying the Kac-Rice formula,
instead of integrating only over the variable $H_{N}\left(\bs\right)$
one has to integrate over both $H_{N}\left(\bs\right)$ and $\ddq H_{N}\left(\bs\right)$,
exactly as we have done in the proof of Theorem \ref{thm:1stmoment}. 

\subsection*{Lipschitz estimates}

The following lemma is based on the proof of \corO{\cite[Lemma 2.2]{BDG1}.}
\begin{lem}
\label{lem:Lipschitz}(Lipschitz continuity of derivatives) For any
$p\geq2$, the pure $p$-spin model $H_{N,p}(\bx)$ satisfies 
\begin{equation}
  \mathbb{P}\left\{ \exists1\leq k\leq p-1,\,\bx,\bx'\in\mathcal{B}_{N}:\,\vphantom{\geq2Kp^{1/2}\frac{p!}{(p-k-1)!}\frac{\|\bx-\bx'\|}{N^{k/2}}}\right.
  \left.\frac{1}{k!}\|\nabla_{E}^{k}H_{N,p}(\bx)-\nabla_{E}^{k}H_{N,p}(\bx')\|_{\infty}\geq2Kp^{3/2}\binom{p-1}{k}\frac{\|\bx-\bx'\|}{N^{k/2}}\right\} \leq e^{-K^{2}pN/2},
\label{eq:0407-01}
\end{equation}
and 
\begin{equation}
  \mathbb{P}\left\{ \exists1\leq k\leq p:\,\vphantom{\geq2Kp^{1/2}\frac{p!}{(p-k-1)!}\frac{\|\bx-\bx'\|}{N^{k/2}}}\right.
  \left.\sup_{\bx\in\mathcal{B}_{N}}\frac{1}{k!}\|\nabla_{E}^{k}H_{N,p}(\bx)\|_{\infty}\geq2Kp^{1/2}\binom{p}{k}N^{-\left(k-1\right)/2}\right\} \leq e^{-K^{2}pN/2},
\label{eq:0407-01-2}
\end{equation}
where $\mathcal{B}_{N}=\{\bx\in\mathbb{R}^{N}:\,\|\bx\|\leq\sqrt{N}\}$,
$K>0$ is a universal constant. For $k=p$, $\nabla_{E}^{k}H_{N,p}(\bx)$
is constant in $\bx\in\mathcal{B}_{N}$, and for $k>p$, $\nabla_{E}^{k}H_{N,p}(\bx)=0$,
almost surely.
\end{lem}

\begin{proof}
Throughout the proof we will use $\{i_{1},...,i_{p}\}$ to denote
multisets of indices. For any such multiset $A$, define $J_{A}^{(p)}$
as the average of all $J_{i_{1},...,i_{p}}^{(p)}$ with $\{i_{1},...,i_{p}\}=A$.
We can write the pure $p$-spin Hamiltonian and its derivatives as
\begin{equation}
\begin{aligned}H_{N,p}(\bx) & =N^{-\left(p-1\right)/2}\sum_{i_{1},...,i_{p}=1}^{N}J_{\{i_{1},...,i_{p}\}}^{(p)}x_{i_{1}}\cdots x_{i_{p}},\\
\frac{d}{dx_{j_{1}}}\cdots\frac{d}{dx_{j_{k}}}H_{N,p}(\bx) & =N^{-\left(p-1\right)/2}\frac{p!}{(p-k)!}\sum_{i_{1},...,i_{p-k}=1}^{N}J_{\{i_{1},...,i_{p-k},j_{1},...,j_{k}\}}^{(p)}x_{i_{1}}\cdots x_{i_{p-k}},
\end{aligned}
\label{eq:1507-06}
\end{equation}
where $k\leq p$, as otherwise the derivatives are $0$. For any points
$\bx$, $\bx'\in\mathcal{B}_{N}$ and $\mathbf{y}^{(1)},...,\mathbf{y}^{(k)}$,
$k-1\leq p$, with $\|\mathbf{y}^{(j)}\|$, 
\begin{equation}
\begin{aligned} & \frac{1}{\sqrt{N}}\sum_{j_{1},...,j_{k}=1}^{N}(\frac{d}{dx_{j_{1}}}\cdots\frac{d}{dx_{j_{k}}}H_{N,p}(\bx)-\frac{d}{dx_{j_{1}}}\cdots\frac{d}{dx_{j_{k}}}H_{N,p}(\bx'))y_{j_{1}}^{(1)}\cdots y_{j_{k}}^{(k)}\\
 & =N^{-p/2}\frac{p!}{(p-k)!}\sum_{j_{1},...,j_{k}=1}^{N}\sum_{i_{1},...,i_{p-k}=1}^{N}J_{\{i_{1},...,i_{p-k},j_{1},...,j_{k}\}}^{(p)}\left(x_{i_{1}}\cdots x_{i_{p-k}}-x_{i_{1}}^{\prime}\cdots x_{i_{p-k}}^{\prime}\right)y_{j_{1}}^{(1)}\cdots y_{j_{k}}^{(k)}\\
 & =\sum_{l=1}^{p-k}N^{-p/2}\frac{p!}{(p-k)!}\sum_{j_{1},...,j_{k}=1}^{N}\sum_{i_{1},...,i_{p-k}=1}^{N}J_{\{i_{1},...,i_{p-k},j_{1},...,j_{k}\}}^{(p)}x_{i_{1}}\cdots x_{i_{l-1}}(x_{i_{l}}-x_{i_{l}}^{\prime})x_{i_{l+1}}^{\prime}\cdots x_{i_{p-k}}^{\prime}y_{j_{1}}^{(1)}\cdots y_{j_{k}}^{(k)}\\
 & \leq N^{-(p-1)/2}\frac{p!}{(p-k-1)!}\|\mathbf{J}_{N,p}\|_{\infty}(\max\{\|\bx\|,\|\bx'\|\})^{p-k-1}\|\bx-\bx'\|,
\end{aligned}
\label{eq:1707-01}
\end{equation}
where the second equality follows by writing $x_{i_{1}}\cdots x_{i_{p-1}}-x_{i_{1}}^{\prime}\cdots x_{i_{p-1}}^{\prime}$
as a telescopic sum and
\[
\mathbf{J}_{N,p}:=\left(J_{\{i_{1},...,i_{p}\}}^{(p)}\right)_{i_{1},...,i_{p}\leq N}.
\]
Note that for $k=p$, the derivative (\ref{eq:1507-06}) is independent
of $\bx$, and the difference of derivatives (\ref{eq:1707-01}) is
$0$.

\corO{By \cite[Theorem 2.1]{ZLQ}, if $\mathbf T$ is a symmetric tensor 
then the supremum in \eqref{eq:inftynorm-1} is obtained with 
$\mathbf y^{(1)}=\cdots = \mathbf y^{(p)}$. Hence, since 
$\mathbf J_{N,p}$ is symmetric by definition,}
\[
\|\mathbf{J}_{N,p}\|_{\infty}=\frac{1}{N}\sup_{\|\mathbf{y}\|=\sqrt{N}}\left|H_{N,p}(\mathbf{y})\right|.
\]

Using Dudley's entropy bound \cite[Theorem 1.3.3]{RFG} it is standard
to show that
\begin{equation}
\mathbb{E}\|\mathbf{J}_{N,p}\|_{\infty}\leq K\sqrt{p},\label{eq:2809-01}
\end{equation}
where $K>0$ is a universal constant. This is a consequence, for example,
of Lemma 19 of \cite{pspinext} and a bound on the canonical metric
of $H_{N,p}(\bs)$ which can be proved similarly to Lemma 20 of \cite{pspinext}.\footnote{There, a bound is computed for the canonical (pseudo) metric of a
modified Hamiltonian $\tilde{H}_{N,p}(T(x))$, where $T(x_{1},...,x_{N-1})=\sqrt{N}(x_{1},...,x_{N-1},\sqrt{1-\sum x_{i}^{2}})$.
Here a similar, but simpler, bound is needed for the metric $d(x,y)=\left(\mathbb{E}(H_{N,p}(T(x))-H_{N,p}(T(y)))^{2}\right)^{1/2}$.
Following the general argument of Lemma 20 of \cite{pspinext} and
using the estimate the Taylor expansion between (7.26) and (7.27)
of the same paper, it is straightforward to show that $d^{2}(x,y)\leq NpC\|x-y\|$
for some constant $C$ independent of $p$.} The event in (\ref{eq:0407-01}) is contained in $\left\{ \|\mathbf{J}_{N,p}\|_{\infty}\geq2K\sqrt{p}\right\} $.
Thus, by the Borell-TIS inequality \cite{Borell,TIS} (see also \cite[Theorem 2.1.1]{RFG}),
its probability is bounded from above by
\begin{equation}
\mathbb{P}\left\{ \|\mathbf{J}_{N,p}\|_{\infty}\geq2K\sqrt{p}\right\} \leq e^{-K^{2}pN/2},\label{eq:0808-02}
\end{equation}
as required.

To prove (\ref{eq:0407-01-2}) note that similarly to (\ref{eq:1707-01}),
for $k\leq p$,
\[
\begin{aligned} & \frac{1}{\sqrt{N}}\sum_{j_{1},...,j_{k}=1}^{N}\frac{d}{dx_{j_{1}}}\cdots\frac{d}{dx_{j_{k}}}H_{N,p}(\bx)y_{j_{1}}^{(1)}\cdots y_{j_{k}}^{(k)}\\
 & =N^{-p/2}\frac{p!}{(p-k)!}\sum_{j_{1},...,j_{k}=1}^{N}\sum_{i_{1},...,i_{p-k}=1}^{N}J_{\{i_{1},...,i_{p-k},j_{1},...,j_{k}\}}^{(p)}x_{i_{1}}\cdots x_{i_{p-k}}y_{j_{1}}^{(1)}\cdots y_{j_{k}}^{(k)}\\
 & \leq N^{-(p-1)/2}\frac{p!}{(p-k)!}\|\mathbf{J}_{N,p}\|_{\infty}\|\bx\|^{p-k},
\end{aligned}
and
\]
and proceed as before.
\end{proof}
\begin{cor}
\label{cor:gradbd}For any mixed model $H_{N}\left(\bx\right)$, assuming
that $\varlimsup p^{-1}\log\gamma_{p}<0$, there exist a constant
$c>0$ such that, with $C_{k}=2K\sum_{p\geq k+1}\gamma_{p}p^{3/2}\binom{p-1}{k}$
and $\tilde{C}_{k}=2K\sum_{p\geq k}\gamma_{p}p^{1/2}\binom{p}{k}$
where $K$ is as in Lemma \ref{lem:Lipschitz},
\begin{equation}
\mathbb{P}\left\{ \exists k\geq1,\,\bx,\bx'\in\mathcal{B}_{N}:\,\frac{1}{k!}\|\nabla_{E}^{k}H_{N}(\bx)-\nabla_{E}^{k}H_{N}(\bx')\|_{\infty}\geq C_{k}\frac{\|\bx-\bx'\|}{N^{k/2}}\right\} \leq e^{-cN},\label{eq:1507-01}
\end{equation}
and
\begin{equation}
\mathbb{P}\left\{ \exists k\geq1:\,\frac{1}{k!}\sup_{\bx\in\mathcal{B}_{N}}\|\nabla_{E}^{k}H_{N}(\bx)\|_{\infty}\geq N^{-(k-1)/2}\tilde{C}_{k}\right\} \leq 
e^{-cN}.\label{eq:1507-03}
\end{equation}
\end{cor}

\begin{proof}
The bound of (\ref{eq:1507-01}) (respectively, (\ref{eq:1507-03}))
follows by a union bound from (\ref{eq:0407-01}) (\ref{eq:0407-01-2}),
since $\|\nabla_{E}^{k}H_{N}(\bx)\|_{\infty}\leq\sum_{p\geq2}\gamma_{p}\|\nabla_{E}^{k}H_{N,p}(\bx)\|_{\infty}$
and, for large $N$, $\sum_{p\geq2}e^{-K^{2}pN/2}\leq e^{-K^{2}N/2}=:e^{-cN}$. 
\end{proof}
\begin{lem}
\label{lem:LipF}For any finite mixture $\nu(x)=\sum_{p=2}^{p_{0}}\gamma_{p}^{2}x^{p}$,
$F_{N,\beta}=\frac{1}{N}\log Z_{N,\beta}$ is a Lipschitz function
of the Gaussian disorder coefficients $J_{i_{1},...,i_{p}}$ with
Lipschitz constant $\beta\sqrt{\frac{\nu(1)}{N}}$.
\end{lem}

\begin{proof}
Write
\[
H_{N}(\mathbf{J},\bs)=\sum_{p=2}^{p_{0}}\gamma_{p}N^{-\frac{p-1}{2}}\sum_{i_{1},...,i_{p}=1}^{N}J_{i_{1},...,i_{p}}\sigma_{i_{1}}\cdots\sigma_{i_{p}},
\]
where $\mathbf{J}=(J_{i_{1},...,i_{p}})$ is the array of all the
disorder coefficients. For any $i_{1},...,i_{k}$, set
\begin{align*}
D_{i_{1},...,i_{k}} & :=\frac{d}{dx_{i_{1},...,i_{p}}}\log\left(\int_{\mathbb{S}^{N-1}(\sqrt{N})}\exp\left\{ -\beta H_{N}(\mathbf{J},\bs)\right\} d\boldsymbol{\sigma}\right)\\
 & =-\beta\gamma_{p}N^{-\frac{p-1}{2}}\cdot\frac{\int_{\mathbb{S}^{N-1}(\sqrt{N})}\sigma_{i_{1}}\cdots\sigma_{i_{p}}\exp\left\{ -\beta H_{N}(\mathbf{J},\bs)\right\} d\boldsymbol{\sigma}}{\int_{\mathbb{S}^{N-1}(\sqrt{N})}\exp\left\{ -\beta H_{N}(\mathbf{J},\bs)\right\} d\boldsymbol{\sigma}}.
\end{align*}
The ratio of integrals in the last equation can be viewed as an expectation
under the Gibbs measure. Denote expectation by this measure by $\left\langle \,\cdot\,\right\rangle $,
so that the ratio is simply $\langle\sigma_{i_{1}}\cdots\sigma_{i_{p}}\rangle$.
We then have 
\begin{equation}
\sum_{p=2}^{p_{0}}\sum_{i_{1},...,i_{p}=1}^{N}\left(D_{i_{1},...,i_{p}}\right)^{2} 
=\sum_{p=2}^{p_{0}}\beta^{2}\gamma_{p}^{2}N^{-(p-1)}\sum_{i_{1},...,i_{p}=1}^{N}\langle\sigma_{i_{1}}\cdots\sigma_{i_{p}}\rangle^{2}
  \leq\sum_{p=2}^{p_{0}}\beta^{2}\gamma_{p}^{2}N^{-(p-1)}\sum_{i_{1},...,i_{p}=1}^{N}\langle(\sigma_{i_{1}}\cdots\sigma_{i_{p}})^{2}\rangle.
\label{eq:0304-1}
\end{equation}
Note that
\[
\sum_{i_{1},...,i_{p}=1}^{N}\langle(\sigma_{i_{1}}\cdots\sigma_{i_{p}})^{2}\rangle=\left\langle \left\Vert \bs\right\Vert _{2}^{2p}\right\rangle =N^{p}.
\]
Therefore, $\frac{1}{N}\log Z_{N,\beta}$ has Lipschitz constant bounded
from above by
\[
\frac{1}{N}\left(\sum_{p=2}^\infty\beta^{2}\gamma_{p}^{2}N\right)^{1/2}=\beta\sqrt{\frac{\nu(1)}{N}}.
\]
\end{proof}
\begin{cor}
\label{cor:concF}For any mixture $\nu(x)=\sum_{p=2}^{\infty}\gamma_{p}^{2}x^{p}$
such that $\varlimsup p^{-1}\log\gamma_{p}<0$ and any $t>0$, the
free energy $F_{N,\beta}=\frac{1}{N}\log Z_{N,\beta}$ satisfies
\begin{equation}
\mathbb{P}\left\{ \Big|F_{N,\beta}-\mathbb{E}F_{N,\beta}\Big|>t\right\} \leq3\exp\left\{ -Nt^{2}/2\beta^{2}\nu(1)\right\} .\label{eq:concF}
\end{equation}
\end{cor}

\begin{proof}
For finite mixtures, i.e., such that $\gamma_{p}=0$ for all $p\geq p_{0}$,
the corollary follows from Lemma \ref{lem:LipF} and standard concentration
results (see e.g. \cite[Lemma 2.3.3]{Matrices}), with prefactor $2$ instead
of $3$. For the infinite case we will truncate the mixture.

Let $F_{N,\beta}^{(p)}$ be the partition function corresponding to
the truncated Hamiltonian 
\[
H_{N,\beta}^{(p)}(\bs)=\sum_{k=2}^{p}\gamma_{k}N^{-\frac{k-1}{2}}\sum_{i_{1},...,i_{k}=1}^{N}J_{i_{1},...,i_{k}}\sigma_{i_{1}}\cdots\sigma_{i_{k}},
\]
which we assume to be defined using the same disorder variables $J_{i_{1},...,i_{k}}$
as $H_{N}(\bs)$, defined by corresponding infinite sum. For any $p\geq2$
and $\epsilon>0$, the left-hand
side of (\ref{eq:concF}) is bounded from above by
\begin{equation}
  \mathbb{P}\left\{ \Big|F_{N,\beta}^{(p)}-\mathbb{E}F_{N,\beta}^{(p)}\Big|>(t-2\epsilon)\right\} 
  +\mathbb{P}\left\{ \Big|F_{N,\beta}-F_{N,\beta}^{(p)}\Big|>\epsilon\right\} +\mathbb{P}\left\{ \Big|\mathbb{E}F_{N,\beta}^{(p)}-\mathbb{E}F_{N,\beta}\Big|>\epsilon\right\} .
\label{eq:2809-02}
\end{equation}
For any $\delta$, for large enough $p$, from (\ref{eq:2809-01}),
\[
\mathbb{E}\sup_{\bs}|H_{N,\beta}^{(p)}(\bs)-H_{N,\beta}(\bs)|<N\delta.
\]
Hence, for fixed $\epsilon$, for large enough $p$, the last two
summands above are smaller than the right-hand side of (\ref{eq:concF}),
with prefactor $1$. Thus, from the finite case we have that 
\[
\mathbb{P}\left\{ \Big|F_{N,\beta}-\mathbb{E}F_{N,\beta}>t\right\} \leq3\exp\left\{ -N(t-2\epsilon)^{2}/2\beta^{2}\nu(1)\right\} .
\]
By taking $\epsilon\to0$, we obtain (\ref{eq:concF}).
\end{proof}
\bibliographystyle{plain}
\bibliography{master}

\end{document}